\documentclass[10pt]{amsart}
\usepackage{mathtools}
\usepackage{mathabx}

\usepackage{amssymb}
\usepackage{amsmath}
\usepackage{amsfonts}
\usepackage{amsthm}
\usepackage{stmaryrd}
\usepackage[all]{xy}
\usepackage{mathrsfs}
\usepackage{graphicx}
\usepackage{hyperref}
\usepackage{color}
\usepackage{tikz-cd}
\usepackage{tikz}
\usetikzlibrary{matrix}
\usepackage{amsmath,amscd}
\usepackage{pifont}
\usetikzlibrary{arrows}
\usepackage[all]{xy}
\usepackage{enumerate}

\usepackage{CJK}
\usepackage{cite}

\usepackage{graphicx}
\usepackage{subfigure}

\newcommand{\Matrix}[4]{\big( \begin{smallmatrix}
		#1&#2\\#3&#4
	\end{smallmatrix}\big)}

\newtheorem{thm}{Theorem}[subsection]
\newtheorem{theorem}[thm]{Theorem}
\newtheorem{conjecture}[thm]{Conjecture}

\newtheorem{corollary}[thm]{Corollary}
\newtheorem{lemma}[thm]{Lemma}
\newtheorem{proposition}[thm]{Proposition}

\theoremstyle{definition}

\newtheorem*{claim*}{Claim}

\newtheorem{convention}[thm]{Convention}
\newtheorem{definition}[thm]{Definition}

\newtheorem{notation}[thm]{Notation}

\newtheorem{remark}[thm]{Remark}

\numberwithin{equation}{section}

\numberwithin{mytheorem}{subsection}

\numberwithin{mytheorem}{subsection}
\numberwithin{myconjecture}{subsection}
\numberwithin{mydefinition}{subsection}
\numberwithin{myremark}{subsection}
\numberwithin{mysituation}{subsection}
\numberwithin{myhypothesis}{subsection}
\numberwithin{myquestion}{subsection}
\numberwithin{mynotation}{subsection}
\numberwithin{myfact}{subsection}
\numberwithin{myexamples}{subsection}
\numberwithin{myexample}{subsection}
\numberwithin{myconstruction}{subsection}
\numberwithin{mycaution}{subsection}
\numberwithin{myproposition}{subsection}
\numberwithin{mylemma}{subsection}
\numberwithin{mycorollary}{subsection}

\def\bbB{\mathbf{B}}
\def\bbC{\mathbf{C}}

\def\bbM{\mathbf{M}}

\def\bbT{\mathbf{T}}

\def\AAA{\mathbb{A}}

\def\CC{\mathbb{C}}

\def\FF{\mathbb{F}}
\def\GG{\mathbb{G}}

\def\NN{\mathbb{N}}

\def\QQ{\mathbb{Q}}
\def\RR{\mathbb{R}}

\def\ZZ{\mathbb{Z}}

\def\calA{\mathcal{A}}

\def\calC{\mathcal{C}}
\def\calD{\mathcal{D}}

\def\calI{\mathcal{I}}

\def\calM{\mathcal{M}}
\def\calN{\mathcal{N}}
\def\calO{\mathcal{O}}

\def\calU{\mathcal{U}}

\def\calW{\mathcal{W}}
\def\calX{\mathcal{X}}

\def\calZ{\mathcal{Z}}

\def\gothd{\mathfrak{d}}

\def\gothl{\mathfrak{l}}
\def\gothm{\mathfrak{m}}
\def\gothn{\mathfrak{n}}

\def\gothp{\mathfrak{p}}

\def\gothD{\mathfrak{D}}

\def\gothR{\mathfrak{R}}

\def\rmH{\mathrm{H}}
\def\rmI{\mathrm{I}}

\def\rmM{\mathrm{M}}

\newcommand{\AL}{\mathrm{AL}}

\def\Ind{\mathrm{Ind}}
\def\Iw{\mathrm{Iw}}

\def\wt{\mathrm{wt}}
\def\Spc{\mathrm{Spc}}

\def\Fp{\FF_p}

\def\wt{\mathrm{wt}}

\def\Nm{\textrm{Nm}}
\def\ula{\underline{\lambda}}

\def\us{\underline{s}}
\def\ualpha{\underline{\alpha}}

\DeclareMathOperator{\End}{End}

\DeclareMathOperator{\GL}{GL}

\DeclareMathOperator{\Hom}{Hom}

\DeclareMathOperator{\Res}{Res}

\begin{document}\large
	
\begin{CJK}{UTF8}{gkai}	
	
	\title{Spectral halo for Hilbert modular forms}

	\author{Rufei Ren}
\address{Department of Mathematics, Fudan University\\220 Handan Rd., Yangpu District, Shanghai 200433, China.}
\email{rufeir@fudan.edu.cn}
	\author{Bin Zhao}
	\address{Bin Zhao, Morningside Center of Mathematics, Academy of Mathematics and Systems Science, Chinese Academy of Sciences, University of the Chinese Academy of Sciences
		Beijing, 100190, China.}
	\email{bin.zhao@amss.ac.cn}
	\date{\today}
	\begin{abstract}
		Let $F$ be a totally real field and $p$ be an odd prime which splits completely in $F$. We prove that the eigenvariety associated to a definite quaternion algebra over $F$ satisfies the following property: over a boundary annulus of the weight space, the eigenvariety is a disjoint union of countably infinitely many connected components which are finite over the weight space; on each fixed connected component, the ratios between the $U_\gothp$-slopes of points and the $p$-adic valuations of the $\gothp$-parameters are bounded by explicit numbers, for all primes $\gothp$ of $F$ over $p$. Applying Hansen's $p$-adic interpolation theorem, we are able to transfer our results to Hilbert modular eigenvarieties. In particular, we prove that on every irreducible component of Hilbert modular eigenvarieties, as a point moves towards the boundary, its $U_p$ slope goes to zero. In the case of eigencurves, this completes the proof of Coleman-Mazur's `halo' conjecture.
	\end{abstract}

	\maketitle
	
	\setcounter{tocdepth}{1}
	\tableofcontents

\section{Introduction}
	
The theory of $p$-adic analytic families of modular forms grew out of the study of congruences of modular forms. In \cite{serre1973}, Serre gave the first example of a $p$-adic family of modular forms-an Eisenstein family. After more than one decade, Hida made the first major step towards the construction of $p$-adic families of cuspforms. In \cite{hida1986b} and \cite{hida1986a}, he constructed $p$-adic families of ordinary modular forms. His theory led Mazur to develop his general theory of deformations of Galois representations in \cite{mazur1989}, which turns out to be a crucial ingredient of Wiles' proof of Fermat's last theorem. The theory of $p$-adic families of modular forms reached a culmination in Coleman and Mazur's celebrated work \cite{coleman-mazur} in which they constructed eigencurves. An eigencurve is a rigid analytic curve whose points correspond to finite slopes normalized overconvergent eigenforms. When pursuing a description of the geometry of the eigencurves, Coleman and Mazur raised the following question based on detailed computations for small primes: does the slope always tend to $0$ as one approaches the `boundary' of the weight space? This problem has been made into a much more delicate conjecture based on the computation of Buzzard and Kilford in \cite{buzzard20052}. It is expected that over the boundary of the weight space, the eigencurve is a disjoint union of infinitely many connected components that are finite flat over the weight space. We refer to \cite[Conjecture~1.2]{liu2017eigencurve} for the precise statement of the conjecture. 

\begin{figure}[h]
	\centering
	\includegraphics[width=4cm,height=4cm]{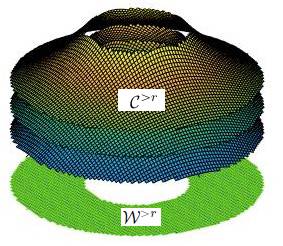}
	\caption{Spectral halo of eigencurves}
\end{figure}
	
For eigencurves associated to a definite quaternion algebra $D$ over $\QQ$, this conjecture was studied by R. Liu, D. Wan, and L. Xiao in \cite{liu2017eigencurve}. To be more precise, we fix an odd prime $p$ and denote by $\Spc_D$ the spectral curve associated to the overconvergent automorphic forms on $D/\QQ$ (with some tame level). It admits a weight map $\wt:\Spc_D\rightarrow \calW$ to the weight space $\calW$, where $\calW$ is the rigid analytification of the Iwasawa algebra $\ZZ_p\llbracket \ZZ_p^\times\rrbracket$ and a slope map $a_p:\Spc_D\rightarrow \GG_m$. For $r\in (0,1)$, we use $\calW^{>r}$ to denote the subspace of $\calW$ where the fixed parameter $T:= [\exp(p)]-1$ satisfies $|T|_p\in (r,1)$ and let $\Spc_D^{>r}=\wt^{-1}(\calW^{>r})$. Under the above notations, Liu-Wan-Xiao proved the following theorem.
\begin{theorem}
	The space $\Spc_D^{>1/p}$ can be  decomposed  into a disjoint union \[X_0\coprod X_{(0,1)}\coprod X_1\coprod X_{(1,2)}\coprod X_2\coprod\cdots\]
	of rigid analytic spaces which are finite and flat over $\calW^{>1/p}$ via $\wt$ such that for each point $x \in X_I$ with $I$ denoting the interval $n=[n,n]$ or $(n, n+1)$, we have $v_p(a_p(x)) \in (p-1)v_p(T_{\wt(x)})\cdot I.$
	
	In particular, as $x$ varies on each irreducible component of $\Spc_D$ with $\wt(x)$ approaching the
	boundary of weight space, i.e. $|T_{\wt(x)}|_p\rightarrow 1^-$, the slope $v_p(a_p(x))\rightarrow 0$. 
\end{theorem}
By $p$-adic family version of the Jacquet-Langlands correspondence (see \cite[Theorem~3]{chenevier2005correspondance}),
they also deduce the similar results for most of the components of the original Coleman-Mazur eigencurves.

The goal of this paper is to generalize the result in \cite{liu2017eigencurve}
to the eigenvarieties associated to ($p$-adic) overconvergent automorphic forms for a definite quaternion algebra over a totally real field $F$ in which $p$ splits completely. Combining with the $p$-adic family versions of base change and Jacquet-Langlands correspondence, this result allows us to determine the boundary behavior for the entire Coleman-Mazur eigencurves, and hence answers the first part of  original question raised by Coleman and Mazur (see \cite[Conjecture~1.2(1)]{liu2017eigencurve}) in complete generality.

We set up a few notations before stating our main result. Let $F$ be a totally real field of degree $g$ and $p$ be an odd prime number which splits completely in $F$.  Let $\calO_F$ be the ring of integers of $F$ and set $\calO_p:= \calO_F\otimes_{\ZZ}\ZZ_p$. Let $I:=\Hom(F,\bar{\QQ}_p)$ and for each $i\in I$, we choose a uniformizer $\pi_i$ of the completion $F_{\gothp_i}$ of $F$ at the prime $\gothp_i$ induced by $i$. 
 The weight space $\calW$ is the rigid analytification of $\ZZ_p\llbracket \calO_p^\times\times\ZZ_p^\times \rrbracket$. It has a full set of parameters $\big\{(T_i:= [\exp(\pi_i),1]-1)_{i\in I},T:= [1,\exp(p)]-1 \big\}$. For $r\in (0,1)$, we use $\calW^{>r}$ to denote the subspace of $\calW$ where $|T_i|_p\in (r,1)$ for all $i\in I$. Let $D/F$ be a totally definite quaternion division algebra over $F$, which is split at all places over $p$. Fix a tame level structure. Let $\calX_D$ (resp. $\calZ_D$) be the eigenvariety (resp. spectral variety) associated to the overconvergent automorphic forms for $D^\times$ constructed in \cite[Part~$\uppercase\expandafter{\romannumeral3}$]{buzzard320eigenvarieties}. We put $w:\calX_D\rightarrow \calW$  be the weight map  and  denote by $\calX_D^{>r}$ the preimage $w^{-1}(\calW^{>r})$. For each $x\in \calX_D(\CC_p)$, it corresponds to a $\CC_p$-valued system of eigenvalues (we refer to \cite[\S5]{buzzard320eigenvarieties} for the precise definition). For each $i\in I$, there is a Hecke operator $U_{\pi_i}$ acting on the space of overconvergent automorphic forms and we use $a_i(x)$ to denote the eigenvalue of the $U_{\pi_i}$-operator for $x$. The constructions of the eigenvarieties and Hecke operators will be carefully recalled in  \S\ref{subsection: Hecke operators} and \S\ref{subsection:the spectral varieties and eigenvarieties}.

$$
\xymatrix @=1.5cm{
	\calX_D \ar[r]^{(a_i)_{i\in I}} \ar[d]^{w} & (\GG_m)^g  \\
	\calW & 
}
$$

Under the above notations, we have the following theorem.
\begin{theorem}\label{T:spectral halo for eigenvarieties for D}
	We denote by $\Sigma$ the subset $\{0 \}\bigcup \{1+2k|k\in \ZZ_{\geq 0} \}$ of $\ZZ$. The eigenvariety $\calX_D^{>1/p}$ is a disjoint union 
	\[
	\calX_D^{>1/p}=\bigsqcup_{l\in \Sigma^I,\ \sigma\in \{\pm\}^I}\calX_{l,\sigma}
	\]
	of (possibly empty) rigid analytic spaces which are finite over $\calW^{>1/p}$ via $w$, such that for each closed point $x\in \calX_{l,\sigma}(\CC_p)$ with $l=(l_i)_{i\in I}\in \Sigma^I$ and $\sigma=(\sigma_i)_{i\in I}\in\{\pm \}^I$, we have
	\[\begin{cases}
	v_p(a_i(x)) = (p-1)v_p(T_{i,w(x)})\cdot l_i,&\textrm{for}\ \sigma_i=-,\\
	v_p(a_i(x)) \in (p-1)v_p(T_{i,w(x)})\cdot (l_i,l_i+2),&\textrm{for}\ \sigma_i=+ \textrm{~and~} l_i\neq 0,\\
	v_p(a_i(x)) \in (p-1)v_p(T_{i,w(x)})\cdot (0,1),&\textrm{for}\ \sigma_i=+ \textrm{~and~} l_i=0,
	\end{cases}\]
	for all $i\in I$. 
	
	In particular, as $x$ varies on each irreducible component of $\calX_D$ with $w(x)$ approaching the boundary of the weight space, i.e. $v_p(T_{i,w(x)})\rightarrow 0$, the slopes $v_p(a_i(x))\rightarrow 0$ for all $i\in I$.
\end{theorem}

\begin{remark}
	\begin{enumerate}
		\item In our main theorem, the decomposition of the eigenvariety $\calX_D^{>1/p}$ is characterized by all the $U_{\pi_i}$-operators ($i\in I$). We cannot work solely with the spectral varieties since only the eigenvalues of $U_{\pi}:=\prod_{i\in I}U_{\pi_i}$ can be read from the spectral varieties. This is a significant difference to the eigencurve case. 
		\item There have been several generalizations of Liu-Wan-Xiao's results to more general eigenvarieties. In \cite{Lynnelle2019slopesunitary}, Ye generalized Liu-Wan-Xiao's estimation on slopes in eigencurves to eigenvarieties for definite unitary groups of arbitrary rank. Ye gave a lower bound and upper bound of the Newton polygon of characteristic power series for the $U_p$-operator, but they do not match at any point on the Newton polygon. So she cannot prove a similar result as in Theorem~\ref{T:spectral halo for eigenvarieties for D} for eigenvarieties associated to definite unitary groups.
		
		There are also generalizations to Hilbert modular eigenvarieties. Let $F$ and $p$ be as above. In \cite{johansson2018parallel}, Johansson and Newton defined a one-dimensional  `partial' eigenvarieties interpolating Hilbert modular forms over $F$ with weight varying
		only at a single place above $p$. They proved that over the boundary annulus of the weight space, the partial eigenvarieties decompose as a disjoint union of components that are finite over weight space, and the components have a similar property as described in the main result of Liu-Wan-Xiao. Our result agrees with theirs when restricting to their `partial' eigenvarieties.
		\item The assumption that $p$ splits completely in $F$ is essential in our argument. In fact, the philosophical analogy between the Artin--Schreier--Witt tower and the Igusa tower of Hilbert modular Shimura varieties explained in \cite{ren2018slopes} and the main theorem there suggests that further modifications are needed in order to formulate a  
		`reasonable' conjecture that generalizes Theorem~\ref{T:spectral halo for eigenvarieties for D} to general $F$. We refer to \S\ref{sub:1.3} for more discussion towards the (conjectural) generalization of our theorem.
	\end{enumerate}	
\end{remark}
\subsection{Idea of the proof of Theorems~\ref{T:spectral halo for eigenvarieties for D}}

We now explain how we deduce our main result.
\begin{enumerate}
	\item[\textbf{I.}]	Going back to Liu-Wan-Xiao's work in \cite{liu2017eigencurve}, the way they deduce their main results over $\QQ$ relies crucially on two key ingredients: 
	\begin{enumerate}
		\item a sharp estimate of the action of the $U_p$-operator on the space of overconvergent automorphic forms, in the form of providing a lower bound of the Hodge polygon, and 
		\item  the observation that, at classical weights, the subspaces of classical automorphic forms provide  known points of the corresponding Newton polygon which happens to lie on the Hodge polygon.
	\end{enumerate}
	\item[\textbf{II.}]
	While we choose to work with automorphic forms associated to a definite quaternion algebra over $F$ as opposed to the usual overconvergent modular forms, we circumvent the complication of the geometry of the Hilbert modular Shimura varieties. We will define the spaces of integral $p$-adic automorphic forms. It contains the spaces of overconvergent automorphic forms and have the same characteristic power series of the $U_\pi$-operator. 
	An important observation is that the $U_\pi$-operator on the spaces of integral $p$-adic automorphic forms can be written reasonably explicitly, as explained in  \S\ref{subsection: explicit expression of Up operator on the space of automorphic forms}.  This was inspired by the thesis \cite{jacobs2003slopes} of D. Jacobs  (a former student of Buzzard), and the generalization in  \cite{liu2017eigencurve}.
	\item[\textbf{III.}]
	The major difficulty we encounter here is that the action of each individual $U_{\pi_i}$-operator on the space of overconvergent automorphic forms is not compact, whereas the action of their product is. So if we generalize the method in \cite{liu2017eigencurve} naively, one would have to work with the operator $U_\pi=\prod\limits_{i\in I}U_{\pi_i}$. It is possible to give a lower bound of the corresponding Hodge polygon similar to $\textbf{I}(a)$, which is associated to inequalities regarding the sum of all the $U_{\pi_i}$-slopes. But for $\textbf{I}(b)$, the space of classical forms are characterized by the properties of the type: the $U_{\pi_i}$-slope is less than or equal to $k_i-1$, for every individual $i\in I$. This does not provide a known 
	point on the Newton polygon for the action of $U_\pi$, because the ordering mechanism of bases are incompatible. More precisely, there exists no orthonormal basis, in the sense of \cite[Definition~\textbf{I}.1.3]{colmez2010fonctions}, of the space of continuous function $\calC(\calO_p,\CC_p)$, such that the first $d:= \prod\limits_{i\in I}(k_i-1)$ elements of $\Omega$, gives a basis of $\calC(\calO_p,\CC_p)^{\deg\leq k-2}$ for all $k=(k_i)_{i\in I}\in \ZZ_{\geq 2}^I$, where $\calC(\calO_p,\CC_p)^{\deg\leq k-2}$ is the $\CC_p$-subspace of $\calC(\calO_p,\CC_p)$ spanned by the polynomial functions $\prod\limits_{i\in I}z_i^{l_i}$, for all $0\leq l_i\leq k_i-2$, $i\in I$, and $z_i:\calO_p=\prod\limits_{i\in I}\calO_{\gothp_i}\rightarrow \CC_p$ is the projection to the $i$th component. Therefore, the lower bound of the Hodge polygon of the $U_\pi$-operator does not touch its actual Newton polygon in general when the degree $g=[F:\QQ]\geq 2$.
	\item[\textbf{IV.}]
	To circumvent this difficulty, we work with the space of generalized integral $p$-adic automorphic forms. In the introduction, we explain our idea in the simplest case when $[F:\QQ]=2$. Let $\gothp_1,\gothp_2$ be the two places of $F$ over $p$. 
	
	Let $A$ be a topological ring in which $p$ is topologically nilpotent. For a continuous homomorphism $\kappa:\calO_p^\times\times \calO_p^\times\rightarrow A^\times$, we define the space of integral $p$-adic automorphic forms $S_{\kappa,I}^D(K^p,A)$ and the space of generalized $p$-adic automorphic forms $S_{\kappa,1}^D(K^p,A)$. The latter space consists of generalized automorphic forms that are like automorphic forms at the place $\gothp_1$, but are like the continuous dual of the completed homology at the place $\gothp_2$. We have explicit isomorphisms of these spaces: $$S_{\kappa,I}^D(K^p,A)\cong \bigoplus\limits_{k=0}^{s-1}\calC(\calO_p,A)\textrm{~and~}S_{\kappa,1}^D(K^p,A)\cong \bigoplus\limits_{k=0}^{s-1}\calC_1(\kappa,A),$$ where $\calC_1(\kappa,A)$ is a suitably defined subspace of $\calC(\calO_{\gothp_1}\times \Iw_{\pi_2},A)$ that contains $\calC(\calO_p,A)$. We use $S_{\kappa,I}^D(K^p,A)^\vee$ (resp. $S_{\kappa,1}^D(K^p,A)^\vee$) to denote the continuous $A$-dual space of $S_{\kappa,I}^D(K^p,A)$ (resp. $S_{\kappa,1}^D(K^p,A)$). Our argument can be exhibited in the following diagram:
	$$
	\xymatrix @=1.2cm{
		S_{\kappa,1}^D(K^p,A)^\vee \ar@{-}[r]^{dual} \ar[d] & 
		S_{\kappa,1}^D(K^p,A) & \curvearrowleft U_{\pi_1} \\
		S_{\kappa,I}^D(K^p,A)^\vee \ar@{-}[r]^{dual}& S_{\kappa,I}^D(K^p,A) \ar[u] & \curvearrowleft U_{\pi_1}, U_{\pi_2}.
	}
	$$
	Here the right vertical arrow is an embedding $S_{\kappa,I}^D(K^p,A)\hookrightarrow S_{\kappa,1}^D(K^p,A)$, which identifies $S_{\kappa,I}^D(K^p,A)$ as the invariant subspace of $S_{\kappa,1}^D(K^p,A)$ under an action of the Borel subgroup $B(\calO_{\gothp_2})$ of $\GL_2(\calO_{\gothp_2})$; and the left vertical arrow is the dual of this embedding, i.e. taking the $B(\calO_{\gothp_2})$-coinvariants of $S_{\kappa,1}^D(K^p,A)^\vee$.

	We first embed $S_{\kappa,I}^D(K^p,A)$ into the larger space $S_{\kappa,1}^D(K^p,A)$. The latter space carries an extra structure of right $A\llbracket P'_2\rrbracket$-modules, where $P'_2\subset \mathrm{SL}_2(\calO_{\gothp_2})$ is some explicit  open compact pro-$p$-subgroup. The sacrifice of doing this is that there is only $U_{\pi_1}$-operator, but no $U_{\pi_2}$-operator defined on $S_{\kappa,1}^D(K^p,A)$. Thus the space $S_{\kappa,1}^D(K^p,A)^\vee$ is an infinite free left $A\llbracket P'_2\rrbracket$-module  and the induced $U_{\pi_1}$-operator on it is $A\llbracket P'_2\rrbracket$-linear. Under a suitable chosen basis of $S_{\kappa,1}^D(K^p,A)^\vee$, we verify that the infinite matrix $M$ corresponding to the $U_{\pi_1}$-operator admits a similar estimation as obtained in \cite{liu2017eigencurve}. On the other hand, we have a characterization of the image of the spaces of classical automorphic forms in $S_{\kappa,1}^D(K^p,A)$ (when $A=\CC_p$ and $\kappa$ is locally algebraic). An analogous argument of $\textbf{I}(b)$ provides us known points on the Newton polygon of $M$ modulo the augmentation ideal of $A\llbracket P'_2\rrbracket$. In \S\ref{section: Newton-Hodge decomposition over certain noncommutative rings}, we prove a Newton-Hodge decomposition theorem for infinite matrices over certain noncommutative rings, which is a generalization of the Newton-Hodge decomposition theorem over valuation rings. We apply this theorem to the matrix $M$, and obtain a filtration $\{\tilde{F}_{\alpha} \}$ of the space $S_{\kappa,1}^D(K^p,A)^\vee$. Under the surjective map $S_{\kappa,1}^D(K^p,A)^\vee\rightarrow S_{\kappa,I}^D(K^p,A)^\vee$, we obtain a filtration $\{F_\alpha \}$ of $S_{\kappa,I}^D(K^p,A)^\vee$. We will show that this filtration is stable under the $U_{\pi_1}$ and $U_{\pi_2}$-operators on $S_{\kappa,I}^D(K^p,A)^\vee$. When $\kappa:\calO_p^\times\times \calO_p^\times\rightarrow A^\times$ is associated to a point $x\in \calW^{>1/p}(\CC_p)$, the graded pieces of this filtration are characterized by the condition in Theorem~\ref{T:spectral halo for eigenvarieties for D} for the single place $i=\gothp_1$. We run the above argument to the graded pieces of the filtration $\{F_\alpha \}$ and the $U_{\pi_2}$-operator on it. We get a filtration for every graded piece, and the graded pieces of all these filtrations are characterized by the desired property in Theorem~\ref{T:spectral halo for eigenvarieties for D}. Our main theorem follows form the existence of such filtrations.
\end{enumerate}

\subsection{Applications}
Let $F$ and $p$ be as above and $\gothn$ be an ideal of $F$ prime to $p$. We use $\calX_{\GL_{2/F}}(\gothn)$ to denote the Hilbert modular eigenvariety of tame level $\gothn$, which admits a weight map $w:\calX_{\GL_{2/F}}(\gothn)\rightarrow \calW$. Similar as before, we can define $\calX_{\GL_{2/F}}(\gothn)^{>1/p}$. 
An important consequence of Theorem~\ref{T:spectral halo for eigenvarieties for D} is the following description of the full Hilbert modular eigenvarieties over the boundary of the weight space.

\begin{theorem}
	The eigenvariety $\calX_{\GL_{2/F}}(\gothn)^{>1/p}$ is a disjoint union 
	\[
	\calX_{\GL_{2/F}}(\gothn)^{>1/p}=\bigsqcup_{l\in \Sigma^I, \sigma\in \{\pm\}^I}\calX_{l,\sigma}
	\]
	of (possibly empty) rigid analytic spaces which are finite over $\calW^{>1/p}$ via $w$, such that for each closed point $x\in \calX_{l,\sigma}(\CC_p)$ with $l=(l_i)_{i\in I}\in \Sigma^I$ and $\sigma=(\sigma_i)_{i\in I}\in\{\pm \}^I$, we have
	\[\begin{cases}
	v_p(a_i(x)) = (p-1)v_p(T_{i,w(x)})\cdot l_i,&\textrm{for}\ \sigma_i=-,\\
	v_p(a_i(x)) \in (p-1)v_p(T_{i,w(x)})\cdot (l_i,l_i+2),&\textrm{for}\ \sigma_i=+ \textrm{~and~} l_i\neq 0,\\
	v_p(a_i(x)) \in (p-1)v_p(T_{i,w(x)})\cdot (0,1),&\textrm{for}\ \sigma_i=+ \textrm{~and~} l_i=0,
	\end{cases}\]
	for all $i\in I$. 
\end{theorem}
This theorem will be proved in \S\ref{section:Application to Hilbert modular eigenvarieties}. The main tool of our proof is Hansen's $p$-adic interpolation theorem (\cite[Theorem~5.1.6]{hansen2017universal}). When the degree $[F:\QQ]$ is even, there is an isomorphism between $\calX_{\GL_{2/F}}(\gothn)$ and the eigenvariety $\calX_D(\gothn)$ for the totally definite quaternion algebra $D$ over $F$ with discriminant $1$. The  theorem above follows directly from Theorem~\ref{T:spectral halo for eigenvarieties for D}. When $[F:\QQ]$ is odd, we take a quadratic extension $F'/F$ such that $F'$ is totally real and $p$ splits completely in $F'$. Set $\gothn'=\gothn\calO_F$. We show that there exists a morphism $\calX_{\GL_{2/F}}(\gothn)\rightarrow \calX_{\GL_{2/F'}}(\gothn')$ that interpolates the quadratic base change from $F$ to $F'$ on non-critical classical points. Then the theorem for $\calX_{\GL_{2/F}}(\gothn)$ follows from that for $\calX_{\GL_{2/F'}}(\gothn')$.

 A notable consequence of the above theorem is that every irreducible component of the (full) Hilbert modular eigenvarieties
contains a classical point of parallel weight $2$. This result has been proven for most irreducible components by Johansson-Newton in \cite{johansson2018parallel}. They use this result to prove the parity conjecture for Hilbert modular forms unconditionally when the degree $[F:\QQ]$ is even, and when the degree $[F:\QQ]$ is odd, they need to impose the assumption that the automorphic representation corresponding to the Hilbert modular form is not principal series at some place prime to $p$. The last assumption is to guarantee that the Hilbert modular form corresponds to an automorphic form for a definite quaternion algebra over $F$ under the Jacquet-Langlands correspondence. This assumption can be removed now.

\subsection{Further questions}\label{sub:1.3}
Theorem~\ref{T:spectral halo for eigenvarieties for D} is not a complete description of the boundary behavior of the eigenvariety $\calX_D$. Inspired by Coleman-Mazur-Buzzard-Kilford conjecture (\cite[Conjecture~1.2]{liu2017eigencurve}), we make the following conjecture.
\begin{conjecture}\label{C:refined halo conjecture}
	When $r\in (0,1)$ is sufficiently close to $1^-$, there exists a sequence of rational numbers $\Sigma_i=\{\alpha_{i,0}\leq \alpha_{i,1}\leq \dots \}$ for every $i\in I$, such that if we denote $\Sigma:=\prod\limits_{i\in I}\Sigma_i$, then the eigenvariety $\calX_D^{>r}$ is a disjoint union 
	\[
	\calX_D^{>r}=\bigsqcup_{\alpha=(\alpha_i)_{i\in I}\in \Sigma} \calX_\alpha
	\]
	of (possibly empty) rigid analytic spaces which are finite over $\calW^{>r}$ via $w$. For every $\alpha\in \Sigma$ and each closed point $x\in \calX_\alpha(\CC_p)$, we have $v_p(a_i(x))=(p-1)v_p(T_{i,w(x)})\cdot \alpha_i$ for all $i\in I$. Moreover, each sequence $\Sigma_i$ is a disjoint union of finitely many arithmetic progressions, counted with multiplicities.
\end{conjecture}
When $F=\QQ$, Conjecture~\ref{C:refined halo conjecture} has been proved for the eigenvariety $\calX_D^{>r}$ for some (explicit) rational number $r\in (\frac{1}{p},1)$ in \cite[Theorem~1.5]{liu2017eigencurve}. Their proof is based on a careful analysis of the characteristic power series of the $U_p$-operator and the fact that the coefficients of the characteristic power series belong to the Iwasawa algebra $\Lambda=\ZZ_p\llbracket \ZZ_p^\times\rrbracket$. It seems to us that more ideas are needed in order to generalize their result to the case of Hilbert modular eigenvarieties. In \cite[Conjecture~1]{birkbeck2019slopes}, Birkbeck made a similar conjecture on the $p$-adic slopes of overconvergent Hilbert modular forms without any assumption on the totally real field $F$. Our Conjecture~\ref{C:refined halo conjecture} makes the assumption that $p$ splits in $F$, and hence gives a more precise characterization of the decomposition of the eigenvarieties. We also refer to Birkbeck's thesis
\cite[Chapter~8]{birkbeck2017eigenvarieties} for numerical evidence towards Conjecture~\ref{C:refined halo conjecture}. 

Finally, we make a comment about the case for $p=2$. Most of our arguments  work for $p=2$ with mild modifications. However, there is one subtle place in Notation~\ref{N:J component of integer rings, matrix groups and characters} that the oddness of $p$ is crucial. We explain the case for $F=\QQ$ as an example and refer to Notation~\ref{N:J component of integer rings, matrix groups and characters} for details. Let $\Iw_p=\Matrix{\ZZ_p^\times}{\ZZ_p}{p\ZZ_p}{\ZZ_p^\times}$ be the Iwahori subgroup of $\GL_2(\ZZ_p)$ and $D(\ZZ_p^\times)=\left\{ \Matrix \alpha 00\alpha |\alpha\in \ZZ_p^\times\right\}$	be the subgroup of $\Iw_p$ consisting of scalar matrices. When $p>2$, we show that the inclusion map $D(\ZZ_p)\hookrightarrow \Iw_p$ has a section, i.e. there exists a subgroup $P$ of $\Iw_p$ such that the multiplication map $D(\ZZ_p)\times P\rightarrow \Iw_p$ is an isomorphism. We encourage the audience to explore how to modify the argument for $p=2$.
	
\subsection{Structure of the paper}

The \S\ref{section: Newton-Hodge decomposition over certain noncommutative rings} is devoted to prove a Newton-Hodge decomposition theorem for infinite matrices over certain noncommutative rings. This section is technical and the readers can assume the main result Theorem~\ref{T:Newton-Hodge decomposition with Hodge bound restriction for infinite matrices over noncommutative rings} and skip its lengthy proof at first. In \S\ref{section:Automorphic forms for definite quaternion algebras and completed homology}, we first recall the notions of overconvergent automorphic forms, Hecke operators and eigenvarieties. Then we construct a space of generalized integral $p$-adic automorphic forms and explain its relation with the spaces of classical automorphic forms.  In \S\ref{section:A filtration on the space of integral $p$-adic automorphic forms}, we give an explicit expression of the Hecke operator $U_{\pi_i}$ on the space of generalized integral $p$-adic automorphic forms, and use the Newton-Hodge decomposition Theorem~\ref{T:Newton-Hodge decomposition with Hodge bound restriction for infinite matrices over noncommutative rings} to obtain a filtration of this space, whose graded piece has a nice description in terms of the $U_{\pi_i}$-action, for a fixed $i\in I$. We complete the proof of Theorem~\ref{T:spectral halo for eigenvarieties for D} in \S\ref{section:Proof of the main theorem}. The idea is to inductively apply the argument in the previous section to all places of $F$ over $p$, and then get a filtration of the space of integral $p$-adic automorphic forms, whose graded pieces can be described in terms of all the $U_{\pi_i}$-operators. The decomposition of the eigenvarieties follows from the existence of such a filtration. In \S\ref{section:Application to Hilbert modular eigenvarieties}, we use Hansen's $p$-adic interpolation theorem to translate our results to Hilbert modular eigenvarieties.

\subsection{Acknowledgment}

We thank Liang Xiao for sharing his ideas and answering many questions on this topic. We would also like to thank Yiwen Ding, Yongquan Hu and Daqing Wan for helpful comments and conversations. Finally we would like to thank the anonymous referees for their impressively helpful report that significantly improved the exposition of this paper.

\subsection{Notations}

For every prime $p$, we fix an embedding $\iota_p:\bar{\QQ}\rightarrow \CC_p$.
We use $v_p(\cdot)$ (resp. $|\cdot|_p$) to denote the $p$-adic valuation (resp. $p$-adic norm) on $\CC_p$, normalized by $v_p(p)=1$ (resp. $|p|_p=p^{-1}$).

For two topological spaces $X$ and $Y$, we use $\calC(X,Y)$ to denote the set of continuous maps from $X$ to $Y$.

\section{Newton-Hodge decomposition over certain noncommutative rings }
\label{section: Newton-Hodge decomposition over certain noncommutative rings}

\subsection{Notations}\label{subsection: notations in N-H decomposition}

\begin{itemize}
	\item Let $\NN=\{1,2,\dots\}$. For an integer $n>0$, we denote by $[n]:= \{1,2,\dots, n \}\subset \NN$ and $[\infty]:=\NN\cup \{\infty \}$.
	\item For an increasing sequence $\ula=(\lambda_1,\lambda_2,\dots )$ of real numbers, we denote by  $\ula^{[n]}:=(\lambda_1,\dots ,\lambda_n )$ the $n$-th truncated subsequence of $\ula$. More generally, for two integers $0<m<n$, we put $\ula^{(m,n]}:=(\lambda_{m+1},\dots, \lambda_n)$.
	\item For $m,n\in [\infty]$  and a ring $R$, we denote by $\rmM_{m\times n}(R)$ the set of  $m\times n$ matrices with entries in $R$ and abbreviate it by $\rmM_n(R)$ when $m=n$.
	\item For any  $n\in [\infty]$, $M\in \rmM_n(R)$ and   $I,J$ two subsets of $[n]$ with cardinalities $|I|=k$, $|J|=l$, we denote by $M_{I,J}$ the $k\times l$ matrix consisting of entries of $M$ with row indices in $I$ and column indices in $J$, and abbreviate it by $M_I$
	when $I=J$. 
	\item We denote by $v_T(-)$ the $T$-adic valuation of the formal power series ring $\Fp\llbracket T \rrbracket$ normalized by $v_T(T)=1$.
	\item We denote by $\bbC$ an algebraic closure of the fractional field $\FF_p(\!(T)\!)$ of $\FF_p\llbracket T\rrbracket$. The valuation $v_T(-)$ extends uniquely to $\bbC$, which is still denoted by $v_T(-)$.
\end{itemize}

\subsection{Newton-Hodge decomposition for matrices in $ M_n(\Fp\llbracket T\rrbracket)$}

\begin{definition}\label{D:Hodge and Newton functions and polygons}

	\begin{enumerate}
		\item The Newton  and Hodge functions of a matrix $M\in M_n(\Fp\llbracket T \rrbracket )$ are defined by
		\begin{equation}\label{E:Newton function }
		N(M, k):= v_T\left(\sum\limits_{N : \textrm{~principal~} k\times k \textrm{~minor of~} M} \det(N)\right)
		\end{equation}
		and 
		\begin{equation}\label{E:Hodge function}
		H(M, k):=\min\{v_T(\det(N))\;|\; N\ \textrm{is a}\ k\times k\ \textrm{minor of}\ M\}
		\end{equation}
		for every $k\in [n]$.
		\item Moreover, we set $H(M,0):=0$ and $N(M,0):=0$. The Hodge polygon (resp. Newton polygon) of $M$ is the lower convex hull of the points $\{(k,H(M,k))\}_{k=0}^n$ (resp. $\{(k, N(M,k))\}_{k=0}^n$).
	\end{enumerate}
\end{definition}

\begin{remark}\label{R:relation between slopes of Newton polygon and eigenvalues}
	Let $\alpha_1,\dots,\alpha_n$ be the eigenvalues of $M$ in $\bbC$ (counted with multiplicities) such that $v_T(\alpha_1)\leq\dots\leq v_T(\alpha_n)$. Then the slopes of the Newton polygon are exactly $v_T(\alpha_1),\dots, v_T(\alpha_n)$.
\end{remark}

Let $\ula=(\lambda_1, \lambda_2, \dots )$ be a (not necessarily strictly) increasing sequence of nonnegative real numbers such that 
$
\lim\limits_{n\rightarrow +\infty}\lambda_n=+\infty.
$
For every integer $n>0$, we denote by $D(\ula^{[n]})$ (resp. $D(\ula)$) the diagonal matrix $\mathrm{Diag}(T^{\lambda_1},\dots,T^{\lambda_n})\in \rmM_n(\Fp\llbracket T \rrbracket)$ (resp. $\mathrm{Diag}(T^{\lambda_1},\dots)\in \rmM_\infty(\Fp\llbracket T \rrbracket)$). For every two integers $0<m<n$, we denote by $D(\ula^{(m,n]})$ the diagonal matrix $\mathrm{Diag}(T^{\lambda_{m+1}},\dots,T^{\lambda_n})\in \rmM_{n-m}(\Fp\llbracket T \rrbracket)$.

\begin{definition}\label{D:lambda Hodge bounded and lambda stable}
	For  $\ell\in \NN$ and $n\in \NN$, 
	\begin{enumerate}
		\item  a matrix $M=(m_{i,j})_{1\leq i\leq n, 1\leq j\leq \ell}\in \rmM_{n\times \ell}(\Fp\llbracket T \rrbracket)$ is called $\ula^{[n]}$-Hodge bounded if $M=D(\ula^{[n]})M'$ for some $M'\in \rmM_{n\times \ell}(\Fp\llbracket T \rrbracket)$, or equivalently, $v_T(m_{ij})\geq \lambda_i$ for all $1\leq i\leq n$ and $ 1\leq j\leq \ell$. When $\ell=n$ and $M'\in\GL_n(\Fp\llbracket T \rrbracket)$, we call $M$ as \emph{strictly $\ula^{[n]}$-Hodge bounded}.
		\item A matrix $M\in \rmM_\infty (\Fp\llbracket T \rrbracket)$ is called \emph{$\ula$-Hodge bounded} if $M=D(\ula)M'$ for some $M'\in \rmM_{\infty}(\Fp\llbracket T \rrbracket)$.
		\item  A matrix $A\in \rmM_n(\Fp\llbracket T \rrbracket)$ is called \emph{$\ula^{[n]}$-stable} if for every $\ula^{[n]}$-Hodge bounded matrix $B\in \rmM_n(\Fp\llbracket T \rrbracket)$, $AB$ is also $\ula^{[n]}$-Hodge bounded.
		\item A matrix $A\in \rmM_\infty(\Fp\llbracket T \rrbracket)$ is called \emph{$\ula$-stable} if for every $\ula$-Hodge bounded matrix $B\in \rmM_\infty(\Fp\llbracket T \rrbracket)$, $AB$ is also $\ula$-Hodge bounded.
	\end{enumerate}
\end{definition}

\begin{remark}
	\begin{enumerate}
		\item For $A,B\in \rmM_n(\Fp\llbracket T \rrbracket)$ if $B$ is $\ula^{[n]}$-Hodge bounded, then $BA$ is also $\ula^{[n]}$-Hodge bounded.
		\item Let $M\in \rmM_\infty(\Fp\llbracket T \rrbracket)$ be a $\ula$-Hodge bounded matrix. Since $\lambda_i\xrightarrow{i\to \infty}  \infty$, the functions \eqref{E:Newton function } and \eqref{E:Hodge function} are well-defined on $M$, which are called the \emph{Newton and Hodge functions} for $M$, respectively.
		Moreover, we define the Newton and Hodge polygons for $M$ in a similar way to the ones for matrices of finite dimensions. 
	\end{enumerate}
\end{remark}

\begin{lemma}\label{L:criterion of lambda-stable}
	For $n\in[\infty]$ and a matrix $A=(a_{ij})\in \rmM_n(\Fp\llbracket T \rrbracket)$, the following statements are equivalent:
	\begin{enumerate}
		\item $A$ is $\ula^{[n]}$-stable;
		\item $D(\ula^{[n]})^{-1}AD(\ula^{[n]})\in \rmM_n(\Fp\llbracket T \rrbracket)$;
		\item $v_T(a_{ij})\geq \lambda_i-\lambda_j$ for every $1\leq i\leq n$ and $1\leq j\leq n$. 
	\end{enumerate}
\end{lemma} 
\begin{proof}
	$(1)\Rightarrow (2)$ Since $A$ is $\ula^{[n]}$-stable and $D(\ula^{[n]})$ is $\ula^{[n]}$-Hodge bounded, $AD(\ula^{[n]})$ is also  $\ula^{[n]}$-Hodge bounded, and hence $AD(\ula^{[n]})=D(\ula^{[n]})B$ for some $B\in\rmM_n(\Fp\llbracket T \rrbracket)$. This proves $D(\ula^{[n]})^{-1}AD(\ula^{[n]})\in \rmM_n(\Fp\llbracket T \rrbracket)$.
	
	$(2)\Rightarrow (3)$ Suppose that we have $AD(\ula^{[n]})=D(\ula^{[n]})B$ for some $B=(b_{ij})\in\rmM_n(\Fp\llbracket T \rrbracket)$. Comparing the $(i,j)$-entries of the matrices on the two sides of this equality, we have $T^{\lambda_j}a_{ij}=T^{\lambda_i}b_{ij}$, and hence $$v_T(a_{ij})=v_T(b_{ij})+\lambda_i-\lambda_j\geq \lambda_i-\lambda_j.$$
	
	$(3)\Rightarrow (1)$ For every $\ula^{[n]}$-Hodge bounded  matrix $M=(m_{ij})\in \rmM_n(\Fp\llbracket T \rrbracket)$, we have $v_T(m_{ij})\geq \lambda_i$ for all $1\leq i\leq n$ and $1\leq j\leq n$. Then the $(i,j)$-entry of $AM$ is 
	$
	\sum\limits_{k=1}^n a_{ik}m_{kj}.
	$
	Since $v_T(a_{ik})\geq \lambda_i-\lambda_k$, we have $v_T(a_{ik}m_{kj})\geq \lambda_i$, and hence
	$
	v_T(\sum\limits_{k=1}^n a_{ik}m_{kj})\geq \lambda_i.
	$
	Therefore, we conclude that $AM$ is $\ula^{[n]}$-Hodge bounded and $A$ is $\ula^{[n]}$-stable.
\end{proof}

\begin{lemma}\label{L:second criterion of lambda-stable}
	Let $M\in \rmM_n(\Fp\llbracket T \rrbracket)$ be strictly  $\ula^{[n]}$-Hodge bounded. For a matrix  $A\in \rmM_n(\Fp\llbracket T \rrbracket)$ if $AM$ is $\ula^{[n]}$-Hodge bounded, then $A$ is $\ula^{[n]}$-stable.
\end{lemma}

\begin{proof}
	Since $M$ is strictly  $\ula^{[n]}$-Hodge bounded and $AM$ is $\ula^{[n]}$-Hodge bounded, there are matrices $M'\in \GL_n(\Fp\llbracket T \rrbracket)$ and $B\in \rmM_n(\Fp\llbracket T \rrbracket)$ such that $M=D(\ula^{[n]})M'$ and $AM=D(\ula^{[n]})B$. Therefore, we have $A=D(\ula^{[n]})BM'^{-1}D(\ula^{[n]})^{-1}$. Combined with Lemma~\ref{L:criterion of lambda-stable}, this equality implies that $A$ is $\ula^{[n]}$-stable.
\end{proof}

\begin{lemma}\label{L:Newton function determines Hodge function for finite matrix}
	For $n\in\NN$ and $M\in M_n(\Fp\llbracket T \rrbracket)$ if $M$ is $\ula^{[n]}$-Hodge bounded and satisfies $N(M,n)=\sum\limits_{i=1}^n\lambda_i$, then
	$
	H(M,k)=\sum\limits_{i=1}^k\lambda_i \text{~for every~} 1\leq k\leq n,
	$
	i.e. the slopes of the Hodge polygon of $M$ are $\lambda_1,\lambda_2,\dots,\lambda_n$. In particular, $M$ is strictly $\ula^{[n]}$-Hodge bounded.
\end{lemma}

\begin{proof}
	Since $M$ is $\ula^{[n]}$-Hodge bounded,  for every $1\leq k\leq n$ we have 
	$
	H(M,k)\geq \sum\limits_{i=1}^k \lambda_i.
	$
	Suppose 
	$
	H(M,k)>\sum\limits_{i=1}^k \lambda_i \text{~for some~} k.
	$
	Let $I:=[k]$. By Laplace expansion, we have
	\[
	\det(M)=\sum_{J\subset [n], |J|=k} \mathrm{sgn}(I,J)\det(M_{I,J})\det(M_{I,J}^{comp}),
	\]
	where $\mathrm{sgn}(I,J)=(-1)^{\sum_{i\in I}i+\sum_{j\in J}j}$ is the signature of the permutation determined by $I$ and $J$, and $M_{I,J}^{comp}$ is the complement $(n-k)\times (n-k)$ minor of $M_{I,J}$. 
	
	Since $M$ is $\ula^{[n]}$-Hodge bounded, we have
	\[
	v_T(\det(M_{I,J}))\geq H(M,k)>\sum\limits_{i=1}^k \lambda_i \textrm{\ and\ } v_T(\det(M_{I,J}^{comp}))\geq \sum_{i=k+1}^n\lambda_i,
	\]
	which together imply
	$ N(M,n)=v_T(\det(M))>\sum_{i=1}^n\lambda_i,$ a contradiction. Therefore we have 
	$
	H(M,k)=\sum\limits_{i=1}^k\lambda_i, \text{~for all~} 1\leq k\leq n.
	$
	
	Since $M$ is $\ula^{[n]}$-Hodge bounded, $M=D(\ula^{[n]})M'$ for some $M'\in \rmM_n(\Fp\llbracket T \rrbracket)$. Taking determinants and then $T$-adic valuations on both sides of this equality, we have $v_T(\det(M'))=0$, and hence $M'\in\GL_n(\Fp\llbracket T \rrbracket)$. So $M$ is strictly $\ula^{[n]}$-Hodge bounded.
\end{proof}

\begin{corollary}\label{P:Newton function determines Hodge function for infinite matrix}
	Let $M\in M_{\infty}(\Fp\llbracket T \rrbracket)$ be $\ula$-Hodge bounded. If there exists an strictly increasing infinite sequence $\us=(s_1, s_2,\dots)$ of positive integers such that 
	$
	N(M,s_n)=\sum\limits_{i=1}^{s_n}\lambda_i \text{~for every~} n\geq 1,
	$
	then for every $k\geq 1$ we have
	$
	H(M,k)=\sum\limits_{i=1}^k\lambda_i.
	$
\end{corollary}

\begin{proof}
	Since $M$ is $\ula$-Hodge bounded, for every $k\geq 1$ we have
	$
	H(M,k)\geq \sum\limits_{i=1}^k\lambda_i.
	$	
	For a fixed integer $k\geq 1$, we can choose $n\geq 1$ such that $s_n>k$. From our hypotheses that $N(M,s_n)=\sum\limits_{i=1}^{s_n}\lambda_i$ and that $M$ is 	$\ula$-Hodge bounded, there exists an $s_n\times s_n$ principal minor $M_1$ of $M$ such that \begin{equation}\label{u1}
	v_T(\det(M_1))=\sum\limits_{i=1}^{s_n}\lambda_i.
	\end{equation} 
	Since $\ula$ is increasing, $M_1$ is $\ula^{[s_n]}$-Hodge bounded. Combined with \eqref{u1} and Lemma~\ref{L:Newton function determines Hodge function for finite matrix} for $M_1$, this implies $H(M_1,k)=\sum\limits_{i=1}^k\lambda_i$. Note that every $k\times k$ minor of $M_1$ is also a $k\times k$ minor of $M$. We have $H(M,k)\leq H(M_1,k)=\sum\limits_{i=1}^k\lambda_i$, and hence
	$
	H(M,k)=\sum\limits_{i=1}^k\lambda_i.
	$
	Since we choose $k$ arbitrarily, the last equality completes the proof.
\end{proof}

\begin{definition}\label{D:touching vertex}
	For every $n\in [\infty]$ and $M\in \rmM_n(\Fp\llbracket T \rrbracket)$ we call an integer $k$ a \emph{touching vertex} of $M$ if it satisfies
	\begin{itemize}
		\item $(k,N(M,k))$ is a vertex of the Newton polygon of $M$, and
		\item $N(M,k)=H(M,k)$.
	\end{itemize}
\end{definition}

\begin{theorem}[\cite{kedlaya2010p} Theorem~4.3.11]\label{T:Newton-Hodge decomposition}
	For every $n\in\NN$ and $M\in \rmM_n(\Fp\llbracket T \rrbracket)$ if $1\leq k<n$ is a touching vertex of $M$, then there exists an invertible matrix $P\in \GL_n(\Fp\llbracket T \rrbracket)$ such that 
	\[
	PMP^{-1}=
	\quad
	\Biggl[\mkern-5mu
	\begin{tikzpicture}[baseline=-.65ex]
	\matrix[
	matrix of math nodes,
	column sep=1ex,
	] (m)
	{
		M_{11} & M_{12} \\
		0 & M_{22} \\
	};
	\draw[dotted]
	([xshift=0.5ex]m-1-1.north east) -- ([xshift=2ex]m-2-1.south east);
	\draw[dotted]
	(m-1-1.south west) -- (m-1-2.south east);
	\node[above,text depth=1pt] at ([xshift=3.5ex]m-1-1.north) {$\scriptstyle k$};  
	\node[left,overlay] at ([xshift=-1.2ex,yshift=-2ex]m-1-1.west) {$\scriptstyle k$};
	\end{tikzpicture}\mkern-5mu
	\Biggr],
	\]
	where $M_{11}$ and $M_{22}$ are $k\times k$ and $(n-k)\times (n-k)$ matrices such that the matrix $M_{11}$ accounts for the first $k$ slopes of the Hodge and Newton polygon of $M$, while $M_{22}$ accounts for the others. 
\end{theorem}
\begin{proof}
	See \cite[Theorem~4.3.11]{kedlaya2010p}.
\end{proof}

\begin{lemma}\label{L:conjugation of a matrix to be lambda-Hodge bounded}
	For every $n\in \NN$ and $M\in \rmM_n(\Fp\llbracket T \rrbracket)$ if $\mu_1\leq \mu_2\leq\dots\leq \mu_n$ are the slopes of the Hodge polygon of some matrix, then there exists an invertible matrix $P\in \GL_n(\Fp\llbracket T \rrbracket)$ such that $PMP^{-1}$ is $\underline\mu$-Hodge bounded, where $\underline\mu=(\mu_1, \mu_2, \dots, \mu_n)$.
\end{lemma}

\begin{proof}
	By \cite[Theorem~4.3.4]{kedlaya2010p}, there exist matrices $U,V\in \GL_n(\Fp\llbracket T \rrbracket)$ such that
	$UMV=D(\underline\mu).$
	Right-multiplying $V^{-1}U^{-1}$ to both sides of this equality, we prove that $UMU^{-1}=D(\underline\mu)V^{-1}U^{-1}$ is $\underline\mu$-Hodge bounded.
\end{proof}

\begin{theorem}\label{T:Newton-Hodge decomposition with Hodge bound restriction for finite matrices}
	Let $M\in M_n(\Fp\llbracket T \rrbracket)$ be a $\ula^{[n]}$-Hodge bounded matrix. Assume that $1<s<n$ is a touching vertex of $M$ and 
	$
	H(M,n)=\sum\limits_{i=1}^n\lambda_i.
	$
	Then there exists a matrix $W\in \GL_n(\Fp\llbracket T \rrbracket)$ such that $WMW^{-1}$ is $\ula^{[n]}$-Hodge bounded and a block upper triangular matrix of the form
	\[
	WMW^{-1}=
	\quad
	\Biggl[\mkern-5mu
	\begin{tikzpicture}[baseline=-.65ex]
	\matrix[
	matrix of math nodes,
	column sep=1ex,
	] (m)
	{
		M_{11} & M_{12} \\
		0 & M_{22} \\
	};
	\draw[dotted]
	([xshift=0.5ex]m-1-1.north east) -- ([xshift=2ex]m-2-1.south east);
	\draw[dotted]
	(m-1-1.south west) -- (m-1-2.south east);
	\node[above,text depth=1pt] at ([xshift=3.5ex]m-1-1.north) {$\scriptstyle s$};  
	\node[left,overlay] at ([xshift=-1.2ex,yshift=-2ex]m-1-1.west) {$\scriptstyle s$};
	\end{tikzpicture}\mkern-5mu
	\Biggr].
	\]
	Moreover, $W$ is $\ula^{[n]}$-stable.
\end{theorem}

\begin{proof}
	By Lemma~\ref{L:Newton function determines Hodge function for finite matrix}, for every $1\leq k\leq n$ we have
	$
	H(M,k)=\sum\limits_{i=1}^{k}\lambda_i.
	$
	By Theorem~\ref{T:Newton-Hodge decomposition}, there exists a matrix $P_1\in \GL_n(\Fp\llbracket T \rrbracket)$ such that
	\[
	P_1MP_1^{-1}=
	\quad
	\Biggl[\mkern-5mu
	\begin{tikzpicture}[baseline=-.65ex]
	\matrix[
	matrix of math nodes,
	column sep=1ex,
	] (m)
	{
		M'_{11} & M'_{12} \\
		0 & M'_{22} \\
	};
	\draw[dotted]
	([xshift=0.5ex]m-1-1.north east) -- ([xshift=2ex]m-2-1.south east);
	\draw[dotted]
	(m-1-1.south west) -- (m-1-2.south east);
	\node[above,text depth=1pt] at ([xshift=3.5ex]m-1-1.north) {$\scriptstyle s$};  
	\node[left,overlay] at ([xshift=-1.2ex,yshift=-2ex]m-1-1.west) {$\scriptstyle s$};
	\end{tikzpicture}\mkern-5mu
	\Biggr],
	\]
	and $M'_{11}$ accounts for the first $s$-th slopes of the Newton and Hodge polygons of $M$, while $M'_{22}$ accounts for the others.
	
	By Lemma~\ref{L:conjugation of a matrix to be lambda-Hodge bounded}, there exist matrices $Q_1\in \GL_s(\Fp\llbracket T \rrbracket)$ and $Q_2\in \GL_{n-s}(\Fp\llbracket T \rrbracket)$ such that $M_{11}=Q_1M'_{11}Q_1^{-1}$ and $M_{22}=Q_2M'_{22}Q_2^{-1}$ are $\ula^{[s]}$ and $\ula^{(s,n]}$-Hodge bounded, respectively. Set
	\[
	P_2:= \begin{bmatrix}
	Q_1&0\\
	0&Q_2\end{bmatrix} \in \GL_n(\Fp\llbracket T \rrbracket)\textrm{\ and \ } N:= P_2(P_1MP_1^{-1})P_2^{-1}=
	\quad
	\Biggl[\mkern-5mu
	\begin{tikzpicture}[baseline=-.65ex]
	\matrix[
	matrix of math nodes,
	column sep=1ex,
	] (m)
	{
		M_{11} & M_{12} \\
		0 & M_{22} \\
	};
	\draw[dotted]
	([xshift=0.5ex]m-1-1.north east) -- ([xshift=2ex]m-2-1.south east);
	\draw[dotted]
	(m-1-1.south west) -- (m-1-2.south east);
	\node[above,text depth=1pt] at ([xshift=3.5ex]m-1-1.north) {$\scriptstyle s$};  
	\node[left,overlay] at ([xshift=-1.2ex,yshift=-2ex]m-1-1.west) {$\scriptstyle s$};
	\end{tikzpicture}\mkern-5mu
	\Biggr],
	\]
	where $M_{12}=Q_1M'_{12}Q_2^{-1}$.
	
	Let $v_1,\dots,v_s$ be the column vectors of $M_{11}$ and $w_1,\dots,w_{n-s}$ be the column vectors of $M_{12}$. From $\det(M_{11})\neq 0$ we know that  $\{v_1,\dots,v_s\}$ forms a basis of the vector space $(\FF_p(\!(T)\!))^s$. Hence, for each integer $1\leq k\leq n-s$ there exist a set of scalars $\{a_1,\dots, a_s\}$ in  $ \FF_p(\!(T)\!)$ such that
	$
	w_k=\sum\limits_{j=1}^s a_jv_j.
	$
	By Cramer's rule, we have $a_j=\frac{\det(M_{11,k})}{\det(M_{11})}$, where $M_{11,k}$ is the $s\times s$ matrix obtained by replacing the $k$-th column vector of $M_{11}$ by $w_k$. Since $M_{11,k}$ is an $s\times s$ minor of $N$, by definition of Hodge functions, we have
	$v_T(\det(M_{11,k}))\geq v_T(\det(M_{11})),$ and hence $v_T(a_{j})\geq 0$ for every $j=1,\dots, s,$ which implies that $M_{12}$ is $\ula^{[s]}$-Hodge bounded. Let $W=P_2P_1\in \GL_n(\Fp\llbracket T \rrbracket)$. From the above computation, we obtain that $WMW^{-1}$ is $\ula^{[n]}$-Hodge bounded, and hence so is $WM$. 
	Combined with Lemma~\ref{L:second criterion of lambda-stable} and that $M$ is strictly $\ula^{[n]}$-Hodge bounded, this implies that $W$ is $\ula^{[n]}$-stable. 
\end{proof}

\begin{lemma}\label{L:valuations of eigenvalues bound Hodge functions}
	Fix an integer $n\geq 1$. Let $A\in M_n(\Fp\llbracket T \rrbracket)$ and $\ualpha:= \{\alpha_1,\dots,\alpha_n\}$ be the set of eigenvalues of $A$ in $\bbC$. Then there exists $b\in \RR$ such that 
	\[
	H(A^k,1)\geq ks+b \text{~and ~} H(A^{-k},1)\geq -kS+b
	\]
	for every $k\geq 1$, where
	\[
	S:=\max_{1\leq i\leq n}\{v_T(\alpha_i) \}\textrm{\ and \ } s:=\min_{1\leq i\leq n}\{v_T(\alpha_i) \}.
	\]
\end{lemma}

\begin{proof}
	There exists a matrix $P\in \GL_n(\bbC)$ such that $J=PAP^{-1}$ is in Jordan canonical form with diagonal entries $\alpha_1,\dots,\alpha_n$. Let $D\in\rmM_n(\bbC)$ be the diagonal matrix with $\alpha_1,\dots,\alpha_n$ on the diagonal and set $N:=J-D$. Note that $N$ is nilpotent and  commutes with $D$. Let $t:= \min\{m|N^m=0 \}$. From the decomposition 
	$$(D+N)^k=\sum_{i=0}^{k}\binom{k}{i} D^iN^{k-i},$$ 
	we have
	the following estimations
	\[
	H(J^k,1)\geq \begin{cases}
	sk &\text{~if~} H(N,1)\geq s,\\
	s(k-t+1)+(t-1)H(N,1)& \text{~if~} H(N,1)< s.
	\end{cases}
	\]
	Hence there exists $b_1\in \RR$ such that $H(J^k,1)\geq ks+b_1$ for every $k\geq 1$. Similarly, there exists $b_2\in \RR$ such that $H(J^{-k},1)\geq -kS+b_2$ for every $k\geq 1$.
	
	From $A^k=P^{-1}J^kP$ and $A^{-k}=P^{-1}J^{-k}P$, we have
	\[
	H(A^k,1)\geq H(P^{-1},1)+H(P,1)+H(J^k,1)\geq  H(P^{-1},1)+H(P,1)+b_1+ks,
	\]
	and 
	\[
	H(A^{-k},1)\geq  H(P^{-1},1)+H(P,1)+H(J^{-k},1)\geq  H(P^{-1},1)+H(P,1)+b_2-kS.
	\]
	Setting $b:=\min(b_1,b_2)+ H(P^{-1},1)+H(P,1)$, we complete the proof.
\end{proof}

\begin{lemma}\label{L:comparison of valuations of eigenvalues of (1,1)-block and (2,2)-block}
	Let  $1\leq s<n$ be two integers and \[M=\quad
	\Biggl[\mkern-5mu
	\begin{tikzpicture}[baseline=-.65ex]
	\matrix[
	matrix of math nodes,
	column sep=1ex,
	] (m)
	{
		A & B \\
		C & D \\
	};
	\draw[dotted]
	([xshift=0.5ex]m-1-1.north east) -- ([xshift=0.5ex]m-2-1.south east);
	\draw[dotted]
	(m-1-1.south west) -- (m-1-2.south east);
	\node[above,text depth=1pt] at ([xshift=2.5ex]m-1-1.north) {$\scriptstyle s$};  
	\node[left,overlay] at ([xshift=-1.2ex,yshift=-2ex]m-1-1.west) {$\scriptstyle s$};
	\end{tikzpicture}\mkern-5mu
	\Biggr]\in  M_n(\Fp\llbracket T \rrbracket)
	\] be a $\ula^{[n]}$-Hodge bounded matrix. 
	Assume that
	\begin{enumerate}
		\item $s$ is a touching vertex of $M$;
		\item $H(C,1)-\lambda_s>0$;
		\item $A$ is strictly $\ula^{[s]}$-Hodge bounded.
	\end{enumerate}
	Then we have
	\begin{equation}\label{E: inequality in comparison of valuations of eigenvalues of (1,1)-block and (2,2)-block}
	\max\{v_T(\alpha)\;|\; \alpha \text{~is an eigenvalue of ~} A \}<
	\min\{v_T(\beta)\;|\; \beta \text{~is an eigenvalue of ~} D \}.
	\end{equation}
\end{lemma}

\begin{proof}
	Write $$a:=\max\{v_T(\alpha)\;|\; \alpha \text{~is an eigenvalue of ~} A \} \textrm{\ and\ } b:=\min\{v_T(\beta)\;| \;\beta \text{~is an eigenvalue of ~} D \}.$$  From Remark~\ref{R:relation between slopes of Newton polygon and eigenvalues}, $a$ (resp.$b$) is the largest slope (resp. smallest slope) of the Newton polygon of the matrix $A$ (resp. $D$). Since Newton polygon always lies on or above Hodge polygon, we have $a\leq \lambda_s$ and $b\geq \lambda_{s+1}\geq \lambda_s$.

	Set \[M_0:=\quad
	\Biggl[\mkern-5mu
	\begin{tikzpicture}[baseline=-.65ex]
	\matrix[
	matrix of math nodes,
	column sep=1ex,
	] (m)
	{
		A & B \\
		0 & D \\
	};
	\draw[dotted]
	([xshift=0.5ex]m-1-1.north east) -- ([xshift=0.5ex]m-2-1.south east);
	\draw[dotted]
	(m-1-1.south west) -- (m-1-2.south east);
	\node[above,text depth=1pt] at ([xshift=2.5ex]m-1-1.north) {$\scriptstyle s$};  
	\node[left,overlay] at ([xshift=-1.2ex,yshift=-2ex]m-1-1.west) {$\scriptstyle s$};
	\end{tikzpicture}\mkern-5mu
	\Biggr]\in \rmM_n(\Fp\llbracket T\rrbracket).
	\]

	Suppose that \eqref{E: inequality in comparison of valuations of eigenvalues of (1,1)-block and (2,2)-block} fails.  Then we have $a=b=\lambda_s$. Denote by $k_1$ and $k_2$ with $k_1<k_2$ the $x$-coordinates of two endpoints of the line segment with the slope $\lambda_{s}$ on the Newton polygon of $M_0$. Then we have $k_1<s<k_2$, $\lambda_i=\lambda_s$ for $k_1+1\leq i\leq k_2$ and  
	\begin{equation}\label{NH}
	N(M_0, k_t)=H(M, k_t)=\sum_{i=1}^{k_t}\lambda_i \textrm{\ for}\ t=1,2,
	\end{equation}
	where the second equality follows from the hypothesis~(2).
	Now we prove \begin{equation}\label{v1}
	N(M,k_t)=N(M_0,k_t) \textrm{\ for\ } t=1,2. 
	\end{equation}

	For an integer $m\in [n]$, we denote by $F(M,m)$ the sum of all the principal $m\times m$ minors of $M$ and define $F(M_0,m)$ in the same way. We write $M=(m_{i,j})_{1\leq i,j\leq n}$ and have
	\begin{equation}\label{expansion}
	F(M,k_t)=	\sum_{\substack{I\subset[n]\\|I|=k_t}}\sum_{\sigma\in S_{k_t}}\mathrm{Sgn}(\sigma)\prod_{i\in I}m_{i,\sigma(i)}.
	\end{equation}
	Note that we have a similar description for $F(M_0,k_t)$. From our construction of $M_0$, the difference $F(M,k_t)-F(M_0,k_t)$ is the sum of terms $\mathrm{Sgn}(\sigma)\prod\limits_{i\in I}m_{i,\sigma(i)}$ such that at least one of the entries $\{m_{i,\sigma(i)}|i\in I \}$ belonging to the block $C$. Combined with
	the hypothesis~(2) and
	that $M$ is $\ula^{[n]}$-Hodge bounded, this implies
	$$v_T(F(M,k_t)-F(M_0,k_t))>\sum\limits_{i=1}^{k_t}\lambda_i=H(M,k_t). $$
	Together with \eqref{NH}, this inequality proves \eqref{v1}.
	
	Therefore, the point $(s,N(M,s))$ lies on the line segment with endpoints  $(k_t,N(M_0,k_t))$ for $t=1,2$, and hence cannot be a vertex of the Newton polygon of $M$, which is a contradiction to our  hypothesis~(1). 
\end{proof}

\begin{lemma}\label{L:estimation of Hodge functions of negative powers of a strictly Hodge bounded matrix}
	For every integer $n\geq 1$ and $A\in M_n(\Fp\llbracket T \rrbracket)$ if $A$ is strictly $\ula^{[n]}$-Hodge bounded, then for every integer $m\geq 1$ we have $H(A^{-m},1)\geq -m\lambda_n$.
\end{lemma}

\begin{proof}
	We denote by $A^{ad}=(a'_{k,l})\in M_n(\Fp\llbracket T \rrbracket)$ the adjoint matrix of $A$, i.e. $a'_{k,l}=(-1)^{k+l}\det(A_{k,l})$, where $A_{k,l}$ is the $(n-1)\times (n-1)$ minor of $A$ by removing the $k$-th row and $l$-th column. So we have $H(A^{ad},1)\geq H(A, n-1)$ and hence for every $m\geq 1$, $$H((A^{ad})^m,1)\geq mH(A^{ad},1)\geq mH(A,n-1)=m\sum\limits_{i=1}^{n-1}\lambda_i.$$
	Combined with $A^{-1}=\det(A)^{-1}A^{ad}$, this inequality implies  \[H(A^{-m},1)=H((A^{ad})^m,1)-mv_T(\det(A))\geq m\sum\limits_{i=1}^{n-1}\lambda_i-m\sum\limits_{i=1}^{n}\lambda_i=-m\lambda_n.\qedhere\]
\end{proof}

\begin{lemma}\label{new1}
	For every $n\in \NN$ if $A\in \GL_n(\FF_p[\![T]\!])$ and $B\in M_n(\FF_p[\![T]\!])$ such that $H(B,1)>0$, then $A+B\in \GL_n(\FF_p[\![T]\!])$.
\end{lemma}
\begin{proof}
	Let \[C:=\left(\sum_{i=0}^\infty (-A^{-1}B)^{i}\right)A^{-1}. \] From the hypothesis $H(B,1)>0$, we know that $C$ is well-defined and \[(A+B)C=C(A+B)=I_n.\qedhere\]
\end{proof}

\begin{lemma}\label{L:conjugating a 'good' matrix by lower triangular matrix into an upper triangular matrix}
	Let  $1\leq s<n$ be two integers and \[M=\quad
	\Biggl[\mkern-5mu
	\begin{tikzpicture}[baseline=-.65ex]
	\matrix[
	matrix of math nodes,
	column sep=1ex,
	] (m)
	{
		A & B \\
		C & D \\
	};
	\draw[dotted]
	([xshift=0.5ex]m-1-1.north east) -- ([xshift=0.5ex]m-2-1.south east);
	\draw[dotted]
	(m-1-1.south west) -- (m-1-2.south east);
	\node[above,text depth=1pt] at ([xshift=2.5ex]m-1-1.north) {$\scriptstyle s$};  
	\node[left,overlay] at ([xshift=-1.2ex,yshift=-2ex]m-1-1.west) {$\scriptstyle s$};
	\end{tikzpicture}\mkern-5mu
	\Biggr]\in  M_n(\Fp\llbracket T \rrbracket).
	\] be a $\ula^{[n]}$-Hodge bounded matrix. 
	
	Assume that $M$ satisfies all the three hypotheses in Lemma~\ref{L:comparison of valuations of eigenvalues of (1,1)-block and (2,2)-block}.
	Then there exists a matrix $X\in \rmM_{(n-s)\times s}(\Fp\llbracket T \rrbracket)$ with $H(X,1)\geq H(C,1)-\lambda_{s}$ such that 
	\[
	\begin{bmatrix}
	I_{s}&0\\
	X& I_{n-s}\end{bmatrix} \begin{bmatrix}
	A&B\\
	C& D\end{bmatrix}  \begin{bmatrix}
	I_{s}&0\\
	-X& I_{n-s}\end{bmatrix}=
	\
	\Biggl[\mkern-5mu
	\begin{tikzpicture}[baseline=-.65ex]
	\matrix[
	matrix of math nodes,
	column sep=1ex,
	] (m)
	{
		A' & B' \\
		0\ & D' \\
	};
	\draw[dotted]
	([xshift=0.5ex]m-1-1.north east) -- ([xshift=0.5ex]m-2-1.south east);
	\draw[dotted]
	(m-1-1.south west) -- (m-1-2.south east);
	\node[above,text depth=1pt] at ([xshift=2.5ex]m-1-1.north) {$\scriptstyle s$};  
	\node[left,overlay] at ([xshift=-1.2ex,yshift=-2ex]m-1-1.west) {$\scriptstyle s$};
	\end{tikzpicture}\mkern-5mu
	\Biggr],
	\]
	and $\begin{bmatrix}
	I_{s}&0\\
	X& I_{n-s}\end{bmatrix}$ is $\ula^{[n]}$-stable.
\end{lemma}

\begin{proof}
	We first prove the following claim.
	
	\begin{claim*}
		Let $\epsilon:= H(C,1)-\lambda_{s}$. 
		For any $\ula^{[n]}$-Hodge bounded matrix $
	M_k:=\ 	\Biggl[\mkern-5mu
		\begin{tikzpicture}[baseline=-.65ex]
			\matrix[
			matrix of math nodes,
			column sep=1ex,
			] (m)
			{
				A_k & B_k \\
				C_k & D_k\\
			};
			\draw[dotted]
			([xshift=0.5ex]m-1-1.north east) -- ([xshift=0.5ex]m-2-1.south east);
			\draw[dotted]
			(m-1-1.south west) -- (m-1-2.south east);
			\node[above,text depth=1pt] at ([xshift=2.5ex]m-1-1.north) {$\scriptstyle s$};  
			\node[left,overlay] at ([xshift=-1.2ex,yshift=-2ex]m-1-1.west) {$\scriptstyle s$};
		\end{tikzpicture}\mkern-5mu
		\Biggr]
$		
 that satisfies all three hypotheses in Lemma~\ref{L:comparison of valuations of eigenvalues of (1,1)-block and (2,2)-block} and $H(C_k,1)-\lambda_{s}\geq k\epsilon$ for some integer $k>0$, there exists a matrix $X_k\in \rmM_{(n-s)\times s}(\Fp\llbracket T \rrbracket)$ such that if we set
		\[
		M_{k+1}:=\ 
		\Biggl[\mkern-5mu
		\begin{tikzpicture}[baseline=-.65ex]
		\matrix[
		matrix of math nodes,
		column sep=1ex,
		] (m)
		{
			A_{k+1} & B_{k+1} \\
			C_{k+1} & D_{k+1} \\
		};
		\draw[dotted]
		([xshift=0.5ex]m-1-1.north east) -- ([xshift=0.5ex]m-2-1.south east);
		\draw[dotted]
		(m-1-1.south west) -- (m-1-2.south east);
		\node[above,text depth=1pt] at ([xshift=2.5ex]m-1-1.north) {$\scriptstyle s$};  
		\node[left,overlay] at ([xshift=-1.2ex,yshift=-2ex]m-1-1.west) {$\scriptstyle s$};
		\end{tikzpicture}\mkern-5mu
		\Biggr]
		=
		\begin{bmatrix}
		I_{s}&0\\
		X_k& I_{n-s}\end{bmatrix} M_k \begin{bmatrix}
		I_{s}&0\\
		-X_k& I_{n-s}\end{bmatrix},
		\]
		then
		\begin{enumerate}[(i)]
			\item $H(X_k,1)\geq k\epsilon$;
			\item $H(C_{k+1},1)-\lambda_{s}\geq (k+1)\epsilon$;
			\item the matrix $ \begin{bmatrix}
			I_{s}&0\\
			X_k& I_{n-s}\end{bmatrix}$ is $\ula^{[n]}$-stable;
			\item the matrix $M_{k+1}$ also satisfies all the three hypotheses in Lemma~\ref{L:comparison of valuations of eigenvalues of (1,1)-block and (2,2)-block}.   
		\end{enumerate}	
	\end{claim*} 
	We now prove that $X_k:=-\sum\limits_{i=0}^\infty D_k^iC_kA_k^{-i-1}$ satisfies all desired properties.	
	We first check that $X_k$ is well-defined. Let
	\[
	S_{A_k}:=\max\{v_T(\alpha)\;|\; \alpha \text{~is an eigenvalue of ~} A_k\} \text{~and~} s_{D_k}:=\min\{v_T(\beta)\;|\; \beta \text{~is an eigenvalue of ~} D_k\}.
	\]
	By Lemmas~\ref{L:valuations of eigenvalues bound Hodge functions} and ~\ref{L:comparison of valuations of eigenvalues of (1,1)-block and (2,2)-block}, we have $S_{A_k}<s_{D_k}$ and that there exists $b_k\in\RR$ such that $$H(A_k^{-i},1)\geq -iS_{A_k}+b_k \textrm{\ and\ } H(D_k^i,1)\geq is_{D_k}+b_k \textrm{\ for every\ }  i\geq 0.$$ Combining them, we have
	\begin{multline}\label{v2}
	H(D_k^iC_kA_k^{-i-1},1)\geq H(D_k^i,1)+H(C_k,1)+H(A_k^{-i-1},1)\\
	\geq is_{D_k}+b_k+\lambda_{s}+k\epsilon-(i+1)S_{A_k}+b_k\xrightarrow{i\rightarrow\infty} \infty.
	\end{multline}
	Therefore, the series $X_k=-\sum\limits_{i=0}^\infty D_k^iC_kA_k^{-i-1}$ converges and hence $X_k$ is well-defined. 
	
	By Lemma~\ref{L:estimation of Hodge functions of negative powers of a strictly Hodge bounded matrix}, we have $H(A_k^{-i-1},1)\geq -(i+1)\lambda_{s}$. Since $H(C_k,1)\geq \lambda_{s}+k\epsilon$ and $H(D_k,1)\geq \lambda_{s+1}\geq \lambda_{s}$, we have $$H(D_k^iC_kA_k^{-i-1},1)\geq i\lambda_{s}+\lambda_{s}+k\epsilon-(i+1)\lambda_{s}=k\epsilon,$$ and hence $H(X_k,1)\geq k\epsilon$, which proves $(\romannumeral1)$.
	
	From $X_kA_k+C_k-D_kX_k=0$, we have
	\[
	M_{k+1}=\begin{bmatrix}
	A_k-B_kX_k&B_k\\
	-X_kB_kX_k& D_k+X_kB_k\end{bmatrix}.
	\]
	Since $A_k$ is strictly $\ula^{[s]}$-Hodge bounded and $B_k$ is $\ula^{[s]}$-Hodge bounded, we have $A_k^{-1}B_k\in \rmM_{s\times(n-s)}(\Fp\llbracket T \rrbracket)$. Combined with $X_kB_k=-\sum\limits_{i=0}^{\infty}D_k^iC_kA_k^{-i}(A_k^{-1}B_k)$ and $$H(D_k^iC_kA_k^{-i}(A_k^{-1}B_k),1)\geq i\lambda_{s}+\lambda_{s}+k\epsilon-i\lambda_{s}+0=\lambda_{s}+k\epsilon \textrm{\ for every\ } i\geq 0,$$ this implies $$H(X_kB_k,1)\geq \lambda_{s}+k\epsilon,$$ and hence $$H(C_{k+1},1)=H(-X_kB_kX_k,1)\geq \lambda_{s}+k\epsilon+\epsilon=\lambda_{s}+(k+1)\epsilon.$$ This completes the proof of $(\romannumeral2)$.
	
	By Lemma~\ref{L:criterion of lambda-stable}, to show that the matrix $\begin{bmatrix}
	I_{s}&0\\
	X_k& I_{n-s}\end{bmatrix}$ is $\ula^{[n]}$-stable, it is enough to show that the matrix
	\begin{multline*}
	\begin{bmatrix}
	D(\ula^{[s]})&0\\
	0& D(\ula^{(s,n]})\end{bmatrix}^{-1}
	\begin{bmatrix}
	I_{s}&0\\
	X_k& I_{n-s}\end{bmatrix} 
	\begin{bmatrix}
	D(\ula^{[s]})&0\\
	0& D(\ula^{(s,n]})\end{bmatrix} 
	=
	\begin{bmatrix}
	I_{s}&0\\
	D(\ula^{(s,n]})^{-1}X_kD(\ula^{[s]})& I_{n-s}\end{bmatrix} 
	\end{multline*}
	belongs to $\rmM_n(\Fp\llbracket T\rrbracket)$, or equivalently, 
	\[
	D\left(\ula^{(s,n]}\right)^{-1}X_kD\left(\ula^{[s]}\right)\in \rmM_{(n-s)\times s}(\Fp\llbracket T \rrbracket).
	\]
	Note that \begin{equation}\label{v3}
	D(\ula^{(s,n]})^{-1}X_kD(\ula^{[s]})=-\sum\limits_{i=0}^\infty D(\ula^{(s,n]})^{-1}D_k^iC_kA_k^{-i}A_k^{-1}D(\ula^{[s]}).
	\end{equation}
	%
	Since $A_k$ is strictly $\ula^{[s]}$-Hodge bounded, we have $$A_k^{-1}D(\ula^{[s]})\in \GL_{s}(\Fp\llbracket T\rrbracket),\quad  (\textrm{i.e. }H(A_k^{-1}D(\ula^{[s]}),1)\geq 0). $$
	Combined with $$H(D_k^{i-1}C_kA_k^{-i},1)\geq (i-1)\lambda_{s}+\lambda_{s}+k\epsilon-i\lambda_{s}=k\epsilon>0 \textrm{\ for every\ } i\geq 1$$ and
	$$H(D(\ula^{(s,n]})^{-1}D_k,1)\geq 0,$$
	we have \begin{equation}\label{new2}
		D(\ula^{(s,n]})^{-1}D_k^iC_kA_k^{-i}A_k^{-1}D(\ula^{[s]})\in \rmM_{(n-s)\times s}(\Fp\llbracket T\rrbracket) \textrm{\ for every\ } i\geq 1.
	\end{equation}
	Note that by $H(D(\ula^{(s,n]})^{-1}C_k,1)\geq 0 $,  the belonging relation \eqref{new2} holds for $i=0$ as well. 
	Therefore,  further combining it with \eqref{v3}, we complete the proof of $(\romannumeral3)$.
	
	Since $M_k$ satisfies the hypothesis~(1) and
	conjugating by matrices in $\GL_n(\Fp\llbracket T \rrbracket)$ does not change either its Newton or Hodge polygons, we prove the hypothesis~(1) for $M_{k+1}$. 
	
	The hypothesis~(2) for $M_{k+1}$ follows directly from $(\romannumeral2)$. 
	
	Now we are left to prove hypothesis~(3) for $M_{k+1}$. Note that $$A_{k+1}=A_k+B_kX_k=A_k(I_{s}+A_k^{-1}B_kX_k). $$
	From our previous discussion, we know that $A_k^{-1}B_k\in \rmM_{s\times (n-s)}(\Fp\llbracket T\rrbracket)$. Combined with hypothesis~(2) on $M_k$ and Lemma~\ref{new1}, this implies $H(A_k^{-1}B_kX_k,1)\geq \epsilon>0$ and that $I_{s}+A_k^{-1}B_kX_k$ is invertible. Therefore, $A_{k+1}$ is strictly $\ula^{[s]}$-Hodge bounded, which proves $(\romannumeral4)$.
	
	Note that the matrix $M_1:=M$ satisfies all the conditions in our claim for $k=1$. Therefore, inductively using this claim, we obtain two infinite sequences of matrices $\underline X\in \rmM_{(n-s)\times s}(\Fp\llbracket T \rrbracket)$ and $\underline M\in \rmM_{(n-s)\times n}(\Fp\llbracket T \rrbracket)$ such that
	\begin{itemize}
		\item 	$H(X_k,1)\geq k\epsilon$;
		\item $M_k:=\quad\Biggl[\mkern-5mu
		\begin{tikzpicture}[baseline=-.65ex]
		\matrix[
		matrix of math nodes,
		column sep=1ex,
		] (m)
		{
			A_k & B_k\\
			C_k & D_k \\
		};
		\draw[dotted]
		([xshift=0.5ex]m-1-1.north east) -- ([xshift=0.5ex]m-2-1.south east);
		\draw[dotted]
		(m-1-1.south west) -- (m-1-2.south east);
		\node[above,text depth=1pt] at ([xshift=3.5ex]m-1-1.north) {$\scriptstyle s$};  
		\node[left,overlay] at ([xshift=-1.2ex,yshift=-2ex]m-1-1.west) {$\scriptstyle s$};
		\end{tikzpicture}\mkern-5mu
		\Biggr]=Y_kM_{k-1}Y_k^{-1}$, where $Y_k:= \begin{bmatrix}
		I_{s}&0\\
		X_k& I_{n-s}\end{bmatrix}$;
		\item $Y_k$ is $\ula^{[n]}$-stable;
		\item $H(C_k,1)\geq \lambda_{s}+k\epsilon$ for all $k\geq 1$.
	\end{itemize}
	Notice that $$Y_kY_{k-1}\dots Y_1= \begin{bmatrix}
	I_{s}&0\\
	\sum\limits_{i=1}^k X_i& I_{n-s}\end{bmatrix}.$$ Since $H(X_k,1)\geq k\epsilon$, the series $\sum\limits_{k=1}^\infty X_k$ converges to a matrix $X\in \rmM_{(n-s)\times s}(\Fp\llbracket T\rrbracket)$. From the above construction, we see that $X$ satisfies all the required properties.
\end{proof}

Now we state the Newton-Hodge decomposition theorem for matrices in $\rmM_\infty(\Fp\llbracket T\rrbracket)$.

\begin{theorem}\label{T:Newton-Hodge decomposition with Hodge bound restriction for infinite matrices}
	Let $M\in M_\infty(\Fp\llbracket T \rrbracket)$ be a $\ula$-Hodge bounded matrix and $\Omega=\{0=s_0<s_1<\cdots \}$ be an infinite subset of touching vertices of $M$.
	If 
	\[
	H(M,s)=\sum_{i=1}^{s}\lambda_i \text{~for all~} s\in \Omega,
	\]
	then there exists $W\in \GL_\infty(\Fp\llbracket T \rrbracket)$ such that 
	\begin{enumerate}
		\item $W$ is $\ula$-stable, and in particular $WMW^{-1}$ is $\ula$-Hodge bounded.
		\item $WMW^{-1}$ is a block upper triangular matrix of the form
		\[
		WMW^{-1}=\quad\Biggl[\mkern-5mu
		\begin{tikzpicture}[baseline=-.65ex]
		\matrix[
		matrix of math nodes,
		column sep=1ex,
		] (m)
		{
			N_{11} & N_{12}\\
			\ 0\ \  & N_{22} \\
		};
		\draw[dotted]
		([xshift=0.5ex]m-1-1.north east) -- ([xshift=0.5ex]m-2-1.south east);
		\draw[dotted]
		(m-1-1.south west) -- (m-1-2.south east);
		\node[above,text depth=1pt] at ([xshift=3.5ex]m-1-1.north) {$\scriptstyle s_1$};  
		\node[left,overlay] at ([xshift=-1.2ex,yshift=-2ex]m-1-1.west) {$\scriptstyle s_1$};
		\end{tikzpicture}\mkern-5mu
		\Biggr].
		\]
	\end{enumerate}

\end{theorem}

\begin{remark}\label{remark1}
	\begin{enumerate}
		\item Although the conclusion of Theorem~\ref{T:Newton-Hodge decomposition with Hodge bound restriction for infinite matrices} only refers to the touching vertex $s_1$, the assumption that there exist infinitely many touching vertices of $M$ is necessary, at least according to our proof. The reason is because in the Newton-Hodge Decomposition Theorem~\ref{T:Newton-Hodge decomposition} for a finite matrix $M\in \rmM_n(\FF_p\llbracket T\rrbracket)$, we have an automatic condition $H(M,n)=N(M,n)$, i.e. $(n,N(M,n))$ is always a touching vertex of $M$, while  we have no such condition for infinite matrix $M\in \rmM_\infty(\FF_p\llbracket T\rrbracket)$. Therefore we have to impose extra condition on the infinite $M$ to obtain the desired decomposition. Here we assume that $M$ has infinitely many touching vertices, which holds in our applications. It is an interesting question whether this assumption can be weakened.
		\item Since $\lim\limits_{n\rightarrow\infty}\lambda_n=\infty$, replacing $\Omega$ by a subset if necessary, we can assume that $\lambda_{s_1}<\lambda_{s_2}$. We will make this assumption throughout the proof of Theorem~\ref{T:Newton-Hodge decomposition with Hodge bound restriction for infinite matrices}.
	\end{enumerate}
	
\end{remark}

\begin{remark}\label{remark:infinite product of infinite matrices}
	One must be careful when compute the infinite product of infinite matrices. It is not true in general that the product of an countable infinitely many invertible matrices is invertible, even if the product is well-defined. For example, we denote by $Q_n\in \GL_\infty(\Fp\llbracket T\rrbracket)$ the infinite matrix obtained by switching the $1$-st and $n$-th rows of the identity matrix $\rmI_\infty$. In particular, set $Q_1=\rmI_\infty$. Then the sequence $(Q_n\cdots Q_1\;|\;n\geq 1)$ converges to the infinite matrix $ \begin{bmatrix}
	0&0\\
	\rmI_\infty&0\end{bmatrix}$, which is obviously not invertible. What we will use is the following result:
	
	Let $\{Q_n \}_{n\geq 1}$ be a family of matrices in $\GL_\infty(\Fp\llbracket T\rrbracket)$. Assume that there exists a sequence of integers $(\alpha_n)_{n\geq 1}$ such that $\lim\limits_{n\rightarrow \infty}\alpha_n=\infty$ and $Q_n-\rmI_\infty\in T^{\alpha_n}\rmM_\infty(\Fp\llbracket T\rrbracket)$. Then the sequence $(Q_n\cdots Q_1\;|\;n\geq 1)$ converges to a matrix $Q\in \GL_\infty(\Fp\llbracket T\rrbracket)$. In fact, the condition on $Q_n$'s implies that $Q_n^{-1}-\rmI_\infty\in T^{\alpha_n}\rmM_\infty(\Fp\llbracket T\rrbracket)$ and hence $(Q_1^{-1}\cdots Q_n^{-1}|n\geq 1 )$ converges to a matrix which gives the two-sided inverse of $Q$. Moreover, if $Q_n$ is $\ula$-stable for all $n\geq 1$, the matrix $Q$ is also $\ula$-stable.
\end{remark}

\begin{lemma}\label{L:permute an infinite matrix and compare Newton polygons of infinite matrix and its truncation}
	Let $M$ be the matrix in Theorem~\ref{T:Newton-Hodge decomposition with Hodge bound restriction for infinite matrices}. Then 
	\begin{enumerate}
		\item for every $k\geq 1$ there exists a $\ula$-stable matrix $P_k\in \GL_\infty(\Fp\llbracket T \rrbracket)$ such that if we write
		\[
		P_kMP_k^{-1}=\quad\Biggl[\mkern-5mu
		\begin{tikzpicture}[baseline=-.65ex]
		\matrix[
		matrix of math nodes,
		column sep=1ex,
		] (m)
		{
			M_{11} & M_{12}\\
			M_{21} & M_{22} \\
		};
		\draw[dotted]
		([xshift=0.5ex]m-1-1.north east) -- ([xshift=0.5ex]m-2-1.south east);
		\draw[dotted]
		(m-1-1.south west) -- (m-1-2.south east);
		\node[above,text depth=1pt] at ([xshift=3.5ex]m-1-1.north) {$\scriptstyle s_k$};  
		\node[left,overlay] at ([xshift=-1.2ex,yshift=-2ex]m-1-1.west) {$\scriptstyle s_k$};
		\end{tikzpicture}\mkern-5mu
		\Biggr],
		\]
		then 
		$
		v_T(\det(M_{11}))=\sum\limits_{i=1}^{s_k}\lambda_i.
		$
		\item There exists an integer $\ell_M$ such that for every $n\geq \ell_M$ and every $\ula$-stable matrix $P\in \GL_\infty(\FF_p[\![T]\!]) $ the Newton polygon of $(PMP^{-1})_{[n]}$ coincides with the Newton polygon of $M$ when one restricts the range of the $x$-coordinates to $[0,s_2]$. In particular, $s_1$ is a touching vertex of $M_{[n]}$.
	\end{enumerate}
\end{lemma}

\begin{proof}
	\noindent(1) Since $s_k$ is a touching vertex of $M$, we have 
	$$
	N(M,s_k)=H(M,s_k)=\sum\limits_{i=1}^{s_k}\lambda_i.
	$$
	Let
	$m_1:=\min \{j\;|\;\lambda_j=\lambda_{s_k} \}$ and $m_2:=\max \{j\;|\;\lambda_j=\lambda_{s_k} \}.$
	Since $M$ is $\ula$-Hodge bounded, there exists a principal $s_k\times s_k$ minor $N$ of $M$ such that
	$
	v_T(\det(N))=\sum\limits_{i=1}^{s_k}\lambda_i.
	$
	Moreover, $N$ contains $M_{[m_1-1]}$ as a principal minor and is contained in the principal minor $M_{[m_2]}$. So there exists a permutation matrix $P'\in \GL_{m_2-m_1}(\Fp\llbracket T \rrbracket)$ (that is, a matrix obtained by permuting the row vectors of the identity matrix) such that if we set 
	\[
	N':=\begin{bmatrix}
	I_{m_1}&0\\
	0&P'\end{bmatrix} M_{[m_2]} \begin{bmatrix}
	I_{m_1}&0\\
	0&P'\end{bmatrix}^{-1},
	\]
	then $N'_{[s_k]}=N$. 
	Set
	\[
	P:=\begin{bmatrix}
	I_{m_1}&0& 0\\
	0&P'&0\\
	0&0& I_\infty\end{bmatrix}.
	\]
	By Lemma~\ref{L:criterion of lambda-stable}, $P$ is $\ula$-stable. Clearly, the matrix $P$ has the desired property.
	
	\noindent(2) 
	From our hypothesis 
	$
	\lim\limits_{n\rightarrow\infty}\lambda_n=\infty
	$, there exists an integer $\ell_M> s_2$ such that 
	$\lambda_{\ell_M}> \sum\limits_{i=1}^{s_2}\lambda_i$, and
	$
	(\ell_M, \sum\limits_{i=1}^{\ell_M} \lambda_i)
	$
	lies above the straight line determined by the line segment in the Hodge polygon of $M$ with slope $\lambda_{s_2}$. Combined with that for every $\ula$-stable matrix $P\in \GL_\infty(\FF_p[\![T]\!]) $, $PMP^{-1}$ has the same Newton polygon to $M$  and  is $\ula$-Hodge bounded, these hypotheses on $\ell_M$ implies  that for every $n\geq \ell_M$ and  for every $0\leq j<s_2$
	we have 
	\[N((PMP^{-1})_{[n]}, j)\begin{cases}
	=N(PMP^{-1}, j)=N(M, j), & \textrm{if\ } N(M, j)\leq \sum_{i=1}^{s_2}\lambda_i, \\
	>\sum_{i=1}^{s_2}\lambda_i, &\textrm{else,}
	\end{cases}\]
	and
	$$N((PMP^{-1})_{[n]}, s_2)=N(PMP^{-1}, s_2)=N(M, s_2)=\sum_{i=1}^{s_2}\lambda_i,$$
	and hence complete the proof.
\end{proof}

\begin{notation}
	From now on, we denote by $s$ a fixed vertex of the Newton polygon of $M$ such that $s\geq \ell_M$ as in Lemma~\ref{L:permute an infinite matrix and compare Newton polygons of infinite matrix and its truncation}. We also set $\Delta:=\lambda_s-\lambda_{s_1}$.
\end{notation}

\begin{lemma}\label{L:conjugating a lambda Hodge bounded matrix to make lower left block closer to 0}
	Under the notations and conditions in Theorem~\ref{T:Newton-Hodge decomposition with Hodge bound restriction for infinite matrices}, there exists a $\ula$-stable matrix $Q\in \GL_\infty(\Fp\llbracket T\rrbracket )$ such that if we set 
	\[
	QMQ^{-1}:=
	\quad\Biggl[\mkern-5mu
	\begin{tikzpicture}[baseline=-.65ex]
	\matrix[
	matrix of math nodes,
	column sep=1ex,
	] (m)
	{
		A_{11} & A_{12}\\
		A_{21} & A_{22} \\
	};
	\draw[dotted]
	([xshift=0.5ex]m-1-1.north east) -- ([xshift=0.5ex]m-2-1.south east);
	\draw[dotted]
	(m-1-1.south west) -- (m-1-2.south east);
	\node[above,text depth=1pt] at ([xshift=3.5ex]m-1-1.north) {$\scriptstyle s_1$};  
	\node[left,overlay] at ([xshift=-1.2ex,yshift=-2ex]m-1-1.west) {$\scriptstyle s_1$};
	\end{tikzpicture}\mkern-5mu
	\Biggr],
	\]
	then we have $H(A_{21},1)\geq \lambda_{s}$ and $A_{11}$ is strictly $\lambda^{[s_1]}$-Hodge bounded.
\end{lemma}

\begin{proof}
	By  Lemma~\ref{L:permute an infinite matrix and compare Newton polygons of infinite matrix and its truncation}(1), we may assume $v_T(\det(M_{[s]}))=\sum_{i=1}^s \lambda_i.$
	Therefore, $M_{[s]}$ satisfies all the conditions in Theorem~\ref{T:Newton-Hodge decomposition with Hodge bound restriction for finite matrices}. By Theorem~\ref{T:Newton-Hodge decomposition with Hodge bound restriction for finite matrices},  there exists a $\ula^{[s]}$-stable matrix $P$ such that $PM_{[s]}P^{-1}$ is of the form
	\[
	PM_{[s]}P^{-1}=
	\quad
	\Biggl[\mkern-5mu
	\begin{tikzpicture}[baseline=-.65ex]
	\matrix[
	matrix of math nodes,
	column sep=1ex,
	] (m)
	{
		B_{11} & B_{12} \\
		0 & B_{22} \\
	};
	\draw[dotted]
	([xshift=0.5ex]m-1-1.north east) -- ([xshift=2ex]m-2-1.south east);
	\draw[dotted]
	(m-1-1.south west) -- (m-1-2.south east);
	\node[above,text depth=1pt] at ([xshift=3.5ex]m-1-1.north) {$\scriptstyle s_1$};  
	\node[left,overlay] at ([xshift=-1.2ex,yshift=-2ex]m-1-1.west) {$\scriptstyle s_1$};
	\end{tikzpicture}\mkern-5mu
	\Biggr],
	\]
	and $B_{11}$ is strictly $\lambda^{[s_1]}$-Hodge bounded.
	Let $Q:=\begin{bmatrix}
	P&0\\
	0&I_\infty\end{bmatrix}$. Clearly, $Q$ is $\ula$-stable as $P$ is $\ula^{[s]}$-stable. Since $M$ is $\ula$-Hodge bounded, the constructed matrix $Q$ satisfies the required properties.
\end{proof}

\begin{lemma}\label{L:inductively conjugating a lambda Hodge bounded matrix to make lower left block closer to 0}
	Let $M:=\quad\Biggl[\mkern-5mu
	\begin{tikzpicture}[baseline=-.65ex]
	\matrix[
	matrix of math nodes,
	column sep=1ex,
	] (m)
	{
		M_{11} & M_{12}\\
		M_{21} & M_{22} \\
	};
	\draw[dotted]
	([xshift=0.5ex]m-1-1.north east) -- ([xshift=0.5ex]m-2-1.south east);
	\draw[dotted]
	(m-1-1.south west) -- (m-1-2.south east);
	\node[above,text depth=1pt] at ([xshift=3.5ex]m-1-1.north) {$\scriptstyle s_1$};  
	\node[left,overlay] at ([xshift=-1.2ex,yshift=-2ex]m-1-1.west) {$\scriptstyle s_1$};
	\end{tikzpicture}\mkern-5mu
	\Biggr]\in M_\infty(\Fp\llbracket T \rrbracket)$ be the matrix in Theorem~\ref{T:Newton-Hodge decomposition with Hodge bound restriction for infinite matrices}. 
	If $M_{11}$ is strictly $\lambda^{[s_1]}$-Hodge bounded and $H(M_{21},1)\geq \lambda_{s}+i\Delta$ for some integer $i\geq 0$, then there exists a $\ula$-stable matrix $Q:=\begin{bmatrix}
	I_{s_1}&0\\
	Q'&I_\infty\end{bmatrix}\in \GL_\infty(\Fp\llbracket T \rrbracket)$ such that
	\begin{itemize}
		\item  $H(Q',1)\geq (i+1)\Delta$, and
		\item  if we set $$QMQ^{-1}:=\quad\Biggl[\mkern-5mu
		\begin{tikzpicture}[baseline=-.65ex]
		\matrix[
		matrix of math nodes,
		column sep=1ex,
		] (m)
		{
			A_{11} & A_{12}\\
			A_{21} & A_{22} \\
		};
		\draw[dotted]
		([xshift=0.5ex]m-1-1.north east) -- ([xshift=0.5ex]m-2-1.south east);
		\draw[dotted]
		(m-1-1.south west) -- (m-1-2.south east);
		\node[above,text depth=1pt] at ([xshift=3.5ex]m-1-1.north) {$\scriptstyle s_1$};  
		\node[left,overlay] at ([xshift=-1.2ex,yshift=-2ex]m-1-1.west) {$\scriptstyle s_1$};
		\end{tikzpicture}\mkern-5mu
		\Biggr],$$
		then $H(A_{21},1)\geq \lambda_s+(i+1)\Delta$.
	\end{itemize} 
\end{lemma}

\begin{proof}
	We set $s':=s-s_1$ and rewrite $M$ as
	\[
	M:=
	\quad\Biggl[\mkern-5mu
	\begin{tikzpicture}[baseline=-.65ex]
	\matrix[
	matrix of math nodes,
	column sep=1ex,
	] (m)
	{
		M_{11} & M_{12} &M_{13}\\
		M_{21} & M_{22} &M_{23} \\
		M_{31} & M_{32} &M_{33}\\
	};
	\draw[dotted]
	([xshift=0.5ex]m-1-1.north east) -- ([xshift=0.5ex]m-3-1.south east);
	\draw[dotted]
	([xshift=0.5ex]m-1-2.north east) -- ([xshift=0.5ex]m-3-2.south east);
	\draw[dotted]
	(m-1-1.south west) -- (m-1-3.south east);
	\draw[dotted]
	(m-2-1.south west) -- (m-2-3.south east);
	\node[above,text depth=1pt] at ([xshift=3.5ex]m-1-1.north) {$\scriptstyle s_1$};  
	\node[above,text depth=1pt] at ([xshift=3.5ex]m-1-2.north) {$\scriptstyle s$}; 
	\node[left,overlay] at ([xshift=-1.2ex,yshift=-2ex]m-1-1.west) {$\scriptstyle s_1$};
	\node[left,overlay] at ([xshift=-1.2ex,yshift=-2ex]m-2-1.west) {$\scriptstyle s$};
	\end{tikzpicture}\mkern-5mu
	\Biggr].
	\]
	Let 
	\[
	P_1:=
	\begin{bmatrix}
	I_{s_1} & 0 &0\\
	0& I_{s'} & 0 \\
	-M_{31}M_{11}^{-1} & 0 & I_\infty
	\end{bmatrix}.
	\]
	Then we have 
	\begin{multline*}
	M':=	\begin{bmatrix}
	M'_{11} & M'_{12} &M'_{13}\\
	M'_{21} & M'_{22} &M'_{23} \\
	M'_{31} & M'_{32} &M'_{33}
	\end{bmatrix}= P_1MP_1^{-1}
	\\
	=
	\begin{bmatrix}
	M_{11}+M_{13}M_{31}M_{11}^{-1} & M_{12} &M_{13}\\
	M_{21}+M_{23}M_{31}M_{11}^{-1} & M_{11} &M_{23} \\
	(M_{33}-M_{31}M_{11}^{-1}M_{13})M_{31}M_{11}^{-1} & M_{32}-M_{31}M_{11}^{-1}M_{12} & M_{33}-M_{31}M_{11}^{-1}M_{13}
	\end{bmatrix}.
	\end{multline*}
	
	From our hypotheses that $M_{11}$ is strictly $\ula^{[s_1]}$-Hodge bounded and that $M_{12}$, $M_{13}$ are $\ula^{[s_1]}$-Hodge bounded, we have 
	$M_{11}^{-1}M_{12}\in \rmM_{s_1\times s}(\Fp\llbracket T\rrbracket)$ and $M_{11}^{-1}M_{13}\in \rmM_{s_1\times \infty}(\Fp\llbracket T\rrbracket)$. 
	Combining $H(M_{31},1)\geq \lambda_{s}+i\Delta$ with Lemma~\ref{L:estimation of Hodge functions of negative powers of a strictly Hodge bounded matrix}, we have $$H(M_{31}M_{11}^{-1},1)\geq \lambda_{s}+i\Delta-\lambda_{s_1}=(i+1)\Delta.$$
	Therefore, $M'$ is $\ula$-Hodge bounded; and we have the estimations $$H(	M_{21}+M_{23}M_{31}M_{11}^{-1},1)\geq \lambda_{s}+i\Delta$$ and \begin{multline*}
	H(M'_{31},1)=H((-M_{31}M_{11}^{-1}M_{13}+M_{33})M_{31}M_{11}^{-1},1)\\
	\geq H(-M_{31}M_{11}^{-1}M_{13}+M_{33},1)+H(M_{31}M_{11}^{-1},1)
	\geq \lambda_{s}+(i+1)\Delta. 
	\end{multline*}
	
	By Lemmas~\ref{new1} and \ref{L:permute an infinite matrix and compare Newton polygons of infinite matrix and its truncation}(2), the matrix $M'_{[s]}= \begin{bmatrix}
	M'_{11}&M'_{12}\\
	M'_{21}&M'_{22}\end{bmatrix}$ satisfies all the conditions in Lemma~\ref{L:conjugating a 'good' matrix by lower triangular matrix into an upper triangular matrix} with $n:=s$ and  $s:=s_1$. Therefore, there exists a matrix $X\in \rmM_{s\times s_1}(\Fp\llbracket T\rrbracket)$ such that
	\begin{enumerate}
		\item $H(X,1)\geq H(M'_{21},1)-\lambda_{s_1}\geq\lambda_{s}+i\Delta-\lambda_{s_1}=(i+1)\Delta$; 
		\item The matrix $ \begin{bmatrix}
		I_{s_1}&0\\
		X&I_{s}\end{bmatrix}$ is $\ula^{[s]}$-stable;
		\item $\begin{bmatrix}
		I_{s_1}&0\\
		X&I_{s}\end{bmatrix} 
		\begin{bmatrix}
		M'_{11}&M'_{12}\\
		M'_{21}&M'_{22}\end{bmatrix}
		\begin{bmatrix}
		I_{s_1}&0\\
		-X&I_{s}\end{bmatrix}= \begin{bmatrix}
		M''_{11}&M''_{12}\\
		0&M''_{22}\end{bmatrix}$.
	\end{enumerate}
	Set $P_2:=
	\begin{bmatrix}
	I_{s_1} & 0 &0\\
	X& I_{s} &0 \\
	0 & 0 & I_{\infty}
	\end{bmatrix}.
	$
	Then $P_2\in\rmM_\infty(\Fp\llbracket T\rrbracket)$ is $\ula$-stable and 
	\[P_2MP^{-1}_2:=\begin{bmatrix}
	M''_{11} & M''_{12} &M'_{13}\\
	0& M''_{22} & M'_{23}+XM'_{33} \\
	M'_{31}-XM'_{32} & M'_{32} & M'_{33}.
	\end{bmatrix}.\]
	
	Combining the inequalities $$H(XM'_{32},1)\geq H(X,1)+H(M'_{32},1)\geq (i+1)\Delta+\lambda_{s}\textrm{\ and\ }H(M'_{31},1)\geq \lambda_{s}+(i+1)\Delta,$$ we have $H(M'_{31}-XM'_{32},1)\geq \lambda_{s}+(i+1)\Delta$.
	
	 Setting
	$$Q:= P_2P_1=\begin{bmatrix}
	I_{s_1} & 0 &0\\
	X& I_{s} &0 \\
	-M_{31}M_{11}^{-1} & 0 & I_{\infty}
	\end{bmatrix},$$ from the above discussion we know that $Q$ satisfies all the required properties.
\end{proof}

Now we can prove Theorem~\ref{T:Newton-Hodge decomposition with Hodge bound restriction for infinite matrices}, which is an easy consequence of the above lemmas. 

\begin{proof}[Proof of Theorem~\ref{T:Newton-Hodge decomposition with Hodge bound restriction for infinite matrices}]

	By Lemma~\ref{L:conjugating a lambda Hodge bounded matrix to make lower left block closer to 0}, there exists a $\ula$-stable matrix $Q_0\in\GL_\infty(\Fp\llbracket T\rrbracket)$ such that if we write 
	\[
	M_0:= Q_0MQ_0^{-1}=
	\quad\Biggl[\mkern-5mu
	\begin{tikzpicture}[baseline=-.65ex]
	\matrix[
	matrix of math nodes,
	column sep=1ex,
	] (m)
	{
		M_{0,11} & M_{0,12}\\
		M_{0,21} & M_{0,22} \\
	};
	\draw[dotted]
	([xshift=0.5ex]m-1-1.north east) -- ([xshift=0.5ex]m-2-1.south east);
	\draw[dotted]
	(m-1-1.south west) -- (m-1-2.south east);
	\node[above,text depth=1pt] at ([xshift=3.5ex]m-1-1.north) {$\scriptstyle s_1$};  
	\node[left,overlay] at ([xshift=-1.2ex,yshift=-2ex]m-1-1.west) {$\scriptstyle s_1$};
	\end{tikzpicture}\mkern-5mu
	\Biggr],
	\]
	then $H(M_{0,21},1)\geq \lambda_{s}=\lambda_{s}+0\cdot \Delta$ and $M_{0,11}$ is strictly $\lambda^{[s_1]}$-Hodge bounded. Then we can apply Lemma~\ref{L:inductively conjugating a lambda Hodge bounded matrix to make lower left block closer to 0} inductively on $k$, and get a sequence of $\ula$-stable matrices 
	\[
	\left(Q_k= \begin{bmatrix}
	I_{s_1}&0\\
	Q_k'&I_{\infty}\end{bmatrix}\in \GL_\infty(\Fp\llbracket T\rrbracket)\; \Big| \; k\geq 1\right)
	\]
	with $H(Q_k',1)\geq (k+1)\Delta$ such that  we write 
	\[
	M_k:=(Q_k\dots Q_1Q_0)M(Q_k\dots Q_1Q_0)^{-1}=
	\quad\Biggl[\mkern-5mu
	\begin{tikzpicture}[baseline=-.65ex]
	\matrix[
	matrix of math nodes,
	column sep=1ex,
	] (m)
	{
		M_{k,11} & M_{k,12}\\
		M_{k,21} & M_{k,22} \\
	};
	\draw[dotted]
	([xshift=0.5ex]m-1-1.north east) -- ([xshift=0.5ex]m-2-1.south east);
	\draw[dotted]
	(m-1-1.south west) -- (m-1-2.south east);
	\node[above,text depth=1pt] at ([xshift=3.5ex]m-1-1.north) {$\scriptstyle s_1$};  
	\node[left,overlay] at ([xshift=-1.2ex,yshift=-2ex]m-1-1.west) {$\scriptstyle s_1$};
	\end{tikzpicture}\mkern-5mu
	\Biggr],
	\]
	then $H(M_{k,21},1)\geq \lambda_{s}+k\Delta$. By Remarks~\ref{remark1} and ~\ref{remark:infinite product of infinite matrices}, the product $Q_k\dots Q_1Q_0$ converges as $k\rightarrow \infty$, and we denote by $W$  this limit. Then $W\in \GL_\infty(\Fp\llbracket T\rrbracket)$ is $\ula$-stable and 
	\[
	WMW^{-1}=
	\quad
	\Biggl[\mkern-5mu
	\begin{tikzpicture}[baseline=-.65ex]
	\matrix[
	matrix of math nodes,
	column sep=1ex,
	] (m)
	{
		N_{11} & N_{12} \\
		0 & N_{22} \\
	};
	\draw[dotted]
	([xshift=0.5ex]m-1-1.north east) -- ([xshift=2ex]m-2-1.south east);
	\draw[dotted]
	(m-1-1.south west) -- (m-1-2.south east);
	\node[above,text depth=1pt] at ([xshift=3.5ex]m-1-1.north) {$\scriptstyle s_1$};  
	\node[left,overlay] at ([xshift=-1.2ex,yshift=-2ex]m-1-1.west) {$\scriptstyle s_1$};
	\end{tikzpicture}\mkern-5mu
	\Biggr].\qedhere
	\]

\end{proof}

\subsection{Newton-Hodge decomposition for matrices over certain non-commutative rings.}

In this subsection,	we will denote by  $R$ a ring with unit but not necessarily commutative. Let $\gothm\subset R$ be a two-sided ideal, and $\tilde{T}\in R$ a central element which is not a zero divisor of $R$. We assume that
\begin{enumerate}
	\item $R$ is complete under the $(\gothm,\tilde{T})$-adic topology, where $(\gothm,\tilde{T})$ is the (two-sided) ideal of $R$ generated by $\gothm$ and $\tilde{T}$, i.e. we have 
	$
	R\cong \varprojlim\limits_{n\to \infty} R/(\gothm, \tilde T)^n.
	$
	\item There exists a surjective ring homomorphism $\pi:R\rightarrow \Fp\llbracket T \rrbracket$ with $\ker(\pi)=\gothm$ and $\pi(\tilde{T})=T$.
	\item For $x\in R\setminus \{0\}$, we define $v_{\tilde{T}}(x):=\sup\{n|x\in \tilde{T}^nR \}$ and $v_{\tilde{T}}(0):=\infty$. Then for each $\bar{x}\in \Fp\llbracket T \rrbracket$, there exists $x\in R$ such that $\pi(x)=\bar{x}$ and $v_{\tilde{T}}(x)=v_T(\bar{x})$. The element $x$ with such properties is called a special lift of $\bar{x}$.
\end{enumerate} 

\begin{remark}
	\begin{enumerate}
		\item The number $v_{\tilde{T}}(x)=\sup\{n|x\in \tilde{T}^nR \}$ is finite for $x\neq 0$, and we have the inequality $v_{\tilde{T}}(x)\leq v_T(\bar{x})$ with $\bar{x}=\pi(x)$.
		\item Since $\tilde{T}$ is central and not a zero divisor of $R$, we can consider the localization $R[\frac{1}{\tilde{T}}]$ of $R$ with respect to the multiplicative set $\{\tilde{T}^n|n\geq 0 \}$. The natural homomorphism $R\rightarrow R[\frac{1}{\tilde{T}}]$ is injective.
		\item Since $\tilde{T}$ is not a zero divisor of $R$, if $x\in \tilde{T}^k R$, then there exists a unique $x'\in R$ with $x=\tilde{T}^kx'$. We write $x'=\tilde{T}^{-k}x$. 
	\end{enumerate}
\end{remark}	

\begin{lemma}\label{L:T-divisible property in gothm}
	Let $x\in \tilde{T}^kR$ for some $k\geq 1$. If $x\in \gothm$, then $\tilde{T}^{-k}x\in \gothm$.
\end{lemma}

\begin{proof}
	We write $x=\tilde{T}^ky$ for some $y\in R$. Applying the homomorphism $\pi$ on both sides of this equality, we get $\pi(x)=T^k\pi(y)=0$, and hence $\pi(y)=0$. This implies $y\in \gothm$. 
\end{proof}

\begin{definition}\label{D:special lift of matrices}
	For a matrix $\bar{A}=(\bar{a}_{ij})\in \rmM_{m\times n}(\Fp\llbracket T\rrbracket)$, we call $A=(a_{ij})\in \rmM_{m\times n}(R)$ a \emph{special lift} of $\bar{A}$ if $a_{ij}$ is a special lift of $\bar{a}_{ij}$ for every $1\leq i\leq m$ and $1\leq j\leq n$.
\end{definition}

We denote by  $D_R(\ula^{[n]})$ (resp. $D_R(\ula)$)  the diagonal matrix $\mathrm{Diag}(\tilde{T}^{\lambda_1},\dots,\tilde{T}^{\lambda_n})\in \rmM_n(R)$ (resp. $\mathrm{Diag}(\tilde{T}^{\lambda_1},\dots)\in \rmM_\infty(R)$). For two integers  $m<n$ we denote by  $D_R(\ula^{(m,n]})$  the diagonal matrix $\mathrm{Diag}(\tilde{T}^{\lambda_{m+1}},\dots,\tilde{T}^{\lambda_n})\in \rmM_{n-m}(R)$.
Similar to Definition~\ref{D:lambda Hodge bounded and lambda stable}, we define $\ula$-Hodge bounded matrices and $\ula$-stable matrices over $R$ as follows.

\begin{definition}\label{D:lambda Hodge bounded and lambda stable over noncommutative rings}
	Let $m,n$ be two positive integers.
	\begin{enumerate}
		\item  A matrix $M\in \rmM_{n\times m}(R)$ is called $\ula^{[n]}$-Hodge bounded with respect to $\tilde{T}\in R$ if $M=D_R(\ula^{[n]})M'$ for some $M'\in \rmM_{n\times m}(R)$. More generally, for two integers $0\leq n_1<n_2$ a matrix $M\in \rmM_{(n_2-n_1)\times m}(R)$ is called \emph{$\ula^{(n_1,n_2]}$-Hodge bounded} with respect to $\tilde{T}\in R$ if $M=D_R(\ula^{(n_1,n_2]})M'$ for some $M'\in \rmM_{(n_2-n_1)\times m}(R)$.
		\item A matrix $M\in \rmM_\infty (R)$ is called \emph{$\ula$-Hodge bounded} with respect to $\tilde{T}\in R$ if $M=D_R(\ula)M'$ for some $M'\in \rmM_{\infty}(R)$.
		\item  A matrix $A\in \rmM_n(R)$ is called \emph{$\ula^{[n]}$-stable} with respect to $\tilde{T}\in R$ if for every $\ula^{[n]}$-Hodge bounded matrix $B\in \rmM_n(R)$, $AB$ is also $\ula^{[n]}$-Hodge bounded.
		\item A matrix $A\in \rmM_\infty(R)$ is called \emph{$\ula$-stable} with respect to $\tilde{T}\in R$ if for every $\ula$-Hodge bounded matrix $B\in \rmM_\infty(R)$, $AB$ is also $\ula$-Hodge bounded.
	\end{enumerate}
\end{definition}	

\begin{convention}
	In the rest of this section, we will fix an element $\tilde{T}\in R$. When we say that a matrix is $\ula^{[n]}$-Hodge bounded or $\ula^{[n]}$-stable, we mean that it is so with respect to $\tilde{T}$. 
\end{convention}	

Since determinants do not behave well for matrices over noncommutative rings, we do not define general Newton functions and Hodge functions for matrices over $R$. But we make the following.
\begin{definition}\label{D:Hodge function for matrices over noncommutative rings}
	For a nonzero matrix $A=(a_{i,j})_{1\leq i\leq m,1\leq j\leq n}\in \rmM_{m\times n}(R)$, we set $$H(A,1):=\min \left\{ v_{\tilde{T}}(a_{i,j})\;|\;1\leq i\leq m,1\leq j\leq n \right\},$$ which is a well-defined integer. If $H(A,1)\geq k$ for some $k\in\ZZ_{\geq 0}$, we denote by  $\tilde{T}^{-k}A$  the (unique) matrix $N$ in $\rmM_{m\times n}(R)$ satisfying $\tilde{T}^k N=A$.
\end{definition}

We also have a similar criterion of $\ula$-stability as Lemma~\ref{L:criterion of lambda-stable}:
\begin{lemma}\label{L:criterion of lambda-stable over noncommutative rings}
	For $n\in [\infty]$ and  a matrix $A=(a_{ij})\in \rmM_n(R)$, the following statements are equivalent:
	\begin{enumerate}
		\item $A$ is $\ula^{[n]}$-stable.
		\item There exists a matrix $B\in \rmM_n(R)$ such that  $AD_R(\ula^{[n]})=D_R(\ula^{[n]})B$.
		\item For all $i>j$, we have $v_{\tilde{T}}(a_{ij})\geq \lambda_i-\lambda_j$, i.e. $a_{ij}\in \tilde{T}^{\lambda_i-\lambda_j}R$.
	\end{enumerate}
\end{lemma}
The proof is almost identical to that of Lemma~\ref{L:criterion of lambda-stable} so we omit the proof here, and it has the following direct consequence.
\begin{corollary}\label{C:special lift of lambda stable is lambda stable}
	Let $\bar{A}\in \rmM_n(\Fp\llbracket T\rrbracket)$ be a $\ula^{[n]}$-stable matrix. If $A\in \rmM_n(R)$ is a special lift of $\bar{A}$, then $A$ is $\ula^{[n]}$-stable.
\end{corollary}

\begin{lemma}\label{L:any lift of an invertible matrix is invertible}
	Let $n\in[\infty]$ and $\bar{P}\in \GL_n(\Fp\llbracket T\rrbracket)$. Then any lift matrix $P\in \rmM_n(R)$ of $\bar{P}$ is invertible.
\end{lemma}

\begin{proof}
	From $\bar{P}\in \GL_n(\Fp\llbracket T\rrbracket)$, there exists $\bar{Q}\in \GL_n(\Fp\llbracket T\rrbracket)$ such that $\bar{P}\bar{Q}=\bar{Q}\bar{P}=\rmI_n$. Let $Q$ be an arbitrary lift of $\bar{Q}$. Then $B:=\rmI_n-PQ\in \rmM_n(\gothm)$ and the series $B':=\sum\limits_{i=0}^\infty B^i$ converges. By a direct computation, we have $PQB'=\rmI_n$. Similarly, let $C:=\rmI_n-QP\in \rmM_n(\gothm)$ and the series $C':=\sum\limits_{i=0}^\infty C^i$ converges. So we have $C'QP=\rmI_n$ and hence $P$ is invertible in $\rmM_n(R)$.
\end{proof}

\begin{proposition}\label{P:conjugating a 'good' matrix by lower triangular matrix into an upper triangular matrix over noncommutative rings}
	Let $n_1<n$ be two positive integers, $M:=\quad
	\Biggl[\mkern-5mu
	\begin{tikzpicture}[baseline=-.65ex]
	\matrix[
	matrix of math nodes,
	column sep=1ex,
	] (m)
	{
		A & B \\
		C & D \\
	};
	\draw[dotted]
	([xshift=0.5ex]m-1-1.north east) -- ([xshift=0.5ex]m-2-1.south east);
	\draw[dotted]
	(m-1-1.south west) -- (m-1-2.south east);
	\node[above,text depth=1pt] at ([xshift=2.5ex]m-1-1.north) {$\scriptstyle n_1$};  
	\node[left,overlay] at ([xshift=-1.2ex,yshift=-2ex]m-1-1.west) {$\scriptstyle n_1$};
	\end{tikzpicture}\mkern-5mu
	\Biggr]\in \rmM_n(R)$ be a $\ula^{[n]}$-Hodge bounded matrix, and $\bar{M}\in \rmM_n(\Fp\llbracket T\rrbracket)$ be its reduction by $\gothm$. Fix an integer $\alpha\geq \lambda_{n_1}$. Assume that 
	\begin{enumerate}
		\item  $H(C,1)\geq \alpha$ and  $C\equiv 0\pmod \gothm$;
		\item $n_1$ is a touching vertex of $\bar{M}$, and $\bar{A}$ is strictly $\ula^{[n_1]}$-Hodge bounded.
	\end{enumerate}
	Then there exists a matrix $X\in \rmM_{(n-n_1)\times n_1}(\gothm)$ with $H(X,1)\geq \alpha-\lambda_{n_1}$ such that if we denote $Y:= \begin{bmatrix}
	I_{n_1}&0\\
	X&I_{n-n_1}\end{bmatrix} $, then $Y$ is $\ula^{[n]}$-stable and $YMY^{-1}$ is of the form
	\[
	YMY^{-1}=
	\quad
	\Biggl[\mkern-5mu
	\begin{tikzpicture}[baseline=-.65ex]
	\matrix[
	matrix of math nodes,
	column sep=1ex,
	] (m)
	{
		A' & B' \\
		0\  & D' \\
	};
	\draw[dotted]
	([xshift=0.5ex]m-1-1.north east) -- ([xshift=0.5ex]m-2-1.south east);
	\draw[dotted]
	(m-1-1.south west) -- (m-1-2.south east);
	\node[above,text depth=1pt] at ([xshift=2.5ex]m-1-1.north) {$\scriptstyle n_1$};  
	\node[left,overlay] at ([xshift=-1.2ex,yshift=-2ex]m-1-1.west) {$\scriptstyle n_1$};
	\end{tikzpicture}\mkern-5mu
	\Biggr].
	\]
\end{proposition}
First we need the following lemma.

\begin{lemma}\label{L:conjugating a lambda Hodge bounded matrix to make lower left block closer to 0 over noncommutative rings}
	Let $M\in \rmM_n(R)$ be the matrix in Proposition~\ref{P:conjugating a 'good' matrix by lower triangular matrix into an upper triangular matrix over noncommutative rings}. Assume further that $\tilde{T}^{-\alpha}C\equiv 0 \pmod {\gothm^k}$ for some integer $k\geq 1$. Then there exists a matrix $X_k\in \rmM_{(n-n_1)\times n_1}(R)$ such that  
	\begin{enumerate}
		\item $H(X_k,1)\geq \alpha-\lambda_{n_1}$ and $\tilde{T}^{\lambda_{n_1}-\alpha}X_k\equiv 0 \pmod {\gothm^k}$.
		\item The matrix $Y_k:=  \begin{bmatrix}
		I_{n_1}&0\\
		X_k&I_{n-n_1}\end{bmatrix}$ is $\ula^{[n]}$-stable.
		\item If we write 
		\[
		Y_kMY_k^{-1}:=
		\quad
		\Biggl[\mkern-5mu
		\begin{tikzpicture}[baseline=-.65ex]
		\matrix[
		matrix of math nodes,
		column sep=1ex,
		] (m)
		{
			A' & B' \\
			C' & D' \\
		};
		\draw[dotted]
		([xshift=0.5ex]m-1-1.north east) -- ([xshift=0.5ex]m-2-1.south east);
		\draw[dotted]
		(m-1-1.south west) -- (m-1-2.south east);
		\node[above,text depth=1pt] at ([xshift=2.5ex]m-1-1.north) {$\scriptstyle n_1$};  
		\node[left,overlay] at ([xshift=-1.2ex,yshift=-2ex]m-1-1.west) {$\scriptstyle n_1$};
		\end{tikzpicture}\mkern-5mu
		\Biggr],
		\]
		then $H(C',1)\geq \alpha$ and $\tilde{T}^{-\alpha}C'\equiv 0 \pmod {\gothm^{k+1}}$.
	\end{enumerate}
\end{lemma}

\begin{proof}
	
	Since $M$ is $\ula^{[n]}$-Hodge bounded, so is its reduction $\bar{M}$. By Lemma~\ref{L:comparison of valuations of eigenvalues of (1,1)-block and (2,2)-block}, we have $S_{\bar{A}}<s_{\bar{D}}$, where $$S_{\bar{A}}:=\max \{v_T(\beta)\;|\;\beta \text{~is an eigenvalue of ~} \bar{A}\}\textrm{\ and\ }s_{\bar{D}}:=\max \{v_T(\beta)\;|\;\beta \text{~is an eigenvalue of ~} \bar{D}\}.$$ By Lemma~\ref{L:valuations of eigenvalues bound Hodge functions}, there exists $b\in\RR$ such that for every $i\geq 1$ we have $$H(\bar{A}^{-i},1)\geq -iS_{\bar{A}}+b\textrm{\ and\ }H(\bar{D}^i,1)\geq is_{\bar{D}}+b. $$
	
	For every integer $i\geq 1$, we set $t_i:=i\lambda_{n_1}$ and construct matrices as follows.
	\begin{itemize}
		\item Since $T^{t_i}\bar{A}^{-i}$ belongs to $\rmM_{n_1}(\Fp\llbracket T\rrbracket)$, we can find a special lift $A'(-i)\in \rmM_{n_1}(R)$ of $T^{t_i}\bar{A}^{-i}$ and set $A(-i):=\tilde{T}^{-t_i}A'(-i)\in \rmM_{n_1}(R[\frac{1}{\tilde{T}}])$.
		\item We fix a special lift  $D(i)\in \rmM_{n-n_1}(R)$   of $\bar{D}^i\in \rmM_{n-n_1}(\FF_p\llbracket T \rrbracket)$.
		\item To simplify our notations, we set $A(0):=I_{n_1}\in\rmM_{n_1}(R)$ and $D(0):=I_{n-n_1}\in\rmM_{n-n_1}(R)$.
		\item Since $A$ is $\underline{\lambda}^{[n_1]}$-Hodge bounded, we can write $A=D_R(\ula^{[n_1]})\cdot A'$ for some $A'\in \rmM_{n_1}(R)$. As the reduction $\bar{A}$ of $A$ by $\gothm$ is strictly $\ula^{[n_1]}$-Hodge bounded, the reduction $\bar{A}'$ of $A'$ by $\gothm$ is invertible. It follows from Lemma~\ref{L:any lift of an invertible matrix is invertible} that $A'\in \GL_{n_1}(R)$. We define $\tilde{A}^{-1}:= A'^{-1}\cdot D_R(\ula^{[n_1]})^{-1}\in \rmM_{n_1}(R[\frac{1}{\tilde{T}}])$. Although $\tilde{A}^{-1}$ satisfies $A\cdot \tilde{A}^{-1}=\tilde{A}^{-1}\cdot A=I_{n_1}$, $A$ is not invertible in $\rmM_{n_1}(R)$ and $\tilde{A}^{-1}$ is the inverse of $A$ in $\rmM_{n_1}(R[\frac{1}{\tilde{T}}])$. To emphasize this fact, we use the notation $\tilde{A}^{-1}$ instead of $A^{-1}$. Finally we remark that $\tilde{T}^{\lambda_{n_1}}\tilde{A}^{-1}\in \rmM_{n_1}(R)$.
	\end{itemize}
	The  construction above give us $$H(\tilde{T}^{t_i}A(-i),1)=H(A'(-i),1)=H(T^{t_i}\bar{A}^{-i},1)=t_i+H(\bar{A}^{-i},1)\textrm{\ and\ }H(D(i),1)=H(\bar{D}^i,1).$$ 
	
	Set
	\begin{multline}\label{E:equation to define X_k in L:conjugating a lambda Hodge bounded matrix to make lower left block closer to 0 over noncommutative rings}
	X_k:=-\left(C\tilde{A}^{-1}+\sum_{i=1}^\infty DD(i-1)CA(-i)\tilde{A}^{-1}\right)
	\\=-\left(C+\sum_{i=1}^\infty DD(i-1)CA(-i)\right)\tilde{A}^{-1}.
	\end{multline}

	We first verify that each term $S$ in the sum of \eqref{E:equation to define X_k in L:conjugating a lambda Hodge bounded matrix to make lower left block closer to 0 over noncommutative rings}  satisfies $H(S,1)\geq \alpha-\lambda_{n_1}$ and in particular belongs to $\rmM_{(n-n_1)\times n_1}(R)$. 
	%
	In fact, for every integer $i\geq 1$ we have 
	\begin{multline}\label{v4}
	H(D(i-1)CA(-i),1)\geq H(D(i-1),1)+H(C,1)+H(A(-i),1)\\
	=H(\bar{D}^{i-1},1)+H(C,1)+H(\bar{A}^{-i},1)
	\geq (i-1)\lambda_{n_1}+\alpha-i\lambda_{n_1}
	=\alpha-\lambda_{n_1},
	\end{multline}
	where the second inequality is from  Lemma~\ref{L:estimation of Hodge functions of negative powers of a strictly Hodge bounded matrix}.
	
	Combined with the fact that $\tilde{T}^{-\lambda_{n_1}}D\in \rmM_{n-n_1}(R)$ and $\tilde{T}^{\lambda_{n_1}}\tilde{A}^{-1}\in\rmM_{n_1}(R)$, this chain of inequality implies  $H(DD(i-1)CA(-i)\tilde{A}^{-1},1)\geq \alpha-\lambda_{n_1}$
	. On the other hand, from $H(C,1)\geq \alpha\geq \lambda_{n_1}$ and $\tilde{T}^{\lambda_{n_1}}\tilde{A}^{-1}\in \rmM_{n_1}(R)$, we have $$C\tilde{A}^{-1}\in \rmM_{(n-n_1)\times n_1}(R)\textrm{\ and\ }H(C\tilde{A}^{-1},1)\geq \alpha-\lambda_{n_1}.  $$
	
	Now we verify the convergence of the infinite series in \eqref{E:equation to define X_k in L:conjugating a lambda Hodge bounded matrix to make lower left block closer to 0 over noncommutative rings}. From $s_{\bar{D}}>S_{\bar{A}}$, we have 
	\begin{multline*}
	H(DD(i-1)CA(-i)\tilde{A}^{-1},1)\geq \lambda_{n_1}+(i-1)s_{\bar{D}}+b+\alpha-iS_{\bar{A}}+b-\lambda_{n_1}
	\\
	=i(s_{\bar{D}}-S_{\bar{A}})-s_{\bar{D}}+2b+\alpha\xrightarrow{i\rightarrow \infty} \infty.
	\end{multline*}
	Therefore, the series in \eqref{E:equation to define X_k in L:conjugating a lambda Hodge bounded matrix to make lower left block closer to 0 over noncommutative rings} converges to a matrix $X_k\in \rmM_{(n-n_1)\times n_1}(R)$, and \begin{equation}\label{v6}
	H(X_k, 1)\geq \alpha-\lambda_{n_1}.
	\end{equation}
	Moreover, from $\tilde{T}^{-\alpha}C\equiv 0\pmod {\gothm^k}$, we have 
	\begin{equation}\label{v10}
	\tilde{T}^{\lambda_{n_1}-\alpha}X_k\equiv 0\pmod {\gothm^k}\textrm{\ and\ }\tilde{T}^{-\alpha}X_kD_R(\ula^{(n_1,n]})\equiv 0\pmod {\gothm^k}.
	\end{equation}	
	Now we set $Y_k:= \begin{bmatrix}
	I_{n_1}&0\\
	X_k&I_{n-n_1}\end{bmatrix}$, and have $$Y_kD_R(\ula^{[n]})=\begin{bmatrix}
	D_R(\ula^{[n_1]})&0\\
	X_kD_R(\ula^{[n_1]})&D_R(\ula^{(n_1,n]})\end{bmatrix}.$$ From $\tilde{A}^{-1}=A'^{-1}D_R(\ula^{[n_1]})^{-1}$, we have $$X_kD_R(\ula^{[n_1]})=-\left(C+\sum\limits_{i=1}^\infty DD(i-1)CA(-i)\right)A'^{-1}.$$ 
	Combined with \eqref{v4} and that both $C$ and $D$ are $\ula^{(n_1,n]}$-Hodge bounded, this equality implies  that $X_kD_R(\ula^{[n_1]})$ is $\ula^{(n_1,n]}$-Hodge bounded and \begin{equation}\label{v7}
	H(X_kD_R(\ula^{[n_1]}),1)\geq \alpha.
	\end{equation} Hence there exists $Y'_k\in \rmM_n(R)$ such that $Y_kD_R(\ula^{[n]})=D_R(\ula^{[n]})Y'_k$. By Lemma~\ref{L:criterion of lambda-stable over noncommutative rings},  $Y_k$ is $\ula^{[n]}$-stable.

	Note that
	\[
	Y_kMY_k^{-1}= \begin{bmatrix}
	A-BX_k&B\\
	X_kA+C-DX_k-X_kBX_k&X_kB+D\end{bmatrix}.
	\]
	We set $C':=X_kA+C-DX_k-X_kBX_k$ and need to prove that $$H(C',1)\geq \alpha\textrm{\ and \ } \tilde{T}^{-\alpha}C'\equiv 0\pmod {\gothm^{k+1}}.$$ The first inequality follows directly from $H(C,1)\geq \alpha$, \eqref{v6} and \eqref{v7}. It remains to prove the second congruence relation. From our construction, for every $i\geq 1$ we have
	\begin{equation*}
	\begin{split}
	&\tilde{A}^{-1}A\equiv I_{n_1}\pmod \gothm\\
	&\tilde{T}^{(i+1)\lambda_{n_1}}A(-i)\tilde{A}^{-1}\equiv \tilde{T}^{(i+1)\lambda_{n_1}}A(-i-1)\tilde{A}^{-1}A \pmod {\gothm},\\
	&\tilde{T}^{-(i+1)\lambda_{n_1}}D^2D(i-1)\equiv \tilde{T}^{-(i+1)\lambda_{n_1}}DD(i)\pmod {\gothm}.
	\end{split}
	\end{equation*}

	Combining these congruence relations above with $\tilde{T}^{-\alpha}C\equiv 0\pmod{\gothm^k}$, we have \begin{equation}\label{v8}
	\tilde{T}^{-\alpha}D^2D(i-1)CA(-i)\tilde{A}^{-1}\equiv \tilde{T}^{-\alpha}DD(i)CA(-i-1)\tilde{A}^{-1}A \pmod {\gothm^{k+1}}
	\end{equation} and \begin{equation}\label{v5}
	\tilde{T}^{-\alpha}C\equiv \tilde{T}^{-\alpha}C\tilde{A}^{-1}A \pmod {\gothm^{k+1}}. 
	\end{equation}
	By $\tilde{T}^{\lambda_{n_1}}A(-1)\tilde{A}^{-1}A\equiv \tilde{T}^{\lambda_{n_1}}\tilde{A}^{-1}\pmod \gothm,$ we have \begin{equation}\label{v9}
	\tilde{T}^{-\alpha}DC\tilde{A}^{-1}\equiv \tilde{T}^{-\alpha}DCA(-1)\tilde{A}^{-1}A\pmod {\gothm^{k+1}}.
	\end{equation}

	By \eqref{v10}, we have $X_k\equiv 0\pmod {\gothm^k}$ and $\tilde T^{-\alpha}X_kB\equiv 0\pmod {\gothm^k}$. Combined with our hypothesis $k\geq 1$, the above two congruence relations imply $$\tilde T^{-\alpha}X_kBX_k\equiv 0\pmod {\gothm^{k+1}}. $$

	Combined with \eqref{E:equation to define X_k in L:conjugating a lambda Hodge bounded matrix to make lower left block closer to 0 over noncommutative rings}, \eqref{v8}, \eqref{v5} and \eqref{v9}, this congruence implies \[\tilde T^{-\alpha}C'\equiv 0\pmod {\gothm^{k+1}}.\qedhere\]
\end{proof}

\begin{remark}
	We emphasize that in the above argument, all the congruence relations modulo powers of $\gothm$ are for matrices with entries in $R$ instead of $R[\frac{1}{\tilde{T}}]$. The reason is because $\gothm^k\cdot R[\frac{1}{\tilde{T}}]\cap R$ and $\gothm^k$ are not necessary to be the same. Equivalently, we do not know whether $\tilde{T}$ is a non-zero divisor in the ring $R/\gothm^k$ in general. However, this is true for $k=1$ by Lemma~\ref{L:T-divisible property in gothm}.
\end{remark}

Now we  prove Proposition~\ref{P:conjugating a 'good' matrix by lower triangular matrix into an upper triangular matrix over noncommutative rings}.
\begin{proof}[Proof of Proposition~\ref{P:conjugating a 'good' matrix by lower triangular matrix into an upper triangular matrix over noncommutative rings}]
	Combining our hypotheses $C\equiv 0\pmod \gothm$ and $H(C,1)\geq \alpha$ with Lemma~\ref{L:T-divisible property in gothm}, we have $\tilde{T}^{-\alpha}C\equiv 0\pmod \gothm$. We can apply Lemma~\ref{L:conjugating a lambda Hodge bounded matrix to make lower left block closer to 0 over noncommutative rings} inductively to get a sequence of matrices $(X_k)_{k\geq 1}$ in $\rmM_{(n-n_1)\times n_1}(R)$  with the following properties:
	\begin{itemize}
		\item $H(X_k,1)\geq \alpha-\lambda_{n_1}$, and  $\tilde{T}^{\lambda_{n_1}-\alpha}X_k\equiv 0\pmod {\gothm^k}$.
		\item The matrix $Y_k=\begin{bmatrix}
		I_{n_1}&0\\
		X_k&I_{n-n_1}\end{bmatrix}$ is $\ula^{[n]}$-stable.
		\item If we set
		\[
		(Y_k\dots Y_1)M(Y_k\dots Y_1)^{-1}:=
		\quad
		\Biggl[\mkern-5mu
		\begin{tikzpicture}[baseline=-.65ex]
		\matrix[
		matrix of math nodes,
		column sep=1ex,
		] (m)
		{
			A_k' & B_k' \\
			C_k' & D_k' \\
		};
		\draw[dotted]
		([xshift=0.5ex]m-1-1.north east) -- ([xshift=0.5ex]m-2-1.south east);
		\draw[dotted]
		(m-1-1.south west) -- (m-1-2.south east);
		\node[above,text depth=1pt] at ([xshift=2.5ex]m-1-1.north) {$\scriptstyle n_1$};  
		\node[left,overlay] at ([xshift=-1.2ex,yshift=-2ex]m-1-1.west) {$\scriptstyle n_1$};
		\end{tikzpicture}\mkern-5mu
		\Biggr],
		\]
		then $H(C_k',1)\geq \alpha$ and $\tilde{T}^{-\alpha}C_k'\equiv 0\pmod {\gothm^{k+1}}$.
	\end{itemize}
	The series $\sum\limits_{i=1}^\infty X_i$ converges to a matrix $X\in \rmM_{(n-n_1)\times n_1}(R)$ with $H(X,1)\geq \alpha-\lambda_{n_1}$. Therefore the products $\{Y_k\dots Y_1|k\geq 1\}$ converge to the matrix $Y=\begin{bmatrix}
	I_{n_1}&0\\
	X&I_{n-n_1}\end{bmatrix}$. By our construction as above, it is straightforward to verify that $Y$ is $\ula^{[n]}$-stable and $YMY^{-1}$ is of the form
	$\quad
	\Biggl[\mkern-5mu
	\begin{tikzpicture}[baseline=-.65ex]
	\matrix[
	matrix of math nodes,
	column sep=1ex,
	] (m)
	{
		A' & B' \\
		0\  & D' \\
	};
	\draw[dotted]
	([xshift=0.5ex]m-1-1.north east) -- ([xshift=0.5ex]m-2-1.south east);
	\draw[dotted]
	(m-1-1.south west) -- (m-1-2.south east);
	\node[above,text depth=1pt] at ([xshift=2.5ex]m-1-1.north) {$\scriptstyle n_1$};  
	\node[left,overlay] at ([xshift=-1.2ex,yshift=-2ex]m-1-1.west) {$\scriptstyle n_1$};
	\end{tikzpicture}\mkern-5mu
	\Biggr].$
\end{proof}

\begin{theorem}\label{T:Newton-Hodge decomposition with Hodge bound restriction for infinite matrices over noncommutative rings}
	Let $M\in \rmM_\infty(R)$ be a $\ula$-Hodge bounded matrix and  $\bar{M}\in \rmM_\infty(\Fp\llbracket T\rrbracket)$ be its reduction by $\gothm$. Suppose that $\Omega=\{s_1<s_2<\dots \}$ is a set of touching vertices of $\bar{M}$ such that $N(\bar{M},s)=\sum\limits_{i=1}^{s}\lambda_i$, for all $s\in \Omega$. Then there exists a $\ula$-stable matrix $Q\in \GL_\infty(R)$ with the following properties:
	\begin{enumerate}
		\item $Q$ is $\ula$-stable, and in particular $QMQ^{-1}$ is $\ula$-Hodge bounded; and
		\item $QMQ^{-1}$ is of the following block upper triangular shape: 
		\[
		QMQ^{-1}=\quad\Biggl[\mkern-5mu
		\begin{tikzpicture}[baseline=-.65ex]
		\matrix[
		matrix of math nodes,
		column sep=1ex,
		] (m)
		{
			M_{11} & M_{12}\\
			\ 0\ \  & M_{22} \\
		};
		\draw[dotted]
		([xshift=0.5ex]m-1-1.north east) -- ([xshift=0.5ex]m-2-1.south east);
		\draw[dotted]
		(m-1-1.south west) -- (m-1-2.south east);
		\node[above,text depth=1pt] at ([xshift=3.5ex]m-1-1.north) {$\scriptstyle s_1$};  
		\node[left,overlay] at ([xshift=-1.2ex,yshift=-2ex]m-1-1.west) {$\scriptstyle s_1$};
		\end{tikzpicture}\mkern-5mu
		\Biggr].
		\]
	\end{enumerate}
\end{theorem}

\begin{proof}
	By Theorem~\ref{T:Newton-Hodge decomposition with Hodge bound restriction for infinite matrices}, there exists a $\ula$-stable matrix $\bar{Q}'\in\GL_\infty(\Fp\llbracket T\rrbracket)$ such that $\bar{Q}'\bar{M}\bar{Q}'^{-1}$ is block upper triangular of the shape 
	$\quad\Biggl[\mkern-5mu
	\begin{tikzpicture}[baseline=-.65ex]
	\matrix[
	matrix of math nodes,
	column sep=1ex,
	] (m)
	{
		N_{11} & N_{12}\\
		\ 0\ \  & N_{22} \\
	};
	\draw[dotted]
	([xshift=0.5ex]m-1-1.north east) -- ([xshift=0.5ex]m-2-1.south east);
	\draw[dotted]
	(m-1-1.south west) -- (m-1-2.south east);
	\node[above,text depth=1pt] at ([xshift=3.5ex]m-1-1.north) {$\scriptstyle s_1$};  
	\node[left,overlay] at ([xshift=-1.2ex,yshift=-2ex]m-1-1.west) {$\scriptstyle s_1$};
	\end{tikzpicture}\mkern-5mu
	\Biggr].$
	Let $Q'\in\rmM_\infty(R)$ be a special lift of $\bar{Q}'$. By Lemma~\ref{L:any lift of an invertible matrix is invertible}, $Q'$ is invertible; and by Corollary~\ref{C:special lift of lambda stable is lambda stable}, $Q'$ is $\ula$-stable. Thus we may assume that $\bar{M}$ is block upper triangular with the prescribed shape. We first take an infinite subsequence $\us':=(s_1',\dots)$ of $\Omega$ such that $s_1'>\ell_{\bar M}$ as in Lemma~\ref{L:permute an infinite matrix and compare Newton polygons of infinite matrix and its truncation}.
	
	By Proposition~\ref{P:conjugating a 'good' matrix by lower triangular matrix into an upper triangular matrix over noncommutative rings} with $M:=M_{[s_1']}\in \rmM_{s_1'}(R)$, we get a matrix $X_1\in \rmM_{(s_1'-s_1)\times s_1}(R)$ such that 
	\begin{itemize}
		\item $H(X_1,1)\geq \lambda_{s_1'}-\lambda_{s_1}$.
		\item 	If we set $Y_1:=\begin{bmatrix}
		I_{s_1}&0\\
		X_1&I_{s_1'-s_1}\end{bmatrix}$, then $Y_1$ is $\ula^{[s_1']}$-stable and $Y_1M_{[s_1']}Y_1^{-1}$ is of the form  $\quad
		\Biggl[\mkern-5mu
		\begin{tikzpicture}[baseline=-.65ex]
		\matrix[
		matrix of math nodes,
		column sep=1ex,
		] (m)
		{
			A_1' & B_1' \\
			\ 0 \   & D_1' \\
		};
		\draw[dotted]
		([xshift=0.5ex]m-1-1.north east) -- ([xshift=0.5ex]m-2-1.south east);
		\draw[dotted]
		(m-1-1.south west) -- (m-1-2.south east);
		\node[above,text depth=1pt] at ([xshift=2.5ex]m-1-1.north) {$\scriptstyle s_1$};  
		\node[left,overlay] at ([xshift=-1.2ex,yshift=-2ex]m-1-1.west) {$\scriptstyle s_1$};
		\end{tikzpicture}\mkern-5mu
		\Biggr].$
	\end{itemize}
	Let $P_1:=\begin{bmatrix}
	Y_1&0\\
	0&I_\infty\end{bmatrix}=\begin{bmatrix}
	I_{s_1}&0&0\\
	X_1&I_{s_1'-s_1}&0\\
	0&0&I_\infty\end{bmatrix}.$ Then $P_1$ is $\ula$-stable and $M_1:=P_1MP_1^{-1}$ is of the form $ \begin{bmatrix}
	A_1'&B_1'&\ast\\
	0&D_1'&\ast\\
	\ast&\ast&\ast\end{bmatrix}.$ 
	
	Now we consider the matrix $(M_1)_{[s_2']}\in \rmM_{s_2'}(R)$ which satisfies all the hypotheses in Proposition~\ref{P:conjugating a 'good' matrix by lower triangular matrix into an upper triangular matrix over noncommutative rings} with $\alpha=\lambda_{s_2'}$. Moreover, if we write $(M_1)_{[s_2']}$ in the form $(M_1)_{[s_2']}=\quad
	\Biggl[\mkern-5mu
	\begin{tikzpicture}[baseline=-.65ex]
	\matrix[
	matrix of math nodes,
	column sep=1ex,
	] (m)
	{
		A & B \\
		C & D \\
	};
	\draw[dotted]
	([xshift=0.5ex]m-1-1.north east) -- ([xshift=0.5ex]m-2-1.south east);
	\draw[dotted]
	(m-1-1.south west) -- (m-1-2.south east);
	\node[above,text depth=1pt] at ([xshift=2.5ex]m-1-1.north) {$\scriptstyle s_1$};  
	\node[left,overlay] at ([xshift=-1.2ex,yshift=-2ex]m-1-1.west) {$\scriptstyle s_1$};
	\end{tikzpicture}\mkern-5mu
	\Biggr],$ then $H(C,1)\geq \lambda_{s_2'}$. Hence there exists $X_2\in \rmM_{(s_2'-s_1)\times s_1}(\gothm)$ such that if we set $Y_2:=\begin{bmatrix}
	I_{s_1}&0\\
	X_2&I_{s_2'-s_1}\end{bmatrix}$, then $Y_2$ is $\ula^{[s_2']}$-stable and $Y_2(M_1)_{[s_2']}Y_2^{-1}$ is of the form 
	$\quad
	\Biggl[\mkern-5mu
	\begin{tikzpicture}[baseline=-.65ex]
	\matrix[
	matrix of math nodes,
	column sep=1ex,
	] (m)
	{
		A_2' & B_2' \\
		\ 0 \   & D_2' \\
	};
	\draw[dotted]
	([xshift=0.5ex]m-1-1.north east) -- ([xshift=0.5ex]m-2-1.south east);
	\draw[dotted]
	(m-1-1.south west) -- (m-1-2.south east);
	\node[above,text depth=1pt] at ([xshift=2.5ex]m-1-1.north) {$\scriptstyle s_1$};  
	\node[left,overlay] at ([xshift=-1.2ex,yshift=-2ex]m-1-1.west) {$\scriptstyle s_1$};
	\end{tikzpicture}\mkern-5mu
	\Biggr].$ Moreover, we have $H(X_2,1)\geq \lambda_{s_2'}-\lambda_{s_1}$. 
	Let $P_2:= \begin{bmatrix}
	Y_2&0\\
	0&I_\infty\end{bmatrix} = \begin{bmatrix}
	I_{s_1}&0&0\\
	X_2&I_{s_2'-s_1}&0\\
	0&0&I_\infty\end{bmatrix}.$ Then $P_2$ is $\ula$-stable and $P_2M_1P_2^{-1}=(P_2P_1)M(P_2P_1)^{-1}$ is of the form 
	$\begin{bmatrix}
	A_2'&B_2'&\ast\\
	0&D_2'&\ast\\
	\ast&\ast&\ast\end{bmatrix}$.
	
	We repeat this iteration and get a sequence of matrices $$\left(P_k= \begin{bmatrix}
	I_{s_1}&0\\
	P_k'&I_\infty\end{bmatrix}\;\Big|\;k\in \NN\right)\subset\GL_\infty(R)$$ such that $P_k$ is $\ula$-stable and $H(P_k',1)\geq \lambda_{s_k'}-\lambda_{s_1}$.
	Moreover, $(P_k\dots P_1)M(P_k\dots P_1)^{-1}$ is of the form $ \begin{bmatrix}
	A_k'&B_k'&\ast\\
	0&D_k'&\ast\\
	\ast&\ast&\ast\end{bmatrix}$. 
	The infinite series $\sum\limits_{i=1}^\infty P_i'$ converges to a matrix $P'$ and hence the product $P_k\dots P_1$ converges to $P= \begin{bmatrix}
	I_{s_1}&0\\
	P'&I_\infty\end{bmatrix}$. By our construction, $P$ is $\ula$-stable and $PMP^{-1}$ is of the form $\quad
	\Biggl[\mkern-5mu
	\begin{tikzpicture}[baseline=-.65ex]
	\matrix[
	matrix of math nodes,
	column sep=1ex,
	] (m)
	{
		N_{11} & N_{12} \\
		0 & N_{22} \\
	};
	\draw[dotted]
	([xshift=0.5ex]m-1-1.north east) -- ([xshift=2ex]m-2-1.south east);
	\draw[dotted]
	(m-1-1.south west) -- (m-1-2.south east);
	\node[above,text depth=1pt] at ([xshift=3.5ex]m-1-1.north) {$\scriptstyle s_1$};  
	\node[left,overlay] at ([xshift=-1.2ex,yshift=-2ex]m-1-1.west) {$\scriptstyle s_1$};
	\end{tikzpicture}\mkern-5mu
	\Biggr].$
\end{proof}

\section{Automorphic forms for definite quaternion algebras and completed homology}\label{section:Automorphic forms for definite quaternion algebras and completed homology}

\subsection{Notations}\label{subsection:Notations in Automorphic forms for definite quaternion algebras and completed homology}

\begin{enumerate}
	\item Let $F$ be a totally real field of degree $g$ over $\QQ$ and  $\calO_F$ be its ring of integers. Let $I:=\Hom(F,\bar{\QQ})$ denote the set of embeddings of $F$ into $\bar{\QQ}$. 
	\item Fix an odd prime $p$ which splits completely in $F$. We set $\calO_p:=\calO_F\otimes_{\ZZ}\ZZ_p$ and $F_p:=F\otimes_{\QQ}\QQ_p$. For every $i\in I$, let $\gothp_i$ be the prime of $\calO_F$ induced by the composite $i_p=\iota_p\circ i:F\rightarrow \bar{\QQ}_p$.  Let $F_{\gothp_i}$ (resp. $\calO_{\gothp_i}$) denote the completion of $F$ (resp. $\calO_F$) at $\gothp_i$.
	For $\alpha\in F_p$ and $i\in I$, we denote by $\alpha_i$ the $i$-component of $\alpha$ under the natural isomorphism $F_p\cong\prod\limits_{i\in I}F_{\gothp_i}$. Since $p$ splits in $F$, the embedding $\QQ_p\rightarrow F_{\gothp_i}$ is an isomorphism $ \QQ_p\cong F_{\gothp_i}$ and hence get $\ZZ_p\cong \calO_{\gothp_i}$.  For every $i\in I$, we fix a uniformizer $\pi_i$ of $\calO_{\gothp_i}$ and let $\pi:=\prod\limits_{i\in I}\pi_i\in \calO_p$. 
	\item For every $i\in I$, we use $\Delta_i$ to denote the torsion subgroup of $\calO_{\gothp_i}^\times$. Hence we have a decomposition of the multiplicative group $\calO_{\gothp_i}^\times\cong \Delta_i\times (1+\pi_i\calO_{\gothp_i})$.
	\item Let $A_{F,f}$ (resp. $A_{F,f}^{(p)}$) denote the ring of finite adeles (resp. finite prime-to-$p$ adeles) of $F$.
	\item We use $\Nm_{F/\QQ}:F^\times\rightarrow \QQ^\times$ to denote the norm map. For simplicity, the norm map $F_p^\times\rightarrow \QQ_p^\times$ will also be denoted by $\Nm_{F/\QQ}$.
	\item Following \cite{buzzard320eigenvarieties}, we define $\calN:=\{|x|_p\;|\;x\in\bar{\QQ}_p^\times\}\cap (0,1]=p^\QQ\cap (0,1]$ and $\calN^\times:=\calN\setminus \{1 \}=p^\QQ\cap (0,1)$.
	\item We use $\AAA^1$ (resp. $\GG_m$) to denote the rigid analytification of the affine line (resp. affine line with zero removed). 
	\item For $r=(r_i)_{i\in I}\in \calN^I$ and $\alpha=(\alpha_i)_{i\in I}\in \calO_p$, we define $B(\alpha,r)$ to be the $\QQ_p$-polydisc  such that $B(\alpha,r)(\CC_p):=\{(x_i)_{i\in I}\in \CC_p^I\;|\; |x_i-i_p(\alpha_i)|_p\leq r_i\text{~for all~} i\in I \}$, and $\bbB_r:=\prod\limits_{\alpha\in \calO_p}B(\alpha,r)$. If $r\in(\calN^\times)^I$, we define $\bbB_r^\times:=\prod\limits_{\alpha\in \calO_p^\times}B(\alpha,r)$.
	\item Let $K$ be a complete field extension of $\QQ_p$ and $X:=\mathrm{Sp}(A)$ be a $K$-affinoid space. By \cite[Lemma~8.2(a)]{buzzard320eigenvarieties}, there is a bijection $\iota:\calO(X)^\times\cong \Hom_{\QQ_p\text{-rigid space~}}(X,\GG_m)$. 
By \cite[Proposition~8.3]{buzzard320eigenvarieties},  for any continuous character $\chi:\calO_p^\times\rightarrow A^\times$, there exists at least one $r\in (\calN^\times)^I$ that satisfies the following:
	
	 There is a unique map of $K$-rigid spaces $\beta_r:\bbB_r^\times\times X\rightarrow \GG_m$ such  that for every $\alpha\in \calO_p^\times$ we have 
	$ \iota\circ\chi(\alpha)=\beta_r(\alpha,-)$,
	where 
	  $\chi(\alpha)\in A^\times=\calO(X)^\times$ and $\beta_r(\alpha,-):X\rightarrow \GG_m$ is obtained by evaluating $\beta_r$ at $\alpha\in \bbB_{r}^\times(\QQ_p)$.
	
	We call $\beta_r$ the \emph{$r$-thickening} of $\chi$.
	\item Let $\chi:\calO_p^\times\rightarrow \CC_p^\times$ be a  continuous character. For every $i\in I$, we denote by $\chi_i$ the $i$-component of $\chi$, i.e. $\chi_i$ is the composite $\calO_{\gothp_i}^\times\hookrightarrow \calO_p^\times\xrightarrow{\chi}\CC_p^\times$. 
	We define the coordinates $T:=(T_i)_{i\in I}$ of $\chi$ by $T_i:=T_{\chi_i}= \chi_i(\exp(\pi_i))-1$ for all $i\in I$. For $r=(r_i)_{i\in I}\in (\calN^\times)^I$, the character $\chi$ admits the $r$-thickening if and only if $v_p(T_i)>\frac{pr_i}{p-1}$ for all $i\in I$.
	\item Let $\iota_F:\calO_F^\times \hookrightarrow \calO_p^\times \times \calO_p^\times, x\mapsto (x,x^2)$ be an embedding. We always regard $\calO_F^\times$ as a subgroup of $\calO_p^\times \times \calO_p^\times$ by $\iota_F$.  A weight is defined to be a continuous group homomorphism $\kappa=(n,\nu):\calO_p^\times \times \calO_p^\times\rightarrow A^\times$ for some affinoid $\QQ_p$-space $X=\mathrm{Sp}(A)$, such that $\ker(\kappa)$ contains a subgroup of $\calO_F^\times$ of finite index. For such a weight $\kappa$, we define $r(\kappa)$ to be the largest element $r$ in $(\calN^\times)^I$ (with the obvious partial order), such that the continuous homomorphism $n:\calO_p^\times\rightarrow \calO(X)^\times$ admits the $r$-thickening.
	\item Let $D/F$ be a totally definite quaternion algebra over $F$ with discriminant $\gothd$. Assume that $(p,\gothd)=1$. Let $D_f:=D\otimes_F A_{F,f}$ and $D_f^{(p)}:=D\otimes_F A_{F,f}^{(p)}$. Fix a maximal order $\calO_D$ of $D$. For each finite place $v$ of $F$ prime to $\gothd$, we fix an isomorphism $\calO_D\otimes_{\calO_F}\calO_{F,v}\cong \rmM_2(\calO_{F,v})$. In particular, we have an isomorphism $\calO_D\otimes_{\calO_F}\calO_p\cong \rmM_2(\calO_p)$. For an ideal $\gothn$ of $\calO_F$ prime to $\gothd$, we denote by $U_0^{D,(p)}(\gothn)$ and $U_1^{D,(p)}(\gothn)$ the subgroups of $(\calO_D\otimes \hat{\ZZ}^{(p)})^\times$ consisting of matrices which are congruent to $\begin{bmatrix}
	\ast&\ast\\
	0&\ast\end{bmatrix}$ and $\begin{bmatrix}
	\ast&\ast\\
	0&1\end{bmatrix}$  $\mod \gothn$, respectively. 
	\item For a $\QQ_p$-Banach algebra $A$ with norm $|\cdot|$, we define a subring $A\{\!\{X\}\!\}$ of $A\llbracket X \rrbracket$ by
	\[
	A\{\!\{X\}\!\}:=\left\{\sum_{n\geq 0}c_nX^n\in A\llbracket X \rrbracket \;\Bigg|\; \lim_{n\rightarrow \infty}|c_n|R^n=0, \text{~for any~} R\in \RR_{>0} \right\}.
	\]
\end{enumerate}

\subsection{Subgroups of $\GL_2(F_p)$}\label{section:subgroups of GL2(Fp)}

\begin{enumerate}
	\item For a commutative ring $R$ with unit $1$, we define several subgroups of $\GL_2(R)$ by
	\[
	B(R):=\begin{bmatrix}
	R^\times&R\\
	0&R^\times\end{bmatrix}, \
T(R):=\begin{bmatrix}
	R^\times&0\\
	0&R^\times\end{bmatrix}, \
	N(R):=\begin{bmatrix}
	1&R\\
	0&1\end{bmatrix}\text{~and~}
	D(R):=\left\{ \begin{bmatrix}
	\alpha&0\\
	0&\alpha\end{bmatrix}\Big|\ \alpha\in R^\times \right\}.
	\]
	For a subgroup $H$ of $R^\times$, we set
	\[
	T(H):=\begin{bmatrix}
	H^\times&0\\
	0&H^\times\end{bmatrix}
\text{~and~} D(H):=\left\{ \begin{bmatrix}
	\alpha&0\\
	0&\alpha\end{bmatrix}\ \Big|\ \alpha\in H \right\}.
	\]
	\item Fix $i\in I$ and a positive integer $t_i$. We define the following subgroups of $\GL_2(\calO_{\gothp_i})$:
	\[
	\Iw_{\pi_i^{t_i}}:=\begin{bmatrix}
	\calO_{\gothp_i}^\times&\calO_{\gothp_i}\\
	\pi_i^{t_i}\calO_{\gothp_i}&\calO_{\gothp_i}^\times\end{bmatrix}
,\
	\bar{N}(\pi_i^{t_i}\calO_{\gothp_i}):=\begin{bmatrix}
	1&0\\
	\pi_i^{t_i}\calO_{\gothp_i}&1\end{bmatrix} \text{~and~}
	\bar{B}(\pi_i^{t_i}\calO_{\gothp_i}):=
	\begin{bmatrix}
	\calO_{\gothp_i}^\times&0\\
	\pi_i^{t_i}\calO_{\gothp_i}&\calO_{\gothp_i}^\times\end{bmatrix}.
	\]
	Given $t=(t_i)_{i\in I}\in \NN^I$, we put $\Iw_{\pi^t}:=\prod\limits_{i\in I}\Iw_{\pi_i^{t_i}}$ and similarly define $\bar{N}(\pi^t\calO_p),\ \bar{B}(\pi^t\calO_p)$. Note that every $g\in \Iw_{\pi^t}$ can be represented by $(g_i)_{i\in I}$, with $g_i\in \Iw_{\pi_i^{m_i}}$, and every $g_i\in \Iw_{\pi_i^{t_i}}$ can be viewed as an element in $\Iw_{\pi^t}$ whose $j$-component is the identity matrix for all $j\neq i$. 
	All the groups we define above are profinite and considered as topological groups endowed with the profinite topology. 
	\item For $i\in I$ and $t_i\in\NN$, we will identify $\calO_{\gothp_i}$ with $\bar{N}(\pi_i^{t_i}\calO_{\gothp_i})$ via the isomorphism $\bar{n}_i:\calO_{\gothp_i}\rightarrow  \bar{N}(\pi_i^{t_i}\calO_{\gothp_i}),\ z\mapsto \begin{bmatrix}
	1&0\\
	\\ \pi_i^{t_i} z&1\end{bmatrix}$. For $t\in\NN^I$, we will identify $\calO_p$ with  $\bar{N}(\pi^t\calO_p)$ via the isomorphism $\bar{n}:\calO_p\rightarrow  \bar{N}(\pi^t\calO_p),\ z\mapsto \begin{bmatrix}
	1&0\\
	\\ \pi^t z&1\end{bmatrix}$.
	\item For $i\in I$, we denote by $\Iw_{\pi_i,1}:= \begin{bmatrix}
	1+\pi_i\calO_{\gothp_i} &\calO_{\gothp_i}\\
	\\ \pi_i\calO_{\gothp_i} z&1+\pi_i\calO_{\gothp_i}\end{bmatrix}$ the the pro-$p$-subgroup of $\Iw_{\pi_i}$. Then the map $T(\Delta_i)\times \Iw_{\pi,1}\rightarrow \Iw_{\pi}$, $(t,g)\mapsto tg$ is a bijection. We put $\Iw_{\pi,1}:=\prod\limits_{i\in I}\Iw_{\pi_i,1}$.  
	\item The Iwasawa decomposition is the following bijection:
	\[
	N(\calO_p)\times T(\calO_p)\times \bar{N}(\pi^t\calO_p)\rightarrow \Iw_{\pi^t},\ (N,T,\bar{N})\mapsto NT\bar{N}.
	\]
	\item For $i\in I$ and $t_i\in \NN$, we define an anti-involution $\ast$ on $\Iw_{\pi_i^{t_i}}$ by $$g=\begin{bmatrix}
	a&b\\
	c&d\end{bmatrix}\in \Iw_{\pi_i^{t_i}}\mapsto g^\ast=\begin{bmatrix}
	a&c/\pi_i^{t_i}\\
	\pi_i^{t_i} b&d\end{bmatrix}.$$ Here ``anti-involution'' means that the map $g\mapsto g^\ast$ satisfies $(g^\ast)^\ast=g$ and $(gh)^\ast=h^\ast g^\ast$ for all $g,h\in \Iw_{\pi_i^{t_i}}$. We use similar formula to define anti-involutions on $\Iw_{\pi^t}$ for  $t\in \NN^I$. 
	\item For $t=(t_i)_{i\in I}\in \NN^I$, we set 
	\[
	\bbM_{\pi^t}:=\left\{\gamma=(\gamma_i)_{i\in I}\in \rmM_2(\calO_p)\;\Bigg|\;\text{~if~} \gamma_i= \begin{bmatrix}
	a_i&b_i\\
	c_i&d_i\end{bmatrix} \text{~then~} \det(\gamma_i)\neq 0, \pi_i^{t_i}|c_i,\pi_i\nmid d_i \right\}.
	\]
	Then $\bbM_{\pi^t}$ contains $\Iw_{\pi^t}$ and is a monoid under multiplication. The involution $\ast$ can be extended to $\bbM_{\pi^t}$ by the same formula. For $i\in I$ and $t_i\in \NN$, we define the monoid $\bbM_{\pi_i^{t_i}}\subset \rmM_2(\calO_{\gothp_i})$ in a similar way.
	
\end{enumerate}

\subsection{Induced representations}\label{section:induced representations}

Let $A$ be a topological ring in which $p$ is topologically nilpotent and  $\kappa=(n,\nu):\calO_p^\times\times \calO_p^\times\rightarrow A^\times$ be a continuous character. 
\begin{convention}\label{Convention:extend the character nu}
	Given the character $\kappa$ as above, we always extend $\nu$ to a continuous homomorphism $F_p^\times\rightarrow A^\times$ by requiring $\nu(\pi_i)=1$ for all $i\in I$.
\end{convention}
The character $\kappa$ induces a character $\kappa_T:T(\calO_p)\rightarrow A^\times,\ \begin{bmatrix}
a&0\\
0&d\end{bmatrix}\mapsto n(d)\cdot \nu(ad)$.  Since $T(\calO_p)$ is a quotient of $B(\calO_p)$ (resp. $\bar{B}(\pi\calO_p)$), this  character $\kappa_T$ extends to a character $\kappa_{B}$ (resp. $\kappa_{\bar{B}}$) of $B(\calO_p)$ (resp. $\bar{B}(\pi\calO_p)$).

Fix $t\in \NN^I$. Consider the induced representation
\begin{equation}\label{E:induced representations}
\Ind_{B(\calO_p)}^{\Iw_{\pi^t}}(\kappa_B):= \Big\{f\in \calC(\Iw_{\pi^t},A)\;\Big|\; f(b g)=\kappa_B(b)f(g) \text{~for all~} b\in B(\calO_p),\ g\in \Iw_{\pi^t} \Big\}.
\end{equation}

The group $\Iw_{\pi^t}$ acts right on $\Ind_{B(\calO_p)}^{\Iw_{\pi^t}}(\kappa_B)$ by $f\circ m(g)=f(gm^\ast)$. By Iwasawa decomposition, we obtain the following bijection
\begin{equation*}
\begin{split}
\Ind_{B(\calO_p)}^{\Iw_{\pi^t}}(\kappa_B)&\rightarrow \calC(\calO_p,A),\\
f&\mapsto h(z):= f(\bar{n}(z))=f\left(\begin{bmatrix}
1&0\\
\pi^t z&1\end{bmatrix}\right).
\end{split}
\end{equation*}
The right action of $\Iw_{\pi^t}$ on $\Ind_{B(\calO_p)}^{\Iw_{\pi^t}}(\kappa_B)$ induces a right action on $\calC(\calO_p,A)$, which can be written in the following explicit formula:
\begin{equation}\label{E:equation to define Iwahori action on continuous functions}
h\circ m(z)=n(cz+d)\nu(ad-bc)h\left(\frac{az+b}{cz+d}\right), \text{~for~} m=\begin{bmatrix}
a&b\\
c&d\end{bmatrix}.
\end{equation}
Using the exact formula in \eqref{E:equation to define Iwahori action on continuous functions}, we can extend this action to the monoid $\bbM_{\pi^t}$.

\subsection{Space of $p$-adic automorphic forms}
We will recall Buzzard's construction of overconvergent automorphic forms
in \cite{buzzard320eigenvarieties}.

Let $A$ be a $\QQ_p$-affinoid algebra and $X:=\mathrm{Sp}(A)$ be the corresponding affinoid space. Let $\kappa=(n,\nu):\calO_p^\times\times \calO_p^\times\rightarrow A^\times$ be a weight. We write $r(\kappa):=(p^{-m_{\kappa,i}})_{i\in I}$. Fix $r=(r_i)_{i\in I}=(p^{-m_i})_{i\in I}\in \calN^I$. Define $\calA_{\kappa,r}$ to be the $\QQ_p$-Banach algebra $\calO(\bbB_r\times X)=\calO_{\bbB_r}\hat{\otimes}_{\QQ_p}A$. Since $\calO_p$ is Zariski dense in $\bbB_r$, we obtain an embedding $\calA_{\kappa,r}\hookrightarrow \calC(\calO_p,A)$ and hence can view $\calA_{\kappa,r}$ as a subspace of $\calC(\calO_p,A)$ via this embedding.

\begin{definition}\label{D:good level structure at p}
We call	$t\in\NN^I$ that \emph{is good for $(\kappa,r)$} if $m_i+t_i\geq m_{\kappa,i}$ for all $i\in I$.
\end{definition}

Fix $t\in \NN^I$ that is good for $(\kappa,r)$. It is easy to check that the right action of $\bbM_{\pi^t}$ on $\calC(\calO_p,A)$ leaves the subspace $\calA_{\kappa,r}$ stable, and hence  \eqref{E:equation to define Iwahori action on continuous functions} defines a right action of the monoid $\bbM_{\pi^t}$ on $\calA_{\kappa,r}$.

Fix an open compact subgroup $K^p\subset D_f^{(p),\times}$ and $t\in\NN^I$ which is good for $(\kappa,r)$. As in \cite{buzzard320eigenvarieties}, we define the space of $r$-overconvergent automorphic forms of weight $\kappa$ and level $K^p\Iw_{\pi^t}$ to be the $A$-module
\[
S_\kappa^D(K^p\Iw_{\pi^t},r)=\{\phi: D^\times \setminus D_f^\times /K^p\rightarrow \calA_{\kappa,r}\;|\;\phi(gu)=\phi(g)\circ u,\text{~for all~} g\in D_f^\times,u\in \Iw_{\pi^t} \}.
\]

\subsection{Hecke operators}\label{subsection: Hecke operators}

Let $\kappa=(n,\nu):\calO_p^\times\times \calO_p^\times\rightarrow A^\times$, $r=(r_i)_{i\in I}\in \calN^I$, and $t\in \NN^I$ be the same as the previous section. Fix an open compact subgroup $K^p$ of $D_f^{(p),\times}$. In the rest of this paper, we will make the following convention on the choice of $K^p$.
\begin{convention}
	$K^p$ is of the form $U_0^{(p)}(\gothn)$ or $U_1^{(p)}(\gothn)$ for some ideal $\gothn$ of $\calO_F$ prime to $p\gothd$.
\end{convention}
Let $K=K^p\Iw_{\pi^t}$ be the open compact subgroup of $ D_f^\times$. Now we recall the definition of Hecke operators on the space $S_\kappa^D(K^p\Iw_{\pi^t},r)$.

Let $v$ be a finite place of $F$ where $D$ splits and fix a uniformizer $\pi_v$ of $F_v$. In particular, when $v=i\in I$ is over $p$, we choose $\pi_v=\pi_i$ as before. Let $\eta_v=\left[ \begin{array}{cc}
\pi_v&0\\
0&1\end{array}\right]\in \GL_2(F_v)$ which is viewed as an element in $D_f^\times$ whose components away from $v$ are the identity element. 
We fix a double coset decomposition 
\begin{equation*}\label{E:double coset decomposition to define the tame Hecke operators Uv}
K\eta_v K=\bigsqcup_j Kx_j,
\end{equation*}
with $x_j\in D_f^\times$. We use $x_{j,p}$ to denote the $v$-component of $x_j$. The Hecke operator $U_v$ on  $S_\kappa^D(K^p\Iw_{\pi^t},r)$ is defined by
\begin{equation}\label{E:formula to define tame Hecke operator Uv}
U_{v}(\phi):=\sum_j\phi|_{x_j} 
\end{equation}
where $\phi|_{x_j}(x):=\phi(xx_j^{-1})\circ x_{j,p}$ for $\phi\in S_\kappa^D(K^p\Iw_{\pi^t},r).$


When $v=i\in I$, the operator $U_i$ depends on the choice of the uniformizer $\pi_i$, and we will write $U_{\pi_i}$ for $U_i$. For later computations, we give a more explicit expression of the $U_{\pi_i}$ operator. 
We fix a double coset decomposition 
\begin{equation*}\label{E:double coset decomposition to define Hecke operators}
\Iw_{\pi^t}\eta_i\Iw_{\pi^t}=\bigsqcup_{j=0}^{p-1}\Iw_{\pi^t}v_{i,j},
\end{equation*}
with $v_{i,j}=\begin{bmatrix}
\pi_i&0\\
j\pi_i^{t_i}&1\end{bmatrix}\in \bbM_{\pi_i^{t_i}}\subset \bbM_{\pi^t}$ for  $j=0,\dots,p-1$.  
The $U_{\pi_i}$-operator can be described by 
\begin{equation}\label{E:formula to define Up operator}
U_{\pi_i}(\phi)=\sum_j\phi|_{v_{i,j}} \text{~with~} \phi|_{v_{i,j}}(x)=\phi(xv_{i,j}^{-1})\circ v_{i,j}.
\end{equation}

When $v$ is prime to $\gothn p$, we view $\pi_v$ as a central element in $D_f^\times$. Choose a double coset decomposition 
\begin{equation*}\label{E:double coset decomposition to define the tame Hecke operators Sv}
K\pi_v K=\bigsqcup_j Ky_j,
\end{equation*}
with $y_j\in D_f^\times$. The Hecke operator $S_v$ on  $S_\kappa^D(K^p\Iw_{\pi^t},r)$ is defined by
\begin{equation}\label{E:formula to define tame Hecke operator Sv}
S_{v}(\phi)=\sum_j\phi|_{y_j}.
\end{equation}

 The Hecke operators $U_v$'s and $S_v$'s for all possible $v$'s commute with each other. In particular,  we define $U_{\pi}=\prod\limits_{i\in I}U_{\pi_i}$. 

By \cite[Lemma~12.2]{buzzard320eigenvarieties}, the operator $U_{\pi}$ acting on $S_\kappa^D(K^p\Iw_{\pi^t},r)$ is compact. Thus, its characteristic power series is well-defined by
 \[
 \mathrm{Char}(U_{\pi};S_\kappa^D(K^p\Iw_{\pi^t},r)):=\det(\mathrm{I}-XU_{\pi}|_{S_\kappa^D(K^p\Iw_{\pi^t},r)})\in A\{\!\{X\}\!\}.
 \]

We make two remarks on the characteristic power series $\mathrm{Char}(U_{\pi};S_\kappa^D(K^p\Iw_{\pi^t},r))$.
\begin{remark}
	\begin{enumerate}
		\item By \cite[Lemma~13.1]{buzzard320eigenvarieties}, $\mathrm{Char}(U_{\pi};S_\kappa^D(K^p\Iw_{\pi^t},r))$ is independent of $r$ (as long as $t$ is good for the pair $(\kappa,r)$).
		\item By \cite[Proposition~11.1]{buzzard320eigenvarieties}, there is a canonical Hecke equivariant isomorphism $$S_\kappa^D(K^p\Iw_{\pi^t},r)\cong S_\kappa^D(K^p\Iw_\pi,r')$$ for some explicit $r'\in (\calN^\times)^I$. It is a well-known phenomenon that the characteristic power series does not see the higher $\Iw_{\pi^t}$-structure. For this reason,  we can and will only work over the space $S_\kappa^D(K^p\Iw_\pi,r)$, i.e. $t=(1,\dots,1)$.
	\end{enumerate}
\end{remark}

\subsection{The eigenvariety datum for $D$}\label{subsection:the spectral varieties and eigenvarieties}
In this section, we recall the construction of the spectral varieties and eigenvarieties associated to $D$ as constructed in \cite[Part~\uppercase\expandafter{\romannumeral3}]{buzzard320eigenvarieties}. We will follow \cite[\S4]{hansen2017universal} to define an eigenvariety datum as we will use Hansen's interpolation theorem to translate our results to Hilbert modular eigenvarieties (see \S\ref{section:Application to Hilbert modular eigenvarieties} below).

We start with the definition of the weight space. As in \cite{andreatta2016arithmetique}, the weight space $\calW$ is defined to be the rigid analytic space over $\QQ_p$ associated to the complete group algebra $\ZZ_p\llbracket \calO_p^\times\times \ZZ_p^\times \rrbracket$. A closed point of $\calW$ is a continuous character $\chi=(\nu,\mu):\calO_p^\times\times \ZZ_p^\times\rightarrow \CC_p^\times$. 

\begin{remark}
	It will be helpful to compare the weight space constructed in \cite{buzzard320eigenvarieties} with the weight space constructed above. 
	
	First we recall its construction in \cite{buzzard320eigenvarieties}. Let $G\subset \calO_F^\times$ be a subgroup of finite index. Recall that we regard $G$ as a subgroup of $\calO_p^\times \times \calO_p^\times$ via the embedding $\iota_F:\calO_F^\times\rightarrow \calO_p^\times \times \calO_p^\times$, $x\mapsto (x,x^2)$.  We denote by $\Gamma_G$ the quotient of $\calO_p^\times \times \calO_p^\times$ by the closure of $G$. By \cite[Lemma~8.2]{buzzard320eigenvarieties}, the functor which sends a $\QQ_p$-rigid space $U$ to the group of continuous group homomorphism $\Gamma_G\rightarrow \calO(U)^\times$ is represented by a $\QQ_p$-rigid space $\calX_{\Gamma_G}$ (actually $\calX_{\Gamma_G}$ is isomorphic to the product of an open unit polydisc and a finite rigid space over $\QQ_p$). The weight space $\calW^{\mathrm{full}}$ is the direct limit $\varinjlim\limits_G \calX_{\Gamma_G} $ as $G$ varies over the set of subgroups of finite index of $\calO_F^\times$. 
	If $G_1\subset G_2\subset \calO_F^\times$ are two subgroups of $\calO_F^\times$ of finite indices, the natural homomorphism $\Gamma_{G_1}\rightarrow \Gamma_{G_2}$ is surjective with finite kernel, and hence the corresponding map $\calX_{\Gamma_{G_2}}\rightarrow \calX_{\Gamma_{G_1}}$ is a closed immersion and geometrically identifies $\calX_{\Gamma_{G_2}}$ with a union of components of $\calX_{\Gamma_{G_1}}$. Therefore, $\{\calX_{\Gamma_{G}}|G \text{~is a subgroup of~} \calO_F^\times \text{~of finite index} \}$ forms an admissible cover of the weight space $\calW$. It follows that the dimension of $\calW^{\mathrm{full}}$ is $g+1+\delta$, where $\delta$ is the Leopoldt defect for $(F,p)$.
	
	Now let's explain the relation between the two weight spaces $\calW$ and $\calW^{\mathrm{full}}$. Define a continuous homomorphism $\phi_\eta:\calO_p^\times\times\calO_p^\times\rightarrow \calO_p^\times\times\ZZ_p^\times$, $(\alpha,\beta)\mapsto (\alpha^{-2}\beta,\Nm_{F/\QQ}(\alpha))$. The kernel $\ker(\phi_\eta)$ contains the group $G=\calO_F^{\times,+}$ and hence $\phi_\eta$ induces a map $\eta:\calW\rightarrow \calX_{\Gamma_{G}}$ of rigid analytic spaces over $\QQ_p$ and hence a map $\eta:\calW\rightarrow \calW^{\mathrm{full}}$. Since $\dim\calW=g+1$, if the Leopoldt's conjecture holds for $(F,p)$, the map $\eta:\calW\rightarrow \calW^{\mathrm{full}}$ would be an immersion which identifies $\calW$ with a set of connected components of $\calW^{\mathrm{full}}$.
	
	There are two reasons why we use the weight space $\calW$ instead of $\calW^{\mathrm{full}}$: first, we do not even know the dimension of the weight space $\calW^{\mathrm{full}}$ without assuming the Leopoldt's conjecture for $(F,p)$ and it is rather complicated to describe the connected components of $\calW^{\mathrm{full}}$ and classical weights on every component; second, as an application we will use $p$-adic Jacquet-Langlands correspondence to translate our results to Hilbert modular eigenvarieties as constructed in \cite{andreatta2016arithmetique}, and the weight space $\calW$ is the one considered there.

	Since the action of the monoid $\bbM_{\pi}$ on the spaces $\calA_{\kappa,r}$ depends on the characters of $\calO_p^\times\times\calO_p^\times$, it is useful to have an explicit expression of the map $\calW\rightarrow \calW^{\mathrm{full}}$ in term of characters. Let $A$ be an affinoid $\QQ_p$-algebra and $\chi=(\nu,\mu):\calO_p^\times\times\ZZ_p^\times\rightarrow A^\times$ be a continuous character. Then the composite $\kappa=\chi\circ \phi_\eta:\calO_p^\times\times\calO_p^\times\rightarrow A^\times$ is of the form $\kappa=(n,\nu)$, where $n:\calO_p^\times\rightarrow A^\times$ is the character defined by $\alpha\mapsto \nu(\alpha)^{-2}\cdot (\mu\circ\Nm_{F/\QQ}(\alpha))$. The character $\kappa$ is called the \emph{character of $\calO_p^\times\times\calO_p^\times$ associated to $\chi$}.
\end{remark}

\begin{convention}\label{Convention: a weight is the character associated to a character of the weight space}
	In the rest of this paper, a weight $\kappa:\calO_p^\times\times\calO_p^\times\rightarrow A^\times$ is always the character associated to some character $\chi$ of $\calO_p^\times\times\ZZ_p^\times$ as constructed above.
\end{convention}

 We fix an open compact subgroup $K^p$ of $D_f^{(p),\times}$ as before. For an affinoid subdomain $U=\mathrm{Sp}(A)\subset \calW$, we denote by $\chi_U:\calO_p^\times\times \ZZ_p^\times\rightarrow A^\times$  the universal character and $\kappa_U:\calO_p^\times\times\calO_p^\times\rightarrow A^\times$ the character associated to $\chi_U$. For simplicity, we use $f_U(X)$ to denote $\mathrm{Char}(U_\pi,S_{\kappa_U}^D(K^p\Iw_{\pi},r(\kappa_U)))\in A\{\!\{X\}\!\}$.

The spectral variety $\mathrm{wt}:\calZ_D\rightarrow \calW$  is defined to be the Fredholm hypersurface of $\calW\times \AAA^1$ associated to the compact operator $U_\pi$. More precisely, for every affinoid $U$ of $\calW$, $\mathrm{wt}^{-1}(U)\subset U\times \AAA^1$ is the zero locus of the characteristic power series $f_U(X)$. We use $\mathrm{wt}:\calZ_D\rightarrow \calW$ (resp. $a_\pi^{-1}:\calZ_D\rightarrow \AAA^1$) to denote the first (resp. second) projection.

\begin{definition}
	An affinoid open subdomain $Y\subset \calZ_D$ is called \emph{slope adapted} if there exist $h\in\QQ$ and an affinoid $U=\mathrm{Sp}(A)\subset \calW$, such that $Y=A\langle p^hX \rangle/(f_U(X))$ is an affinoid open subset of $\calZ_D$ and the natural map $Y\rightarrow U$ is finite and flat. 
\end{definition}
For such an slope adapted affinoid $Y$, the characteristic power series $f_U(X)$ admits a slope $\leq h$-factorization $f_U(X)=Q(X)\cdot R(X)$ and $\calO(Y)\cong A[X]/(Q(X))$. More precisely, we have the following characterization of $Q(X)$ and $R(X)$:
\begin{itemize}
	\item $Q(X)$ is a polynomial whose leading coefficient is a unit and the slopes of the Newton polygon of $Q(X)$ are all $\leq h$; and 
	\item $R(X)$ is an entire power series and the slopes of the Newton polygon of $R(X)$ are all $>h$.
\end{itemize}
Moreover, it follows from \cite[Theorem~3.3]{buzzard320eigenvarieties} that the space $S_{\kappa_U}^D(K^p\Iw_{\pi},r(\kappa_U))$ admits a slope $\leq h$-decomposition $S_{\kappa_U}^D(K^p\Iw_{\pi},r(\kappa_U))=S_{\kappa_U}^{D,\leq h}\oplus S_{\kappa_U}^{D,>h}$ in the sense of \cite[Definition~2.3.1]{hansen2017universal}. In particular, $S_{\kappa_U}^{D,\leq h}$ is an $\calO(Y)\cong A[X]/(Q(X))$-module via the map $X\mapsto U_\pi^{-1}$. Recall that by \cite[Proposition~4.1.4]{hansen2017universal}, the collection of slope adapted affinoids forms an admissible cover of $\calZ_D$. Hence the association $Y\mapsto S_{\kappa_U}^{D,\leq h}$ defines a coherent sheaf $\calM_D$ on $\calZ_D$.

Let $\bbT$ be the Hecke algebra over $\QQ_p$ generated by the Hecke operators $U_v$'s and $S_v$'s for all finite places $v$ of $F$ not dividing $p\gothn\gothd$, and all the $U_{\pi_i}$'s for all $i\in I$ defined in \S\ref{subsection: Hecke operators} and $\psi:\bbT\rightarrow \End_{\calZ_D}(\calM_D)$ be the natural homomorphism of $\QQ_p$-algebras. The tuple $\gothD=(\calW,\calZ_D,\calM_D,\bbT,\psi)$ is an eigenvariety datum in the sense of \cite[Definition~4.2.1]{hansen2017universal}. We use $\calX_D$ to denote the associated eigenvariety together with the finite morphism $\pi:\calX_D\rightarrow \calZ_D$ and a morphism $w:\calX_D\rightarrow \calW$. It follows from \cite[Theorem~4.2.2]{hansen2017universal} that every point $x$ of $\calX_D$ lying over $z\in \calZ_D$ corresponds to a generalized eigenspace for the action of $\bbT$ on $\calM_D(z)$. In particular, we use $a_i(x)$ to denote the eigenvalue of the $U_{\pi_i}$-operator for all $i\in I$.

Let $\calW^\ast$ be the rigid analytic space associated to $\ZZ_p\llbracket \calO_p^\times  \rrbracket$. The homomorphism $\phi_\rho:\calO_p^\times\rightarrow\calO_p^\times\times\ZZ_p^\times$, $\alpha\mapsto(\alpha^{-2},\Nm_{F/\QQ}(\alpha))$ induces a map of rigid spaces $\rho:\calW\rightarrow \calW^\ast$. Explicitly $\rho$ maps a weight $(\nu,\mu)$ of $\calO_p^\times\times\ZZ_p^\times$ to the character $n$ of $\calO_p^\times$ defined above. Each closed point of $\calW^\ast$ corresponds to a continuous character $n:\calO_p^\times\rightarrow \CC_p^\times$. The coordinates $T^\ast_i$ for $i\in I$ of $n$ defined in \S\ref{subsection:Notations in Automorphic forms for definite quaternion algebras and completed homology} form a set of parameters on $\calW^\ast$.
Let $T_i:=\phi_\rho(T_i^\ast)$ and $T:=[1,\exp(p)]-1\in \ZZ_p\llbracket \calO_p\times \ZZ_p\rrbracket $. The set $\{(T_i)_{i\in I}, T \}$ forms a full set of parameters on the weight space $\calW$. Let $J$ be a nonempty subset of $I$ and 
$r_J:=(r_j)_{j\in J}\in (0,1)^J$. We  denote by $\calW^{>r_J}$ the subspace of $\calW$ where $|T_j|_p>r_j$ for all $j\in J$. Set  $\calZ_D^{>r_J}:=\mathrm{wt}^{-1}(\calW^{>r_J})$ and $\calX_D^{>r_J}:=w^{-1}(\calW^{>r_J})$.

\begin{remark}
	 We can use a similar trick as in \cite[Remark~2.14]{liu2017eigencurve} to replace the open compact subgroup $K^p$ be a sufficiently small one. For an ideal $\gothn'$ of $\calO_F$ prime to $p\delta_D$,  let $K'^p=K^p\cap U_0^{(p)}(\gothn')$ or $K^p\cap U_1^{(p)}(\gothn')$. Then the spectral variety (resp. eigenvariety) for $K^p$ is a union of connected components of the spectral variety (resp. eigenvariety) for $K'^p$. Hence it suffices to prove the theorem for sufficiently small $K^p$.
\end{remark}

\subsection{Integral model of the space of $p$-adic automorphic forms}\label{S:Integral model of the space of $p$-adic automorphic forms}

Let $A$ be a topological ring in which $p$ is topologically nilpotent and $\chi=(\nu,\mu):\calO_p^\times\times\ZZ_p^\times\rightarrow A^\times$ be a continuous character. Let $\kappa=(n,\nu):\calO_p^\times\times\calO_p^\times\rightarrow A^\times$ be the associated character of $\calO_p^\times\times\calO_p^\times$. Recall that from $\kappa$, we can define a character $\kappa_T$ (resp. $\kappa_B$, $\kappa_{\bar{B}}$) of $T(\calO_p)$ (resp. $B(\calO_p)$, $\bar{B}(\pi\calO_p)$) as in  \S\ref{section:induced representations}.  For any $i\in I$, we use $\kappa_i=(n_i,\nu_i):\calO_{\gothp_i}^\times\times \calO_{\gothp_i}^\times\rightarrow A^\times$ to denote the $i$-component of $\kappa$. We extend the character $\nu$ to a character $\nu:F_p^\times\rightarrow A^\times$ as before.

 Similar with \cite[\S2.7]{liu2017eigencurve}, we make the following definition.
 \begin{definition}\label{D:space of  integral p-adic automorphic forms}
 	We define the space of  integral $p$-adic automorphic forms for $D$ to be 
 	\begin{align*}
 	S_{\kappa,I}^D(K^p,A) &:= \{\phi:D^\times \setminus D_f^\times /K^p\rightarrow \Ind_{B(\calO_p)}^{\Iw_\pi}(\kappa)| \phi(xu)=\phi(x)\circ u, \text{~for all~} x\in D_f^\times,u\in \Iw_\pi \}\\
 	&=\{\phi:D^\times \setminus D_f^\times /K^p\rightarrow \calC(\calO_p,A)| \phi(xu)=\phi(x)\circ u, \text{~for all~} x\in D_f^\times,u\in \Iw_\pi \}.
 	\end{align*}
 \end{definition}

\begin{remark}\label{R:comparison between spaces of integral p-adic automorphic forms of different levels}
	For any $t\in \NN^I$, we define the space of integral $p$-adic automorphic forms of level $K^p\Iw_{\pi^t}$ by
	\[
	S_{\kappa,I}^D(K^p\Iw_{\pi^t},A) := \{\phi:D^\times \setminus D_f^\times /K^p\rightarrow \calC(\calO_p,A)| \phi(xu)=\phi(x)\circ u, \text{~for all~} x\in D_f^\times,u\in \Iw_{\pi^t} \}.
	\]
	This definition gives no generalization of Definition~\ref{D:space of  integral p-adic automorphic forms}. In fact, for $t\in \NN^I$ and $s\in \ZZ_{\geq 0}^I$, we have an Hecke equivariant isomorphism between the spaces $S_{\kappa,I}^D(K^p\Iw_{\pi^t},A)$ and $S_{\kappa,I}^D(K^p\Iw_{\pi^{t+s}},A)$. This fact is well known to experts. For completeness, we give a sketch of the construction of the isomorphisms and refer \cite[\S11]{buzzard320eigenvarieties}, especially Proposition~$11.1$ for more details.
	
	For $\phi\in S_{\kappa,I}^D(K^p\Iw_{\pi^t},A)$, we define $\psi:D^\times \setminus D_f^\times /K^p\rightarrow \calC(\calO_p,A)$ by 
	\[
	\psi(x)=\phi\left(x\begin{bmatrix}
	\pi^{-s}&0\\
	0&1\end{bmatrix}\right)\circ \begin{bmatrix}
	\pi^{s}&0\\
	0&1\end{bmatrix}.
	\]
	One can check $\psi\in S_{\kappa,I}^D(K^p\Iw_{\pi^{t+s}},A)$ and we obtain a map $\iota_1:S_{\kappa,I}^D(K^p\Iw_{\pi^{t}},A)\rightarrow S_{\kappa,I}^D(K^p\Iw_{\pi^{t+s}},A)$. 
	
	Conversely, given $\psi\in S_{\kappa,I}^D(K^p\Iw_{\pi^{t+s}},A)$, we define $\phi:D^\times \setminus D_f^\times /K^p\rightarrow \calC(\calO_p,A)$ by
	\[
	\phi(x)(z'):=\psi\left(x\begin{bmatrix}
	\pi^{s}&\alpha\\
	0&1\end{bmatrix}\right)(z), \text{~with~} z'=\pi^sz+\alpha \text{~for some ~} z\textrm{~and~}\alpha\in \calO_p.
	\]
	One can verify that this definition is independent of the expression of $z'=\pi^sz+\alpha\in \calO_p$ and $\phi\in S_{\kappa,I}^D(K^p\Iw_{\pi^{t}},A)$. So we obtain another map $\iota_2:S_{\kappa,I}^D(K^p\Iw_{\pi^{t+s}},A)\rightarrow S_{\kappa,I}^D(K^p\Iw_{\pi^{t}},A)$. The two maps $\iota_1$ and $\iota_2$ establish the desired Hecke equivariant isomorphism.
\end{remark}

For a better interpretation of Definition~\ref{D:space of generalized integral p-adic automorphic forms} below, it is useful to explain the relation between the space of integral $p$-adic automorphic forms for $D$ and the completed homology group of the Shimura varieties associated to the algebraic group $\Res_{F/\QQ}D^\times$. Recall the completed homology group (of degree $0$) is defined as 
\[
\tilde{\rmH}_0:= \varprojlim_{K_p}\rmH_0(D^\times \setminus D_f^\times /K^pK_p,\ZZ_p)=\varprojlim_{K_p}\ZZ_p[D^\times \setminus D_f^\times /K^pK_p],
\]
where in the inverse limit, $K_p$ ranges over all open compact subgroups of $\GL_2(\calO_p)$. For each $K_p$, the group $\rmH_0(D^\times \setminus D_f^\times /K^pK_p,\ZZ_p)$ is endowed with the $p$-adic topology and $\tilde{\rmH}_0$ is endowed with the inverse limit topology. The complete homology group $\tilde{\rmH}_0$ carries a natural right action of $\GL_2(F_p)$ induced by the right translation of $\GL_2(F_p)$ on $D_f^\times$. It follows from \cite[Theorem~2.2]{emerton2014completed} that $\tilde{\rmH}_0$ is a finitely generated $\ZZ_p\llbracket K_p\rrbracket$-module for any (nonempty) open compact subgroup $K_p$ of $\GL_2(F_p)$.

\begin{proposition}\label{P:integral model of the space of p-adic autormorphic forms and complete homology groups}
	Assume that $A$ is linearly topologized, i.e. $0\in A$ has a fundamental system of neighborhoods consisting of ideals. There is a natural isomorphism of $A$-modules: 
	\[
	S_{\kappa,I}^D(K^p,A)\rightarrow \Hom_{cts,\bar{B}(\pi\calO_p)}(\tilde{\rmH}_0, A(\kappa_{\bar{B}})).
	\]
\end{proposition}

\begin{proof}
	This result is well known to experts. We only give a sketch of the construction of the isomorphism here. 
	
	By definition, the induced representation $\Ind_{B(\calO_p)}^{\Iw_{\pi}}(\kappa)$ is a subspace of $\calC(\Iw_{\pi},A)$. The right action of $\Iw_{\pi}$ on $\Ind_{B(\calO_p)}^{\Iw_{\pi}}(\kappa)$ can be extended to $\calC(\Iw_{\pi},A)$ by the same formula: $f\circ u(g)=f(gu^\ast)$. It is worthwhile to remark here that the action of the monoid $\bbM_1$ cannot be extended to $\calC(\Iw_{\pi},A)$.
	
	Given a map $\phi:D^\times \setminus D_f^\times /K^p\rightarrow \calC(\Iw_{\pi},A)$, with the property that $\phi(xu)=\phi(x)\circ u$, for all $x\in D_f^\times$ and $u\in \Iw_{\pi}$. We define a map $\tilde{f}:D^\times \setminus D_f^\times /K^p\rightarrow A$ by $\tilde{f}(x)=\phi(x)(I_2)$, where $I_2\in \Iw_{\pi}$ is the identity matrix. Since $\phi(x)\in \calC(\Iw_{\pi},A)$, for all $x\in D_f^\times$, $\tilde{f}$ induces a continuous map $f:\tilde{\rmH}_0\rightarrow A$. Moreover, when $\phi\in S_{\kappa,I}^D(K^p)$, it is straightforward to check that $f\in \Hom_{cts,\bar{B}(\pi\calO_p)}(\tilde{\rmH}_0, A(\kappa_{\bar{B}}))$. 
	
	Conversely, given a continuous map $f:\tilde{\rmH}_0\rightarrow A$, we view $f$ as a map $D^\times \setminus D_f^\times /K^p\rightarrow A$, such that $f \mod I$ is locally constant on cosets of $\GL_2(F_p)$ for all nonempty open ideals $I$ of $A$ as $A$ is linearly topologized. Define $\phi:D^\times \setminus D_f^\times /K^p\rightarrow \calC(\Iw_{\pi},A)$ by $\phi(x)(g)=f(xg^\ast)$, for $x\in D_f^\times$, $g\in \Iw_{\pi}$. It is easy to check that $\phi(xu)=\phi(x)\circ u$, for $x\in D_f^\times$ and $u\in \Iw_{\pi}$, and if $f\in \Hom_{cts,\bar{B}(\pi\calO_p)}(\tilde{\rmH}_0, A(\kappa_{\bar{B}}))$, then $\phi(x)\in \Ind_{B(\calO_p)}^{\Iw_{\pi}}(\kappa)$, for all $x\in D_f^\times$. 
\end{proof}

\begin{remark}
	Conceptually, Proposition~\ref{P:integral model of the space of p-adic autormorphic forms and complete homology groups} follows from the tautological isomorphism
	\begin{equation}\label{E:isomorphism of p-adic automoprhism in term of completed homology}
	S_{\kappa,I}^D(K^p,A)=\Hom_{\Iw_p}(D^\times\setminus D_f^\times/K^p,\Ind_{B(\calO_p)}^{\Iw_\pi}(\kappa))\cong \Hom_{\ZZ_p\llbracket \Iw_\pi\rrbracket}(\tilde{\rmH}_0,\Ind_{B(\calO_p)}^{\Iw_\pi}(\kappa)))
	\end{equation}
	and Frobenius reciprocity. However, the induced representation $\Ind_{B(\calO_p)}^{\Iw_\pi}(\kappa)$ is slightly different from the standard definition. This is why we give a concrete proof here.
\end{remark}

Inspired by Proposition~\ref{P:integral model of the space of p-adic autormorphic forms and complete homology groups}, if we replace the group $\bar{B}(\pi\calO_p)$ by $\bar{B}(\pi_i\calO_{\gothp_i})$ for some $i\in I$, we obtain the space $\Hom_{\bar{B}(\pi_i\calO_{\gothp_i})}(\tilde{\rmH}_0, A(\kappa_{\bar{B}}))$ that contains $S_{\kappa,I}^D(K^p,A)$. Elements in $\Hom_{\bar{B}(\pi_i\calO_{\gothp_i})}(\tilde{\rmH}_0, A(\kappa_{\bar{B}}))$ looks like automorphic forms at the place $i\in I$, and are continuous functions on the complete homology $\tilde{\rmH}_0$ at all other places $i'\neq i$. For our argument, we need a generalization of the space $\Hom_{\bar{B}(\pi_i\calO_{\gothp_i})}(\tilde{\rmH}_0, A(\kappa_{\bar{B}}))$. First we introduce some notations.

\begin{notation}\label{N:J component of integer rings, matrix groups and characters}
	Let $J$ be a subset of $I$. We denote by $J^c$ the complement of $J$ in $I$. 
	\begin{enumerate}
		\item Let $\calO_{p,J}:=\prod\limits_{j\in J}\calO_{\gothp_j}$, $\pi_J:=\prod\limits_{j\in J}\pi_j\in\calO_{p,J}$ and $\Iw_{\pi,J}:=\prod\limits_{j\in J}\Iw_{\pi_j}$ which is regarded as a subgroup of $\Iw_{\pi}$ whose $j'$-component is the identity matrix for all $j'\notin J$.
		\item Similar to (1), we define subgroup $B(\calO_{p,J})$ (resp. $T(\calO_{p,J})$, $\bar{B}(\pi_J\calO_{p,J})$, $\bar{N}(\pi_J\calO_{p,J})$, $D(\calO_{p,J})$) of $B(\calO_p)$ (resp. $T(\calO_p)$, $\bar{B}(\pi\calO_p)$, $\bar{N}(\pi\calO_p)$, $D(\calO_p)$) and the monoid $\bbM_{\pi,J}\subset \rmM_2(\calO_{p,J})$.
		\item Fix $j\in I$. We set $P_j':=\left\{g\in \Iw_{\pi_j,1}| \det(g)=1 \right\}$ and denote by $P_j$  the subgroup of $\Iw_{\pi_j}$ generated by $\begin{bmatrix}
		1&0\\
		0&\Delta_j\end{bmatrix}\subset T(\calO_{\gothp_j})$ and $P_j'$. The map $D(\calO_{\gothp_j})\times P_j\rightarrow \Iw_{\pi_j}$, $(d,g)\mapsto dg$ is an isomorphism of groups (here we use the assumption that $p>2$). $P_j'$ is the pro-$p$ normal subgroup of $P_j$ and $P_j/P_j'\cong 
		\Delta_j$.
		\item Let $P_J:=\prod\limits_{j\in J}P_j$ and $P_J':=\prod\limits_{j\in J}P_j'$.
		\item We use $\chi_J:=(\nu_J,\mu):\calO_{p,J}^\times\times \ZZ_p^\times\rightarrow A^\times$ to denote the restriction to the $J$-component of a continuous homomorphism $\chi=(\nu,\mu):\calO_p^\times\times\ZZ_p^\times\rightarrow A^\times$.
	\end{enumerate}
\end{notation}

In the rest of this section, we fix a nonempty subset $J$ of $I$ and a continuous character $\chi_J=(\nu_J,\mu):\calO_{p,J}^\times\times\ZZ_p^\times\rightarrow A^\times$. We define a character $\kappa_{B,J}:B(\calO_{p,J})\times D(\calO_{p,J^c})\rightarrow A^\times$ by
\[
\kappa_{B,J}\left( \begin{bmatrix}
a_j&b_j\\
0&d_j\end{bmatrix}\right):=\nu_j(a_j/d_j)\mu(\Nm_{F/\QQ}(d_j)), 
\text{~for~} j\in J,
\]
and
\[
\kappa_{B,J}\left( \begin{bmatrix}
a_j&0\\
0&a_j\end{bmatrix}\right):=\mu(\Nm_{F/\QQ}(a_j)), 
\text{~for~} j\notin J.
\]
We remark that if $\chi_J$ is the restriction of a continuous character $\chi:\calO_p^\times\times\ZZ_p^\times\rightarrow A^\times$, then $\kappa_{B,J}$ is the restriction of the character $\kappa_B:B(\calO_p)\rightarrow A^\times$, where $\kappa:\calO_p^\times\times\calO_p^\times\rightarrow A^\times$ is the character associated to $\chi$ and $\kappa_B$ is the character defined in \S\ref{section:induced representations}. 
\begin{definition}
	Under the above notations, we set $\calC_{\chi_J}(\Iw_{\pi},A):=\Ind_{B(\calO_{p,J})\times D(\calO_{p,J^c})}^{\Iw_{\pi}}(\kappa_{B,J})$. In particular, when $J=I$ and $\chi=\chi_J$ is a character of $\calO_p^\times\times\ZZ_p^\times$,  $\calC_{\chi_J}(\Iw_{\pi},A)$ is the induced representation $\Ind_{B(\calO_p)}^{\Iw_{\pi}}(\kappa_B)$ defined in \S\ref{section:induced representations}.
\end{definition}

\begin{remark}\label{remark:some basic properties of the space C_J(kappa,A)}
	\begin{enumerate}
		\item By definition, $\calC_{\chi_J}(\Iw_{\pi},A)$ is an $A$-submodule of the induced representation  $\Ind_{B(\calO_{p,J})}^{\Iw_{\pi}}(\kappa_J)\subset\calC(\Iw_{\pi},A)$. Recall that we have defined a right action of $\Iw_{\pi}$ on $\calC(\Iw_{\pi},A)$ via the formula $f\circ u(g)=f(gu^\ast)$ for $f\in \calC(\Iw_{\pi},A)$ and $u\in \Iw_{\pi}$. This action leaves the two spaces $\calC_{\chi_J}(\Iw_{\pi},A)$ and $\Ind_{B(\calO_{p,J})}^{\Iw_{\pi}}(\kappa_J)$ stable.
		\item Under the Iwasawa decomposition, $\Ind_{B(\calO_{p,J})}^{\Iw_{\pi}}(\kappa_J)$ can be identified with $\calC(\calO_{p,J}\times \Iw_{\pi,J^c},A)$, and the $\Iw_{\pi}$-action on the latter space is given by the following explicit formula:
		\[
		h(z_J,g_{J^c})\circ u=\big(\prod_{j\in J}n_j(c_jz_j+d_j)v_j(\det(u_j)) \big) h(z_J',g_{J^c}u_{J^c}^\ast),
		\]
		where $z_J=(z_j)_{j\in J}\in\calO_{p,J}$, $g_{J^c}=(g_{j'})_{j'\in J^c}\in \Iw_{\pi,J^c}$ and $z_J'=(\frac{a_jz_j+b_j}{c_jz_j+d_j})_{j\in J}\in\calO_{p,J}$. Using this formula, we can extend this action to the monoid $\bbM_{\pi,J}\times \Iw_{\pi,J^c}$. This action also leaves $\calC_{\chi_J}(\Iw_{\pi},A)$ stable.
		\item Under the Iwasawa decomposition and the bijection $D(\calO_{\gothp_j})\times P_j\rightarrow \Iw_{\pi_j}$ for all $j\in J^c$, we can identify $\calC_{\chi_J}(\Iw_{\pi},A)$ with $\calC(\calO_{p,J}\times P_{J^c},A)$.
		\item For $\alpha\in \calO_p^\times$, the diagonal matrix $\mathrm{Diag}(\alpha):= \begin{bmatrix}
		\alpha&0\\
		0&\alpha\end{bmatrix} \in D(\calO_p)\subset \Iw_{\pi}$ with $\alpha$ on the diagonal acts on $\calC_{\chi_J}(\Iw_{\pi},A)$ via multiplication by $\mu(\Nm_{F/\QQ}(\alpha))$. This fact will be used in \S\ref{subsection:Explicit expression of p-adic automorphic forms}.
	\end{enumerate}
\end{remark}

\begin{definition}\label{D:space of generalized integral p-adic automorphic forms}
	Given a nonempty subset $J$ of $I$ and a continuous character $\chi_J:\calO_{p,J}^\times\times\ZZ_p^\times\rightarrow A^\times$, we define the space of generalized integral $p$-adic automorphic forms for $D$ to be
	\[
	S_{\kappa,J}^D(K^p,A):=\{\phi:D^\times \setminus D_f^\times /K^p\rightarrow \calC_{\chi_J}(\Iw_\pi,A)| \phi(xu)=\phi(x)\circ u, \text{~for all~} x\in D_f^\times,u\in \Iw_\pi \}.
	\]
	For $j\in J$, we can use the exact formula \eqref{E:formula to define Up operator} to define the $U_{\pi_j}$-operator on $S_{\kappa,J}^D(K^p,A)$.
\end{definition}

\begin{remark}
	The reason we use the notation $S_{\kappa,J}^D(K^p,A)$ to denote the space of generalized integral $p$-adic automorphic forms is that we want to make it compatible with Definition~\ref{D:space of  integral p-adic automorphic forms}, i.e. when $J=I$, $S_{\kappa,I}^D(K^p,A)$ defined above coincides with the space defined in Definition~\ref{D:space of  integral p-adic automorphic forms}. But we point out that the definition of $S_{\kappa,J}^D(K^p,A)$ only depends on the character $\chi_J$ of $\calO_{p,J}^\times\times\ZZ_p^\times$, even though it may be the restriction of a character $\chi$ of $\calO_p^\times\times\ZZ_p^\times$.
\end{remark}

\begin{remark}\label{R:remark for integral p-adic automorphic forms}
	Let $J_1\subset J_2$ be two nonempty subsets of $I$ and $\chi_{J_2}:\calO_{p,J_2}^\times\times\ZZ_p^\times\rightarrow A^\times$ be a continuous character. Let $\chi_{J_1}$ be the restriction of $\chi_{J_2}$ to $\calO_{p,J_1}^\times\times\ZZ_p^\times$. For later argument, it is useful to explain the relation between the two spaces $S_{\kappa,J_1}^D(K^p,A)$ and $S_{\kappa,J_2}^D(K^p,A)$.
	\begin{enumerate}
		\item Since $B(\calO_{p,J_1})\times D(\calO_{p,J_1^c})$ is a subgroup of $B(\calO_{p,J_2})\times D(\calO_{p,J_2^c})$, we have a natural injection $\calC_{\chi_{J_2}}(\Iw_\pi,A)\hookrightarrow \calC_{\chi_{J_1}}(\Iw_\pi,A)$, which is equivariant under the action of the monoid $\bbM_{\pi,J_1}\times \Iw_{\pi,J_1^c}$. This induces an injection  $S_{\kappa,J_2}^D(K^p,A)\rightarrow S_{\kappa,J_1}^D(K^p,A)$. This map is compatible with the $U_{\pi_j}$-operator on these spaces for all $j\in J_1$.
		\item We can define a right action of $B(\calO_{p,J^c})$ on $ \calC(\calO_{p,J}\times \Iw_{\pi,J^c}, A)$ by $(h\cdot b_{J^c})(z_J,g_{J^c}):=h(z_J,b_{J^c}g_{J^c})$ for $b_{J^c}\in B(\calO_{p,J^c})$ and $h\in \calC(\calO_{p,J}\times \Iw_{\pi,J^c}, A)$. This action leaves $\calC_{\chi_J}(\Iw_\pi,A)$ stable as $D(\calO_{p,J^c})$ is central in $\Iw_{\pi,J^c}$ and hence  induces a right action of $B(\calO_{p,J^c})$ on $S_{\kappa,J}^D(K^p,A)$. Set $J_3:=J_2\setminus J_1$. Then the injection $\calC_{\chi_{J_2}}(\Iw_\pi,A)\hookrightarrow \calC_{\chi_{J_1}}(\Iw_\pi,A)$ (resp. $S_{\kappa,J_2}^D(K^p,A)\rightarrow S_{\kappa,J_1}^D(K^p,A)$) identifies the first space with the  subspace of the latter space on which the Borel subgroup $B(\calO_{p,J_3})$ acts via the character $\kappa_{B,J_2}$ defined above. 
	\end{enumerate}
\end{remark}


\subsection{Spaces of  classical automorphic forms}\label{section:spaces of  classical automorphic forms} There are three goals in this section. 
\begin{enumerate}
	\item Recall the definitions of classical weights and spaces of classical automorphic forms as defined in \cite{buzzard320eigenvarieties}.
	\item Prove an Atkin-Lehner duality result (see Proposition~\ref{P:Atkin-Lehner map and Up operator} below).
	\item Explain how to realize some spaces of classical automorphic forms as subspaces of the space of generalized integral $p$-adic automorphic forms.
\end{enumerate}

Let $\nu=(\nu_i)_{i\in I}\in \ZZ^I$ and $\mu\in \ZZ$, such that $n=(n_i:= \mu-2\nu_i)_{i\in I}\in \NN^I$. Define a character $\chi:\calO_p^\times\times \ZZ_p^\times\rightarrow \QQ_p^\times$, $(\alpha,\beta)\mapsto \beta^r\prod\limits_{i\in I}i_p(\alpha_i)^{\nu_i}$. The associated character  $\kappa:\calO_p^\times\times \calO_p^\times\rightarrow \QQ_p^\times$ is given by $\kappa(\alpha,\beta)=\prod\limits_{i\in I}i_p(\alpha_i)^{n_i}i_p(\beta_i)^{\nu_i}$.
We call such a weight \emph{algebraic}. Let $L$ be a finite extension of $\QQ_p$. A weight $\kappa:\calO_p^\times\times \calO_p^\times\rightarrow L^\times$ is called locally \emph{algebraic}, or \emph{classical}, if $\kappa$ decomposes as $\kappa=\kappa_{alg}\kappa_{fin}$, where $\kappa_{alg}$ (resp. $\kappa_{fin}=\psi=(\psi_1,\psi_2):\calO_p^\times\times\calO_p^\times\rightarrow L^\times$) is an algebraic weight (resp. a finite character). Hence a locally algebraic weight can be represented by  a triple $((n_i)_{i\in I}\in\NN^I,(\nu_i)_{i\in I}\in \ZZ^I,\psi=(\psi_1,\psi_2))$ with the property that $n+2\nu\in \ZZ$ and the character $\psi_1\cdot \psi_2^2:\calO_p^\times\rightarrow L^\times$ factors through the norm map $\Nm_{F/\QQ}:\calO_p^\times\rightarrow \ZZ_p^\times$ (see Convention~\ref{Convention: a weight is the character associated to a character of the weight space}).

Fix such a locally algebraic weight $\kappa=(n=(n_i)_{i\in I},\nu=(\nu_i)_{i\in I},\psi)$. Let $r(\kappa):=(p^{-m_{\kappa,i}})_{i\in I}$ and $t:=(t_i)_{i\in I}\in \NN^I$ such that $\psi:\calO_p^\times\times \calO_p^\times\rightarrow L^\times$ factors through $(\calO_p/\pi^t)^\times\times(\calO_p/\pi^t)^\times$. It is straightforward to verify that $t_i\geq m_{\kappa,i}$ for all $i\in I$. In particular, $t$ is good for $(\kappa,r)$ for all $r\in \calN^I$. From the integer $\nu_i$ and the character $\psi_{2,i}:\calO_{\gothp_i}^\times\rightarrow L^\times$, we define another character $\tau_i:\calO_{\gothp_i}^\times\rightarrow L^\times$, $\beta_i\mapsto i_p(\beta_i)^{\nu_i}\psi_{2,i}(\beta_i)$ and extend it to a character of $F_{\gothp_i}^\times$ by setting $\tau_i(\pi_i)=1$. 

 We put $L_{\kappa}$ to be the $L$-vector space with basis $\Sigma_\kappa:=\left\{\prod\limits_{i\in I}Z_i^{l_i}\;\Big|\;0\leq l_i\leq n_i, \text{~for all~} i\right\}$, where $Z_i$'s are independent indeterminates.
The space $L_{\kappa}$ carries a right action of $\bbM_{\pi^t}$  by 
\[
\left(\prod\limits_{i\in I}Z_i^{l_i}\right)\circ \gamma=\prod_{i\in I}\psi_{1,i}(d_i)\tau_i(\det(\gamma_i))(a_iZ_i+b_i)^{l_i}(c_iZ_i+d_i)^{n_i-l_i},
\]
for $\gamma=(\gamma_i)_{i\in I}\in M_{\pi^t}$ with $\gamma_i=\left[ \begin{array}{cc}
a_i&b_i\\
c_i&d_i\end{array}\right]$ for $i\in I$. 

Let $\calO_L$ be the ring of integers of $L$. We denote by $L_{\kappa,\calO_L}$ the $\calO_L$-lattice of $L_\kappa$ spanned by the polynomials in $\Sigma_\kappa$. Since the character $\kappa$ takes values in $\calO_L^\times$, the action of the monoid $\bbM_{\pi^t}$ on $L_\kappa$ leaves $L_{\kappa,\calO_L}$ stable.

Fix an open compact subgroup $K^p$ of $D_f^{(p),\times}$ and let $K=K^p\Iw_{\pi^t}$ be the subgroup of $D_f^\times$.

\begin{definition}\label{D:classical automorphic forms}
	Under the above notations, define $k=n+2\in \ZZ_{\geq 2}^I$ and $w=n+\nu+1\in \ZZ^I$. The space of classical automorphic forms of weight $(k,w)$, level $K$ and character $\psi$ for $D$ is defined by
	\[
	S_{k,w}^D(K,\psi):=\{\phi: D^\times \setminus D_f^\times /K^p\rightarrow L_\kappa|\phi(xu)=\phi(x)\circ u,\text{~for all~} x\in D_f^\times,u\in \Iw_{\pi^t} \}.
	\] 
	We also define
	\[
	S_{k,w}^D(K,\psi,\calO_L):=\{\phi: D^\times \setminus D_f^\times /K^p\rightarrow L_{\kappa,\calO_L}|\phi(xu)=\phi(x)\circ u,\text{~for all~} x\in D_f^\times,u\in \Iw_{\pi^t} \}.
	\] 
Note that	$S_{k,w}^D(K,\psi,\calO_L)$ is an $\calO_L$-lattice of $S_{k,w}^D(K,\psi)$ and stable under the $U_{\pi_i}$-operators for all $i\in I$.
\end{definition}

\begin{remark}
	Convention~\ref{Convention:extend the character nu} on the weight $\kappa=(n,\nu):\calO_p^\times\times \calO_p^\times\rightarrow L^\times$ makes our definition of Hecke operators on the spaces of classical automorphic forms slightly different from the `usual' one (see \cite[\S2]{hida1988p} for example). To be more precise, we assume that the character $\psi$ above is trivial for simplicity. We will define a second action of the monoid $\bbM_{\pi^t}$ on $L_\kappa$ below. This action coincides with the one that we defined above when restricting to $\Iw_{\pi^t}$. In particular, it defines the same spaces of classical automorphic forms. However, this two actions differ by a power of $\pi$ as an action of the monoid $\bbM_{\pi^t}$, and hence define different Hecke operators. If we use $U_{\pi_i,cl}$ to denote the Hecke operator defined by the second action $\Vert_\gamma$, then it is related with our Hecke operator $U_{\pi_i}$ defined in \S\ref{subsection: Hecke operators} via the equality $U_{\pi_i,cl}=\pi_i^{\nu_i}U_{\pi_i}$. The reason that we renormalize the Hecke operators is that the quantity $\pi_i^{\nu_i}$ does not vary analytically with $\nu$. 
\end{remark}

We keep the notations as before. We define another right action of the monoid $\bbM_{\pi^t}$ on $L_{\kappa} $ by 
\[
\left(\prod_{i\in I}Z_i^{l_i}\right)\Big\Vert_\gamma:=\prod_{i\in I}(\det(\gamma_i))^{\nu_i}((a_iZ_i+b_i)^{l_i}(c_iZ_i+d_i)^{n_i-l_i})
\]
for $\gamma=(\gamma_i)_{i\in I}\in \bbM_{\pi^t}$ and $\gamma_i=\left[ \begin{array}{cc}
a_i&b_i\\
c_i&d_i\end{array}\right]$ for all $i\in I$.
Then the space of classical automorphic forms can be described by  
	\begin{multline*}	
	S_{k,w}^D(K,\psi)=\Big\{\phi: D^\times \setminus D_f^\times /K^p\rightarrow L_{\kappa}\;\Big|\;\phi(xu)=\psi_1(d)\psi_2(\det(u))\phi(x)\Vert_u,
	\\
	\text{~for~} x\in D_f^\times,u=\left[ \begin{array}{cc}
	a&b\\
	c&d\end{array}\right]\in \Iw_{\pi^t} \Big\}.
	\end{multline*}
	Note that since $\psi_2$ factors through $(\calO_p/\pi^t)^\times$, we have $\psi_2(\det(u))=\psi_2(ad)$.

Fix $j\in I$. Define another finite character $\psi'=(\psi_1',\psi_2'):\calO_p^\times\times\calO_p^\times\rightarrow L^\times$ by setting: 
\begin{equation*}
\psi_{1,i}':=
\begin{cases}
\psi_{1,i}, &\mbox{if }i\neq j\\
\psi_{1,i}^{-1}, &\mbox{if }i=j
\end{cases} \textrm{~and~}
\psi_{2,i}':=
\begin{cases}
\psi_{2,i}, &\mbox{if }i\neq j\\
\psi_{1,i}\psi_{2,i}, &\mbox{if }i=j
\end{cases}.
\end{equation*}
In other words, if we regard $\psi_{1,j}$ as a character of $\calO_p^\times$ via the natural projection $\calO_p^\times \rightarrow \calO_{\gothp_j}^\times$, we have $\psi_1'=\psi_1\psi_{1,j}^{-2}$ and $\psi_2'=\psi_2\psi_{1,j}$.

The Atkin-Lehner map is defined by
\[
\AL_j:S_{k,w}^D(K,\psi)\rightarrow S_{k,w}^D(K,\psi'), \phi\mapsto \phi(\bullet v_j^{-1})\Vert_{v_j},
\]
where $v_j=\left[ \begin{array}{cc}
0&1\\
\pi_j^{t_j}&0\end{array}\right]\in \bbM_{\pi_j^{t_j}}$ and we view $v_j$ as an element of $\bbM_{\pi^t}$ whose $i$-component is the identity, for all $i\neq j$. First we verify that the map $\AL_j$ is well-defined. Given $\phi\in S_{k,w}^D(U,\psi)$, $x\in D_f^\times$ and $u_j=\left[ \begin{array}{cc}
a_j&b_j\\
c_j&d_j\end{array}\right]\in \Iw_{\pi_j^{t_j}}$, we set $\tilde{u}_j:= v_ju_jv_j^{-1}=\left[ \begin{array}{cc}
d_j&\pi_j^{-t_j}c_j\\
\pi_j^{t_j}b_j&a_j\end{array}\right]$. By a direct computation, we have
\begin{equation*}
\begin{split}
\AL_j(\phi)(xu_j)&=\phi(xu_jv_j^{-1})\Vert_{v_j}
=\phi(xv_j^{-1}\tilde{u}_j)\Vert_{v_j}=(\psi_{1,j}(a_j)\psi_{2,j}(\det(\tilde{u}_j))\phi(xv_j^{-1})\Vert_{\tilde{u}_j})\Vert_{v_j}
\\
&=\psi_{1,j}(a_j)\psi_{2,j}(\det(\tilde{u}_j))\phi(xv_j^{-1})\Vert_{\tilde{u}_jv_j}=\psi_{1,j}(a_j)\psi_{2,j}(\det(\tilde{u}_j))(\phi(xv_j^{-1})\Vert_{v_j})\Vert_{u_j}
\\
&=\psi_{1,j}^{-1}(d_j)\psi_{1,j}(a_jd_j)\psi_{2,j}(\det(u_j))\AL_j(\phi)(x)\Vert_{u_j}=\psi_{1,j}'(d_j)\psi_{2,j}'(\det(u_j))\AL_j(\phi)(x)\Vert_{u_j}.
\end{split}
\end{equation*}
Thus, we know that $\AL_j(\phi)\in S_{k,w}^D(K,\psi')$ and $\AL_j$ is well-defined. It is clear that $\AL_j$ is an $L$-linear isomorphism and induces an $\calO_L$-linear isomorphism between $ S_{k,w}^D(K,\psi,\calO_L)$ and $ S_{k,w}^D(K,\psi',\calO_L)$.

\begin{definition}\label{D:j-Atkin-Lehner dual weight of a locally algebraic weight}
	The locally algebraic weight $\kappa'$ corresponding to the triple $(n,v,\psi')$ is called the \emph{$j$-Atkin-Lehner dual weight} of $\kappa$. 
\end{definition}

\begin{proposition}\label{P:Atkin-Lehner map and Up operator} 
	Assume that the finite character $\psi_{1,j}:\calO_{\gothp_j}^\times\rightarrow L^\times$ has conductor $t_j\geq 2$. 
	For $\phi\in S_{k,w}^D(K,\psi)$, we have $U_{\pi_j}\circ \AL_j\circ U_{\pi_j}(\phi)=p\pi_j^{n_j}\AL_j\circ S_{\pi_j}(\phi)$, where $S_{\pi_j}:S_{k,w}^D(K,\psi)\rightarrow S_{k,w}^D(K,\psi)$ is the automorphism defined by $\phi\mapsto \phi\left(\bullet \left[ \begin{array}{cc}
	\pi_j^{-1}&0\\
	0&\pi_j^{-1}\end{array}\right]\right)$ .
\end{proposition}

\begin{proof}
	First we make a remark on notations. In this proposition, we have two locally algebraic weights $\kappa=(n,v,\psi)$ and $\kappa'=(n,v,\psi')$, and hence two different actions of the monoid $\bbM_{\pi^t}$ on the $L$-vector space $L_{\kappa}=L_{\kappa'}$. However, for $v_{j,l}=\left[ \begin{array}{cc}
	\pi_j&0\\
	l\pi_j^{t_j}&1\end{array}\right]$, $l=0,\dots,p-1$, the right action of $v_{j,l}$ on $L_{\kappa}=L_{\kappa'}$ is given by $h\circ v_{j,l}=h\Vert_{v_{j,l}}$, no matter which weights $\kappa$ or $\kappa'$ that we consider. Hence for $\phi\in S_{k,w}^D(K,\psi)$ or $S_{k,w}^D(K,\psi')$, we always have $U_{\pi_j}(\phi)(x)=\sum\limits_{l=0}^{p-1}\phi(xv_{j,l}^{-1})\Vert_{v_{j,l}}$.
	
	For $x\in D_f^\times$ and $\phi\in S_{k,w}^D(K,\psi)$, we have  $$U_{\pi_j}\circ \mathrm{AL}_j\circ U_{\pi_j}(\phi)(x)=\sum\limits_{l,m=0}^{p-1}\phi(xv_{j,m}^{-1}v_j^{-1}v_{j,l}^{-1})\Vert_{v_{j,l}v_jv_{j,m}}=\sum\limits_{l,m=0}^{p-1}\phi(xv_j^{-1}(v_jv_{j,m}^{-1}v_j^{-1}v_{j,l}^{-1}))\Vert_{v_{j,l}v_jv_{j,m}}.$$ 
	
	By a direct computation, we have 
	\begin{equation*}
	\begin{split}
	v_jv_{j,m}^{-1}v_j^{-1}v_{j,l}^{-1}= \begin{bmatrix}
	1&-m\pi_j^{-1}\\
	0 &\pi_j^{-1}\end{bmatrix}\begin{bmatrix}
	\pi_j^{-1}&0\\
	-l\pi_j^{t_j-1} &1\end{bmatrix}
	=\begin{bmatrix}
	\pi_j^{-1}&0\\
	-l\pi_j^{t_j-2} &\pi_j^{-1}\end{bmatrix} \begin{bmatrix}
	1+lm\pi_j^{t_j-1}&-m\\
	ml^2\pi_j^{2t_j-2} &1-lm\pi_j^{t_j-1}\end{bmatrix}.
	\end{split}
	\end{equation*}
	
	From the hypothesis $t_j\geq 2$, the matrix $ \begin{bmatrix}
1+lm\pi_j^{t_j-1}&-m\\
ml^2\pi_j^{2t_j-2} &1-lm\pi_j^{t_j-1}\end{bmatrix}$ belongs to $\Iw_{\pi_j^{t_j}}$.  Hence, we have $$\phi(xv_j^{-1}(v_jv_{j,m}^{-1}v_j^{-1}v_{j,l}^{-1}))=\psi_{1,j}(1-lm\pi_j^{t_j-1})\phi\left( xv_j^{-1} \begin{bmatrix}
\pi_j^{-1}&0\\
-l\pi_j^{t_j-2} &\pi_j^{-1}\end{bmatrix} \right)\Bigg\Vert_{ \left[\begin{smallmatrix}
	1+lm\pi_j^{t_j-1}&-m\\
	ml^2\pi_j^{2t_j-2} &1-lm\pi_j^{t_j-1}\end{smallmatrix}\right]},$$ and consequently 
\begin{multline*}
U_{\pi_j}\circ \mathrm{AL}_j\circ U_{\pi_j}(\phi)(x)\\
=\sum\limits_{l,m=0}^{p-1}\psi_{1,j}(1-lm\pi_j^{t_j-1})\phi\left( xv_j^{-1} \begin{bmatrix}
\pi_j^{-1}&0\\
-l\pi_j^{t_j-2} &\pi_j^{-1}\end{bmatrix} \right)\Bigg\Vert_{\left[\begin{smallmatrix}
	1+lm\pi_j^{t_j-1}&-m\\
	ml^2\pi_j^{2t_j-2} &1-lm\pi_j^{t_ij-1}\end{smallmatrix}\right]v_{j,l}v_jv_{j,m}}.
\end{multline*}
	Since $\psi_{1,j}$ is a finite character of conductor $\pi_j^{t_j}$, we have 
	\begin{equation*}
	\sum_{m=0}^{p-1}\psi_{1,j}(1-lm\pi_j^{t_j-1})=
	\begin{cases}
	p, & \text{~for~} l=0, \\
	0, & \text{otherwise},
	\end{cases}
	\end{equation*}
	and hence \[U_{\pi_j}\circ \mathrm{AL}_j\circ U_{\pi_j}(\phi)(x)=p\phi\left( xv_j^{-1}\left[ \begin{array}{cc}
	\pi_j^{-1}&0\\
	0 &\pi_j^{-1}\end{array}\right] \right)\Bigg\Vert_{\left[ \begin{array}{cc}
		\pi_j&0\\
		0 &\pi_j\end{array}\right]v_j}=p\pi_j^{n_j}\mathrm{AL}_j\circ S_{\pi_j}(\phi)(x).\qedhere\]
\end{proof}

\begin{proposition}\label{P:pairing of Newton slopes by Atkin-Lehner}
	The slopes of the $U_{\pi_j}$ operator on the two spaces $S_{k,w}^D(K,\psi)$ and $S_{k,w}^D(K,\psi')$ can be paired such that the slopes in each pair sum to $n_i+1=k_i-1$.
\end{proposition}
We will give the proof of this proposition after introducing the following lemma.

	\begin{lemma}\label{L:a basic linear algebra lemma}
		Let $L$ be a field with a valuation $v_L(\cdot)$ and $\calO_L$ be its valuation ring. Consider three matrices $A,B\in \rmM_n(L)$ and $U\in \GL_n(\calO_L)$, such that $U$ commutes with $B$ and $AB=\alpha U$ for some $\alpha\in \calO_L$ with $v_L(\alpha)=k$. Then one can pair the slopes of the Newton polygons of $A$ and $B$, such that the slopes in each pair sum to $k$. 
	\end{lemma}
	
	\begin{proof}
		Replacing $L$ by its algebraic closure $\bar{L}$ and extending $v_L(\cdot)$ to $\bar{L}$, we can assume that $L$ is algebraically closed. We regard $A,B,U$ as $L$-linear operators on the space $V=L^n$. Then $V$ has the decomposition $V=\bigoplus\limits_{\lambda \text{~eigenvalue of~}U}V_\lambda$, where $V_\lambda=\ker((U-\lambda I_n)^n)$. Since $AB=\alpha U$ and $U$ commutes with $B$, the three matrices $A,B,U$ all commute with the others. Hence $A,B$ stabilizes the spaces $V_\lambda$'s. Hence we may assume that $V=\ker((U-\lambda I_n)^n)$ for some $\lambda\in L$. For the same reason, we can also assume that $V=\ker((A-\lambda_A I_n)^n)$ for some $\lambda_A\in L$. Since $U\in\GL_n(\calO_L)$, we have $\lambda\in \calO_L^\times$ and hence $v_L(\lambda)=0$. Let $\lambda_B=a\lambda\lambda_A^{-1}$. Then $B-\lambda_BI_n=\alpha A^{-1}U-\alpha\lambda\lambda_A^{-1}=\alpha\lambda_A^{-1}A^{-1}(\lambda_AU-\lambda A)=\alpha\lambda_A^{-1}A^{-1}(\lambda_A(U-\lambda I_n)+\lambda(\lambda_AI_n-A))$. Therefore $B-\lambda_BI_n$ is a nilpotent operator on $V$. The eigenvalues of $B$ are $\lambda_B$ with multiplicity $n$. Since $\lambda_A\lambda_B=\alpha\lambda$, we have $v_L(\lambda_A)+v_L(\lambda_B)=v_L(\alpha)+v_L(\lambda)=k$.
	\end{proof}

\begin{proof}[Proof of Proposition~\ref{P:pairing of Newton slopes by Atkin-Lehner}]
	We choose an $\calO_L$-basis $\Omega$ of $S_{k,w}^D(K,\psi,\calO_L)$ which can be viewed as an $L$-basis of $S_{k,w}^D(K,\psi)$. We denote by $A\in \rmM_n(L)$ (resp. $B\in \rmM_n(L)$, resp. $U\in \GL_n(\calO_L)$) the matrix corresponding to the map $\AL_j^{-1}\circ U_{\pi_j}\circ \AL_j: S_{k,w}^D(K,\psi)\rightarrow S_{k,w}^D(K,\psi)$ (resp. $U_{\pi_j}:S_{k,w}^D(K,\psi)\rightarrow S_{k,w}^D(K,\psi)$, resp. $S_{\pi_j}:S_{k,w}^D(K,\psi)\rightarrow S_{k,w}^D(K,\psi)$).

	Notice that the operator $S_{\pi_j}$ preserves $S_{k,w}^D(K,\psi,\calO_L)$ and commutes with $U_{\pi_j}$ as $\left[ \begin{array}{cc}
	\pi_j^{-1}&0\\
	0 &\pi_j^{-1}\end{array}\right]$ is central in $\GL_2(F_{\gothp_j})$. Now the proposition follows from Proposition~\ref{P:Atkin-Lehner map and Up operator} and Lemma~\ref{L:a basic linear algebra lemma}.
\end{proof}

Fix a nonempty subset $J$ of $I$ and a locally algebraic weight $\kappa$ corresponding to the triple $((n_i)_{i\in I}\in\ZZ_{\geq 0}^I,(\nu_i)_{i\in I}\in \ZZ^I,\psi=(\psi_1,\psi_2))$ with the following property: $n_j=\nu_j=0$ and $\psi_{1,j}|_{1+\pi_j\calO_{\gothp_j}}$, $\psi_{2,j}|_{1+\pi_j\calO_{\gothp_j}}$ are trivial, for all $j\notin J$. 
Equivalently, the character $\kappa_{J^c}$ equals to $\psi_{J^c}=(\psi_{1,J^c},\psi_{2,J^c}):\calO_{p,J^c}^\times\times \calO_{p,J^c}^\times\rightarrow L^\times$ and  factors through $(\calO_{p,J^c}/\pi_{J^c})^\times\times(\calO_{p,J^c}/\pi_{J^c})^\times$. The character $\kappa_{J^c,T}:T(\calO_{p,J^c})\rightarrow L^\times$ is trivial on $T(1+\pi_{J^c}\calO_{p,J^c})$ under the decomposition $T(\calO_{p,J^c})=T(\calO_{p,J^c}/\pi_{J^c})\times T(1+\pi_{J^c}\calO_{p,J^c})$.

Since $\Iw_{\pi,1,J^c}$ is a normal subgroup of $\Iw_{\pi,J^c}$ with quotient $\Iw_{\pi,J^c}/\Iw_{\pi,1,J^c}\cong T(\calO_{p,J^c}/\pi_{J^c})$, the character $\kappa_{T,J^c}:T(\calO_{p,J^c})\rightarrow L^\times$ induces a character $\alpha_{J^c}:\Iw_{\pi,J^c}\rightarrow L^\times$. As $p$ splits in $F$, $\alpha_{J^c}$ takes values in $\Delta\subset \ZZ_p^\times$. Moreover, if the locally algebraic weight $\kappa$ is the character associated to a point $\chi\in \calW(L)$, then the character $\alpha_{J^c}$ only depends on $\omega:=\chi|_H$ (or equivalently, the component that $\chi$ belongs to), where $H$ is the torsion subgroup of $\calO_p^\times\times \ZZ_p^\times$ and the set $J$. We choose $t=(t_i)_{i\in I}\in \NN^I$ such that the character $\psi$ factors through $(\calO_p/\pi^t)^\times\times (\calO_p/\pi^t)^\times$. In particular, we take $t_j=1$ for all $j\notin J$.

\begin{definition}
	We call $\alpha_{J^c}$ the character associated to the character $\omega\in H^\vee$ and the set $J$. 
\end{definition}

For the locally algebraic weight $\kappa$ we consider above, 
let $k,w\in \ZZ^I$ be defined as in Definition~\ref{D:classical automorphic forms}. We can regard the $L$-vector space $L_{\kappa}$ as polynomial functions on $\calO_p=\prod\limits_{i\in I}\calO_{\gothp_i}$ and hence obtain an embedding $L_{n,v}\rightarrow \calC(\calO_p,L)$, which is equivariant under the action of the monoid $\bbM_{\pi^t}$. So we have an embedding $S_{k,w}^D(K,\psi)\rightarrow S_{\kappa,I}^D(K^p\Iw_{\pi^t},L)$. Composing with the isomorphism $S_{\kappa,I}^D(K^p\Iw_{\pi^t},L)\cong S_{\kappa,I}^D(K^p,L)$ established in Remark~\ref{R:comparison between spaces of integral p-adic automorphic forms of different levels} and the embedding $S_{\kappa,I}^D(K^p,L)\hookrightarrow S_{\kappa,J}^D(K^p,L)$, we obtain an embedding $S_{k,w}^D(K,\psi)\rightarrow S_{\kappa,J}^D(K^p,L)$. Then we have the following characterization of the image of this embedding.
\begin{proposition}\label{characterization of classical automorphic forms in generalized automorphic forms}
	Under the above assumptions on the weight $\kappa$, the image of the embedding $S_{k,w}^D(K,\psi)\rightarrow S_{\kappa,J}^D(K^p,L)$ lies in the subspace of $S_{\kappa,J}^D(K^p,L)$ on which the group $\Iw_{\pi,J^c}$ acts via the character $\alpha_{J^c}:\Iw_{\pi,J^c}\rightarrow \ZZ_p^\times$ associated to $\omega\in H^\vee$ and the set $J$.
\end{proposition}

Fix a locally algebraic weight $\kappa:\calO_p^\times\times\calO_p^\times\rightarrow L^\times$ that corresponds to the triple $(n,\nu,\psi)$ and we choose $t=(t_i)_{i\in I}\in \NN^I$ such that $\psi$ factors through $(\calO_p/\pi^t)^\times\times (\calO_p/\pi^t)^\times$ as before. Let $K=K^p\Iw_{\pi^t}$ be an open compact subgroup of $D_f^\times$. Fix $i\in I$. Define $n',\nu'\in \ZZ^I$ as follows:
\begin{equation*}
n_j'=
\begin{cases}
n_j, &\mbox{if }j\neq i\\
n_j, &\mbox{if }j=i
\end{cases} \text{~and~}
\nu_j'=
\begin{cases}
\nu_j, &\mbox{if }j\neq i\\
\nu_j+n_j-1, &\mbox{if }j=i
\end{cases}.
\end{equation*}
Let $\kappa'$ be the locally algebraic weight that corresponds to the triple $(n',\nu',\psi)$. Fix $r\in \calN^I$ and we denote by $z_j$ the coordinate of the $j$-component of the polydisc $\bbB_r$ for all $j\in I$. Under the above notations, the differential operator $(\frac{d}{dz_i})^{n_i+1}:\calA_{\kappa,r}\rightarrow \calA_{\kappa',r}$ induces an operator $\theta_i:S_{\kappa}^D(K,r)\rightarrow S_{\kappa'}^D(K,r)$. Moreover, this map is equivariant for the $U_{\pi_i}$-operator on the source and the $\pi_i^{n_i+1}U_{\pi_i}$-operator on the target. When $F=\QQ$, this is proven in \cite[\S7]{buzzard2004p}. The proof of the general case is similar. At the end of this section, we record the following classicality result that characterizes the image of $S_{k,w}^D(K,\psi)$ in $S_{\kappa}^D(K,r)$.

\begin{proposition}\label{P:classicality of overconvergent automorphic forms on D}
	Fix $\phi\in S_{\kappa}^D(K,r)$. Then
	\begin{enumerate}
		\item $\phi\in S_{k,w}^D(K,\psi)$ if and only if $\theta_i(\phi)=0$ for all $i\in I$.
		\item If $\phi$ is an eigenform for the $U_{\pi_i}$-operator with nonzero eigenvalue $\lambda_i$ such that $v_p(\lambda_i)<n_i+1$ for all $i\in I$, then $\phi\in S_{k,w}^D(K,\psi)$.
	\end{enumerate}
\end{proposition}

\begin{proof}
	Part $(2)$ follows from \cite[Theorem~4.3.6]{birkbeck2016jacquet}. Part $(1)$ follows from the proof given there.
\end{proof}

\subsection{Explicit expression of $p$-adic automorphic forms}\label{subsection:Explicit expression of p-adic automorphic forms}
Let $A$ be a topological ring and $\kappa:\calO_p^\times\times\calO_p^\times\rightarrow A^\times$ as before.
Fix a double coset decomposition $D_f^\times=\bigsqcup\limits_{k=0}^{s-1}D^\times \gamma_kK^p\Iw_{\pi}$ with $\gamma_k\in D_f^\times$. We have an isomorphism of $A$-modules:
\[
S_{\kappa,I}^D(K^p)\xrightarrow{\cong}\bigoplus_{k=0}^{s-1}\calC(\calO_p,A)^{\Gamma_k}, \phi\mapsto (\phi(\gamma_k))_{k=0,\dots,s-1},
\]
where $\Gamma_k=\gamma_k^{-1}D^\times \gamma_k\cap K^p\Iw_{\pi}$ and $\Gamma_k$ acts on $\calC(\calO_p,A)$ via its $\GL_2(F_p)$-component.

We embed $\calO_F^\times$ diagonally to $D_f^\times$ and hence view $\calO_F^\times$ as a subgroup of $D_f^\times$. From \cite[Lemma~7.1]{hida1988p} and the assumption that $D/F$ is totally definite,  we can choose $K^p$ small enough such that $\Gamma_k\subset \calO_F^{\times,+}$, for all $k=0,\dots,s-1$. The subgroup $K^p$ with this property is called \emph{neat}. For a neat $K^p$, we have the isomorphism $S_{\kappa,I}^D(K^p)\xrightarrow{\cong}\bigoplus\limits_{k=0}^{s-1}\calC(\calO_p,A)$. Similarly we have explicit descriptions of the spaces of $p$-adic overconvergent automorphic forms by evaluating the functions at $\gamma_k$'s, for $k=0,\dots,s-1$:  $S_\kappa^D(K^p\Iw_{\pi},r)\xrightarrow{\cong} \bigoplus\limits_{k=0}^{s-1}\calA_{\kappa,r}$, where $A$ a $\QQ_p$-affinoid algebra and $(1,\dots, 1)$ good for $(\kappa,r)$.

\begin{convention}
	In the rest of this paper, we always assume that $K^p$ is neat.
\end{convention}

In general, we fix a nonempty subset $J$ of $I$ and a continuous character $\chi_J=(\nu_J,\mu):\calO_{p,J}^\times\times\ZZ_p^\times\rightarrow A^\times$. As observed in Remark~\ref{remark:some basic properties of the space C_J(kappa,A)}, for any $\alpha\in \calO_F^\times$, the matrix $\mathrm{Diag}(\alpha)\in T(\calO_p)$ acts on $\calC_{\chi_J}(\Iw_\pi,A)$ via multiplication by $\mu(\Nm_{F/\QQ}(\alpha))$. So we have an isomorphism of $A$-modules: $S_{\kappa,J}^D(K^p,A)\xrightarrow{\cong}\bigoplus\limits_{k=0}^{s-1}\calC_{\chi_J}(\Iw_\pi,A)$.

\section{A filtration on the space of integral $p$-adic automorphic forms}\label{section:A filtration on the space of integral $p$-adic automorphic forms}

\subsection{Notations}\label{S:notations in filtration on the space of p-adic automorphic forms}

\begin{itemize}
	\item We label the elements in $I=\Hom(F,\bar{\QQ})$ by $I=\{i_1,\dots, i_g\}$. For $1\leq l\leq g-1$, let $J_l=\{i_1,\dots, i_l \}\subset I$.
	\item Let $H$ be the torsion subgroup of $\calO_p^\times\times \ZZ_p^\times$. Hence $\calO_p^\times\times \ZZ_p^\times\cong H\times ((1+\pi\calO_p)\times (1+p\ZZ_p))$ and $\calW\cong \prod\limits_{\omega\in H^\vee}\calW_\omega$, where $H^\vee$ is the character group of $H$, and $\calW_\omega$ is isomorphic to the $(g+1)$-dimensional open unit polydisc.
	\item The homomorphism $\phi_\rho:\calO_p^\times\rightarrow \calO_p^\times\times \ZZ_p^\times$ induces a continuous homomorphism $\ZZ_p\llbracket \calO_p^\times \rrbracket\rightarrow\ZZ_p\llbracket \calO_p^\times\times \ZZ_p^\times \rrbracket$, which is still denoted by $\phi_\rho$. For each $i\in I$, define $T_i:=\phi_\rho([\exp(\pi_i)]-1)\in \ZZ_p\llbracket \calO_p^\times\times \ZZ_p^\times \rrbracket$ and $T:=[1,\exp(p)]-1$. Then $\{(T_i)_{i\in I},T \}$ forms a full set of parameters of the weight space $\calW$.
	\item We denote by $\Lambda$ the complete group ring $\ZZ_p\llbracket \calO_p^\times\times \ZZ_p^\times \rrbracket$, and put $\gothm_\Lambda:=(p,(T_i)_{i\in I})\subset \Lambda$. For every character $\omega\in H^\vee$, let $\Lambda_\omega:=\Lambda\otimes_{\ZZ_p[H],\omega}\ZZ_p$. Under the above notations, we have $\Lambda_\omega=\ZZ_p\llbracket (T_i)_{i\in I},T \rrbracket$.
	\item For a nonempty subset $J$ of $I$, we set $\Lambda_J:=\ZZ_p\llbracket \calO_{p,J}^\times\times\ZZ_p^\times \rrbracket$. We use $H_J\subset H$ to denote the torsion subgroup of $\calO_{p,J}^\times\times \ZZ_p^\times$. Under the isomorphism $\Lambda_J\cong \ZZ_p[H_J]\otimes_{\ZZ_p}\ZZ_p\llbracket (T_j)_{j\in J}, T \rrbracket$, we set $\Lambda_J^{>1/p}:=\ZZ_p[H_J]\otimes_{\ZZ_p}\ZZ_p\llbracket (T_j,\frac{p}{T_j})_{j\in J},T \rrbracket$ and $\gothm_{\Lambda_J^{>1/p}}$ be the ideal of $\Lambda_J^{>1/p}$ generated by the elements $(T_j)_{j\in J}$. Note that since $p=T_j\cdot \frac{p}{T_j}$ in $\Lambda_J^{>1/p}$, we have $p\in \gothm_{\Lambda_J^{>1/p}}$.
	\item When $J=I$, we write $\Lambda^{>1/p}$ (resp. $\gothm_{\Lambda^{>1/p}}$) for $\Lambda_J^{>1/p}$ (resp. $\gothm_{\Lambda_J^{>1/p}}$) for simplicity. In particular, we have $\gothm_{\Lambda^{>1/p}}=\gothm_\Lambda\cdot \Lambda^{>1/p}$, and $\gothm_{\Lambda^{>1/p}}$ is generated by $(T_i)_{i\in I}$ in $\Lambda^{>1/p}$. We also have $p\in \gothm_{\Lambda^{>1/p}}$.
\end{itemize}

\subsection{Explicit expression of $U_{\pi_j}$-operators on the space of $p$-adic automorphic forms}\label{subsection: explicit expression of Up operator on the space of automorphic forms}

First we give an explicit expression of the $U_{\pi_j}$-operator on the space of (generalized) integral $p$-adic automorphic forms. This is a generalization of \cite[Proposition~3.1]{liu2017eigencurve}. 


\begin{proposition}\label{P:explicit expression of Up operators on integral model of p-adic automorphic forms}
	Let $J$ be a nonempty subset of $I$ and $\chi_J:\calO_{p,J}^\times\times\ZZ_p^\times\rightarrow A^\times$ be a continuous character. Fix $j\in J$. 
	Under the isomorphism $S_{\kappa,J}^D(K^p,A)\xrightarrow{\cong}\bigoplus\limits_{k=0}^{s-1}\calC_{\chi_J}(\Iw_\pi,A)$, the $U_{\pi_j}$-operator on this space can be described by the following commutative diagram:
	$$
	\xymatrix@=1.3cm{
		S_{\kappa,J}^D(K^p,A) \ar[r]^{\cong} \ar[d]^{U_{\pi_j}} & \bigoplus\limits_{k=0}^{s-1}\calC_{\chi_J}(\Iw_\pi,A)\ar[d]^{\calU_j} \\
		S_{\kappa,J}^D(K^p,A) \ar[r]^{\cong} & \bigoplus\limits_{k=0}^{s-1}\calC_{\chi_J}(\Iw_\pi,A).
	}
	$$
	Here the right vertical map $\calU_j$ in the above diagram is given by an $s\times s$ matrix with the following descriptions:
	\begin{enumerate}
		\item Each entry of $\calU_j$ is a sum of operators of the form $\circ\delta_p$, where $\delta_p=(\delta_i)_{i\in I}\in \rmM_2(\calO_p)$ have the property that $\delta_i\in \Iw_{\pi_i}$ for all $i\neq j$, and $\delta_j$ belongs to $\left( \begin{array}{cc}
		\pi_j\calO_{\gothp_j}&\calO_{\gothp_j}\\
		\pi_j\calO_{\gothp_j}&\calO_{\gothp_j}^\times\end{array}\right)
		\subset M_{\pi,j}$, where $\circ\delta_p$ is the right action of the monoid $\bbM_{\pi,J}\times \Iw_{\pi,J^c}$ on the spaces $\calC_{\chi_J}(\Iw_\pi,A)$ defined in \S\ref{S:Integral model of the space of $p$-adic automorphic forms}.
		\item There are exactly $p$ such operators appearing in each row and column of $\calU_j$.
	\end{enumerate}
\end{proposition}
\begin{proof}
	The proof is almost identical with that of \cite[Proposition~3.1]{liu2017eigencurve} and we only give a sketch here.
	For every $l=0,\dots,p-1$ and $k=0,\dots,s-1$, we can write $\gamma_kv_{j,l}^{-1}$ uniquely as $\delta_{l,k}\gamma_{\alpha_{l,k}}u_{l,k}$， for $\delta_{l,k}\in D^\times$, $\alpha_{l,k}\in \{0,\dots,s-1 \}$ and $u_{l,k}\in K^p\Iw_\pi$. Then
	\begin{equation*}
	\begin{split}
	U_{\pi_j}(\phi)(\gamma_k)&=\sum_{l=0}^{p-1}\phi(\gamma_{\alpha_{l,k}}u_{l,k})\circ v_{j,l}=\sum_{l=0}^{p-1}(\phi(\gamma_{\alpha_{l,k}})\circ u_{l,k,p})\circ v_{j,l}
	\\
	&=\sum_{l=0}^{p-1}\phi(\gamma_{\alpha_{l,k}})\circ(u_{l,k,p}v_{j,l}),
	\end{split}
	\end{equation*}
	where $u_{l,k,p}\in \Iw_{\pi}$ is the $\GL_2(F_p)$-component of $u_{l,k}$. 
	
	Let $\delta_{l,k,p}=u_{l,k,p}v_{j,l}$ for all $l$'s and $k$'s. It is straightforward to check that $\delta_{l,k,p}$'s satisfy the stated property. 
\end{proof}

\subsection{p-adic analysis}
It follows from Proposition~\ref{P:explicit expression of Up operators on integral model of p-adic automorphic forms} that to study the $U_{\pi_j}$-operators for $j\in I$, it is crucial to understand the action of the monoid $\bbM_\pi\subset\GL_2(F_p)$ on the space $\calC(\calO_p,A)$. When $F=\QQ$ (and hence $\calO_p=\ZZ_p$) and $\kappa:\ZZ_p^\times\rightarrow \Lambda^\times=(\ZZ_p\llbracket\ZZ_p^\times\rrbracket)^\times$ is the universal character, this question has been studied carefully by Liu-Wan-Xiao in \cite[\S3]{liu2017eigencurve}. We will recall their results in this section and temporarily adopt their notations.

For $F=\QQ$, the monoid $\bbM_\pi$ becomes $\left\{ \left[ \begin{array}{cc}
a& b\\
c&d\end{array}\right]\in \rmM_2(\ZZ_p)\;\Big|\; p|c, p\nmid d \text{~and~} ad-bc\neq 0 \right\}$.  The Iwasawa algebra $\Lambda$ decomposes as  $\ZZ_p\llbracket \ZZ_p^\times \rrbracket\cong \ZZ_p[\Delta]\otimes_{\ZZ_p}\ZZ_p\llbracket T \rrbracket$, where $\Delta\subset \ZZ_p^\times$ is the torsion subgroup, $T=[\exp(p)]-1$. Denote $\gothm_{\Lambda}:=(p,T)\subset \Lambda$. The space $\calC(\ZZ_p,\Lambda)$ admits an orthonormal basis $\{e_0=1,e_1=z,e_2=\binom{z}{2},\dots \}$, which is called t\emph{he Mahler basis} (we refer \cite[\S2.16]{liu2017eigencurve} for more details). It also carries a right action of the monoid $\bbM_{\pi}$ defined by
\[
h\circ\delta(z):=\chi(cz+d)h(\frac{az+b}{cz+d}) \text{~for~} h\in \calC(\ZZ_p,\Lambda) \text{~and~} \delta\in \bbM_{\pi}.
\]

For $\delta_p= \left[\begin{array}{cc}
a& b\\
c&d\end{array}\right]\in \bbM_\pi$, we use $P(\delta_p)=(P_{m,n}(\delta_p))_{m,n\geq 0}$ to denote the infinite matrix for the action of $\delta_p$ on $\calC(\ZZ_p,\Lambda)$ with respect to the Mahler basis, i.e. $P_{m,n}(\delta_p)$ is the coefficient of $\binom{z}{m}$ of the function $\binom{z}{n}\circ \delta_p$. Then we have the following estimation.
\begin{proposition}[\cite{liu2017eigencurve}, Proposition~3.24]\label{P:Liu-Wan-Xiao's computation on matrix coefficients of Mahler basis}
\noindent	\begin{enumerate}
		\item When $\delta_p= \left[ \begin{array}{cc}
		a& b\\
		c&d\end{array}\right]\in \left[ \begin{array}{cc}
		p\ZZ_p& \ZZ_p\\
		p\ZZ_p& \ZZ_p^\times\end{array}\right]$, the coefficient $P_{m,n}(\delta_p)$ belongs to $\gothm_{\Lambda}^{\max\{m-\lfloor \frac{n}{p} \rfloor,0  \}}$.
		\item When $\delta_p= \left[ \begin{array}{cc}
		a& b\\
		c&d\end{array}\right]\in \bbM_\pi$, the coefficient $P_{m,n}(\delta_p)$ belongs to $\gothm_{\Lambda}^{\max\{m-n,0  \}}$.
	\end{enumerate}
\end{proposition}

\subsection{Orthonormalizable spaces and compact operators}\label{S:Orthonormalizable spaces and compact operators}

In this section we recall the notions of orthonormal basis and compact operator defined in \cite[\S5]{liu2017eigencurve} and apply the theory to the spaces of (generalized) integral $p$-adic automorphic forms.

\begin{definition}\label{D:orthonormalizable module and compact operator}
	Let $R$ be a complete noetherian ring with ideal of definition $\gothm_R$.
	\begin{enumerate}
		\item A topological $R$-module $M$ is called \emph{orthonormalizable} if it is isomorphic to the topological $R$-module:
		\[
		\hat{\oplus}_{i\in \ZZ_{\geq 0}} Re_i:= \varprojlim_{n}(\oplus_{i\in \ZZ_{\geq 0}}(R/\gothm_R^n)e_i).
		\]
		More precisely, $R$ is orthonormalizable if there exists $\{e_i|i\in \ZZ_{\geq 0} \}\subset M$, such that every element $m$ in $M$ can be written uniquely as $m=\sum\limits_{i\in\ZZ_{\geq 0}}a_ie_i$, with $a_i\in R$ and $\lim\limits_{i\rightarrow \infty}a_i=0$ in $R$. The set $\{e_i|i\in \ZZ_{\geq 0} \}$ is called an \emph{orthonormal basis} of $M$.
		\item Let $M$ be an orthonormalizable $R$-module. Let $U:M\rightarrow M$ be a continuous $R$-linear operator on $M$. $U$ is called \emph{compact} if the induced operator on $M/\gothm_R^nM$ has finitely generated image for all $n\in \ZZ_{\geq 0}$.
	\end{enumerate}
\end{definition}

Keep the notations as in Definition~\ref{D:orthonormalizable module and compact operator}. We fix an orthonormal basis $\{e_m|m\in \ZZ_{\geq 0} \}$ of $M$, and let $P\in \rmM_\infty(R)$ be the matrix associated to $U$ under this basis. We define the characteristic power series of the $U$-operator by $\mathrm{char}(U,M):= \det(\mathrm{I}_\infty-XP)=\lim\limits_{n\rightarrow \infty}\det(\mathrm{I}_\infty-X(P\mod \gothm_R^n))\in R\llbracket X \rrbracket$. The formal power series $\mathrm{char}(U,M)$ is well defined and it does not depend on the choice of the orthonormal basis. We will refer to \cite[Definition~5.1]{liu2017eigencurve} for more details.

We apply the above notions to the spaces of integral $p$-adic automorphic forms and their Hecke operators. 
Let $\chi':\calO_p^\times\times \ZZ_p^\times\rightarrow (\Lambda^{>1/p})^\times$ be the universal character and $\kappa':\calO_p^\times\times\calO_p^\times\rightarrow (\Lambda^{>1/p})^\times$ be the associated character. Recall that we can identify the induced representation $\Ind_{B(\calO_p)}^{\Iw_{\pi}}(\kappa'_B)$ with the space $\calC(\calO_p,\Lambda^{>1/p})$. The latter space admits an orthonormal basis (as a topological $\Lambda^{>1/p}$-module) $\{e_m=\prod\limits_{i\in I}\binom{z_i}{m_i}| m=(m_i)_{i\in I}\in \ZZ_{\geq 0}^I \}$, where $z_i$ is the coordinate of the $i$th component of $\calO_p=\prod\limits_{i\in I}\calO_{\gothp_i}$. Similar to \cite{liu2017eigencurve}, we consider a closed subspace $
\calC(\calO_p,\Lambda^{>1/p})^{mod}=\hat{\oplus}_{m\in \ZZ_{\geq 0}^I}\Lambda^{>1/p}e_m',$
where $e_m'=\prod\limits_{i\in I}T_i^{m_i}\binom{z_i}{m_i}$.

We claim that the space $\calC(\calO_p,\Lambda^{>1/p})^{mod}$ is stable under the action of the monoid $\bbM_\pi$. In fact, for any $i\in I$ and $\delta_i\in \bbM_{\pi_i}$, by Proposition~\ref{P:Liu-Wan-Xiao's computation on matrix coefficients of Mahler basis}$(2)$ and $p=T_i\cdot\frac{p}{T_i}$ in $\Lambda^{>1/p}$, we have $\binom{z_i}{m_i}\circ\delta_i=\sum\limits_{n_i\geq 0}a_{n_i,m_i}\binom{z_i}{n_i}$ with $a_{n_i,m_i}\in T_i^{\max\{n_i-m_i,0\}}\Lambda^{>1/p}$, and hence $(T_i^{m_i}\binom{z_i}{m_i})\circ\delta_i=\sum\limits_{n_i\geq 0}b_{n_i,m_i}T_i^{n_i}\binom{z_i}{n_i}$ with $b_{n_i,m_i}=T_i^{m_i-n_i}a_{n_i,m_i}\in \Lambda^{>1/p}$. This implies that $\calC(\calO_p,\Lambda^{>1/p})^{mod}$ is stable under the action of $\bbM_{\pi_i}$. Combined with $\bbM_\pi=\prod\limits_{i\in I}\bbM_{\pi_i}$, this proves our claim. 

The space of integral $1$-convergent automorphic forms are defined by
\[
S_{\kappa,I}^{D,I}(K^p,\Lambda^{>1/p}):= \{\phi: D^\times \setminus D_f^\times /K^p\rightarrow \calC(\calO_p,\Lambda^{>1/p})^{mod}|\phi(xu)=\phi(x)\circ u, \text{~for all~} x\in D_f^\times,u\in \Iw_{\pi} \}
\]
We have an explicit expression of $S_{\kappa,I}^{D,I}(K^p)$ as before: $S_{\kappa,I}^{D,I}(K^p)\xrightarrow{\cong}\bigoplus\limits_{k=0}^{s-1}\calC(\calO_p,\Lambda^{>1/p})^{mod}$. Hence the topological $\Lambda^{>1/p}$-module $S_{\kappa,I}^{D,I}(K^p)$ has an orthonormal basis $\{e_{k,m}'|k=0,\dots,s-1,m\in\ZZ_{\geq 0}^I \}$, such that $\{e_{k,m}'| m\in \ZZ_{\geq 0}^I \}$ is the orthonormal basis of the $k$-th direct summand in $\bigoplus\limits_{k=0}^{s-1}\calC(\calO_p,\Lambda^{>1/p})^{mod}$ defined above, for every $k=0,\dots,s-1$. 

I claim that the $U_{\pi}$-operator on the space $S_{\kappa,I}^{D,I}(K^p)$ is compact. In fact, by Proposition~\ref{P:Liu-Wan-Xiao's computation on matrix coefficients of Mahler basis} $(1)$, for $\delta_i\in \left[ \begin{array}{cc}
\pi_i\calO_{\gothp_i}& \calO_{\gothp_i}\\
\pi_i\calO_{\gothp_i}& \calO_{\gothp_i}^\times\end{array}\right]$ and $m_i\in \ZZ_{\geq 0}$, we have $\binom{z_i}{m_i}\circ\delta_i=\sum\limits_{n_i\geq 0}a_{n_i,m_i}\binom{z_i}{n_i}$ with $a_{n_i,m_i}\in T_i^{\max\{n_i-\lfloor \frac{m_i}{p} \rfloor,0\}}\Lambda^{>1/p}$. Therefore $(T_i^{m_i}\binom{z_i}{m_i})\circ\delta_i=\sum\limits_{n_i\geq 0}b_{n_i,m_i}T_i^{n_i}\binom{z_i}{n_i}$, with $b_{n_i,m_i}=T_i^{m_i-n_i}a_{n_i,m_i}\in T_i^{m_i-\lfloor \frac{m_i}{p} \rfloor}\Lambda^{>1/p}\subset \gothm_{\Lambda^{>1/p}}^{m_i-\lfloor \frac{m_i}{p} \rfloor}$. Hence for $\delta_p=(\delta_i)\in \left[ \begin{array}{cc}
\pi\calO_p& \calO_p\\
\pi\calO_p& \calO_p^\times\end{array}\right]$ and $m=(m_i)\in \ZZ_{\geq 0}^I$, we have 
\[
e_m'\circ \delta_p=\left(\prod_{i\in I}T_i^{m_i}\binom{z_i}{m_i}\right)\circ \delta_p=
\prod_{i\in I}\left(T_i^{m_i}\binom{z_i}{m_i}\right)\circ\delta_i
=\sum_{n\in\ZZ_{\geq 0}^I}b_{n,m}e_n',
\]
with $b_{n,m}\in \gothm_{\Lambda^{>1/p}}^{\lambda_m}$, where $\lambda_m=\sum\limits_{i\in I}(m_i-\lfloor \frac{m_i}{p} \rfloor)$. It follows from the above estimation and the explicit expression of the $U_{\pi}$-operator on $S_{\kappa,I}^{D,I}(K^p)$ that $U_\pi$ is compact. 

\begin{remark}
	Fix a point $x\in \calW^{>1/p}(\CC_p)$ which corresponds to a continuous homomorphism $\chi:\Lambda^{>1/p}\rightarrow \CC_p$. Let $V:= \calC(\calO_p,\Lambda^{>1/p})^{mod}\hat{\otimes}_{\Lambda^{>1/p},\chi}\CC_p$ be the specialization of $\calC(\calO_p,\Lambda^{>1/p})^{mod}$ at $x$. By \cite[Theorem~I.4.7]{colmez2010fonctions}, there exist $r<r'$ in $\calN^I$ (with the obvious partial order) that depend on the valuations $v_p(\chi(T_i))$'s for all $i\in I$ such that $\calA_{\kappa,r}\subset V\subset \calA_{\kappa,r'}$, and hence $S_{\kappa}^D(K^p\Iw_{\pi},r)\subset S_{\kappa,I}^{D,I}(K^p,\Lambda^{>1/p})\hat{\otimes}_{\Lambda^{>1/p},\chi}\CC_p\subset S_{\kappa}^D(K^p\Iw_{\pi},r')$. Therefore the space $S_{\kappa,I}^{D,I}(K^p,\Lambda^{>1/p})\hat{\otimes}_{\Lambda^{>1/p},\chi}\CC_p$ contains all the finite $U_{\pi}$-slope systems of Hecke eigenvalues.
\end{remark}

\subsection{Continuous functions and distribution algebras}\label{section:continuous functions and distribution algebras}

Let $G$ be a profinite group endowed with the profinite topology. Let $A$ be a $\ZZ_p$-algebra and $I$ be an ideal of $A$. We assume that $A$ is endowed with the $I$-adic topology and is complete under this topology. The typical example we are interested is that $A=\Lambda_J^{>1/p}$ equipped with the $\gothm_{\Lambda_J^{>1/p}}$-adic topology for a nonempty subset $J$ of $I$. On the set $\calC(G,A)$, we give it the uniform topology, i.e. for any $f\in \calC(G,A)$, it has a basis $\{U_n \}$ of open neighborhoods as $U_n=\{g\in \calC(G,A)\;|\;f(x)-g(x)\in I^n\text{~for all~} x\in G \}$. Let $A\llbracket G \rrbracket :=\varprojlim\limits_{U\subset G}A[G/U]$ be the complete group ring of $G$ over $A$, where $U$ ranges over all the open normal subgroups of $G$. For each $U$, the free $A$-module  $A[G/U]$ is endowed with the product topology and $A\llbracket G \rrbracket$ is endowed with the inverse limit topology. We use $\calD(G,A)$ to denote the set $\Hom_{A}(\calC(G,A),A)$ of continuous $A$-linear maps from $\calC(G,A)$ to $A$ (here continuity follows from $A$-linearity). Then we have the following.
\begin{proposition}\label{P:duality between continuous functions and distribution algebras}
	There is a natural isomorphism of $A$-modules $A\llbracket G \rrbracket\cong \calD(G,A)$. Under this isomorphism the topology on $A\llbracket G \rrbracket$corresponds to the weak topology on $\calD(G,A)$.
\end{proposition}

\begin{proof}
	The proof is almost identical to the proof of \cite[Lemma~3.1.3]{johansson2019extended}, which handles the case when $A$ is the unit ball of a Banach-Tate $\ZZ_p$-algebra. For completeness, we give a sketch of proof here. 
	
	For any nonempty open normal subgroup $U$ of $G$, we have an isomorphism of $A$-modules $A[G/U]\xrightarrow{\cong}\Hom_A(\calC(G/U,A),A)$, which is also a homeomorphism. The projection $G\rightarrow G/U$ induces a natural $A$-module homomorphism $\calD(G,A)\rightarrow \calD(G/U,A)$ and hence a map $\calD(G,A)\rightarrow \varprojlim\limits_U \calD(G/U,A)$. If we use $\calC_{sm}(G,A)$ to denote the $A$-submodule of $\calC(G,A)$ consisting of locally constant functions, the map $\calD(G,A)\rightarrow \varprojlim\limits_U \calD(G/U,A)\cong A\llbracket G\rrbracket$ is the natural homomorphism $\Hom_A(\calC(G,A),A)\rightarrow \Hom_A(\calC_{sm}(G,A),A)$ induced by the inclusion map $\calC_{sm}(G,A)\rightarrow \calC(G,A)$. Since $\calC_{sm}(G,A)$ is dense in $\calC(G,A)$, it is clear that the map $\Hom_A(\calC(G,A),A)\rightarrow \Hom_A(\calC_{sm}(G,A),A)$ is an isomorphism of $A$-modules as well as a homeomorphism. 
\end{proof}

\begin{remark}
	\begin{enumerate}
	\item We can also define a convolution product on $\calD(G,A)$ so that the isomorphism in Proposition~\ref{P:duality between continuous functions and distribution algebras} is an isomorphism between $A$-algebras. But we do not need this fact and refer to \cite[\S3.1]{johansson2019extended} to details (in a slightly different setting).
	\item For latter discussion, we recall a similar result as that of Proposition~\ref{P:duality between continuous functions and distribution algebras} discussed in \cite[Proposition~3.1.4]{johansson2019extended}. Let $A$ be a Banach-Tate $\ZZ_p$-algebra with unit ball $A_0$ and a multiplicative pseudo-uniformizer $\varpi$ in the sense of \cite[Definition~2.1.2]{johansson2019extended}. For a profinite group $G$, we define $A_0\llbracket G\rrbracket=\varprojlim_{U\subset G}A_0[G/U]$ and $A\llbracket G\rrbracket =A_0\llbracket G\rrbracket [\frac{1}{\varpi}]$. It follows from \cite[Proposition~3.1.4]{johansson2019extended} that there is a natural $A$-Banach algebra isomorphism $A\llbracket G\rrbracket \xrightarrow{\cong}\calD(G,A)$. It restricts to an $A_0$-algebra isomorphism $A_0\llbracket G\rrbracket \xrightarrow{\cong}\calD(G,A_0)$ that identifies the inverse star topology on the source with the weak star topology on the target. In particular, when $L$ is a closed subfield of $\CC_p$, the $L$-Banach space $L\llbracket G\rrbracket$ is defined and isomorphic to $\calD(G,L)$.
\end{enumerate}
\end{remark}

Let $H $ be a group or more generally  a monoid. Suppose that we have a right action of $H$ on $\calC(G,A)$. The above isomorphism induces a left action of $H$ on $A\llbracket G \rrbracket$ by requiring that 
$h\circ \mu(f)=\mu(f\circ h)$, for $h\in H$, $\mu\in \calD(G,A)$ and $f\in \calC(G,A)$.

We apply the above results to the spaces of (generalized) integral $p$-adic automorphic forms. Fix a nonempty subset $J$ of $I$. Recall that in \S\ref{S:Integral model of the space of $p$-adic automorphic forms}, we have an isomorphism $\calC_{\chi_J}(\Iw_\pi,A)\cong \calC(\calO_{p,J}\times P_{J^c},A)$. Combined with Proposition~\ref{P:duality between continuous functions and distribution algebras}, it gives us an $A$-linear isomorphism $A\llbracket \calO_{p,J}\times P_{J^c} \rrbracket \cong \Hom_A(\calC_{\chi_J}(\Iw_\pi,A),A)$. In  \S\ref{S:Integral model of the space of $p$-adic automorphic forms}, we have defined a right action of the monoid $\bbM_{\pi,J}\times\Iw_{\pi,J^c}$ on $\calC_{\chi_J}(\Iw_\pi,A)$. Therefore, we get a left action of $\bbM_{\pi,J}\times\Iw_{\pi,J^c}$ on $A\llbracket \calO_{p,J}\times P_{J^c} \rrbracket$. Under the isomorphism $A\llbracket \calO_{p,J}\times P_{J^c} \rrbracket \cong A\llbracket \calO_{p,J} \rrbracket\hat{\otimes}_A A\llbracket P_{J^c} \rrbracket$, the monoid $\bbM_{\pi,J}$ acts on $A\llbracket \calO_{p,J} \rrbracket$ and the group $\Iw_{\pi,J^c}$ acts on $A\llbracket P_{J^c}  \rrbracket$. We remark that the latter action has an explicit expression: for $u\in \Iw_{\pi,J^c}$ and $x\in A\llbracket P_{J^c}  \rrbracket$, the left action of $u$ on $x$ is given by $u\circ x=\kappa_{B,J}(u^\ast_2)x\cdot u^\ast_1$, where
$u_1^\ast$ (resp. $u_2^\ast$) is the $P_{J^c}$-component (resp. $D(\calO_{p,J^c})$-component) of $u^\ast\in \Iw_{\pi,J^c}$ under the decomposition $\Iw_{\pi,J^c}=P_{J^c}\times D(\calO_{p,J^c})$, $\kappa_{B,J}$ is the character defined in \S\ref{S:Integral model of the space of $p$-adic automorphic forms}, and  
 $x\cdot u_1^\ast$ is the product in the complete group algebra $A\llbracket P_{J^c}  \rrbracket$.

\begin{convention}\label{convention:left A[P'_J^c]-module structure on the group algebras}
	We always view the complete group algebra $A\llbracket P_{J^c}\rrbracket$ as a left $A\llbracket P'_{J^c}\rrbracket$-module induced by the left multiplication of $P'_{J^c}$ on $P_{J^c}$. As $P'_{J^c}$ is a normal subgroup of $P_{J^c}$ with quotient $P_{J^c}/P'_{J^c}\cong \Delta_{J^c}$, $A\llbracket P_{J^c}\rrbracket$ is a free $A\llbracket P'_{J^c}\rrbracket$ of rank $|\Delta_{J^c}|$. Under the isomorphism $A\llbracket \calO_{p,J}\times P_{J^c} \rrbracket \cong A\llbracket \calO_{p,J} \rrbracket\hat{\otimes}_A A\llbracket P_{J^c} \rrbracket$, the group algebra $A\llbracket \calO_{p,J}\times P_{J^c} \rrbracket$ is also endowed with a left $A\llbracket P'_{J^c}\rrbracket$-module structure. The left action of $\Iw_{\pi,J^c}$ on $A\llbracket P_{J^c}'\rrbracket$ and  $A\llbracket \calO_{p,J}\times P_{J^c} \rrbracket$ defined above are $A\llbracket P'_{J^c}\rrbracket$-linear. On the other hand, although $A\llbracket \calO_{p,J}\rrbracket$ has a (commutative) ring structure, the left action of $\Iw_{\pi,J}$ on $A\llbracket \calO_{p,J}\rrbracket$ defined above is not compatible with this ring structure. On the dual side, this is equivalent to the fact that the right action of $\Iw_{\pi,J}$ on the space $\calC(\calO_{p,J},A)$ does not commute with the obvious translation action of $\calO_{p,J}$ on $\calC(\calO_{p,J},A)$. For this reason we only view $A\llbracket \calO_{p,J}\times P_{J^c} \rrbracket$ as a left $A\llbracket P'_{J^c} \rrbracket$-module in this paper.	
\end{convention}

\begin{lemma}\label{L:isomorphism concerning tensor product and A-linear duals}
	Let $M$ be a finitely generated right $\ZZ_p\llbracket \Iw_\pi\rrbracket$-module and $N$ be a right $A\llbracket \Iw_\pi\rrbracket$-module. Let $\langle \cdot ,\cdot \rangle_N:N\times \Hom_A(N,A)\rightarrow A$ be the natural bilinear map. We endow $\Hom_A(N,A)$ with a left action of $\Iw_{\pi}$ as before, i.e. we have $\langle n\cdot g , l \rangle_N=\langle n, g\cdot l \rangle_N$ for all $n\in N$, $l\in \Hom_A(N,A)$ and $g\in \Iw_{\pi}$. Under the above notations, the map 
	\begin{equation*}
	\begin{split}
	M\times \Hom_A(N,A) &\rightarrow \Hom_A(\Hom_{\ZZ_p\llbracket \Iw_{\pi}\rrbracket}(M,N),A)\\
	(m,l)&\mapsto F_{(m,l)},
	\end{split}
	\end{equation*}
	where $F_{(m,l)}:\Hom_{\ZZ_p\llbracket \Iw_{\pi}\rrbracket}(M,N)\rightarrow A$ is defined by $F_{(m,l)}(\varphi):=\langle \varphi(m) , l \rangle_N$ for any $\varphi\in \Hom_{\ZZ_p\llbracket \Iw_{\pi}\rrbracket}(M,N)$, induces an $A$-linear isomorphism 
	\[
	\iota_{M,N}:M\hat{\otimes}_{\ZZ_p\llbracket \Iw_{\pi}\rrbracket}\Hom_A(N,A)\xrightarrow{\cong}\Hom_A(\Hom_{\ZZ_p\llbracket \Iw_{\pi}\rrbracket}(M,N),A).
	\]
\end{lemma}

\begin{proof}
	First we remark that since $M$ is a right $\ZZ_p\llbracket \Iw_{\pi}\rrbracket$ and $\Hom_A(N,A)$ carries a left $\Iw_{\pi}$-action, the completed tensor product $M\hat{\otimes}_{\ZZ_p\llbracket \Iw_{\pi}\rrbracket}\Hom_A(N,A)$ is meaningful. It is straightforward to verify that the map $M\times \Hom_A(N,A) \rightarrow \Hom_A(\Hom_{\ZZ_p\llbracket \Iw_{\pi}\rrbracket}(M,N),A)$ induces an $A$-linear homomorphism $\iota_{M,N}:M\hat{\otimes}_{\ZZ_p\llbracket \Iw_{\pi}\rrbracket}\Hom_A(N,A)\rightarrow\Hom_A(\Hom_{\ZZ_p\llbracket \Iw_{\pi}\rrbracket}(M,N),A)$ and $\iota_{M,N}$ is an isomorphism when $M$ is free. In the general case, we choose a resolution $F_1\rightarrow F_0\rightarrow M\rightarrow 0$ of $M$ with $F_0$, $F_1$ finite free $\ZZ_p\llbracket \Iw_{\pi}\rrbracket$-module. We apply the functors $M\rightarrow M\hat{\otimes}_{\ZZ_p\llbracket \Iw_{\pi}\rrbracket}\Hom_A(N,A)$ and $M\rightarrow \Hom_A(\Hom_{\ZZ_p\llbracket \Iw_{\pi}\rrbracket}(M,N),A)$ (with $N$ fixed) to the resolution $F_1\rightarrow F_0\rightarrow M\rightarrow 0$. A simple diagram chasing shows that $\iota_{M,N}$ is an isomorphism.
\end{proof}

We denote by $S_{\kappa,I}^D(K^p,A)^\vee$ the $A$-linear dual of the space of $p$-adic automorphic forms $S_{\kappa,I}^D(K^p,A)$, and for any subset $J\subset I$ and a continuous character $\chi_J:\calO_{p,J}^\times\times\ZZ_p^\times\rightarrow A^\times$, we define $S_{\kappa,J}^D(K^p,A)^\vee$ to be the $A$-linear dual of the space $S_{\kappa,J}^D(K^p,A)$. From the tautological isomorphism~\eqref{E:isomorphism of p-adic automoprhism in term of completed homology} and Lemma~\ref{L:isomorphism concerning tensor product and A-linear duals}, we have the following expression of these spaces:
\begin{equation}\label{E:explicit expression of A-linear duals of generalized p-adic automorphic forms}
S_{\kappa,I}^D(K^p,A)^\vee=\tilde{\rmH}_0\hat{\otimes}_{\ZZ_p\llbracket \Iw_{\pi}\rrbracket} A\llbracket \calO_p\rrbracket \text{~and~} S_{\kappa,J}^D(K^p,A)^\vee= \tilde{\rmH}_0\hat{\otimes}_{\ZZ_p\llbracket \Iw_{\pi}\rrbracket} A\llbracket \calO_{p,J}\times P_{J^c} \rrbracket,
\end{equation}
where $\tilde{\rmH}_0$ is the completed homology group defined in \S\ref{S:Integral model of the space of $p$-adic automorphic forms}.

We can translate the results we developed for the spaces $S_{\kappa,I}^D(K^p,A)$ and $S_{\kappa,J}^D(K^p,A)$ to their $A$-linear duals. We summarize these results as follows:
\begin{enumerate}
	\item For $j\in J$, we can define the $U_{\pi_j}$-operator on $S_{\kappa,I}^D(K^p,A)^\vee$ and $S_{\kappa,J}^D(K^p,A)^\vee$ by 
	\[
	U_{\pi_j}(h\hat{\otimes} \mu):=\sum_{l=0}^{p-1} (h\cdot v_{j,l}^{-1})\hat{\otimes} (v_{j,l}\cdot \mu), \text{~for~} h\hat{\otimes}\mu \in \tilde{\rmH}_0\hat{\otimes}_{\ZZ_p\llbracket \Iw_{\pi}\rrbracket} A\llbracket \calO_p\rrbracket  \text{~or ~} \tilde{\rmH}_0\hat{\otimes}_{\ZZ_p\llbracket \Iw_{\pi}\rrbracket} A\llbracket \calO_{p,J}\times P_{J^c} \rrbracket.
	\]
	Under the natural $A$-bilinear pairing $\langle \cdot ,\cdot \rangle:S_{\kappa,I}^D(K^p,A)\times S_{\kappa,I}^D(K^p,A)^\vee\rightarrow A$ (resp. $\langle \cdot , \cdot\rangle_J:S_{\kappa,J}^D(K^p,A)\times S_{\kappa,J}^D(K^p,A)^\vee\rightarrow A$), we have $\langle U_{\pi_j}(\phi),\psi \rangle=\langle \phi,U_{\pi_j}(\psi) \rangle$ (resp. $\langle U_{\pi_j}(\phi),\psi \rangle_J=\langle \phi,U_{\pi_j}(\psi) \rangle_J$).
	\item The right action  of $B(\calO_{p,J^c})$ on $\calC_{\chi_J}(\Iw_{\pi},A)$ defined in Remark~\ref{R:remark for integral p-adic automorphic forms} induces a left action of $B(\calO_{p,J^c})$ on $A\llbracket \calO_{p,J}\times P_{J^c} \rrbracket$. Moreover $A\llbracket \calO_p \rrbracket$ can be identified with the coinvariant of $A\llbracket \calO_{p,J}\times P_{J^c} \rrbracket$ under this action. For any two subsets $J_1\subset J_2$ of $I$, let $J_3=J_2\setminus J_1$. We have a natural surjective $A$-linear map 
	 $S_{\kappa,J_1}^D(K^p,A)^\vee\rightarrow S_{\kappa,J_2}^D(K^p,A)^\vee$, which identifies the latter space with the $B(\calO_{p,J_3})$-coinvariants of the first space. Moreover, it is compatible with the $U_{\pi_j}$-operators on these two spaces for all $j\in J_1$.
	\item Let $\gamma_k$ for $k=0,\dots, s-1$ be the elements of $D_f^\times$ defined in \S\ref{subsection:Explicit expression of p-adic automorphic forms}, which are viewed as elements in $\tilde{\rmH}_0$. By \eqref{E:explicit expression of A-linear duals of generalized p-adic automorphic forms}, we write $S_{\kappa,I}^D(K^p,A)^\vee$ and $S_{\kappa,J}^D(K^p,A)^\vee$ explicitly by
	$$\bigoplus\limits_{k=0}^{s-1}A\llbracket \calO_p \rrbracket \xrightarrow{\cong}S_{\kappa,I}^D(K^p,A)^\vee, \
	(\mu_k)_{k=0,\dots,s-1}\mapsto \sum\limits_{k=0}^{k-1}\gamma_k\hat{\otimes} \mu_k$$ and 	$$\bigoplus\limits_{k=0}^{s-1}A\llbracket \calO_{p,J}\times P_{J^c} \rrbracket\xrightarrow{\cong} S_{\kappa,J}^D(K^p,A)^\vee,\ (\mu_k)_{k=0,\dots,s-1}\mapsto \sum\limits_{k=0}^{k-1}\gamma_k\hat{\otimes} \mu_k.$$
	\item The left $A\llbracket P_{J^c}'\rrbracket$-module structure on $A\llbracket \calO_{p,J}\times P_{J^c}\rrbracket$ defined in Convention~\ref{convention:left A[P'_J^c]-module structure on the group algebras} induces a left $A\llbracket P_{J^c}'\rrbracket$-module structure on the space $S_{\kappa,J}^{D}(K^p,A)^\vee$. The $U_{\pi_j}$-operator on $S_{\kappa,J}^{D}(K^p,A)^\vee$ is $A\llbracket P_{J^c}'\rrbracket$-linear for all $j\in J$.
\end{enumerate}

Recall that in \S\ref{S:Orthonormalizable spaces and compact operators}, we define a closed subspace $\calC(\calO_p,\Lambda^{>1/p})^{mod}$ of $\calC(\calO_p,\Lambda^{>1/p})$, which is closed under the action of the monoid $\bbM_\pi$. We will generalize this construction to define a closed subspace $\calC(\calO_{p,J}\times P_{J^c},\Lambda_J^{>1/p})^{J'-mod}$ of $\calC(\calO_{p,J}\times P_{J^c},\Lambda_J^{>1/p})$ for any nonempty subsets $J'\subset J$ of $I$. We cannot define this space directly as there is no Mahler basis for the space $\calC(\calO_{p,J}\times P_{J^c},\Lambda_J^{>1/p})$. We will pass to the dual side to make the definition and then dual back. 

First we explain the construction in the case $F=\QQ$. Then $\calO_p=\ZZ_p$ and we assume $\Lambda^{>1/p}=\ZZ_p\llbracket T,\frac{p}{T} \rrbracket$. It follows from Amice transformation that the $\Lambda^{>1/p}$-dual space of $\calO(\ZZ_p,\Lambda^{>1/p})$ is isomorphic to $\Lambda^{>1/p}\llbracket X \rrbracket$ and under this isomorphism, the dual basis of the Mahler basis $\{e_m=\binom{z}{m}\;|\;m\geq 0 \}$ is given by $\{X^m\;|\;m\geq 0 \}$. Combining it with Proposition~\ref{P:duality between continuous functions and distribution algebras}, we have an isomorphism $\Lambda^{>1/p}\llbracket \ZZ_p \rrbracket\xrightarrow{\cong}\Lambda^{>1/p}\llbracket X \rrbracket$, under which $[1]-1$ corresponds to $X$. We use $\Lambda^{>1/p}\llbracket X\rrbracket^{mod}$ to denote the $\Lambda^{>1/p}$-dual space of the closed subspace $\calC(\ZZ_p,\Lambda^{>1/p})^{mod}$.  Then as a  $\Lambda^{>1/p}$-module, it is isomorphic to $\Lambda^{>1/p}\llbracket X'\rrbracket$, and under this isomorphism the dual basis of the modified Mahler basis $\{e_m'=T^m\binom{z}{m}\;|\;m\geq 0 \}$ corresponds to $\{X'^m\;|\;m\geq 0 \}$. The $\Lambda^{>1/p}$-dual of the inclusion $\calC(\ZZ_p,\Lambda^{>1/p})^{mod}\rightarrow \calC(\ZZ_p,\Lambda^{>1/p})$ is the $\Lambda^{>1/p}$-algebra homomorphism $\Lambda^{>1/p}\llbracket X\rrbracket\rightarrow \Lambda^{>1/p}\llbracket X' \rrbracket$ which sends $X$ to $TX'$. Moreover, as the right action of the monoid $\bbM_\pi$ on $\calC(\ZZ_p,\Lambda^{>1/p})$ keeps the subspace $\calC(\ZZ_p,\Lambda^{>1/p})^{mod}$ stable, it induces a left action of $\bbM_\pi$ on $\Lambda^{>1/p}\llbracket X' \rrbracket$, and the map $\Lambda^{>1/p}\llbracket X \rrbracket\rightarrow \Lambda^{>1/p}\llbracket X' \rrbracket $ defined above is compatible with the $\bbM_\pi$-actions on these two spaces.

Now we consider the general case. Fix two nonempty subsets $J'\subset J$ of $I$. We define an isomorphism of $\Lambda_J^{>1/p}$-algebras $$\iota_J:\Lambda_J^{>1/p}\llbracket \calO_{p,J} \rrbracket\rightarrow \Lambda_J^{>1/p}\llbracket (X_j)_{j\in J}\rrbracket,\ [1_j]-1\mapsto X_j$$ for all $j\in J$, where $1_j\in\calO_{p,J}=\prod\limits_{j\in J}\calO_{\gothp_j}$ is the element with $j$-component $1$ and $i$-component $0$ for all $i\neq j$. We identify $\Lambda_J^{>1/p}\llbracket \calO_{p,J}\rrbracket$ with $\Lambda_J^{>1/p}\llbracket (X_j)_{j\in J}\rrbracket$ via $\iota_J$. We put $\Lambda_J^{>1/p}\llbracket \calO_{p,J} \rrbracket^{J'-mod}:=\Lambda_J^{>1/p}\llbracket (X_j')_{j\in J'}, (X_j)_{j\in J\setminus J'} \rrbracket$ and define an embedding of $\Lambda_J^{>1/p}$-algebras $\Lambda_J^{>1/p}\llbracket \calO_{p,J}\rrbracket \rightarrow \Lambda_J^{>1/p}\llbracket \calO_{p,J}\rrbracket^{J'-mod}$ by $X_j\mapsto T_jX_j'$ for  $j\in J'$ and $X_j\mapsto X_j$ for $j\in J\setminus J'$. Then we put $\Lambda_J^{>1/p}\llbracket \calO_{p,J}\times P_{J^c}\rrbracket^{J'-mod}:=\Lambda_J^{>1/p}\llbracket \calO_{p,J}\rrbracket^{J'-mod}\hat{\otimes}_{\Lambda_J^{>1/p}}\Lambda_J^{>1/p}\llbracket P_{J^c}\rrbracket$. Under the isomorphism $ \calO_{\gothp_j}\cong \ZZ_p$, we have $\Lambda_J^{>1/p}\llbracket \calO_{\gothp_j}\rrbracket^{j-mod}\cong \Lambda_J^{>1/p}\llbracket \ZZ_p\rrbracket^{mod}$ defined above and hence we can endow $\Lambda_J^{>1/p}\llbracket \calO_{\gothp_j}\rrbracket^{j-mod}$ a left action of the monoid $\bbM_{\pi_j}$. Under the isomorphism $\Lambda_J^{>1/p}\llbracket \calO_{p,J}\rrbracket^{J'-mod}\cong \left(\hat{\otimes}_{j\in J'}\Lambda_J^{>1/p}\llbracket \calO_{\gothp_j}\rrbracket^{j-mod} \right)\hat{\otimes}_{\Lambda_J^{>1/p}}\Lambda_J^{>1/p}\llbracket \calO_{p,J\setminus J'}\rrbracket$, we get a left action of the monoid $\bbM_{\pi,J}$ on $\Lambda_J^{>1/p}\llbracket \calO_{p,J}\rrbracket^{J'-mod}$. Then the embedding $\Lambda_J^{>1/p}\llbracket \calO_{p,J}\rrbracket \hookrightarrow \Lambda_J^{>1/p}\llbracket \calO_{p,J}\rrbracket^{J'-mod}$ (resp. $\Lambda_J^{>1/p}\llbracket \calO_{p,J}\times P_{J^c} \rrbracket\rightarrow \Lambda_J^{>1/p}\llbracket \calO_{p,J}\times P_{J^c}\rrbracket^{J'-mod}$) is equivariant under the actions of the monoid $\bbM_{\pi,J}$ (resp. $\bbM_{\pi,J}\times \Iw_{\pi,J^c}$) on these two spaces.

Define $\calC_{\chi_J}(\Iw_\pi,\Lambda_J^{>1/p})^{J'-mod}$ to be the $\Lambda_J^{>1/p}$-dual space of $\Lambda_J^{>1/p}\llbracket \calO_{p,J}\times P_{J^c}\rrbracket^{J'-mod}$, which carries a right action of the monoid $\bbM_{\pi,J}\times \Iw_{\pi,J^c}$. Similar with the definition of $S_{\kappa,I}^{D,I}(K^p,\Lambda_J^{>1/p})$, we set 
\begin{multline*}	
S_{\kappa,J}^{D,J'}(K^p,\Lambda_J^{>1/p})\\
:=\{\phi: D^\times \setminus D_f^\times /K^p\rightarrow \calC_{\chi_J}(\Iw_\pi,\Lambda_J^{>1/p})^{J'-mod}|\phi(xu)= \phi(x)\circ u \text{~for all~} x\in D_f^\times,u\in \Iw_{\pi} \},
\end{multline*}
and $S_{\kappa,J}^{D,J'}(K^p,\Lambda_J^{>1/p})^\vee$ to be its $\Lambda_J^{>1/p}$-linear dual. Hence, we have 
$$
S_{\kappa,J}^{D,J'}(K^p,\Lambda_J^{>1/p})^\vee=\tilde{\rmH}_0\hat{\otimes}_{\ZZ_p\llbracket \Iw_{\pi}\rrbracket} \Lambda_J^{>1/p}\llbracket \calO_{p,J}\times P_{J^c}\rrbracket^{J'-mod}.
$$

\begin{remark}
	\begin{enumerate}
		\item Since our notation above is somewhat complicated, we make a remark to explain its meaning. In the notation $S_{\kappa,J}^{D,J'}(K^p,\Lambda_J^{>1/p})$, the subscript $J$ under $S$ means that we consider generalized $p$-adic automorphic forms which behave like automorphic forms at places in $J$ and behave like dual of completed homology groups at places in $I\setminus J$. The superscript $J'$ means that we `modify' the basis at places in $J'$, or more precisely, we consider the closed subspace $\calC(\calO_{p,J}\times P_{J^c},\Lambda_J^{>1/p})^{J'-mod}$ of $\calC(\calO_{p,J}\times P_{J^c},\Lambda_J^{>1/p})$ defined  before. In particular, we always have $J'\subset J$ when writing $S_{\kappa,I}^{D,J'}(K^p,\Lambda_J^{>1/p})$, and if $J''\subset J'\subset J$ are three nonempty subsets of $I$, we have a natural map $S_{\kappa,J}^{D,J'}(K^p,\Lambda_J^{>1/p})\rightarrow S_{\kappa,J}^{D,J''}(K^p,\Lambda_J^{>1/p})$, which is equivariant under the $U_{\pi_j}$-operator, for all $j\in J$.
		\item Fix a nonempty subset $J$ of $I$ and let $A=\Lambda_J$ or $\Lambda_J^{>1/p}$. For $j\in J$, under the isomorphism $\ZZ_p\cong \calO_{\gothp_j}$, we can identify $\calC(\ZZ_p,A)$ with $\calC(\calO_{\gothp_j},A)$. So it is meaningful to talk about the Mahler basis (resp. modified Mahler basis) of the space $\calC(\calO_{\gothp_j},A)$ (resp. $\calC(\calO_{\gothp_j},A)^{mod}$). We define the Mahler basis of $A\llbracket \calO_{\gothp_j}\rrbracket$ to be the dual basis of the Mahler basis of $\calC(\calO_{\gothp_j},A)$. Under the isomorphism $A\llbracket \calO_{\gothp_j}\rrbracket \cong A\llbracket X_j\rrbracket$ constructed as above, the Mahler basis of $A\llbracket \calO_{\gothp_j}\rrbracket$ is given by $\{X_j^m|m\geq 0 \}$. In general, from the expression $A\llbracket \calO_{p,J}\rrbracket=\hat{\otimes}_{j\in J}A\llbracket\calO_{\gothp_j}\rrbracket$, we can talk about the Mahler basis of $A\llbracket \calO_{p,J}\rrbracket$. Under the isomorphism $A\llbracket \calO_{p,J}\rrbracket\cong A\llbracket (X_j)_{j\in J}\rrbracket$, the Mahler basis is given by $\left\{\prod\limits_{j\in J} X_j^{n_j}|n_j\geq 0 \right\}$. Since $A\llbracket \calO_{p,J}\times P'_{J^c}\rrbracket =A\llbracket \calO_{p,J}\rrbracket\hat{\otimes}_A A\llbracket P'_{J^c}\rrbracket$, the Mahler basis of $A\llbracket \calO_{p,J}\rrbracket$ becomes a basis of the left $A\llbracket P'_{J^c}\rrbracket$-module $A\llbracket \calO_{p,J}\times P'_{J^c}\rrbracket$, which is called the Mahler basis of $A\llbracket \calO_{p,J}\times P'_{J^c}\rrbracket$. We make similar definitions for the modified Mahler basis of the spaces $A\llbracket \calO_{\gothp_j}\rrbracket^{mod}$, $A\llbracket \calO_{p,J}\rrbracket^{J'-mod}$ and $A\llbracket \calO_{p,J}\times P'_{J^c}\rrbracket^{J'-mod}$.
		\item 	Strictly speaking, it is crucial for us to work with the $\Lambda_J^{>1/p}$-dual spaces of $\calC(\calO_{p,J}\times P_{J^c}, \Lambda_J^{>1/p})$ and $\calC(\calO_{p,J}\times P_{J^c}, \Lambda_J^{>1/p})^{J'-mod}$. To explain the problem, we assume $J=I$ for simplicity. Recall that $\calC(\calO_p,\Lambda^{>1/p})^{mod}$ is the subspace of $\calC(\calO_p,\Lambda^{>1/p})$ spanned by the orthonormal basis $\{e_m'|m\in \ZZ_{\geq 0}^I \}$, i.e. any element in $\calC(\calO_p,\Lambda^{>1/p})^{mod}$ is of the form $\sum\limits_{m\in \ZZ_{\geq 0}^I}a_me_m'$ with $a_m\in \Lambda^{>1/p}$ and $a_m\rightarrow 0$ as $|m|:= \sum\limits_{i\in I}m_i\rightarrow \infty$ under the $\gothm_{\Lambda^{>1/p}}$-adic topology. In particular, $\calC(\calO_p,\Lambda^{>1/p})^{mod}$ is not isomorphic to the direct product $\prod\limits_{m\in \ZZ_{\geq 0}^I}\Lambda^{>1/p}$, while its $\Lambda^{>1/p}$-dual is. The computation in \S\ref{S:Orthonormalizable spaces and compact operators} actually shows that there is a well-defined action of the monoid $\bbM_{\pi}$ on $\Lambda^{>1/p}\llbracket \calO_p\rrbracket^{I-mod}$, such that the inclusion $\Lambda^{>1/p}\llbracket \calO_p\rrbracket\rightarrow \Lambda^{>1/p}\llbracket \calO_p\rrbracket^{I-mod}$ is $\bbM_{\pi}$-equivariant. Hence the spaces $\Lambda^{>1/p}\llbracket \calO_{p,J}\times P_{J^c}\rrbracket^{J'-mod}$ for various subsets $J'\subset J\subset I$ are the correct objects to consider. Since $S_{\kappa,I}^D(K^p,\Lambda^{>1/p})$ and its dual space $S_{\kappa,I}^D(K^p,\Lambda^{>1/p})^\vee$ share the same Hecke eigensystem, we can work with the space $S_{\kappa,I}^D(K^p,\Lambda^{>1/p})^\vee$ and embedded into the space $S_{\kappa,I}^{D,I}(K^p,\Lambda^{>1/p})^\vee$  to study the $U_{\pi_i}$-eigenvalues. In the rest of our proof of Theorem~\ref{T:spectral halo for eigenvarieties for D}, we will insist to work with the dual spaces $S_{\kappa,J}^{D,J'}(K^p,\Lambda_J^{>1/p})^\vee$  of generalized $p$-adic automorphic forms.
	\end{enumerate}	
\end{remark}

\subsection{A filtration on the space of $p$-adic automorphic forms with respect to a $U_{\pi_i}$-operator}\label{subsection:a filtration on the space of p-adic automorphic forms with respect to a Up operator}

Throughout this section, we take $J=\{j \}\subset I$ consisting of a single place $j$ in $I$, and we use $R$ to denote the ring $\Lambda_J^{>1/p}=\ZZ_p[H_J]\otimes_{\ZZ_p}\ZZ_p\llbracket T_j,\frac{p}{T_j},T\rrbracket$. In particular we have $\calO_{p,J}=\calO_{\gothp_j}$. 

Let $\chi_J:\calO_{p,J}^\times\times\ZZ_p^\times=\calO_{\gothp_j}^\times\times\ZZ_p^\times\rightarrow R^\times$ be the universal character.
Recall that the space $S_{\kappa,J}^D(K^p,R)^\vee$ is endowed with a left $R\llbracket P'_{J^c} \rrbracket$-module structure. Moreover, the $U_{\pi_j}$-operator and the explicit expression $S_{\kappa,J}^D(K^p,R)^\vee\xrightarrow{\cong}\bigoplus\limits_{k=0}^{s-1}R\llbracket \calO_{\gothp_j}\times P_{J^c} \rrbracket$ are both $R\llbracket  P'_{J^c} \rrbracket$-linear. We have a `dual' result of Proposition~\ref{P:explicit expression of Up operators on integral model of p-adic automorphic forms}: under the isomorphism $S_{\kappa,J}^D(K^p,R)^\vee\xrightarrow{\cong}\bigoplus\limits_{k=0}^{s-1}R\llbracket \calO_{\gothp_j}\times  P_{J^c} \rrbracket$, the $U_{\pi_j}$-operator can be described by the following commutative diagram:
\[
	\xymatrix@=1.3cm{
	S_{\kappa,J}^D(K^p,R)^\vee \ar[r]^{\cong} \ar[d]^{U_{\pi_j}} & \bigoplus\limits_{k=0}^{s-1}R\llbracket \calO_{\gothp_j}\times P_{J^c} \rrbracket\ar[d]^{\calU_j^\vee} \\
	S_{\kappa,J}^D(K^p,R)^\vee \ar[r]^{\cong} & \bigoplus\limits_{k=0}^{s-1}R\llbracket \calO_{\gothp_j}\times P_{J^c} \rrbracket.
}
\]
	Here the right vertical maps $\calU_j^\vee$ is the `transpose' of the map $\calU_j$ in Proposition~\ref{P:explicit expression of Up operators on integral model of p-adic automorphic forms}. More explicitly,  $\calU_j^\vee$ is given by a  $t\times t$ matrix with the following descriptions:
\begin{enumerate}
	\item Each entry of $\calU_j^\vee$ is a sum of operators of the form $\delta_p\circ$, where $\delta_p=(\delta_i)_{i\in I}\in \rmM_2(\calO_p)$ have the same property as in Proposition~\ref{P:explicit expression of Up operators on integral model of p-adic automorphic forms}.
	\item There are exactly $p$ such operators appearing in each row and column of $\calU_j^\vee$.
\end{enumerate}

Notice that $R\llbracket P_{J^c}\rrbracket$ is a free left $R\llbracket P'_{J^c}\rrbracket$-module of rank $|\Delta_{J^c}|$. If we fix a set of representatives in $P_{J^c}$ of the quotient $P_{J^c}/P'_{J^c}\cong \Delta_{J^c}$, we obtain an $R\llbracket P'_{J^c}\rrbracket$-linear isomorphism $\llbracket P_{J^c}\rrbracket\cong \bigoplus\limits_{k=0}^{|\Delta_{J^c}|}R\llbracket P'_{J^c}\rrbracket$. This induces an isomorphism $R\llbracket \calO_{\gothp_j}\times P_{J^c}\rrbracket \cong \bigoplus\limits_{k=1}^{|\Delta_{J^c}|}R\llbracket \calO_{\gothp_j}\times P'_{J^c}\rrbracket$ and we have an isomorphism $S_{\kappa, J}^D(K^p,R)^\vee\xrightarrow{\cong}\bigoplus\limits_{k=1}^SR\llbracket \calO_{\gothp_j}\times P'_{J^c}\rrbracket$ of left $R\llbracket P'_{J^c}\rrbracket$-modules, where $S=s\cdot |\Delta_{J^c}|$. Similarly, we have an $R\llbracket P'_{J^c}\rrbracket$-linear isomorphism $S_{\kappa, J}^{D,J}(K^p,R)^\vee\xrightarrow{\cong}\bigoplus\limits_{k=1}^SR\llbracket \calO_{\gothp_j}\times P'_{J^c}\rrbracket^{J-mod}$.

Under the above isomorphisms, we choose a basis $\left\{f_{k,m}|1\leq k\leq S,m\in \ZZ_{\geq 0}\right\}$ of $S_{\kappa, J}^{D}(K^p,R)^\vee$ as a left $R\llbracket P'_{J^c}\rrbracket$-module, such that $\{f_{k,m}|m\in \ZZ_{\geq 0}\}$ is the Mahler basis of the $k$-th direct factor of $\bigoplus\limits_{k=0}^{s-1}R\llbracket \calO_{\gothp_j}\times P'_{J^c}\rrbracket$ for all $1\leq k\leq S$. We choose a basis $\{f_{k,m}^{mod}|1\leq k\leq S,m\in \ZZ_{\geq 0} \}$ of $S_{\kappa, J}^{D,J}(K^p,R)^\vee$ in a similar way.

Notice that for any $\delta_i\in \bbM_{\pi_i}$, if we use $P(\delta_i)=(P_{m,n}(\delta_i))_{m,n\geq 0}$ to denote the infinite matrix for the action of $\delta_i$ on $\calC(\calO_{\gothp_i}, A)$ under the Mahler basis, then the matrix for the action of $\delta_i$ on $A\llbracket \calO_{\gothp_i}\rrbracket$ under the Mahler basis is the transpose of the matrix $P(\delta_i)$.

We use $N=(N_{m,n})_{m,n\geq 0}$ (resp. $M=(M_{m,n})_{m,n\geq 0}$) to denote the matrix in $\rmM_\infty(R\llbracket P'_{J^c}\rrbracket)$ which corresponds to the $U_{\pi_j}$-operator on $S_{\kappa,J}^D(K^p,R)^\vee$ (resp. $S_{\kappa,J}^{D,J}(K^p,R)^\vee$) under the basis we choose above. It follows from Propositions~\ref{P:explicit expression of Up operators on integral model of p-adic automorphic forms}, ~\ref{P:Liu-Wan-Xiao's computation on matrix coefficients of Mahler basis}  and the above remark that we have $N_{m,n}\in (T_j)^{\max\{ \lfloor \frac{n}{S} \rfloor-\lfloor \frac{m}{pS} \rfloor,0 \}}$. On the other hand, from the construction of the Mahler basis, the matrix $M$ is the conjugation of $N$ by the infinite diagonal matrix with diagonal entries $\underbrace{1,1,\dots,1}_S,\underbrace{T_j,T_j,\dots,T_j}_{S}, \underbrace{T^2_j,T^2_j,\dots,T^2_j}_{S},\dots$ Define two sequences $\underline{\lambda}$, $\underline{\mu}$ of integers as $\lambda_n=\lfloor \frac{n}{S} \rfloor-\lfloor \frac{n}{pS} \rfloor$ and $\mu_0=0$, $\mu_{n+1}-\mu_n=\lambda
_n$ for all $n\in \ZZ_{\geq 0}$. The above computation implies that the matrix $M\in \rmM_\infty(R\llbracket P'_{J^c}\rrbracket)$ is $\underline{\lambda}$-Hodge bounded with respect to the element $T_j\in R\subset R\llbracket P'_{J^c}\rrbracket$.

In the following discussion, we fix a character $\omega=(\eta_j,\eta):H_J=\Delta_j \times \Delta \rightarrow \ZZ_p^\times$. Denote $R_{\omega_J}=R\otimes_{\ZZ_p[H_J],\omega_J}\ZZ_p$, which is naturally isomorphic to $\ZZ_p\llbracket T_j,\frac{p}{T_j},T\rrbracket$. We construct a chain of ring homomorphisms as follows:
\begin{itemize}
	\item Let $R\rightarrow R_{\omega_J}$ and $R\llbracket P'_{J^c}\rrbracket\rightarrow R_{\omega_J}\llbracket P'_{J^c}\rrbracket$ be the natural homomorphisms.
	\item Let $R_{\omega_J}\llbracket P'_{J^c}\rrbracket\rightarrow R_{\omega_J}$ be the reduction map modulo the augmentation ideal of the complete group algebra $R_{\omega_J}\llbracket P'_{J^c}\rrbracket$.
	\item Under the isomorphism $R_{\omega_J}\cong \ZZ_p\llbracket T_j,\frac{p}{T_j},T\rrbracket$, let $R_{\omega_J}\rightarrow R_j:= \ZZ_p\llbracket T_j,\frac{p}{T_j}\rrbracket$ be the reduction map modulo the ideal generated by $T$.
	\item Let $R_j\rightarrow \FF_p\llbracket T_j\rrbracket$ be the reduction map modulo the ideal generated by $\frac{p}{T_j}$.
\end{itemize}
We apply the above homomorphisms $R\llbracket P'_{J^c}\rrbracket  \rightarrow R_{\omega_J}\llbracket P'_{J^c}\rrbracket\rightarrow R_{\omega_J}\rightarrow R_j\rightarrow \FF_p\llbracket T_j\rrbracket$ to entries of the matrix $M$, and obtain matrices $M_{R_{\omega_J}\llbracket P'_{J^c}\rrbracket}\in \rmM_\infty(R_{\omega_J}\llbracket P'_{J^c}\rrbracket)$, $M_{R_{\omega_J}}\in \rmM_\infty(R_{\omega_J})$, $M_{R_j}\in \rmM_\infty(R_j)$ and $\bar{M}\in \rmM_\infty(\FF_p\llbracket T_j\rrbracket)$. All these matrices are $\underline{\lambda}$-Hodge bounded with respect to the element $T_j$ as the matrix $M$ is so.

Fix $l\in 2\ZZ_{\geq 0}$. For every character $\omega_{J^c}:\Delta_{J^c}\rightarrow \ZZ_p^\times$, we obtain a character $\omega=(\omega_J,\omega_{J^c}):H\rightarrow \ZZ_p^\times$. We construct a point $\chi_l\in \calW(\CC_p)$ whose associated character $\kappa_l:\calO_p^\times\times\calO_p^\times\rightarrow\CC_p^\times$ is locally algebraic and corresponds to the triple $(n\in\ZZ_{\geq 0}^I,\nu\in \ZZ^I,\psi=(\psi_1,\psi_2))$ defined as follows.
\begin{itemize}
	\item Define $\nu:=(\nu_i)_{i\in I}$ by $\nu_j:=-\frac{l}{2}$ and $\nu_i:=0$ for all $i\neq j$, and set $n:=-2\nu\in\ZZ_{\geq 0}^I$.
	\item Define two finite characters $\psi_1,\psi_2:\calO_p^\times\rightarrow \CC_p^\times$ with the following properties: \begin{itemize}
		\item $\psi_1|_{1+\pi_i\calO_{\gothp_i}}$ and $\psi_2|_{1+\pi_i\calO_{\gothp_i}}$ are trivial for all $i\neq j$;
		\item $\psi_2|_{1+\pi_j\calO_{\gothp_j}}$ is a nontrivial character which factors through $(1+\pi_j\calO_{\gothp_j})/(1+\pi_j^2\calO_{\gothp_j})$; 
		\item $\psi_1|_{1+\pi_j\calO_{\gothp_j}}=(\psi_2|_{1+\pi_j\calO_{\gothp_j}})^{-2}$; and
		\item the characters $\psi_1|_{\Delta_p}$ and $\psi_2|_{\Delta_p}$ are uniquely determined by the condition that the point $\chi_l$ belongs to the component $\calW_\omega$ of $\calW$. 
	\end{itemize} 

\end{itemize}
It is straightforward to verify that $T_{i,\chi_l}=0$ for all $i\neq j$, and $v_p(T_{j,\chi_l})=\frac{1}{p-1}\in (0,1)$. 

Note that the $\Iw_{\pi,1,J^c}$-coinvariant space of $S_{\kappa,J}^D(K^p,R_{\omega_J})^\vee$ (resp. $S_{\kappa,J}^{D,J}(K^p,R_{\omega_J})^\vee$) is given by $S_{\kappa,J}^D(K^p,R_{\omega_J})^\vee\otimes_{R_{\omega_J}\llbracket P'_{J^c}\rrbracket}R_j$ (resp. $S_{\kappa,J}^{D,J}(K^p,R_{\omega_J})^\vee\otimes_{R_{\omega_J}\llbracket P'_{J^c}\rrbracket}R_j$), and this space admits an action of $\Iw_{\pi,J^c}/\Iw_{\pi,1,J^c}\cong \Delta_{J^c}$. 

Define a homomorphism $\tau_l:R_j\rightarrow \CC_p$ of $\ZZ_p$-algebras with $\tau_l(T_j)=T_{j,\chi_l}$ and we use $M_{\tau_l}\in \rmM_\infty (\CC_p)$ to denote the infinite matrix obtained by applying $\tau_l$ to entries of $M_{R_j}$. 

Now we consider another character $\omega_J'=(\omega_1^{-1},\omega_2):H_J\rightarrow \ZZ_p^\times$, and obtain another character $\omega'=(\omega_J',\omega_{J^c}):H\rightarrow \ZZ_p^\times$. We construct another point $\chi_l'\in \calW(\CC_p)$ whose associated character $\kappa_l':\calO_p^\times\times\calO_p^\times\rightarrow \CC_p^\times$ corresponds to the triple $(n,\nu,\psi'=(\psi_1^{-1},\psi_2))$. It belongs to the weight disc $\calW_{\omega'}$. Applying the homomorphisms $R\llbracket P'_{J^c}\rrbracket\rightarrow R_{\omega'_J}\llbracket P'_{J^c}\rrbracket\rightarrow R_{\omega'_J}\rightarrow R_j$, we get matrices $M'_{R_{\omega'_J}\llbracket P'_{J^c}\rrbracket}$, $M'_{R_{\omega'_J}}$ and $M'_{R_j}$ as before. Applying the homomorphism $\tau_l$ to entries of $M'_{R_j}$, we get a matrix $M'_{\tau_l}\in \rmM_\infty(\CC_p)$. 

Define $t\in\NN^I$ by $t_j=2$ and $t_i=1$ for all $i\neq j$. Since we have $|\Delta_{J^c}|$ choices of the characters $\omega_{J^c}:\Delta_{J^c}\rightarrow \ZZ_p^\times$, and $\dim S_{k,w}^D(K^p\Iw_{\pi^t},\psi)=sp(l+1)$, it follows from Proposition~\ref{P:pairing of Newton slopes by Atkin-Lehner} that there are $sp(l+1)|\Delta_{J^c}|=Sp(l+1)$ pairs of the slopes of the Newton polygons of $M_{\tau_l}$ and $M'_{\tau_l}$, such that the slopes in each pair sum to $l+1$. Hence the total sum of these slopes is $Sp(l+1)^2$. Since the matrix $M_{R_j}\in \rmM_\infty(R_j)$ is $\underline{\lambda}$-Hodge bounded, it follows from the construction of the matrix $M_{\tau_l}$ that the sum of the first $Sp(l+1)$ slopes of $M_{\tau_l}$ is at least $v_p(T_{j,\chi_l})\sum\limits_{k=0}^{Sp(l+1)-1}\lambda_k$. For the same reason, the first $Sp(l+1)$ slopes of $M_{\tau'_l}$ is at least $v_p(T_{j,\chi'_l})\sum\limits_{k=0}^{Sp(l+1)-1}\lambda_k$. As $v_p(T_{j,\chi_l})=v_p(T_{j,\chi'_l})=\frac{1}{p-1}$, we have $Sp(l+1)^2=v_p(T_{j,\chi_l})\sum\limits_{k=0}^{Sp(l+1)-1}\lambda_k+v_p(T_{j,\chi'_l})\sum\limits_{k=0}^{Sp(l+1)-1}\lambda_k$. Therefore the first $Sp(l+1)$ slopes of the matrix $M_{\tau_l}$ (resp. $M_{\tau'_l}$) sum to exactly $v_p(T_{j,\chi_l})\sum\limits_{k=0}^{Sp(l+1)-1}\lambda_k=\frac{1}{2}Sp(l+1)^2$ (resp. $v_p(T_{j,\chi'_l})\sum\limits_{k=0}^{Sp(l+1)-1}\lambda_k=\frac{1}{2}Sp(l+1)^2$). Let $$\mathrm{char}(M_{R_j}):=\sum\limits_{n\geq 0}c_nX^n\in \ZZ_p\llbracket T_j\rrbracket \llbracket X \rrbracket\subset R_j\llbracket X\rrbracket$$ be the characteristic power series of $M_{R_j}$, with $c_n\in \ZZ_p\llbracket T_j \rrbracket$. 
The above discussion implies that the Newton polygon of $\sum\limits_{n\geq 0}c_n(T_{j,\chi_l})X^n$ passes through $(Sp(l+1),\mu_{Sp(l+1)}v_p(T_{j,\chi_l}))$ for all $l\in 2\ZZ_{\geq 0}$. For simplicity, we denote $n_{l+1}=Sp(l+1)$ for $l\in 2\ZZ_{\geq 0}$. 

Now we can run the exact argument in the proof of \cite[Theorem~$1.3$]{liu2017eigencurve}. We give a summary of their results which will be used in our argument:
\begin{enumerate}
	\item Write $c_n(T_j)=\sum\limits_{m\geq 0}b_{n,m}T_j^m$, with $b_{n,m}\in \ZZ_p$. Then $v_p(b_{n,m})\geq \max\{\mu_n-m,0 \}$. In particular, for $t_j\in \CC_p$ with $v_p(t_j)\in (0,1)$, we have $v_p(c_n(t_j))\geq \mu_nv_p(t_j)$ for all $n\in\ZZ_{\geq 0}$. The equality holds if and only if $b_{n,\mu_n}\in \ZZ_p^\times$.
	\item For any $l\in2\ZZ_{\geq 0}$, there exist two integers $n_{l+1}^-\in [n_{l+1}-S,n_{l+1}]$ and $n_{l+1}^+\in [n_{l+1},n_{l+1}+S]$, such that for any $t_j\in \CC_p$ with $v_p(t_j)\in (0,1)$, we have the following.
	\begin{itemize}
		\item If $n_{l+1}^-\neq n_{l+1}^+$, then the points $(n_{l+1}^-,\mu_{n_{l+1}^-}v_p(t_j))$ and $(n_{l+1}^+,\mu_{n_{l+1}^+}v_p(t_j))$ are two consecutive vertices of the Newton polygon of $\sum\limits_{n\geq 0}c_n(t_j)X^n$. Moreover, the line segment connecting these two vertices has slope $(l+1)(p-1)v_p(t_j)$;
		\item If $n_{l+1}^-=n_{l+1}=n_{l+1}^+$, then $(n_{l+1},\mu_{n_{l+1}v_p(t_j)})$ is a vertex of the Newton polygon of $\sum\limits_{n\geq 0}c_n(t_j)X^n$.
	\end{itemize}
\end{enumerate}

Now we consider the matrix $\bar{M}\in \rmM_{\infty}(\FF_p\llbracket T_j \rrbracket)$. Its characteristic power series $\mathrm{char}(\bar{M})=\sum\limits_{n\geq 0}d_n(T_j)X^n$ can be obtained by applying the homomorphism $R_j\rightarrow \FF_p\llbracket T_j \rrbracket$ to the coefficients of the power series $\mathrm{char}(M_{R_j})=\sum\limits_{n\geq 0}c_n(T_j)X^n$. We define the lower bound polygon of $\mathrm{char}(\bar{M})$ to be the lower convex hull of the points $(n,\mu_n)$, for all $n\in\ZZ_{\geq 0}$. Since $\bar{M}$ is $\underline{\lambda}$-Hodge bounded, the Newton polygon of $\mathrm{char}(\bar{M})$ always lies on or above its lower bound polygon. We observe the following two facts:
\begin{itemize}
	\item The kernel of the composite $\ZZ_p\llbracket T_j \rrbracket\rightarrow R_j\rightarrow \FF_p\llbracket T_j \rrbracket$ is the principal ideal generated by $p$.
	\item The $x$-coordinates of the line segment of the lower bound polygon with slope $(l+1)(p-1)$ belong to the interval $[n_{l+1}-S,n_{l+1}+S]$.
\end{itemize}
Combined the results from \cite{liu2017eigencurve}, these facts imply the following.

If $n_{l+1}^-\neq n_{l+1}^+$, the points $(n_{l+1}^-,v_{T_j}(d_{n_{l+1}^-}(T_j)))$ and $(n_{l+1}^+,v_{T_i}(d_{n_{l+1}^+}(T_j)))$ lie on the lower bound polygon and the slope of the line segment connecting these two vertices has slope $\lambda_{n_{l+1}}=(l+1)(p-1)$; if $n_{l+1}^-=n_{l+1}=n_{l+1}^+$, $(n_{l+1},v_{T_j}(d_{n_{l+1}}(T_j)))$ lies on the lower bound polygon. In both cases, the slopes of the Newton polygon of $\mathrm{char}(\bar{M})$ on the left of the point $(n_{l+1}^-,v_{T_j}(d_{n_{l+1}^-}(T_j)))$ belong to $[0,\lambda_{n_{l+1}})$, and the slopes of the Newton polygon of $\mathrm{char}(\bar{M})$ on the right of the point $(n_{l+1}^+,v_{T_j}(d_{n_{l+1}^+}(T_j)))$ belong to $(\lambda_{n_{l+1}},\infty)$. In particular, $(n_{l+1}^-,v_{T_j}(d_{n_{l+1}^-}(T_j)))$ and $(n_{l+1}^+,v_{T_j}(d_{n_{l+1}^+}(T_j)))$ are touching vertices of the Newton polygon of $\mathrm{char}(\bar{M})=\sum\limits_{n\geq 0}d_n(T_j)X^n$.

Notice that $(0,0)$ is always a touching vertex of $\bar{M}$. The above construction of the numbers $n_{l+1}^+$'s also applies to the point $(0,0)$. It follows that we can find an integer $n_0^+\in [0,S]$, such that $(n_0^+,v_{T_j}(d_{n_0^+}(T_j)))$ is a touching vertex of the Newton polygon of $\mathrm{char}(\bar{M})$, and the slopes of this Newton polygon on the left of the point $(n_0^+,v_{T_j}(d_{n_0^+}(T_j)))$ are all $0$. We denote $n_0^-=0$ for simplicity.

Let $\Omega_{\omega_J}:=\{n_0^-,n_0^+ \}\cup\{n_l^-,n_l^+|l\in 1+2\ZZ_{\geq 0} \}=\{ n_0^-\leq n_0^+\leq n_1^-\leq n_1^+\leq n_3^-\leq \dots  \}$. We can apply Theorem~\ref{T:Newton-Hodge decomposition with Hodge bound restriction for infinite matrices over noncommutative rings} to the matrix $M\in \rmM_{\infty}(R_{\omega_J}\llbracket P_{J^c}'\rrbracket)$ and the set $\Omega_{\omega_J}$, and get the following.
\begin{proposition}\label{P:filtration on the space of generalized automoprhic forms}
	For every character $\omega_J:H_J\rightarrow \ZZ_p^\times$, there exist a basis $\{v_m|m\in \NN \}$ and a filtration $\{\tilde{V}_\alpha|\alpha\in\Omega_{\omega_J} \}$ of the left $R_{\omega_J}\llbracket P_{J^c}'\rrbracket$-module $S_{\kappa,J}^{D,J}(K^p,R_{\omega_J})^\vee$ with the following properties.
	\begin{enumerate}
		\item If we use $N\in \rmM_\infty(R_{\omega_J}\llbracket P_{J^c}'\rrbracket)$ to denote the matrix corresponding to the $U_{\pi_j}$-operator on $S_{\kappa,J}^{D,J}(K^p,R_{\omega_J})^\vee$ under the basis $\{v_m|m\in \NN \}$, the matrix $N$ is $\underline{\lambda}$-Hodge bounded with respect to the element $T_j\in R_{\omega_J}$.
		\item The set $\{\tilde{V}_\alpha|\alpha\in\Omega_{\omega_J} \}$ is an increasing filtration of $S_{\kappa,J}^{D,J}(K^p,R_{\omega_J})^\vee$; more precisely, we have $(0)=\tilde{V}_{n_0^-}\subset \tilde{V}_{n_0^+}\subset \tilde{V}_{n_1^-}\subset \tilde{V}_{n_1^+}\subset \dots$.
		\item For every $\alpha\in\Omega_{\omega_J}$, $\tilde{V}_\alpha$ is a free  $R_{\omega_J}\llbracket P_{J^c}'\rrbracket$-submodule of $S_{\kappa,J}^{D,J}(K^p,R_{\omega_J})^\vee$ of finite rank $d_\alpha\geq 0$, and $\{ v_m|m=1,\dots, d_\alpha \}$ is a basis of $\tilde{V}_\alpha$.
		\item For every $\alpha\in\Omega_{\omega_J}$, $\tilde{V}_\alpha$ is stable under the $U_{\pi_j}$-operator. We use $R_\alpha\in \rmM_{d_\alpha}(R_{\omega_J}\llbracket P_{J^c}'\rrbracket)$ to denote the matrix corresponding to the $U_{\pi_j}$-operator on $\tilde{V}_\alpha$ under the basis $\{v_m|m=1,\dots,d_\alpha \}$. In particular, the matrix $N$ can be written in the blockwise form:
		\[
		N=
		\quad
		\Biggl[\mkern-5mu
		\begin{tikzpicture}[baseline=-.65ex]
		\matrix[
		matrix of math nodes,
		column sep=1ex,
		] (m)
		{
			A_\alpha & B_\alpha \\
			0\ \   & D_\alpha \\
		};
		\draw[dotted]
		([xshift=0.5ex]m-1-1.north east) -- ([xshift=0.5ex]m-2-1.south east);
		\draw[dotted]
		(m-1-1.south west) -- (m-1-2.south east);
		\node[above,text depth=1pt] at ([xshift=2.5ex]m-1-1.north) {$\scriptstyle d_\alpha$};  
		\node[left,overlay] at ([xshift=-1.2ex,yshift=-2ex]m-1-1.west) {$\scriptstyle d_\alpha$};
		\end{tikzpicture}\mkern-5mu
		\Biggr]\in \rmM_{\infty}(R_{\omega_J}\llbracket P_{J^c}'\rrbracket).
		\] 
		\item Let $\alpha\leq \beta$ be two consecutive elements in $\Omega_{\omega_J}$. The matrix $A_\beta$ can be written in the block uppertriangular forms: $\quad
		\Biggl[\mkern-5mu
		\begin{tikzpicture}[baseline=-.65ex]
		\matrix[
		matrix of math nodes,
		column sep=1ex,
		] (m)
		{
			A_{11} & A_{12} \\
			\ 0\ \   & A_{22} \\
		};
		\draw[dotted]
		([xshift=0.5ex]m-1-1.north east) -- ([xshift=0.5ex]m-2-1.south east);
		\draw[dotted]
		(m-1-1.south west) -- (m-1-2.south east);
		\node[above,text depth=1pt] at ([xshift=2.5ex]m-1-1.north) {$\scriptstyle d_\alpha$};  
		\node[left,overlay] at ([xshift=-1.2ex,yshift=-2ex]m-1-1.west) {$\scriptstyle d_\alpha$};
		\end{tikzpicture}\mkern-5mu
		\Biggr]$ with $A_{11}=A_\alpha$. We use $\bar{A}_{22}\in \rmM_d(\FF_p\llbracket T_j \rrbracket)$ to denote the matrix obtained by applying the homomorphism  $R_{\omega_J}\llbracket P_{J^c}'\rrbracket\rightarrow \FF_p\llbracket T_j \rrbracket$ to entries of $A_{22}$, where $d=d_\beta-d_\alpha$. Then the slopes of the Newton polygon of $\bar{A}_{22}$ (with respect to the $T_j$-adic valuation) belong to $\{\lambda_l \}$ if $\alpha=n_l^-$, $\beta=n_l^+$, and belong to $(\lambda_l,\lambda_{l'})$ if $\alpha=n_l^+$, $\beta=n_{l'}^-$. 
	\end{enumerate}
\end{proposition}

\section{Proof of the main theorem}\label{section:Proof of the main theorem}

\begin{notation}
	In this section, we write $I=\{j_1,j_2,\dots, j_g \}$. For $1\leq l\leq g$, we denote by $J_l$ the subset $\{j_1,\dots, j_l\}$ of $I$. We will drop the letter $J$ or $j$ in the previous notations for simplicity. More precisely, for $1\leq l\leq g$, we make the following.
	\begin{enumerate}
		\item Let $\gothp_l$ be the prime of $F$ corresponding to the embedding $j_l\in I$ and $\pi_l=\pi_{j_l}$ be the uniformizer of $\calO_{\gothp_l}$ we fixed before.
		\item Let $H_l:=H_{J_l}$ denote the torsion subgroup of $\calO_{p,J}^\times\times\ZZ_p^\times$ and $\delta_l$ be the torsion subgroup of $\calO_{\gothp_l}^\times$. 
		\item Let $\Lambda_l$ and $R_l$ denote the ring $\Lambda_{J_l}=\ZZ_p\llbracket \calO_{p,J}^\times\times\ZZ_p^\times\rrbracket$ and $\Lambda_{J_l}^{>1/p}$, respectively.
		\item Let $T_l\in \Lambda_l$ be the element $T_{j_l}$ defined in \S\ref{subsection:Notations in Automorphic forms for definite quaternion algebras and completed homology}. In particular, we can write $\Lambda_l=\ZZ_p[H_l]\otimes_{\ZZ_p}\ZZ_p\llbracket (T_m)_{1\leq m\leq l}, T\rrbracket$ and $R_l=\ZZ_p[H_l]\otimes_{\ZZ_p}\ZZ_p\llbracket (T_m,\frac{p}{T_m})_{1\leq m\leq l}, T\rrbracket$.
	\end{enumerate} 
\end{notation}

\subsection{Properties of the filtration on the space of integral $p$-adic automorphic forms}

We apply the construction in \S\ref{subsection:a filtration on the space of p-adic automorphic forms with respect to a Up operator} to the set $J_1$ consisting of a single element $j_1$.
To be more precise, we fix a character $\omega_{1}=(\eta_1,\eta):H_{1}=\Delta_1\times\Delta\rightarrow \ZZ_p^\times$. Let $R_{1,\omega_{1}}=R_1\otimes_{\ZZ_p[H_{1}],\omega_{1}}\ZZ_p$, which is isomorphic to $\ZZ_p\llbracket T_{1},\frac{p}{T_{1}},T\rrbracket$ and $R_{2,\omega_2}=R_2\otimes_{\ZZ_p[H_{1}],\omega_{1}}\ZZ_p$. 

We apply Proposition~\ref{P:filtration on the space of generalized automoprhic forms} to the space $S_{\kappa,J_1}^{D,J_1}(K^p,R_{1,\omega_{1}})^\vee$ and the $U_{\pi_{1}}$-operator on it. Then we get a set $\Omega_{\omega_{1}}=\{n_l^-,n_l^+|l\in \{0\}\cup  1+2\ZZ_{\geq 0} \}$, a basis $\{v_m|m\in \NN\}$ and a filtration $\{\tilde{V}_\alpha|\alpha\in \Omega_{\omega_{1}}\}$ of the space $S_{\kappa,J_1}^{D,J_1}(K^p,R_{1,\omega_{1}})^\vee$. By Remark~\ref{remark:some basic properties of the space C_J(kappa,A)}, we have a surjective map $S_{\kappa,J_1}^{D,J_1}(K^p,R_{2,\omega_{1}})^\vee\rightarrow S_{\kappa,J_2}^{D,J_1}(K^p,R_{2,\omega_{1}})^\vee$, which identifies the latter space as the $B(\calO_{\gothp_{2}})$-coinvariants of the first space. For $\alpha\in \Omega_{\omega_{1}}$, we use $V_\alpha$ to denote the image of $\tilde{V}_\alpha\otimes_{R_{1,\omega_{1}}}R_{2,\omega_{2}}$ under this map.

\begin{proposition}\label{P:filtration with respect to one Up operator is stable under the other Up operators}
	For all  $\alpha\in\Omega_{\omega_{1}}$, the space $V_{\alpha}$ is stable under the $U_{\pi_{2}}$-operator on $S_{\kappa,J_2}^{D,J_1}(K^p,R_{2,\omega_{1}})^\vee$.
\end{proposition}

\begin{proof}
	Throughout the proof, we use $R_1'$ and $R_2'$ to denote the rings $R_{1,\omega_{1}}$ and $R_{2,\omega_{2}}$, respectively.
	
	Let $N=\quad
	\Biggl[\mkern-5mu
	\begin{tikzpicture}[baseline=-.65ex]
	\matrix[
	matrix of math nodes,
	column sep=1ex,
	] (m)
	{
		A_\alpha & B_\alpha \\
		0\ \   & D_\alpha \\
	};
	\draw[dotted]
	([xshift=0.5ex]m-1-1.north east) -- ([xshift=0.5ex]m-2-1.south east);
	\draw[dotted]
	(m-1-1.south west) -- (m-1-2.south east);
	\node[above,text depth=1pt] at ([xshift=2.5ex]m-1-1.north) {$\scriptstyle d_\alpha$};  
	\node[left,overlay] at ([xshift=-1.2ex,yshift=-2ex]m-1-1.west) {$\scriptstyle d_\alpha$};
	\end{tikzpicture}\mkern-5mu
	\Biggr]\in \rmM_{\infty}(R_1'\llbracket P'_{J_1^c} \rrbracket)$ be the infinite matrix as in Proposition~\ref{P:filtration on the space of generalized automoprhic forms}. 
	
	We apply the homomorphisms $R'_1\llbracket P'_{J_1^c}\rrbracket\rightarrow R_1'\rightarrow \FF_p\llbracket T_1\rrbracket$ to the matrix $A_\alpha\in \rmM_{d_\alpha}(R_1'\llbracket P'_{J_1^c}\rrbracket)$, and obtain matrices $A_{\alpha,R_1'}\rmM_{d_\alpha}(R_1')$ and $\bar{A}_{\alpha}\in \rmM_{d_\alpha}(\FF_p\llbracket T_1\rrbracket)$. We use $f(X)=\det(X\cdot I_{d_{\alpha}}-A_{\alpha,R_1'})=X^{d_\alpha}+a_{d_\alpha-1}X^{d_\alpha-1}+\cdot+a_0\in R_1'[X]$ to denote the characteristic polynomial of $A_{\alpha,R_1'}$. Since the matrix $\bar{A}_\alpha$ is strictly $\underline{\lambda}^{[d_\alpha]}$-Hodge bounded, we can write $a_0=T_1^{n_\alpha}\cdot a_0'$, where $n_\alpha=\sum\limits_{k=0}^{d_\alpha-1}\lambda_k$ and $a_0'$ is a unit in $R_1'$.  
	
	As $N$ is $\underline{\lambda}$-Hodge bounded and $\lim\limits_{k\rightarrow \infty}\lambda_k=\infty$, we can find an element $\beta\geq \alpha$ in $\Omega_{\omega_1}$ with the following property: if we write the matrix $N$ in the block uppertriangular form $N=\quad
	\Biggl[\mkern-5mu
	\begin{tikzpicture}[baseline=-.65ex]
	\matrix[
	matrix of math nodes,
	column sep=1ex,
	] (m)
	{
		A_\beta & B_\beta\\
		0\ \   & D_\beta \\
	};
	\draw[dotted]
	([xshift=0.5ex]m-1-1.north east) -- ([xshift=0.5ex]m-2-1.south east);
	\draw[dotted]
	(m-1-1.south west) -- (m-1-2.south east);
	\node[above,text depth=1pt] at ([xshift=2.5ex]m-1-1.north) {$\scriptstyle d_\beta$};  
	\node[left,overlay] at ([xshift=-1.2ex,yshift=-2ex]m-1-1.west) {$\scriptstyle d_\beta$};
	\end{tikzpicture}\mkern-5mu
	\Biggr]$, then $D_\beta\in \rmM_\infty(T_1^{n_\alpha+1}R_1'\llbracket P'_{J_1^c} \rrbracket)$. We write $f(D_\beta)=a_0I_\infty+a_1D_\beta+\dots+D_\beta^{d_\alpha}=T_1^{n_\alpha}(a_0'I_\infty+(T_1^{-n_\alpha}D_\beta)(a_1+\dots+D_\beta^{d_\alpha-1}))$. Since $T_1^{-n_\alpha}D_\beta\in \rmM_\infty( T_1 R'_1\llbracket P'_{J_1^c}\rrbracket)$ and  $a_0'$  is a unit in $R_1'$,  the matrix $D_\beta'=a_0'I_\infty+(T_1^{-n_\alpha}D_\beta)(a_1+\dots+D_\beta^{d_\alpha-1})$ belongs to $\GL_\infty(R'_1\llbracket P'_{J_1^c}\rrbracket)$ by Lemma~\ref{L:any lift of an invertible matrix is invertible}. 
	
	By Proposition~\ref{P:filtration on the space of generalized automoprhic forms}, the matrix $A_\beta\in \rmM_{d_\beta}(R_1'\llbracket P'_{J_1^c}\rrbracket)$ is of the block uppertriangular form $\quad
	\Biggl[\mkern-5mu
	\begin{tikzpicture}[baseline=-.65ex]
	\matrix[
	matrix of math nodes,
	column sep=1ex,
	] (m)
	{
		A_{11} & A_{12} \\
		\ 0 \ \    & A_{22} \\
	};
	\draw[dotted]
	([xshift=0.5ex]m-1-1.north east) -- ([xshift=0.5ex]m-2-1.south east);
	\draw[dotted]
	(m-1-1.south west) -- (m-1-2.south east);
	\node[above,text depth=1pt] at ([xshift=2.5ex]m-1-1.north) {$\scriptstyle d_\alpha$};  
	\node[left,overlay] at ([xshift=-1.2ex,yshift=-2ex]m-1-1.west) {$\scriptstyle d_\alpha$};
	\end{tikzpicture}\mkern-5mu
	\Biggr]$ with $A_{11}=A_\alpha$. 
	
	Denote $d=d_\beta-d_\alpha$ and $g(X)=\det(X\cdot I_d-A_{22,R_1'})\in R_1'[X]$ for the characteristic polynomial of $A_{22,R_1'}$. It follows from Proposition~\ref{P:filtration on the space of generalized automoprhic forms} that the slopes of the Newton polygon of $\bar{A}_{11}\in \rmM_{d_{\alpha}}(\FF_p\llbracket T_1 \rrbracket)$ are strictly less than the slopes of the Newton polygon of $\bar{A}_{22}\in \rmM_{d'}(\FF_p\llbracket T_1 \rrbracket)$ (under the $T_1$-adic valuation). We use $\bar{f}(X)$ (resp. $\bar{g}(X)$) to denote the image of $f(X)$ (resp. $g(X)$) under the homomorphism $R_1'\rightarrow \FF_p\llbracket T_1 \rrbracket$. Then $\bar{f}(X)$ and $\bar{g}(X)$ are coprime in $\FF_p(\!(T_1)\!)[X]$. By \cite[Chapter~1 \S4]{milne2016etale}, $f$ and $g$ are strictly coprime in $R_1'[\frac{1}{T_1}][X]$. So we can find $u(X)$,$v(X)\in R_1'[X]$ and $m_\alpha\in\ZZ_{\geq 0}$, such that $$f(X)u(X)+g(X)v(X)=T_1^{m_\alpha}$$ in $R_1'[X].$ By Hamilton-Cayley's Theorem (over the commutative ring $R_1'$), we have $f(A_{11,R_1'})=0$ and $g(A_{22,R_1'})=0$. Hence we have $f(A_{22,R_1'})u(A_{22,R_1'})=T_1^{m_\alpha}I_{d}$ in $\rmM_d(R_1')$.
	
	We use $\calI_1\subset R_1'\llbracket P'_{J_1^c}\rrbracket$ and $\calI_2\subset R_2'\llbracket P'_{J_1^c}\rrbracket$ to denote the augmentation ideal of the corresponding complete group algebras, respectively. In summary of the above discussion, the matrix $N\in \rmM_\infty(R_1'\llbracket P'_{J_1^c}\rrbracket)$ can be written in the following block uppertriangular form:
	\[
	N=
	\quad\Biggl[\mkern-5mu
	\begin{tikzpicture}[baseline=-.65ex]
	\matrix[
	matrix of math nodes,
	column sep=1ex,
	] (m)
	{
		A_{11} & \ \ast\  & \ast\\
		0 & A_{22} & \ast \\
		0 & \ 0\ \  & D_\beta\\
	};
	\draw[dotted]
	([xshift=0.5ex]m-1-1.north east) -- ([xshift=2ex]m-3-1.south east);
	\draw[dotted]
	([xshift=0.5ex]m-1-2.north east) -- ([xshift=0.5ex]m-3-2.south east);
	\draw[dotted]
	(m-1-1.south west) -- (m-1-3.south east);
	\draw[dotted]
	(m-2-1.south west) -- (m-2-3.south east);
	\node[above,text depth=1pt] at ([xshift=3.5ex]m-1-1.north) {$\scriptstyle d_\alpha$};  
	\node[above,text depth=1pt] at ([xshift=3ex]m-1-2.north) {$\scriptstyle d_\beta$};
	\node[left,overlay] at ([xshift=-1.2ex,yshift=-2ex]m-1-1.west) {$\scriptstyle d_\alpha$};
	\node[left,overlay] at ([xshift=-3ex,yshift=-2ex]m-2-1.west) {$\scriptstyle d_\beta$};
	\end{tikzpicture}\mkern-5mu
	\Biggr],
	\]
	such that the matrix $f(N)=\left[ \begin{array}{ccc}
	f(A_{11})& \ast &\ast\\
	0&f(A_{22})&\ast\\
	0&0&f(D_\beta)\end{array} \right]$ satisfies the following properties:
	\begin{itemize}
		\item $f(A_{11})\in \rmM_{d_\alpha}(\calI_1)$，
		\item $f(A_{22})u(A_{22})-T_1^{m_\alpha}\in \rmM_d(\calI_1)$，
		\item $f(D_\beta)=T_1^{n_\alpha}D'_\beta$, for some $D'_\beta\in\GL_\infty(R_1'\llbracket P'_{J_1^c}\rrbracket)$.
	\end{itemize}

	Under the basis $\{v_m|m\in\NN \}$, we have an $R_1'\llbracket P'_{J_1^c}\rrbracket$-linear isomorphism $S_{\kappa_1,J_1}^{D,J_1}(K^p,R_1')^\vee\cong \prod\limits_{m=1}^\infty R_1'\llbracket P'_{J_1^c}\rrbracket$. 
	
	For any positive integer $k$, we denote by $\gothR_k$ the quotient $R_1'\llbracket P'_{J_1^c}\rrbracket/\calI_1^k$. The above isomorphism induces an isomorphism $$S_{\kappa_1,J_1}^{D,J_1}(K^p,R_1')^\vee\otimes_{R_1'\llbracket P'_{J_1^c}\rrbracket} \gothR_k \cong \prod\limits_{m=1}^\infty \gothR_k.$$ Since the $U_{\pi_1}$-operator on $S_{\kappa_1,J_1}^{D,J_1}(K^p,R_1')^\vee$ is $R_1'\llbracket P'_{J_1^c}\rrbracket$-linear, it induces an $\gothR_k$-linear map on $S_{\kappa_1,J_1}^{D,J_1}(K^p,R_1')^\vee\otimes_{R_1'\llbracket P'_{J_1^c}\rrbracket} \gothR_k$. Under the above isomorphism, this operator corresponds to the infinite matrix $N_{\gothR_k}$, and hence $f(U_{\pi_1})$ corresponds to the matrix $f(N)_{\gothR_k}$. From the above discussion, we have the following facts for $f(N)_{\gothR_k}$.
	\begin{itemize}
		\item The submatrix $f(A_{11})_{\gothR_k}$ is nilpotent as $f(A_{11})\in \rmM_{d_\alpha}(\calI_1)$; more precisely, we have $f(A_{11})_{\gothR_k}^k=0$.
		\item There exists a matrix $B_k\in \rmM_d(\gothR_k)$, such that $f(A_{22})_{\gothR_k}\cdot B_k=B_k\cdot f(A_{22})_{\gothR_k}=T_1^{m_\alpha k}\cdot I_d$; in fact, if we write $f(A_{22})u(A_{22})=T_1^{m_\alpha}\cdot I_d+E$, with $E\in \rmM_d(\calI_1)$, we can choose the matrix $B_k$ to be $u(A_{22})\left(\sum\limits_{l=0}^{k-1} T_1^{m_\alpha l}E^{k-1-l}\right)$ (viewed as a matrix in $\rmM_d(\gothR_k)$ in the obvious way).
		\item $f(D_\beta)_{\gothR_k}=T_1^{n_\alpha}D'_{\beta,\gothR_k}$ with $D'_{\beta,\gothR_k}\in \GL_\infty(\gothR_k)$.
	\end{itemize}
	So we have 
	\[
	\tilde{V}_\alpha/\calI_1^k=\varinjlim\limits_{m} \ker\left(f(U_{\pi_1})^m: S_{\kappa,J_1}^{D,J_1}(K^p,R_1')^\vee\otimes_{R_1'\llbracket P'_{J_1^c}\rrbracket} \gothR_k\rightarrow S_{\kappa,J_1}^{D,J_1}(K^p,R_1')^\vee\otimes_{R_1'\llbracket P'_{J_1^c}\rrbracket} \gothR_k\right).
	\]
	Since we have a surjection $S_{\kappa,J_1}^{D,J_1}(K^p,R_2')^\vee\rightarrow S_{\kappa,J_2}^{D,J_1}(K^p,R_2')^\vee$, and $T_1$ is not a zero divisor in $R_2'\llbracket P'_{J_1^c}\rrbracket/\calI_2^k$, for all $k>0$, we have 
	\[
	V_\alpha/\calI_2^k=\varinjlim\limits_{m} \ker(f(U_{\pi_1})^m: S_{\kappa,J_2}^{D,J_1}(K^p,R_2')^\vee/\calI_2^k\rightarrow S_{\kappa,J_2}^{D,J_1}(K^p,R_2')^\vee/\calI_2^k.
	\]
	Since the operator $U_{\pi_2}$ commutes with $U_{\pi_1}$, it also commutes with $f(U_{\pi_1})$. So it stabilizes $V_\alpha/\calI_2^k$
	for all $k\geq 0$. Since $V_\alpha=\varprojlim\limits_k V_\alpha/\calI_2^k$, we see that $U_{\pi_2}$ stabilizes $V_\alpha$.  
	\end{proof}

\begin{remark}
	The proof of Proposition~\ref{P:filtration with respect to one Up operator is stable under the other Up operators} is a little lengthy and technical. So it would be helpful to explain the intuition of our argument. We fix a continuous homomorphism $\chi:R_{2,\omega_{2}}\rightarrow \CC_p$. We will show in \S5.3 below that $V_\alpha\otimes_{R_{2,\omega_2},\chi}\CC_p$ is the subspace of $S_{\kappa,J_2}^{D,J_1}(K^p,R_{2,\omega_{2}})^\vee\otimes_{R_{2,\omega_{2}},\chi}\CC_p$ on which the slopes of the $U_{\pi_1}$-operator belong to the interval $[0,\lambda_l\cdot v_p(\chi(T_1)))$ (resp. $[0,\lambda_l\cdot v_p(\chi(T_1))]$) if $\alpha=n_l^-$ (resp. $\alpha=n_l^+$). Proposition~\ref{P:filtration with respect to one Up operator is stable under the other Up operators} follows essentially from this characterization of the spaces $V_\alpha$'s and the fact that the operators $U_{\pi_1}$ and $U_{\pi_2}$ commutes with each other. 
\end{remark}
\subsection{A filtration on the space of $p$-adic automorphic forms with respect to all $U_{\pi_i}$-operators}\label{subsection:A filtration on the space of p-adic automorphic forms with respect to all Upi-operators}

In this section, we fix a character $\eta_{2}:\Delta_{2}\rightarrow \ZZ_p^\times$. Together with $\omega_{1}:H_{1}\rightarrow \ZZ_p^\times$, it gives a character $\omega_{2}:H_{2}\rightarrow \ZZ_p^\times$. Denote $R_{2,\omega_{2}}=R_2\otimes_{\ZZ_p[H_{2}],\omega_{2}}\ZZ_p$ and $V_{\alpha,\omega_{2}}=V_\alpha\otimes_{\ZZ_p[H_{2}],\omega_{2}}\ZZ_p$ for all $\alpha\in \Omega_{\omega_{1}}$. The set $\{V_{\alpha,\omega_{2}}|\alpha\in \Omega_{\omega_{1}} \}$ is a filtration of $S_{\kappa,J_2}^{D,J_1}(K^p,R_{2,\omega_{2}})^\vee$, and by Proposition~\ref{P:filtration with respect to one Up operator is stable under the other Up operators}, this filtration is stable under the $U_{\pi_2}$-operator. This filtration induces a filtration $\{V^{mod}_{\alpha,\omega_{2}}|\alpha\in \Omega_{\omega_{1}} \}$ of $S_{\kappa,J_2}^{D,J_2}(K^p, R_{2,\omega_{2}})^
\vee$. 
If $\gamma\leq \alpha$ are two consecutive elements in $\Omega_{\omega_{1}}$, we set $W_{\alpha,\Omega_{\omega_{2}}}:=V_{\alpha,\omega_{2}}/V_{\gamma,\omega_{2}}$ and $W^{mod}_{\alpha,\Omega_{\omega_{2}}}:=V^{mod}_{\alpha,\omega_{2}}/V^{mod}_{\gamma,\omega_{2}}$. Hence， $W_{\alpha,\Omega_{\omega_{2}}}$ also carries an action of the $U_{\pi_2}$-operator. Under the basis $\{v_m|m=1,\dots, d_\alpha \}$ of $\tilde{V}_\alpha$, we have an isomorphism $\tilde{V}_\alpha\cong \prod\limits_{k=1}^{d_\alpha}R_{1,\omega_{1}}\llbracket P'_{J_1^c}\rrbracket$, which induces an isomorphism $V_{\alpha,\omega_{2}}\cong \prod\limits_{k=1}^{d_\alpha}R_{2,\omega_{2}}\llbracket \calO_{\gothp_2}\times P'_{J_2^c}\rrbracket$. So we obtain an isomorphism $W_{\alpha,\omega_{2}}\cong \prod\limits_{k=1}^{r_\alpha}R_{2,\omega_{2}}\llbracket \calO_{\gothp_2}\times P'_{J_2^c}\rrbracket$, where $r_\alpha=d_\alpha-d_\gamma$ under the above notations. We also have isomorphisms $V_{\alpha,\omega_{2}}^{mod}\cong \prod\limits_{k=1}^{d_\alpha}R_{2,\omega_{2}}\llbracket \calO_{\gothp_2}\times P'_{J_2^c}\rrbracket^{j_2-mod}$ and $W_{\alpha,\omega_{2}}^{mod}\cong \prod\limits_{k=1}^{r_\alpha}R_{2,\omega_{2}}\llbracket \calO_{\gothp_2}\times P'_{J_2^c}\rrbracket^{j_2-mod}$. Since $V_{\alpha,\omega_{2}}$ is stable under the $U_{\pi_1}$ and $U_{\pi_2}$-operators, a similar computation in \S\ref{section:continuous functions and distribution algebras} shows that $V_{\alpha,\omega_{2}}^{mod}$ is also stable under the $U_{\pi_1}$ and $U_{\pi_2}$-operators. Hence we have well defined $U_{\pi_1}$ and $U_{\pi_2}$-operators on $W_{\alpha,\omega_{2}}^{mod}$ for all $\alpha$.

Fix $\alpha\in \Omega_{\omega_{1}}$ in the following discussion. We will construct a filtration of the graded piece $W_{\alpha,\omega_{2}}^{mod}$ adapted to the $U_{\pi_2}$-operator based on a similar idea we used in \S\ref{subsection:a filtration on the space of p-adic automorphic forms with respect to a Up operator}.

Under the above isomorphisms, we choose a basis $\{f_{k,m}|1\leq k\leq r_\alpha, m\in \ZZ_{\geq 0} \}$ of $W_{\alpha,\omega_{2}}$ as a left $R_{2,\omega_{2}}\llbracket P'_{J_2^c} \rrbracket$-module, such that $\{f_{k,m}|m\in \ZZ_{\geq 0} \}$ is the Mahler basis of the $k$-th direct factor of $\prod\limits_{k=1}^{r_\alpha}R_{2,\omega_{2}}\llbracket \calO_{\gothp_2}\times P'_{J_2^c}\rrbracket$ for all $1\leq k\leq r_\alpha$. We choose a basis $\{f_{k,m}^{mod}|1\leq k\leq r_\alpha,m\in \ZZ_{\geq 0} \}$ of $W_{\alpha,\omega_{2}}^{mod}$ in a similar way. 

We use $N=(N_{m,n})_{m,n\geq 0}$ (resp. $M=(M_{m,n})_{m,n\geq 0}$) to denote the matrix in $\rmM_\infty(R_{2,\omega_{2}}\llbracket P'_{J_2^c}\rrbracket)$ which corresponds to the $U_{\pi_2}$-operator on $W_{\alpha,\omega_{2}}$ (resp. $W_{\alpha,\omega_{2}}^{mod}$) under the basis we choose above. It follows from Propositions~\ref{P:explicit expression of Up operators on integral model of p-adic automorphic forms} and ~\ref{P:Liu-Wan-Xiao's computation on matrix coefficients of Mahler basis} that $N_{m,n}\in (T_{2})^{\max\{\lfloor \frac{n}{r_\alpha}\rfloor-\lfloor \frac{m}{pr_\alpha}\rfloor,0 \}}$ for all $m,n\in \ZZ_{\geq 0}$. Notice that $M$ is the conjugation of $N$ by the diagonal matrix with diagonal entries $\underbrace{1,1,\dots,1}_{r_\alpha},\underbrace{T_{2},T_{2},\dots,T_{2}}_{r_\alpha}, \underbrace{T^2_{2},T^2_{2},\dots,T^2_{2}}_{r_\alpha},\dots$ Define two sequence of integers $\underline{\lambda}, \underline{\mu}$ as $\lambda_n=\lfloor\frac{n}{r_\alpha}\rfloor-\lfloor \frac{n}{pr_\alpha}\rfloor$, $\mu_0=0$, $\mu_{n+1}-\mu_n=\lambda_n$ for all $n\in \ZZ_{\geq 0}$. Then the matrix $M$ is $\underline{\lambda}$-Hodge bounded with respect to $T_{j_2}\in R_{2,\omega_{J_2}}\llbracket P'_{J_2^c}\rrbracket$.

Let $R_{2,\omega_{2}}\llbracket P'_{J_2^c}\rrbracket\rightarrow R_{2,\omega_{2}}$ be the reduction map modulo the augmentation ideal of the group ring $R_{2,\omega_{J_2}}\llbracket P'_{J_2^c}\rrbracket$. Under the isomorphism $R_{2,\omega_{2}}\cong \ZZ_p\llbracket T_{1},T_{2},\frac{p}{T_{1}},\frac{p}{T_{2}},T\rrbracket$, let $R_{2,\omega_{2}}\rightarrow \FF_p\llbracket T_{2}\rrbracket$ be the reduction map modulo the ideal generated by $T_{1},\frac{p}{T_{1}}, \frac{p}{T_{2}}$ and $T$. Applying the above two homomorphisms to the entries of the matrix $M$, we obtain two matrices $M_{R_{2,\omega_{2}}}\in \rmM_\infty(R_{2,\omega_{2}})$ and $\bar{M}\in \rmM_\infty(\FF_p\llbracket T_2\rrbracket)$. Both of them are $\underline{\lambda}$-Hodge bounded with respect to $T_2$.

Fix an integer $l_1\in 2\ZZ_{\geq 0}$ and a finite character $\varepsilon_1:1+\pi_1\calO_{\gothp_1}\rightarrow \CC_p^\times$ with the following properties:
\begin{itemize}
	\item $l_1+1>\frac{\lambda_l}{p-1}$, where $\alpha=n_l^-$ or $n_l^+$; and 
	\item $\varepsilon_1$ is nontrivial and factors through $(1+\pi_1\calO_{\gothp_1})/(1+\pi^2_1\calO_{\gothp_1})$.
\end{itemize}

For every character $\omega_{J_2^c}:\Delta_{J_2^c}\rightarrow \ZZ_p^\times$, we obtain a character $\omega=(\omega_{2},\omega_{J_2^c}):H\rightarrow \ZZ_p^\times$. 
 
Given the above datum, for any integer $l_2\in 2\ZZ_{\geq 0}$ and any nontrivial character $\varepsilon_2:1+\pi_2\calO_{\gothp_2}\rightarrow \CC_p^\times$ that factors through $1+\pi^2_2\calO_{\gothp_2}$, we construct a point $\chi_2\in\calW(\CC_p)$ such that its associated character $\kappa_{l_2}:\calO_p^\times\times \calO_p^\times\rightarrow \CC_p^\times$ is locally algebraic and corresponds to the triple $(n\in\ZZ_{\geq 0}^I,\nu\in \ZZ^I,\psi=(\psi_1,\psi_2))$ defined as follows.
\begin{itemize}
	\item Define $\nu:=(\nu_i)_{i\in I}$ by \begin{equation*}
	\nu_i=
	\begin{cases}
	-l_1/2,  &\textrm{for~} i=j_1, \\
	-l_2/2,  &\textrm{for~} i=j_2, \\
	0,  & \text{otherwise.}
	\end{cases}.
	\end{equation*}
	and  $n:=-2\nu\in \ZZ_{\geq 0}^I$.
	\item Let $\psi_1,\psi_2:\calO_p^\times \rightarrow \CC_p^\times$ be two finite characters with the following properties: $\psi_1|_{1+\pi_i\calO_{\gothp_i}}$ and $\psi_2|_{1+\pi_i\calO_{\gothp_i}}$ are trivial for all $i\neq j_1,j_2$, $\psi_2|_{1+\pi_1\calO_{\gothp_1}}=\varepsilon_{1}$, $\psi_2|_{1+\pi_2\calO_{\gothp_2}}=\varepsilon_{2}$, and $\psi_1|_{1+\pi_1\calO_{\gothp_1}}=\varepsilon_{1}^{-2}$, $\psi_1|_{1+\pi_2\calO_{\gothp_2}}=\varepsilon_{2}^{-2}$; the characters $\psi_1|_{\Delta_p}$ and $\psi_2|_{\Delta_p}$ are determined by the condition that the point $\chi_{l_2}$ belongs to the component $\calW_\omega$ of $\calW$.
\end{itemize}
Under the above construction, we have $T_{i,\chi_{l_2}}=0$ for all $i\neq j_1,j_2$ and $v_p(T_{i,\chi_{l_2}})=\frac{1}{p-1}$ for $i=j_1,j_2$.

\begin{remark}
	The construction of the point $\chi_2\in\calW(\CC_p)$ depends on two even integers $l_1,l_2$ and two finite characters $\varepsilon_1,\varepsilon_2$. In the following paragraphs, we will construct a filtration of each graded piece $W_{\alpha,\omega_2}^{mod}$ adapted to the $U_{\pi_2}$-operator. Therefore the integer $l_2$ and the finite character $\varepsilon_2$ play similar roles as those of $l$ and $\psi_2$ appearing in the construction in \S$4.6$. On the other hand, since we work over the ring $R_{2,\omega_2}\cong \ZZ_p\llbracket T_1,\frac p{T_1},T_2,\frac p{T_2},T\rrbracket$, we do have a restriction on the $T_1$-parameter of the locally algebraic weight $\chi_2\in \calW(\CC_p)$, that is, $v_p(T_{1,\chi_2})\in (0,1)$. The additional assumption $l_1+1>\frac{\lambda_l}{p-1}$ is to guarantee that the space $V_{\alpha,\omega_2}^{mod}$ lands into the space of classical automorphic forms after specializing to the point $\chi_2\in\calW(\CC_p)$. As we will see in the argument below, the choices of $l_1$ and $\varepsilon_{1}$ do not affect the filtrations adapted to the $U_{\pi_2}$-operator as long as they satisfy the above assumptions.
\end{remark}

The point $\chi_{l_2}\in \calW(\CC_p)$ defines a homomorphism $\tau_{l_2}:R_{2,\omega_{2}}\rightarrow \CC_p$. We apply this homomorphism to entries of the matrix $M_{R_{2,\omega_{2}}}$ and get a matrix $M_{\tau_{l_2}}\in \rmM_\infty(\CC_p)$. Let $R_{2,\omega_{2}}\llbracket P'_{J_2^c}\rrbracket\rightarrow \CC_p$ be the composite of the homomorphisms $R_{2,\omega_{2}}\llbracket P'_{J_2^c}\rrbracket\rightarrow R_{2,\omega_{2}}$ and $\tau_{l_2}:R_{2,\omega_{2}}\rightarrow \CC_p$.  Denote $\CC_p\llbracket \calO_{\gothp_2}\rrbracket^{j_2-mod}=R_{2,\omega_{2}}\llbracket \calO_{\gothp_2}\rrbracket^{j_2-mod}\otimes_{R_{2,\omega_{2}},\tau_{l_2}}\CC_p$. From the explicit expression $V_{\alpha,\omega_{2}}^{mod}\cong \prod\limits_{k=1}^{d_\alpha} R_{2,\omega_{2}}\llbracket \calO_{\gothp_2}\times P'_{J_2^c}\rrbracket^{j_2-mod}$, we have an isomorphism $V_{\alpha,\omega_{2}}^{mod}\otimes_{R_{2,\omega_{2}}\llbracket P'_{J_2^c}\rrbracket}\CC_p\cong \prod\limits_{k=1}^{d_\alpha}\CC_p\llbracket \calO_{\gothp_2}\rrbracket^{j_2-mod}$.

Under the isomorphism $R_{2,\omega_{2}}\llbracket \calO_{\gothp_2}\rrbracket^{j_2-mod}\cong R_{2,\omega_{2}}\llbracket X_2'\rrbracket$, the space $\CC_p\llbracket \calO_{\gothp_2}\rrbracket^{j_2-mod}$ admits a quotient $\CC_p\llbracket X_2'\rrbracket^{\deg\leq l_2}$ consisting of polynomials of $X_2'$ of degree $\leq l_2$. This quotient is stable under the action of the monoid $\bbM_{\pi_2}$. 

Recall that $V_\alpha^{mod}$ is a $R_{2,\omega_{2}}\llbracket P'_{J_2^c}\rrbracket$-submodule of $S_{\kappa,J_2}^{D,J_2}(K^p,R_{2,\omega_{2}})^\vee$, and we have an $R_{2,\omega_{2}}\llbracket P'_{J_2^c}\rrbracket$-linear  isomorphism $V_\alpha^{mod}\cong \prod\limits_{k=1}^{d_\alpha}R_{2,\omega_{2}}\llbracket \calO_{\gothp_2}\times P'_{J_2^c}\rrbracket^{j_2-mod}$. We put $V_{\alpha,\CC_p}^{mod}:=V_{\alpha}^{mod}\otimes_{R_{2,\omega_{2}}\llbracket P'_{J_2^c}\rrbracket}\CC_p$ and obtain an isomorphism $V_{\alpha,\CC_p}^{mod}\cong \prod\limits_{k=1}^{d_\alpha}\CC_p\llbracket \calO_{\gothp_2}\rrbracket^{j_2-mod}$.

Given the locally algebraic weight $\kappa_{l_2}$ defined above, we define a subspace $\calC(\calO_{\gothp_2},\CC_p)^{cl}$ of $\calC(\calO_{\gothp_2},\CC_p)$ consisting of functions $f$ with the property that $f|_{a+\pi_2\calO_{\gothp_2}}$ is a polynomial function of degree less or equal to $l_2$, for all $a\in \calO_{\gothp_2}$. Then we have $\dim_{\CC_p}\calC(\calO_{\gothp_2},\CC_p)^{cl}=p(l_2+1)$. We use $\CC_p\llbracket \calO_{\gothp_2}\rrbracket^{cl}$ to denote the $\CC_p$-dual of $\calC(\calO_{\gothp_2},\CC_p)^{cl}$. Then $\CC_p\llbracket \calO_{\gothp_2}\rrbracket^{cl}$ is a quotient of $\CC_p\llbracket \calO_{\gothp_2}\rrbracket^{j_2-mod}$ and hence $V_{\alpha,\CC_p}^{mod}$ admits a quotient $V_{\alpha,\CC_p}^{cl}:=\prod\limits_{k=1}^{d_\alpha}\CC_p\llbracket \calO_{\gothp_2}\rrbracket^{cl}$. Moreover, this quotient also carries an action of the $U_{\pi_1}$ and $U_{\pi_2}$-operators induced from these operators on $V_{\alpha,\CC_p}^{mod}$. We define $W_{\alpha,\CC_p}^{cl}=V_{\alpha,\CC_p}^{cl}/V_{\gamma,\CC_p}^{cl}$, which is isomorphic to $\prod\limits_{k=1}^{r_\alpha}\CC_p\llbracket \calO_{\gothp_2}\rrbracket^{cl}$. 

We use $V_{\alpha,\CC_p}^{cl,\vee}$ to denote the $\CC_p$-dual of $V_{\alpha,\CC_p}^{cl}$. From the above construction, we see that there exists $r\in \calN^I$, such that $V_{\alpha,\CC_p}^{cl,\vee}$ lies in the following subspace
\[
\{\phi:D^\times\setminus D_f^\times/K^p\rightarrow \calC^{r,an}(\calO_{\gothp_1}\times\calO_{\gothp_2},\CC_p)|\phi(xu)=\phi(x)\circ u, \text{~for all~} x\in D_f^\times, u\in \Iw_{\pi^t} \}
\]
of $S_{\kappa_l,I}^D(K^p,\CC_p)$ with the additional properties:
\begin{itemize}
	\item The slopes of $U_{\pi_1}$-operator on $V_{\alpha,\CC_p}^{cl,\vee}$ is less or equal to $\lambda_l\cdot v_p(\chi_{l_2}(T_j))=\frac{\lambda_l}{p-1}< l_2+1$; and
	\item $\phi(x)|_{\{x_1\}\times \calO_{\gothp_2}}$ is a polynomial function of degree less or equal to $l_2$, for all $x_1\in \calO_{\gothp_1}$.
\end{itemize}

Fix a locally algebraic weight $\kappa_{l_2}$ as above. For $1\leq l\leq g$, we have defined an operator  $\theta_{i_l}:S_{\kappa_{l_2}}^D(K^p,r)\rightarrow S_{\kappa'_{l_2}}^D(K^p,r)$ in \S\ref{section:spaces of  classical automorphic forms}. It follows from the first property and the fact that the $U_{\pi_1}$-slopes on the space $S_{\kappa'_{l_2}}^D(K^p,r)$ are all nonnegative that $V_{\alpha,\CC_p}^{cl,\vee}\subset \ker(\theta_{i_1})$. It follows from the second property and the definition of $\theta_{i_l}$'s that $V_{\alpha,\CC_p}^{cl,\vee}\subset \ker(\theta_{i_l})$ for $l=2,\dots,g$. 
By Proposition~\ref{P:classicality of overconvergent automorphic forms on D}, we conclude that $V_{\alpha,\CC_p}^{cl,\vee}\subset \bigoplus\limits_{\omega_{J_2^c}:\Delta_{J_2^c}\rightarrow \ZZ_p^\times}S_{k,w}^D(K^p,\psi)$.

We make the same construction starting with another character $\omega_{2}'=(\eta_{j_1}^{-1},\eta_{j_2}^{-1},\eta):H_{2}\rightarrow\ZZ_p^\times$. We obtain $\underline{\lambda}$-Hodge bounded matrices $M'_{R_{2,\omega_{J_2}'}}\in \rmM_\infty(R_{2,\omega_{J_2}'})$, $\bar{M}'\in \rmM_\infty(\FF_p\llbracket T_j\rrbracket)$, and for every $l_2\in 2\ZZ_{\geq 0}$, we get a matrix $M'_{\tau_{l_2}}\in\rmM_\infty(\CC_p)$. We also have spaces $(V_{\alpha,\CC_p}^{mod})'$, $(V_{\alpha,\CC_p}^{cl})'$ and $(W_{\alpha,\CC_p}^{cl})'$. It follows from Proposition~\ref{P:pairing of Newton slopes by Atkin-Lehner} that the $U_{\pi_2}$-slopes on the spaces $W_{\alpha,\CC_p}^{cl}$ and $(W_{\alpha,\CC_p}^{cl})'$ can be paired such that the slopes in each pair sum to $l_2-1$. In other words, there exist $r_\alpha p(l_2+1)$ pairs of the slopes of the Newton polygons of the matrices $M_{\tau_{l_2}}$ and $M'_{\tau_{l_2}}$, such that the slopes in each pair sum to $l_2-1$. It follows from a similar argument in \S $4.6$ that the first $r_\alpha p(l_2+1)$ slopes of the matrix $M_{\tau_{l_2}}$ sum to $\frac{1}{2}r_\alpha p(l_2+1)$.

We use $\mathrm{char}(M_{R_{2,\omega_{2}}})=\sum\limits_{n\geq 0}c_nX^n\in R_{2,\omega_{2}}\llbracket X\rrbracket$ to denote the characteristic power series of the matrix $M_{R_{2,\omega_{2}}}$ with $c_n(\underline{T})\in R_{1,\omega_{1}}\llbracket T_2\rrbracket \subset R_{2,\omega_{2}}$ and $\underline{T}=(T_1,T_2,T)$. For all $l_2\in 2\ZZ_{\geq 0}$, the Newton polygon of $\sum\limits_{n\geq 0}c_n(\chi_{l_2}(\underline{T}))X^n$ passes through the point $(r_\alpha p(l_2+1), \mu_{r_\alpha p (l_2+1)}v_p(\chi_{l_2}(T_2)))$. Write $c_n(\underline{T})=\sum\limits_{m\geq 0}b_{n,m}T_2^{m}$ with $b_{n,m}\in R_{1,\omega_{1}}$. Recall that $R_{1,\omega_{1}}$ is isomorphic to $\ZZ_p\llbracket T_1,\frac{p}{T_1},T\rrbracket$. In particular, $R_{1,\omega_{1}}$ is a local ring and we use $\gothm_{1,\omega_{1}}$ to denote its maximal ideal. For $b(T_1,T)\in R_{1,\omega_{1}}$, the following statements are equivalent:
\begin{enumerate}
	\item $b(T_1,T)$ is a unit in $R_{1,\omega_{1}}$;
	\item $b(T_1,T)\notin \gothm_{1,\omega_{1}}$;
	\item there exist $t_1,t\in \CC_p$ with $|t_1|_p\in (\frac{1}{p},1)$ and $|t|_p<1$, such that $|b(t_1,t)|_p=1$, i.e. $b(t_1,t)$ is a $p$-adic unit.
\end{enumerate}

Now we can run the same argument in \S\ref{subsection:a filtration on the space of p-adic automorphic forms with respect to a Up operator} to the matrix $M\in \rmM_\infty(R_{2,\omega_{2}}\llbracket P'_{J_2^c}\rrbracket)$, $M_{R_{2,\omega_{2}}}\in \rmM_\infty(R_{2,\omega_{2}})$ and $\bar{M}\in \rmM_\infty(\FF_p\llbracket T_2\rrbracket)$. We obtain a set \[\Omega_{\omega_{2},\alpha_1}=\left\{n_{\omega_{2},\alpha_1,l}^-, n_{\omega_{2},\alpha_1,l}^+\Big| l\in \{ 0\}\bigcup 1+2\ZZ_{\geq 0} \right\},
\]
and a filtration $\{\tilde{V}_{\alpha_1,\alpha_2}|\alpha_2\in \Omega_{\omega_{2},\alpha_1} \}$ of the graded piece $W_{\alpha_1,\omega_{2}}^{mod}$ as a left $R_{2,\omega_{2}}\llbracket P'_{J_2^c}\rrbracket$-module.

Fix a character $\omega:H\rightarrow \ZZ_p^\times$. For $1\leq l\leq g$, we use $\omega_{l}:H_{l}\rightarrow \ZZ_p^\times$ to denote the restriction of $\omega$ to $H_{l}$.

Now we apply the above construction inductively to the places $j_3,\dots, j_g$. In summary, we get the following datum:
\begin{enumerate}
	\item There exists a set $\Omega_{\omega_1}=\{n_{\omega_{1},l}^-, n_{\omega_{1},l}^+|l\in \{ 0\}\cup 1+2\ZZ_{\geq 0} \}$ and a filtration $\{F_{\alpha_1}|\alpha_1\in \Omega_{\omega_1} \}$ of $S_{\kappa,I}^{D,I}(K^p,\Lambda_\omega^{>1/p})^\vee$ consisting of free $\Lambda_\omega^{>1/p}$-modules. We use $\{G_{\alpha_1}|\alpha_1\in \Omega_{\omega_1} \}$ to denote the graded pieces of this filtration, i.e. $G_{\alpha_1}=F_{\alpha_1}/F_{\gamma_1}$, if $\gamma_1\leq \alpha_1$ are two consecutive elements in $\Omega_{1}$ (we denote $G_{n_{\omega_{1},0}^-}=F_{n_{\omega_{1},0}^-}=(0)$).
	\item For any $\alpha_1\in \Omega_{\omega_{1}}$, there exist a set $\Omega_{\omega_{2},\alpha_1}=\{n_{\omega_{2},\alpha_1,l}^-, n_{\omega_{2},\alpha_1,l}^+|l\in \{ 0\}\cup 1+2\ZZ_{\geq 0} \}$ and a filtration $\{F_{\alpha_1,\alpha_2}|\alpha_2\in \Omega_{\omega_{2},\alpha_1} \}$ of $G_{\alpha_1}$ consisting of free $\Lambda_\omega^{>1/p}$-modules; define $G_{\alpha_1,\alpha_2}=F_{\alpha_1,\alpha_2}/F_{\alpha_1,\gamma_2}$ if $\gamma_2\leq \alpha_2$ are two consecutive elements in $\Omega_{\omega_{2},\alpha_1}$.
	\item For $2\leq k\leq g-1$, given the set $\Omega_{\omega_{k},\alpha_1,\dots,\alpha_{k-1}}$ and the space $G_{\alpha_1,\dots, \alpha_k}$, there exist a set $\Omega_{\omega_{k+1},\alpha_1,\dots,\alpha_k}=\{n_{\omega_{k+1},\alpha_1,\dots,\alpha_k,l}^-, n_{\omega_{k+1},\alpha_1,\dots,\alpha_k,l}^+|l\in \{ 0\}\cup 1+2\ZZ_{\geq 0}  \}$ and a filtration $\{F_{\alpha_1,\dots,\alpha_{k+1}}|\alpha_{k+1}\in \Omega_{\omega_{k+1},\alpha_1,\dots,\alpha_k} \}$ of $G_{\alpha_1,\dots,\alpha_k}$ consisting of free $\Lambda_\omega^{>1/p}$-modules; define $G_{\alpha_1,\dots,\alpha_{k+1}}=F_{\alpha_1,\dots,\alpha_{k+1}}/F_{\alpha_1,\dots,\gamma_{k+1}}$ if $\gamma_{k+1}\leq \alpha_{k+1}$ are two consecutive elements in $\Omega_{\omega_{k+1},\alpha_1,\dots,\alpha_k}$.
\end{enumerate}
Moreover, we have the following properties of the above filtrations:
\begin{enumerate}
	\item all the filtrations are stable under the $U_{\pi_i}$-operators, for all $i\in I$; and
	\item the graded pieces $G_{\alpha_1,\dots,\alpha_g}$ obtained above are free $\Lambda_\omega^{>1/p}$-modules of finite rank (possibly $0$).
\end{enumerate}

\subsection{Conclusion}

For any $x\in \calW^{>1/p}_\omega(\CC_p)$, we use $\chi_x:\Lambda_\omega^{>1/p}\rightarrow \CC_p$ to denote the corresponding homomorphism. For any free $\Lambda_\omega^{>1/p}$-module $G_{\alpha_1,\dots, \alpha_g}$ we obtained in the previous section, we denote $G_{\alpha_1,\dots,\alpha_g}(x)=G_{\alpha_1,\dots, \alpha_g}\otimes_{\Lambda_\omega^{>1/p},\chi_x}\CC_p$. 

\begin{proposition}\label{P:characterization of the filtration of the space of integral p-adic automorphic forms}
	For all $1\leq k\leq g$, if $\alpha_k=n_{\omega_{J_{k}},\alpha_1\dots,\alpha_{k-1},l}^-$ for some $l\in \{ 0\}\cup 1+2\ZZ_{\geq 0} $, the slopes of the $U_{\pi_k}$-operator on $G_{\alpha_1,\dots,\alpha_g}(x)$ are all equal to $(p-1)v_p(\chi_x(T_{i_k}))\cdot l$; if $\alpha_k=n_{\omega_{J_{k}},\alpha_1\dots,\alpha_{k-1},l}^+$ for some $l\in \{ 0\}\cup 1+2\ZZ_{\geq 0} $, the slopes of the $U_{\pi_k}$-operator on $G_{\alpha_1,\dots,\alpha_g}(x)$ all belong to the interval $(p-1)v_p(\chi_x(T_{i_k}))\cdot (l,l+2)$ if $l\neq 0$, and the interval $(p-1)v_p(\chi_x(T_{i_k}))\cdot (0,1)$ if $l=0$. 
\end{proposition}

\begin{proof}
	We give a proof for $l=1$. The argument for other $l$'s is similar. 
	
	Recall that by Proposition~\ref{P:filtration on the space of generalized automoprhic forms}, we have a filtration $\{\tilde{V}_{\alpha_1}|\alpha_1\in \Omega_{\omega_{1}} \}$ of the space $S_{\kappa_1,J_1}^{D,J_1}(K^p,R_{1,\omega_{1}})^\vee$. We have an isomorphism $\tilde{V}_{\alpha_1}\cong \prod\limits_{k=1}^{d_{\alpha_1}}R_{1,\omega_{1}}\llbracket P'_{J_1^c}\rrbracket$ of left $R_{1,\omega_{1}}\llbracket P'_{J_1^c}\rrbracket$-modules for all $\alpha_1\in \Omega_{\omega_{1}}$. These isomorphisms induces isomorphisms $\tilde{W}_{\alpha_1}:= \tilde{V}_{\gamma_1}/\tilde{V}_{\alpha_1}\cong \prod\limits_{k=1}^{r_{\alpha_1}}R_{1,\omega_{1}}\llbracket P'_{J_1^c}\rrbracket$. 
	
	Fix $\alpha_1\in \Omega_{\omega_{1}}$ and let $\gamma_1\in \Omega_{\omega_{1}}$ such that $\alpha_1\leq \gamma_1$ are two consecutive elements in $\Omega_{\omega_{1}}$. Under the above isomorphisms, the $U_{\pi_1}$-operator on $\tilde{W}_{\alpha_1}$ corresponds to a matrix $U_{\alpha_1}\in \rmM_{r_{\alpha_1}}(R_{1,\omega_{1}}\llbracket P'_{J_1^c}\rrbracket)$. This matrix is strictly $\lambda^{(\alpha_1,\gamma_1]}$-Hodge bounded with respect to the element $T_1\in R_{1,\omega_{1}}\llbracket P'_{J_1^c}\rrbracket$. 
	
	First we assume that $\alpha_1=n_{\omega_{1},l}^-$ for some $l\in \{0\}\cup 1+2\ZZ_{\geq 0}$. We can assume that $\alpha_1\neq \gamma_1$, or otherwise $G_{\alpha_1,\dots,\alpha_g}=0$ and the proposition is trivial in this case. From the construction in Proposition~\ref{P:filtration on the space of generalized automoprhic forms}, we see that $\lambda_{\alpha_1+1}=\lambda_{\gamma_1}=(p-1)l$. Hence there exists an invertible matrix $U'_{\alpha_1}\in \GL_{r_{\alpha_1}}(R_{1,\omega_{1}}\llbracket P'_{J_1^c}\rrbracket)$, such that $U_{\alpha_1}=T_1^{(p-1)l}U'_{\alpha_1}$. Under the isomorphism $\tilde{W}_{\alpha_1}\cong \prod\limits_{k=1}^{r_{\alpha_1}}R_{1,\omega_{J_1}}\llbracket P'_{J_1^c}\rrbracket$, the matrix $U'_{\alpha_1}$ defines an invertible $R_{1,\omega_{J_1}}\llbracket P'_{J_1^c}\rrbracket$-linear operator $U'_{\pi_1}$ on $\tilde{W}_{\alpha_1}$.  Recall that $V_{\alpha_1}$ is the image of $\tilde{V}\otimes_{R_{1,\omega_{1}}}R_{2,\omega_{2}}$ under the map $S_{\kappa,J_1}^{D,J_1}(K^p,R_{2,\omega_{2}})^\vee\rightarrow S_{\kappa,J_2}^{D,J_1}(K^p,R_{2,\omega_{2}})^\vee$. Hence we have isomorphisms $V_{\alpha_1}\cong \prod\limits_{k=1}^{d_{\alpha_1}}R_{2,\omega_{2}}\llbracket \calO_{\gothp_2}\times P'_{J_2^c}\rrbracket$ and $V^{mod}_{\alpha_1}\cong \prod\limits_{k=1}^{d_{\alpha_1}}R_{2,\omega_{2}}\llbracket \calO_{\gothp_2}\times P'_{J_2^c}\rrbracket^{j_2-mod}$. It induces isomorphisms $W_{\alpha_1}\cong \prod\limits_{k=1}^{r_{\alpha_1}}R_{2,\omega_{2}}\llbracket \calO_{\gothp_2}\times P'_{J_2^c}\rrbracket$ and $W^{mod}_{\alpha_1}\cong \prod\limits_{k=1}^{r_{\alpha_1}}R_{2,\omega_{2}}\llbracket \calO_{\gothp_2}\times P'_{J_2^c}\rrbracket^{j_2-mod}$. Under these isomorphisms, the matrix $U'_{\alpha_1}$ defines invertible operators $U'_{\pi_1}$ on $W_{\alpha_1}$ and $W_{\alpha_1}^{mod}$, but notice that these operators are only $R_{2,\omega_{2}}\llbracket  P'_{J_2^c}\rrbracket$-linear. Moreover, the operator $U_{\pi_1}'$ preserves the filtrations $\{W_{\alpha_1,\alpha_2}^{mod}|\alpha_2\in \Omega_{\omega_{2},\alpha_1} \}$ as the operator $U_{\pi_1}$ does. We apply the above argument inductively to the places $i_2,\dots, i_g$ and conclude that there is an operator $U_{\pi_1}'$ on the graded pieces $G_{\alpha_1,\dots, \alpha_g}$ such that $U_{\pi_1}'$ is invertible and $\Lambda_\omega^{>1/p}$-linear, and satisfies $U_{\pi_1}=T_1^{(p-1)l}U'_{\pi_1}$. 
	
	Notice that the homomorphism $\chi_x:\Lambda_\omega^{>1/p}\rightarrow \CC_p$ factors through $\calO_{\CC_p}$ and hence $G_{\alpha_1,\dots, \alpha_g}(x)$ admits a lattice $L_{\alpha_1,\dots,\alpha_g}(x)=G_{\alpha_1,\dots,\alpha_g}\otimes_{\Lambda_\omega^{>1/p},\chi_x}\calO_{\CC_p}$. Both the operators $U_{\pi_1}$ and $U_{\pi_1}'$ preserve this lattice. Since $U_{\pi_1}'$ is invertible on $L_{\alpha_1,\dots, \alpha_g}(x)$, the slopes of $U'_{\pi_1}$ are all $0$. So the slopes of $U_{\pi_1}$ on $G_{\alpha_1,\dots, \alpha_g}(x)$ are all equal to $(p-1)v_p(\chi_x((T_1)))\cdot l$. 
	
	Now we assume that $\alpha_1=n_{\omega_{1},l}^+$ for some $l\in \{0 \}\cup 1+2\ZZ_{\geq 0}$. As in the first case, we can find a matrix $U'_{\alpha_1}\in \rmM_{r_{\alpha_1}}(R_{1,\omega_{1}}\llbracket P'_{J_1^c}\rrbracket)$, such that $U_{\alpha_1}=T_1^{(p-1)l}U_{\alpha_1}'$, but this matrix is not invertible in general. It follows that there exists a $\Lambda_\omega^{>1/p}$-linear operator $U'_{\pi_1}$ on $G_{\alpha_1,\dots,\alpha_g}$, such that $U_{\pi_1}=T_1^{(p-1)l}U_{\pi_1}'$. Since $U'_{\pi_1}$ preserves the lattice $L_{\alpha_1,\dots,\alpha_g}(x)$, the slopes of $U_{\pi_1}'$ on $G_{\alpha_1,\dots,\alpha_g}(x)$ are all nonnegative. So the slopes of $U_{\pi_1}$ on $G_{\alpha_1,\dots,\alpha_g}(x)$ are all no less than $(p-1)v_p(\chi_x((T_1)))\cdot l$. 
	
	Recall that $R_{1,\omega_{1}}$ is isomorphic to $\ZZ_p\llbracket T_1,\frac{p}{T_1},T\rrbracket$. In particular, it is a local ring with residue field $\FF_p$. Hence we have a natural surjective map $R_{1,\omega_{1}}\rightarrow \FF_p$. We apply the homomorphisms $R_{1,\omega_{1}}\llbracket P'_{J_1^c}\rrbracket\rightarrow \FF_p\llbracket P'_{J_1^c}\rrbracket$ and $\FF_p\llbracket P'_{J_1^c}\rrbracket\rightarrow \FF_p$ to entries of the matrix $U'_{\alpha_1}\in \rmM_{r_{\alpha_1}}(R_{1,\omega_{1}}\llbracket P'_{J_1^c}\rrbracket)$, and get matrices $U'_{\alpha_1,\FF_p\llbracket P'_{J_1^c}\rrbracket}\in \rmM_{r_{\alpha_1}}(\FF_p\llbracket P'_{J_1^c}\rrbracket)$ and $U'_{\alpha_1,\FF_p}\in \rmM_{r_{\alpha_1}}(\FF_p)$. By Proposition~\ref{P:filtration on the space of generalized automoprhic forms}, the slopes of the matrix $\bar{U}_{\alpha_1}\in \rmM_{r_{\alpha_1}}(\FF_p\llbracket T_1\rrbracket)$ (with respect to the $T_1$-adic valuation) all belong to the interval $(\lambda_{\alpha_1},\lambda_{\gamma_1})=((p-1)l,(p-1)(l+2))$. So the slopes of $\bar{U}'_{\alpha_1}\in \rmM_{r_{\alpha_1}}(\FF_p\llbracket T_1\rrbracket)$ are all positive. If we use $f(X)\in \FF_p[X]$ to denote the characteristic polynomial of the matrix $U'_{\alpha_1,\FF_p}$, then we have $f(X)=X^{r_{\alpha_1}}$. 
	
	We use $\calI\subset \FF_p\llbracket P'_{J_1^c}\rrbracket$ to denote the augmentation ideal of this complete group ring. It follows from Hamilton-Cayley Theorem that $f(U'_{\alpha_1,\FF_p})=(U'_{\alpha_1,\FF_p})^{r_{\alpha_1}}=0$, and hence $(U'_{\alpha_1,\FF_p\llbracket P'_{J_1^c}\rrbracket})^{r_{\alpha_1}}\in \rmM_{r_{\alpha_1}}(\calI)$.
	
	Define $\bar{L}_{\alpha_1,\dots,\alpha_g}(x)=L_{\alpha_1,\dots,\alpha_g}(x)\otimes_{\calO_{\CC_p}}\bar{\FF}_p$. It is a finite $\bar{\FF}_p$-vector space and $\Iw_{\pi}$ acts on it continuously. It follows that the action of $\calI$ on $\bar{L}_{\alpha_1,\dots,\alpha_g}(x)$ is nilpotent and hence $U'_{\alpha_1,\FF_p\llbracket P'_{J_1^c}\rrbracket}$ is a nilpotent operator on $\bar{L}_{\alpha_1,\dots,\alpha_g}(x)$. So the slopes of the $U'_{\pi_1}$-operator on $G_{\alpha_1,\dots,\alpha_g}(x)$ are all positive, and we conclude that the slopes of the $U'_{\pi_1}$-operator on $G_{\alpha_1,\dots,\alpha_g}(x)$ are all strictly larger than $(p-1)v_p(\chi_x((T_1)))\cdot l$. 
	
	Since the matrix $U_{\alpha_1}$ is strictly $\lambda^{(\alpha_1,\gamma_1]}$-Hodge bounded, we can find a matrix $V_{\alpha_1}\in \rmM_{r_{\alpha_1}}(R_{1,\omega_{J_1}}\llbracket P'_{J_1^c}\rrbracket)$, such that $U_{\alpha_1}V_{\alpha_1}=V_{\alpha_1}U_{\alpha_1}=T_1^{(p-1)(l+2)}\cdot I_{r_{\alpha_1}}$.
	
	The matrix $V_{\alpha_1}$ induces a $\Lambda_\omega^{>1/p}$-linear operator $V_{\pi_1}$ on $G_{\alpha_1,\dots, \alpha_g}$ and $G_{\alpha_1,\dots, \alpha_g}(x)$. It preserves the lattice $L_{\alpha_1,\dots, \alpha_g}(x)$ and satisfies $U_{\pi_1}V_{\pi_1}=V_{\pi_1}U_{\pi_1}=T_1^{(p-1)(l+2)}\cdot \mathrm{Id}$. It follows from Lemma~\ref{L:a basic linear algebra lemma} that the slopes of $U_{\pi_1}$ and $V_{\pi_1}$ on $G_{\alpha_1,\dots, \alpha_g}(x)$ can be paired such that the slopes in each pair sum to $(p-1)v_p(\chi_x((T_1)))\cdot (l+2)$. We have seen that the slopes of $\bar{U}_{\alpha_1}\in \rmM_{r_{\alpha_1}}(\FF_p\llbracket T_1\rrbracket)$ all belong to the interval $((p-1)l,(p-1)(l+2))$. Hence the slopes of $\bar{V}_{\alpha_1}$ are all positive. By a similar argument as above, we can show that the slopes of $V_{\pi_1}$ on  $G_{\alpha_1,\dots, \alpha_g}(x)$ are all positive. It follows that the slopes of $U_{\pi_1}$ on $G_{\alpha_1,\dots, \alpha_g}(x)$ are all strictly less than $(p-1)v_p(\chi_x((T_1)))\cdot (l+2)$.
\end{proof}

Now we are ready to prove our main theorem~\ref{T:spectral halo for eigenvarieties for D}. Let $X=\mathrm{Sp}(A)\subset \calW_\omega^{>1/p}$ be an affinoid subdomain that corresponds to the continuous homomorphism $\chi:\Lambda_\omega^{>1/p}\rightarrow A$. For $l\in \{1,\dots, g \}$ and a set of indices $\{\alpha_k| k=1,\dots, l \}$ with $\alpha_1\in \Omega_{\omega_{1}}$, $\alpha_{k+1}\in \Omega_{\omega_{k+1},\alpha_1,\dots, \alpha_k}$ for $k=1,\dots, l-1$, we define $F_{\alpha_1,\dots,\alpha_l,A}:= F_{\alpha_1,\dots,\alpha_l}\hat{\otimes}_{\Lambda_\omega^{>1/p},\chi}A$ and $G_{\alpha_1,\dots,\alpha_l,A}:= G_{\alpha_1,\dots,\alpha_l}\hat{\otimes}_{\Lambda_\omega^{>1/p},\chi}A$. Let $S_A:= S_{\kappa,I}^{D,I}(K^p,\Lambda_\omega^{>1/p})\hat{\otimes}_{\Lambda_\omega^{>1/p},\chi}A$ and $\psi_A:\bbT\rightarrow \End_A(S_A)$ be the natural homomorphism, where $\bbT$ is the (abstract) $\QQ_p$-Hecke algebra defined in \S\ref{subsection:the spectral varieties and eigenvarieties}. Let $\bbT_A$ be the image of $\psi_A$. Under the above notations and by our previous construction in \S\ref{subsection:A filtration on the space of p-adic automorphic forms with respect to all Upi-operators}, we obtain a filtration $\{F_{\alpha_1}|\alpha_1\in \Omega_{\omega_{1}}\}$ of the $A$-Banach module $S_A$ and the Hecke operators on $S_A$ all preserve this filtration. It follows from Proposition~\ref{P:characterization of the filtration of the space of integral p-adic automorphic forms} that the induced homomorphism $\bbT_A\rightarrow \prod\limits_{\alpha_1\in \Omega_{\omega_{1}}}\End_A(G_{\alpha_1})$ is injective. We inductively run the above argument, and obtain an injective homomorphism $\bbT_A\rightarrow \prod\limits_{\alpha_1,\dots,\alpha_g}\End_A(G_{\alpha_1,\dots,\alpha_g})$. The main theorem \ref{T:spectral halo for eigenvarieties for D} follows from the injectivity of this map, the construction of the eigenvariety $\calX_D$, and Proposition~\ref{P:characterization of the filtration of the space of integral p-adic automorphic forms}.

\section{Application to Hilbert modular eigenvarieties}\label{section:Application to Hilbert modular eigenvarieties}

\subsection{p-adic Langlands Functoriality}

\begin{definition}\label{D:eigenvariety datum}
	An eigenvariety datum is a tuple $\gothD=(\calW,\calZ,\calM,\bbT,\psi)$, where $\calW$ is a separated reduced equidimensional, relatively factorial (see \cite[\S4.1]{hansen2017universal} for the precise definition) rigid analytic space, $\calZ\subset \calW\times\AAA^1$ is a Fredholm hypersurface, $\calM$ is a coherent analytic sheaf over $\calZ$, $\bbT$ is a commutative $\QQ_p$-algebra and $\psi$ is a $\QQ_p$-algebra homomorphism $\psi:\bbT\rightarrow \End_{\calO_\calZ}(\calM)$.
\end{definition}

\begin{remark}
	In \S\ref{subsection:the spectral varieties and eigenvarieties}, we construct a quasi-coherent sheaf $\calM$ on $\calW$ coming from the spaces of of overconvergent automorphic forms for $D$. We can also construct a coherent sheaf $\calM^\ast$ over the spectral variety $\calZ_D$ coming from the spaces $S_{\kappa}^D(K^p\Iw_{\pi},r)$'s. The construction need a special admissible cover of $\calZ_D$ which consists of slope adapted affinoids of $\calZ_D$. We refer to \cite[Chapter $4$]{buzzard320eigenvarieties}  and \cite[$\S4$]{hansen2017universal}  for the details of the construction. 
\end{remark}

Given an eigenvariety datum $\gothD$ as above, we use $\calX=\calX(\gothD)$ to denote the eigenvariety associated to $\gothD$, and we have a finite morphism $\pi:\calX\rightarrow \calZ$ and a morphism $w:\calX\rightarrow \calW$. Recall that in \cite[Definition~$5.1.1$]{hansen2017universal} , the core $\calX^\circ$ of $\calW$ is defined to be the union of $\dim\calW$-dimensional irreducible components of the nilreduction $\calX^{\mathrm{red}}$, and $\calX$ is unmixed if $\calX^\circ\cong \calX$.

\begin{remark}
	The eigenvarieties we will consider in this section are the eigenvarieties $\calX_D$ constructed in \S\ref{subsection:the spectral varieties and eigenvarieties} and Hilbert modular eigenvarieties $\calX_{\GL_{2/F}}$ constructed in \cite[$\S5$]{andreatta2016arithmetique}. These eigenvarieties are unmixed by \cite[Theorem~2]{birkbeck2016jacquet}. Hence we assume that all the eigenvarieties are unmixed in the rest of this paper.
\end{remark}
Given a point $z\in \calZ$ and any $T\in \bbT$, we write $D(T,X)(z)\in k(z)[X]$ for the characteristic polynomial $\det(1-\psi(T)X)|_{\calM(z)}$. 

Now we can state Hansen's interpolation theorem (\cite[Theorem~$5.1.6$]{hansen2017universal}), which is the main tool to translate our results to Hilbert modular eigenvarieties.

\begin{theorem}\label{T:p-adic interpolation theorem}
	Given two eigenvariety data $\gothD_i=(\calW_i,\calZ_i,\calM_i,\bbT_i,\psi_i)$, let $\calX_i=\calX(\gothD_i)$ be the associated eigenvariety for $i=1,2$. Suppose that we are given the following additional data:
	\begin{enumerate}
		\item a closed immersion $\jmath:\calW_1\rightarrow \calW_2$ and we use $j$ to denote the closed immersion $\jmath\times \mathrm{id}:\calW_1\times \AAA^1\rightarrow \calW_2\times\AAA^1$;
		\item a  homomorphism of $\QQ_p$-algebras $\sigma:\bbT_2\rightarrow \bbT_1$;
		\item a very Zariski dense set $\calZ_1^{\mathrm{cl}}\subset \calZ_1$ with $j(\calZ_1^{\mathrm{cl}})\subset \calZ_2$, such that $D(\sigma(T),X)(z)$ divides $D(T,X)(j(z))$ in $k(z)[X]$ for all $z\in \calZ_1^{\mathrm{cl}}$ and all $T\in \bbT_2$.
	\end{enumerate}
	Then there exists a morphism $i:\calX_1\rightarrow \calX_2$ with the commutative diagrams
	$$
	\xymatrix{
		\calX_1 \ar[r]^{i} \ar[d]^{w_1} & \calX_2 \ar[d]^{w_2} \\
		\calW_1 \ar[r]^{\jmath} & \calW_2
	} ,
	\xymatrix{
		\calO(\calX_2) \ar[r]^{i^\ast}  & \calO(\calX_1)  \\
		\bbT_2 \ar[u]^{\phi_2} \ar[r]^{\sigma} & \bbT_1 \ar[u]^{\phi_1}
	}.
	$$
	Moreover, $i$ is a composite of a finite morphism followed by a closed immersion.
\end{theorem}

\begin{remark}
	When $\calW_1=\calW_2=\calW$, $\bbT_1=\bbT_2=\bbT$ and $\jmath:\calW_1\rightarrow \calW_2$ and $\sigma:\bbT_2\rightarrow\bbT_1$ are the identity maps, the map $i:\calX_1\rightarrow \calX_2$ in the above theorem is a closed immersion.
\end{remark}

\subsection{Application to Hilbert modular eigenvarieties}

Recall that $F$ is a totally real field in which $p$ splits and $D$ is a totally definite quaternion algebra over $F$ with discriminant $\gothd$. We assume that $(p,\gothd)=1$. 

Fix a prime ideal $\gothn$ of $F$p prime to $\gothd p$. Set
\begin{equation*}
U_1(\gothn):=\left \{ \gamma\in\GL_2(\calO_F\otimes\hat{\ZZ})|\gamma\equiv  \left( \begin{array}{cc}
	\ast&\ast\\
	0&1\end{array}\right)\mod \gothn \right \}.
\end{equation*}
When $\gothn$ is prime to $\gothd$, we set 
\begin{equation*}
U^D_1(\gothn):=\left \{ \gamma\in(\calO_D\otimes\hat{\ZZ})^\times|\gamma\equiv  \left( \begin{array}{cc}
\ast&\ast\\
0&1\end{array}\right)\mod \gothn \right \}.
\end{equation*}

Let $\gothD_1=(\calW,\calZ_D,\calM_D,\bbT,\psi_D)$ be the eigenvariety datum associated to the spaces of overconvergent automorphic forms of tame level $U^D_1(\gothn)$ as constructed in \S\ref{subsection:the spectral varieties and eigenvarieties}. We use $\calX_D(U_1(\gothn))$ to denote the corresponding eigenvariety. Let $\gothD_2=(\calW,\calZ,\calM,\bbT,\psi)$ be the eigenvariety datum associated to the spaces of overconvergent cuspidal Hilbert modular forms of tame level $U_1(\gothn)$ as constructed in \cite{andreatta2016arithmetique} $\S5$ and let $\calX_{\GL_{2/F}}(U_1(\gothn))$ be the corresponding eigenvariety.  

We define $\calZ_D^{cl}\subset \calZ_D(\CC_p)\subset (\calW\times \AAA^1)(\CC_p)$ to be the set of points $z=(\chi,\alpha^{-1})$ consisting of a classical weight $\chi\in \calW(\CC_p)$ corresponding to $(v=(v_i)\in \ZZ^I,r\in \ZZ)$ and $\alpha\in \CC_p^\ast$ with $v_p(\alpha)<\min\limits_{i\in I}\{n_i+1 \}$. Then it follows from the proof of \cite[Proposition~$3.5$]{chenevier2005correspondance} or the proof of \cite[Proposition~$6.7$]{birkbeck2016jacquet}  that the set $\calZ_D^{cl}$ is very Zariski dense in $\calZ_D$. From the classicality theorem for overconvergent automorphic forms for $D$ (\cite[Theorem~2.3]{yamagami2007p}) and for overconvergent Hilbert modular forms (\cite[Theorem~1]{tian2016arithmetique}), we have
\[
\calM_D(z)=S_{k,w}^D(U_1^D(\gothn\pi))^{U_{\pi}=\alpha}, \text{~and~} \calM(z)=S_{k,w}(U_1(\gothn\gothd\pi))^{U_\pi=\alpha}. 
\]
Now we can apply Theorem~\ref{T:p-adic interpolation theorem} to the eigenvariety data $\gothD_1$ and $\gothD_2$ together with the additional data $\mathrm{id}:\calW\rightarrow \calW$ and $\mathrm{id}:\bbT\rightarrow \bbT$,  and get the following theorem.
\begin{theorem}\label{T:p-adic Jacquet-Langlands correspondence}
	(\cite[Theorem~$1$]{birkbeck2016jacquet}) There is a closed immersion $i_D:\calX_D(\gothn)\rightarrow \calX_{\GL_{2/F}}(\gothn\gothd)$ interpolating the Jacquet-Langlands correspondence on non-critical classical points. Moreover, when $g=[F:\QQ]$ is even, one can choose $D$ with $\gothd=1$ so that $i_D$ is an isomorphism. 
\end{theorem}
When $g=[F:\QQ]$ is even, we can choose $D$ with $\gothd=1$ and use the above theorem to translate our main theorem \ref{T:spectral halo for eigenvarieties for D} to the Hilbert modular eigenvarieties. When $g$ is odd, the discriminant $\gothd$ cannot be $1$ and the immersion $i_D$ in Theorem~\ref{T:p-adic Jacquet-Langlands correspondence} is not surjective. We need more work to get the desired results for Hilbert modular eigenvarieties.

We choose a quadratic extension $F'/F$, such that $F'$ is totally real and $p$ splits completely in $F'$. Let $\gothD_3=(\calW',\calZ',\calM',\bbT',\psi')$ be the eigenvariety datum associated to the spaces of overconvergent Hilbert modular forms of tame level $U'_1(\gothn')$, where $\gothn'=\gothn\calO_{F'}$ and 
\begin{equation*}
U'_1(\gothn')=\left \{ \gamma\in\GL_2(\calO_{G'}\otimes\hat{\ZZ})|\gamma\equiv  \left( \begin{array}{cc}
\ast&\ast\\
0&1\end{array}\right)\mod \gothn' \right \}.
\end{equation*}

Let $\calO_p':=\calO_{F'}\otimes \ZZ_p$. The norm map $\Nm_{F'/F}:\calO_p'^\times\rightarrow \calO_p^\times$ induces a continuous homomorphism of the completed group rings: $\phi:\ZZ_p\llbracket \calO_p'^\times\times \ZZ_p^\times \rrbracket\rightarrow \ZZ_p\llbracket \calO_p^\times\times \ZZ_p^\times \rrbracket$, and hence gives a closed immersion $\jmath:\calW\rightarrow \calW'$. Define a homomorphism $\sigma:\bbT'\rightarrow \bbT$ of $\QQ_p$-algebras as follows: for a prime $\gothl$ of $F$, define
\begin{enumerate}
	\item $\sigma(T_{\gothl_1})=\sigma(T_{\gothl_2})=T_\gothl$, if $\gothl$ splits as $\gothl\calO_{F'}=\gothl_1\cdot\gothl_2$ in $F'$;
	\item $\sigma(T_{\gothl'})=T_\gothl^2-2lS_\gothl$, if $\gothl$ is inert in $\calO_{F'}$ and set $\gothl'=\gothl\calO_{F'}$;
	\item $\sigma(T_{\gothl'})=T_\gothl$, if $\gothl$ is ramified in $F'$, and let $\gothl'$ be the unique prime of $F'$ over $\gothl$. 
\end{enumerate}

Applying Theorem~\ref{T:p-adic interpolation theorem} to the above datum, we have the following theorem.
\begin{theorem}\label{T:p-adic base change}
	There is a morphism $i_{F'/F}:\calX_{\GL_{2/F}}(\gothn)\rightarrow \calX_{\GL_{2/F'}}(\gothn')$ interpolating the quadratic base change from $F$ to $F'$ on non-critical classical points.
\end{theorem}

To apply the above theorem to Hilbert modular eigenvarieties, we need to make some explicit computations. Let $I'=\Hom(F,\bar{\QQ})=\Hom(F,\bar{\QQ}_p)$. We identify the set $I$ (resp. $I'$) with the set of primes of $F$ (resp. $F'$) over $p$. Since $p$ splits completely in $F'$, every prime $i$ in $I$ splits into two primes in $I'$, and we use $i_1'$ and $i_2'$ to denote the two primes. Under the natural map $\calO_{\gothp_i}\rightarrow \calO'_{\gothp_{i'_j}}$, the image $\pi_{i'_j}$ of $\pi_i$ is a uniformizer of $\calO'_{\gothp_{i'_j}}$, for $j=1,2$. As in \S\ref{S:notations in filtration on the space of p-adic automorphic forms}, we can use these $\pi_{i'_j}$'s for all $i\in I$ and $j=1,2$ to define a full set $\{(T_{i'_j})_{i\in I,j=1,2}, T \}$ of parameters on the weight space $\calW'$. Under these notations, the closed immersion $\jmath:\calW\rightarrow \calW'$ defined above can be described in term of parameters: if $x\in \calW(\CC_p)$ has parameters $((w_i)_{i\in I},w_0)$, and $\jmath(x)\in\calW'(\CC_p)$ has parameters $((w_{i'})_{i'\in I'},w_p')$, then $w_{i_1'}=w_{i_2'}=w_i$ for all $i\in I$ and $w_0=w_0'$. In particular, for any $r=(r_i)_{i\in I}\in \calN^I$, if we define $r'=(r_{i'})_{i'\in I'}\in\calN^{I'}$ by $r_{i_1'}=r_{i_2'}=r_i$ for all $i\in I$, then $\jmath^{-1}(\calW'^{>r'})=\calW^{>r}$. 

For any closed point $x'$ of $\calX_{\GL_{2/F'}}(\gothn')$, recall that we use $a_{i'}(x)$ to denote its corresponding $U_{\pi_{i'}}$-eigenvalue for all $i'\in I'$. From the explicit description of the homomorphism $\sigma:\bbT'\rightarrow \bbT$, we have $a_{i_1'}(i_{F'/F}(x))=a_{i_2'}(i_{F'/F}(x))=a_i(x)$ for any closed point $x$ of $\calX_{\GL_{2/F}}(\gothn)$. Combining Theorems~\ref{T:spectral halo for eigenvarieties for D}, ~\ref{T:p-adic Jacquet-Langlands correspondence} and ~\ref{T:p-adic base change}, we have the following description of the boundary behavior of Hilbert modular eigenvarieties:

\begin{theorem}\label{T:spectral halo for Hilbert modular eigenvarieties}
	We use $\Sigma$ to denote the subset $\{0 \}\bigcup \{1+2k|k\in \ZZ_{\geq 0} \}$ of $\ZZ$. The eigenvariety $\calX_{\GL_{2/F}}(\gothn)^{>1/p}$ is a disjoint union 
	\[
	\calX_{\GL_{2/F}}(\gothn)^{>1/p}=\bigsqcup_{l\in \Sigma^I, \sigma\in \{\pm\}^I}\calX_{l,\sigma}
	\]
	of (possibly empty) rigid analytic spaces which are finite over $\calW^{>1/p}$ via $w$, such that for each closed point $x\in \calX_{l,\sigma}(\CC_p)$ with $l=(l_i)_{i\in I}\in \Sigma^I$ and $\sigma=(\sigma_i)_{i\in I}\in\{\pm \}^I$, we have
	\[\begin{cases}
	v_p(a_i(x)) = (p-1)v_p(T_{i,w(x)})\cdot l_i,&\textrm{if}\ \sigma_i=-;\\
	v_p(a_i(x)) \in (p-1)v_p(T_{i,w(x)})\cdot (l_i,l_i+2),&\textrm{if}\ \sigma_i=+ \textrm{~and~} l_i\neq 0;\\
	v_p(a_i(x)) \in (p-1)v_p(T_{i,w(x)})\cdot (0,1),&\textrm{if}\ \sigma_i=+ \textrm{~and~} l_i=0,
	\end{cases}\]
	for all $i\in I$. 
\end{theorem}

\bibliographystyle{plain}
\bibliography{bibliography.bib}
\end{CJK}
\end{document}